\documentclass[11pt]{amsart}

\usepackage{amsfonts}
\usepackage{amssymb}
\usepackage{hyperref}
\usepackage{stmaryrd}

\usepackage{amscd}

\usepackage{tikz}
\usetikzlibrary{matrix,arrows,decorations.pathmorphing}

\usepackage{hyperref}
\usepackage[capitalise]{cleveref}
     
\theoremstyle{plain}
\newtheorem{thm}{Theorem}[section]
\newtheorem{lem}[thm]{Lemma}
\newtheorem{cor}[thm]{Corollary}
\newtheorem{prop}[thm]{Proposition}
\newtheorem{conj}[thm]{Conjecture}
 
\theoremstyle{definition}
\newtheorem{defi}[thm]{Definition}
      
\theoremstyle{remark}
\newtheorem{remark}[thm]{Remark}
\newtheorem{notation}[thm]{Notation}
\newtheorem{example}[thm]{Example}

\newcommand{\Ker}{\mathrm{Ker}}
\newcommand{\Gal}{\mathrm{Gal}}

\def \F {\mathbb{F}}

\def \Z {\mathbb{Z}}

\def \lcm {\mathrm{lcm}}
\def \tr {\operatorname{tr}}
\def \Frob {\operatorname{Frob}}

\def \rank {{\mathbf r}}
\def \swan {\mathbf {sw} }
\def \cond {{\mathbf c}}
\def \slope {\operatorname{slope}}
\def \dro {\mathbf{d}}
\def \rad {\operatorname{rad}}

\makeatletter
\newcommand*{\defeq}{\mathrel{\rlap{%
                     \raisebox{0.27ex}{$\m@th\cdot$}}%
                     \raisebox{-0.27ex}{$\m@th\cdot$}}%
                     =}

\newcommand*{\eqdef}{=\mathrel{\rlap{%
                     \raisebox{0.27ex}{$\m@th\cdot$}}%
                     \raisebox{-0.27ex}{$\m@th\cdot$}}%
                     }

\numberwithin{equation}{section}
\makeatother

\makeatletter
\def\@setcopyright{}
\def\serieslogo@{}
\makeatother

\input xy
\xyoption{all}

\title[M\"obius on polynomial sequences and quadratic Bateman-Horn]{M\"{o}bius cancellation on polynomial sequences and the quadratic Bateman-Horn conjecture over function fields}

\author{Will Sawin}\thanks{W.S. served as a Clay Research Fellow while working on this paper.}

\author{Mark Shusterman}

\address{Department of Mathematics, Columbia University, New York, NY 10027, USA}
\email{sawin@math.columbia.edu}

\address{Department of Mathematics, Harvard University, 1 Oxford Street, Cambridge, MA 02138, USA}
\email{mshusterman@math.harvard.edu}

\begin{document}

\begin{abstract}

We establish cancellation in short sums of certain special trace functions over $\F_q[u]$ below the P\'{o}lya-Vinogradov range,
with savings approaching square-root cancellation as $q$ grows.
This is used to resolve the $\F_q[u]$-analog of Chowla's conjecture on cancellation in M\"{o}bius sums over polynomial sequences, 
and of the Bateman-Horn conjecture in degree $2$, for some values of $q$.
A final application is to sums of trace functions over primes in $\F_q[u]$.

\end{abstract}

\maketitle

\tableofcontents

\section{Introduction}

%

\subsection{Quadratic Bateman-Horn}



The history of interest in prime values of integral polynomials dates back at least to Euler,
with early conjectural contributions also by Bunyakovsky, Landau, and Schinzel.
Quantifying the existing qualitative predictions, 
Bateman and Horn conjectured that for every irreducible monic polynomial $F(T) \in \mathbb{Z}[T]$,
we have
\begin{equation} \label{BHconjectureEq}
\#\{X \leq n \leq 2X : F(n) \ \text{is prime} \} \sim \mathfrak{S}(F) \cdot \frac{X}{\log X}
\end{equation}
where 
\begin{equation}
\mathfrak{S}(F) = \frac{1}{\deg(F)} \prod_p \frac{1 - \frac{1}{p} \#\{x \in \mathbb{Z}/p\mathbb{Z} : F(x) \equiv 0 \ \mathrm{mod} \ p\} }{1 - \frac{1}{p}}. 
\end{equation}

Even though the only completely resolved case is $\deg(F) = 1$, which is the prime number theorem,
significant progress on this conjecture has been made in other cases as well.
For example, it was shown by Iwaniec in \cite{Iw} that there are $\gg X/\log X$ integers $n \in [X, 2X]$ for which $n^2 + 1$ is a product of at most two primes.
For an exposition of the proof of Iwaniec, a generalization to other quadratic polynomials,
and a discussion of related results with $\deg(F) > 2$, we refer to \cite{LO}.

Building and improving on a succession of previous works, 
Merikoski has shown in \cite{Mer} that there are infinitely many integers $n$ with $n^2 + 1$ having a prime factor exceeding $n^{1.279}$ (or exceeding $n^{1.312}$ if Selberg's eigenvalue conjecture is assumed).
Results in this vein have also been obtained in case $\deg(F) > 2$, see for instance \cite{dlB} and references therein.

Among results on multivariate analogs of the Bateman-Horn conjecture, 
we would like to mention the work \cite{FI} of Friedlander-Iwaniec obtaining an asymptotic for the number of primes of the form $n^2+m^4$,
the paper \cite{HM} by Heath-Brown--Moroz on counting primes represented by bivariate cubic polynomials,
and the article \cite{May} of Maynard on incomplete norm forms.
We also refer to \cite{Yau, BR} and their references for results on the Bateman-Horn conjecture `on average over the polynomial $F$'.

Here we are concerned with the function field analog of the Bateman-Horn conjecture.
We fix throughout an odd prime number $p$ and a power $q$ of $p$. We denote by $\F_q$ the field with $q$ elements.
In our function field analogy, the ring $\mathbb{Z}$ is replaced by the univariate polynomial ring $\F_q[u]$.
Throughout this work, we use $\pi$ to denote a prime (monic irreducible) polynomial in $\F_q[u]$.
We define the norm of a nonzero polynomial $f \in \mathbb F_q[u]$  to be
\begin{equation}
|f| = q^{\deg(f)} = |\mathbb F_q[u]/(f)|,
\end{equation}
where $\deg(f) = \deg_u(f)$ is the degree of $f$, and $(f)$ is the ideal of $\F_q[u]$ generated by $f$.
The degree of the zero polynomial is negative $\infty$, so we set its norm to be $0$. 

\begin{conj} \label{BHFFConj}

Let $F(T) \in \F_q[u][T]$ be an irreducible separable monic polynomial with coefficients in $\F_q[u]$.
Then we have
\begin{equation}
\#\{g \in \F_q[u] : |g| = X, \ g \ \text{is monic,} \ F(g) \ \text{is prime} \} \sim \mathfrak{S}_q(F) \cdot \frac{X}{\log_q X}
\end{equation}
as $X \to \infty$ through powers of $q$, and 
\begin{equation}
\mathfrak{S}_q(F) = \frac{1}{\deg_T(F)} \prod_{\pi} \frac{1 - \frac{1}{|\pi|} \#\{x \in \F_q[u]/(\pi) : F(x) \equiv 0 \ \mathrm{mod} \ \pi\} }{1 - \frac{1}{|\pi|}}. 
\end{equation}

\end{conj}

A polynomial $F$ in the variable $T$ with coefficients from $\mathbb F_q[u]$ is \emph{separable} if it is squarefree over an algebraic closure $\overline{\F_q(u)}$ of $\F_q(u)$.
For an irreducible polynomial $F(T) \in \F_q[u][T]$ to be separable, it is necessary and sufficient that $F$ is not a polynomial in $T^p$. 

\cref{BHFFConj} is a fairly straightforward adaptation of the Bateman-Horn conjecture to function fields,
excluding inseparable polynomials over $\F_q[u]$, a family of polynomials that does not have a counterpart over $\Z$.
The importance of singling out the inseparable case, which we do not study here, was first highlighted in the works of Conrad-Conrad-Gross who also put forth a version of \cref{BHFFConj} for this case in \cite[Conjecture 6.2]{CCG}.

Apart from discussing the prior translation of existing results on the Bateman-Horn conjecture from $\mathbb{Z}$ to $\F_q[u]$, 
see \cite[Introduction]{Pol},
Pollack shows that for certain polynomials $F$ in \cref{BHFFConj} that do not depend on the variable $u$ (namely $F \in \F_q[T]$), 
there exist infinitely many monic $g \in \F_q[u]$ for which $F(g)$ is prime.
The polynomials $g$ that Pollack substitutes into $F$ are all monomials, 
so his method does not provide a lower bound that is comparable to the one in \cref{BHFFConj}. 

The main result of this work is the resolution of the function field quadratic Bateman-Horn conjecture over certain finite fields.

\begin{thm} \label{MainRes}

Let  $p$ be an odd prime number, and let $q$ be a power of $p$ with $q > 2^{10} 3^2 e^2 p^4$. 
Then \cref{BHFFConj} holds in case $\deg_T(F) = 2$.

\end{thm}

We obtain the asymptotic in \cref{BHFFConj} with a power saving error term.
For somewhat larger (fixed) values of $q$, this power saving is inversely proportional to $p^2$.
We also have uniformity in the quadratic polynomial $F$, 
allowing the norm of its coefficients to grow almost as fast as $X^2$ when $q$ is large, see \cref{UniformBHstatement} for a more detailed statement.

Bateman and Horn also made a conjecture for the `reducible' or `split' case,
predicting simultaneous primality of the values of several irreducible polynomials,
which in the case of linear polynomials specializes to the Hardy-Littlewood conjecture.
For some results in the direction of that conjecture see our previous work \cite{SS} (and references therein) on which this paper builds.
In particular \cref{MainRes} is the nonsplit analog of the twin prime number theorem \cite[Theorem 1.1]{SS}, 
obtained therein under the assumption $q > 685090p^2$.
The values of $q$ satisfying \cref{MainRes} are somewhat smaller than those in \cite[Theorem 1.1]{SS} for some very small primes $p$,
but are otherwise larger.
This is due to a new kind of difficulty appearing in one of the ranges in the proof of \cref{MainRes}, as will be explained later.

One of the difficulties in making progress on the Bateman-Horn conjecture is the parity barrier, or in other words,
producing many integers $n$ with $F(n)$ having an odd number of prime factors. 
This is implicit for example in the aforementioned work \cite{Iw} whose strategy is sieve-theoretic.
We shall now elaborate on this problem and on our resolution of a function field analog.

\subsection{Chowla's conjecture on polynomial sequences}

In \cite[Eq. (341)]{Ch} Chowla conjectured that for every (monic) squarefree polynomial $F \in \mathbb{Z}[T]$ one should have
\begin{equation}
\sum_{n \leq X} \mu(F(n)) = o(X).
\end{equation}
As in the Bateman-Horn conjecture, the only resolved case is the linear one.
For progress with multivariate polynomials $F$, we refer to \cite{Hel, FH, La} and references therein. 
Notable progress has also been made in case $F$ splits as a product of linear factors, see \cite[Introduction]{MRT}.

Recall that the M\"{o}bius function of a polynomial $f \in \F_q[u]$ is $0$ if $f$ is divisible by a square of a nonconstant polynomial,
and is otherwise given by $(-1)^r$ where $r$ is the number of prime factors of $f$.

\begin{thm} \label{NewRes}

Fix an odd prime number $p$, an integer $k \geq 1$, and a power $q$ of $p$ satisfying $q > 4 e^2k^2p^2  $. Let $F(T) \in \F_q[u][T]$ be a separable polynomial of degree $k$ in $T$.
Then
\begin{equation} \label{ResEq}
\sum_{\substack{f \in \mathbb F_q[u] \\ |f| \leq X}} \mu \left( F \left(f \right) \right) = o(X), \quad X \to \infty.
\end{equation}
\end{thm}

The result builds on and complements \cite{CCG} which deals with certain squarefree inseperable polynomials $F$,
for which \cref{ResEq} is shown not to hold. 

In fact, we obtain \cref{ResEq} with a power saving.
This saving approaches $\frac{1}{2p}$ for fixed $p$ and growing $q$.
Moreover we can take the coefficients of $F$ to be as large as any fixed power of $X$, 
by allowing some increase in $q$.
An effective error term and wide uniformity in $F$ are crucial (but not quite sufficient on their own) in our approach to establishing \cref{MainRes}.
We could likely obtain a similar cancellation in case the sum in \cref{ResEq} is restricted to prime polynomials $f \in \F_q[u]$, following \cite[Corollary 6.1]{SS}.


An analog of \cref{BHFFConj} and \cref{NewRes}, not considered in this work, is to fix $X$ and let $q \to \infty$ (thus allowing $F$ to change as well).
Refining many previous works, Entin in \cite{Ent16, Ent18} and then Kowalski in \cite{Kow} resolved the `large finite field' variants of \cref{BHFFConj} and Chowla's conjecture on polynomial sequences,
obtaining an error term of size $O(q^{-\frac{1}{2}})$ with the implied constant depending on $X$.
It is plausible that our arguments can be used to obtain superior error terms for certain special cases of these works.

Our proof of \cref{NewRes} also builds on and refines arguments from the proof of \cite[Theorem 1.3]{SS} where $F$ is assumed to be a product of (distinct) linear factors. 
The power savings and uniformity in $F$ obtained here are similar to those in \cite{SS}.
What follows is an overview of our proof of \cref{NewRes}, which leads to the technical heart of our work - cancellation in short sums of trace functions.   

We start, as in \cite{SS}, by restricting in \cref{SameDerivativeEq} to subsums over polynomials $f \in \F_q[u]$ sharing the same derivative,
and applying Pellet's formula from \cref{PelletEq} to write the value of M\"{o}bius in \cref{MobConvertedToResEq} as a (quadratic) character of the resultant of the values at $f$ of a pair of bivariate polynomials closely related to $F$ and the aforementioned derivative.
This allows us to restate a good deal of the arithmetic problem in terms of the geometry of the two plane curves given by the vanishing loci of our pair of bivariate polynomials, a strategy succesfully employed in previous works on factorization statistics of polynomials over finite fields by Conrad-Conrad-Gross, Entin, Kowalski, and others. 

Adapting a result from \cite{CCG}, we obtain in \cref{ccg-like} an expression for the above resultant in terms of the intersection numbers of our curves.
We can then write in \cref{SaSh-like} a character of our resultant as a Jacobi symbol. 
Our problem becomes that of obtaining cancellation for very short sums in the \'{e}tale $\F_q$-algebra $\F_q[u]/(M)$ of Jacobi symbols of the form 
\begin{equation} \label{ShortCharSumEq}
\sum_s \left( \frac{W(s)}{M}\right)
\end{equation}
with $W(s)$ a polynomial in $s$ with coefficients in $\F_q[u]/(M)$, and $M \in \F_q[u]$ a squarefree polynomial.

The problem of cancellation in short multiplicative character sums with $W$ linear in $s$ has been addressed in \cite[Theorem 1.4]{SS}, going below the Burgess range. 
The vanishing cycles argument used in the proof of that theorem,
reducing the problem to bounds of Weil and Deligne, turns out to be insufficient for controlling \cref{ShortCharSumEq}
in part due to the lack of multiplicativity in $s$ for a nonlinear polynomial $W$.
Indeed, obtaining significant cancellation in \cref{ShortCharSumEq} for general $W$ remains out of our reach.
We refer to \cite[Section 4, 4.3]{Saw20} for a further discussion of vanishing cycles in this context.

Sums as in \cref{ShortCharSumEq} have been studied, over the integers, in several works of Burgess such as \cite{Bur}, 
and for multivariate integral polynomials $W$ in \cite{MC}.
Burgess works with prime $M$, and obtains stronger results under the assumption that $W$ has a linear factor or even splits completely.

Although the arguments of Burgess are probably not directly applicable to getting cancellation in sums as short as ours,
along analogous lines we are able, after making a linear change of variable in the original polynomial $F$,
to show that the vast majority of our fixed derivative subsums give rise to short character sums with a prime factor of $M$ mod which $W$ is a power of a linear polynomial. 
This involves an application of a quantitative form of Hilbert's irreducibility theorem due to S. D. Cohen from \cite{Coh} as adapted to a function field setting by Bary-Soroker and Entin in \cite{BSE}.

The most novel part of our work is in establishing cancellation in sums satisfying this assumption on $W$ and $M$,
and more general short sums of certain special \emph{trace functions} that arise in our approach to proving \cref{MainRes}, 
which we now discuss.

\subsection{Strategy for proving the main result}

To prove \cref{MainRes}, we use a convolution identity expressing the indicator of primes in terms of the M\"{o}bius function.
Roughly speaking, this introduces three different ranges of summation.
In the first range we manipulate with Euler products and use classical bounds for $L$-functions to single out and calculate the singular series main term of \cref{MainRes}.
For the second range, a uniform version of \cref{NewRes} with a power saving cancellation is sufficient.
This part of our approach is similar to arguments from \cite{SS}, one difference is the need of a greater uniformity here. 

A more significant difference is that in \cite{SS} the third range did not present substantial difficulties, 
because a similar problem has already been handled by Fouvry-Michel over $\mathbb{Z}$.
Here however, in the third range we need (roughly speaking) to count (with good savings) the number of values of a quadratic polynomial having a prime factor of size somewhat larger than their square root.
This problem has not yet been resolved over $\mathbb{Z}$, and we refer to the aforementioned work \cite{Mer} for upper bounds and a discussion of the possibility of further progress.

In our solution of the problem over $\F_q[u]$, 
we first follow a strategy similar to some parts of \cite{Mer},
applying Poisson summation, completion, and the theory of binary quadratic forms.
This approach has its roots in the work \cite{Hoo} of Hooley.
Due to the lack of an appropriate reference, and our desire to obtain \cref{MainRes} with significant uniformity,
we develop for that matter the necessary parts of binary quadratic form theory over function fields.
This allows us to reduce the problem in the third range to a version of \cref{NewRes} twisted by a Kloosterman fraction,
see \cref{MobiusVStraceThm} for a more general twisted sum.

Our approach to proving \cref{NewRes} is also helpful for its twisted variants, leading again to short sums of trace functions.
One difference is that in the twisted case the modulus of the resulting exponential sum is not squarefree,
so we use a simple sieve in \cref{MobiusBeatsKloostermanCor} to reduce to squarefree moduli.

\subsection{Trace functions}

In various works, Fouvry, Kowalski, and Michel highlighted the importance of \emph{trace functions} to number theory over the integers, see for instance \cite{FKMS}.
These are functions on the integers modulo a prime $p$, equivalently, functions on the integers that are periodic with period $p$, that arise from the trace of Frobenius on an $\ell$-adic sheaf on the affine line over $\mathbb F_p$. 
Examples include multiplicative characters, additive characters, compositions of multiplicative characters or additive characters with rational functions, Kloosterman sums such as
\begin{equation} \label{KloostermanExampleEq}
t(x) = \frac{1}{\sqrt{p}} \sum_{\substack{a,b \in \F_p^\times \\ ab = x}} e \left(\frac{a + b}{p}\right),
\quad e(y) = e^{2 \pi i y},
\end{equation} 
compositions of Kloosterman sums with rational functions, and products or sums of any of these functions. 
Despite this vast generality, it is possible to obtain nontrivial results for all (or almost all) trace functions. 

More generally, as in \cite{WX}, one can work with periodic functions with squarefree period, which are products of trace functions modulo distinct primes. These behave similarly to trace functions, although most results have not yet been proven at this level of generality.

We define trace functions over $\mathbb F_q[u]$ in an analogous way, 
as functions on $\mathbb F_q[u]/(\pi)$ for a prime $\pi \in \F_q[u]$ arising from sheaves on $\mathbb A^1_{\mathbb F_q[u]/(\pi)}$, or products of these for distinct primes $\pi$.  

\begin{defi} \label{SheafFuncDictionaryDef}

Fix throughout an auxiliary prime number $\ell$ different from $p$ and an embedding $\iota \colon \overline{\mathbb Q_\ell} \hookrightarrow \mathbb C$.
We work with the abelian category of constructible $\overline{\mathbb{Q}_\ell}$-sheaves on a variety in characteristic $p$, see \cite[Part 2, Section 8]{KR},
and call its object simply `sheaves'.
Let $\pi \in \mathbb F_q[u]$ be a prime, and let $\mathcal{F}$ be a sheaf on $\mathbb A^1_{\mathbb F_q[u]/(\pi)}$. 
We can think of any $x \in \F_q[u]/(\pi)$ as a point on $\mathbb A^1_{\mathbb F_q[u]/(\pi)}$ and thus as a map 
\begin{equation}
x \colon \mathrm{Spec}(\F_q[u]/(\pi)) \to \mathbb A^1_{\mathbb F_q[u]/(\pi)}.
\end{equation}

For a geometric point $\overline{x}$ over $x$,
the stalk $\mathcal{F}_{\overline{x}}$ of $\mathcal{F}$ at $\overline{x}$ is the underlying finite-dimensional vector space over $\overline{\mathbb{Q}_\ell}$ of the pullback $x^* \mathcal{F}$ of $\mathcal{F}$ to $\mathrm{Spec}(\F_q[u]/(\pi))$.
This vector space is equipped with a linear action of $\mathrm{Frob}_{q^{\deg(\pi)}}$, so we can define
\begin{equation} \label{DefTraceFuncFrobEq}
t \colon \mathbb F_q[u]/(\pi) \to \mathbb C, \quad t(x) = \iota (\tr(\mathrm{Frob}_{q^{\deg(\pi)}}, \mathcal{F}_{\overline{x}}))
\end{equation}
independently of the choice of $\overline{x}$.

A function $t$ as above is called a \emph{trace function}, 
and is sometimes denoted by $t_{\mathcal{F}}$ in order to emphasize that $t$ arises from $\mathcal{F}$ via \cref{DefTraceFuncFrobEq}.
It is at times convenient to think of $t$ as a $\pi$-periodic function on $\F_q[u]$.
In the sequel, abusing notation we drop $\iota$ from our formulas. 

Note that the construction above suggests an extension of the function $t$ to any finite field extension $\kappa$ of $\F_q[u]/(\pi)$,
by considering the action of $\mathrm{Frob}_{|\kappa|}$ on $\mathcal{F}_{\overline{x}}$ for every $\kappa$-valued point
$x$ of $\mathbb{A}^1_{\F_q[u]/(\pi)}$.
We say that $\mathcal{F}$ is punctually pure of weight $w \in \mathbb{R}$ if for every $\kappa$-valued point $x$ of $\mathbb{A}^1_{\F_q[u]/(\pi)}$, all the eigenvalues of $\mathrm{Frob}_{|\kappa|}$ on $\mathcal{F}_{\overline{x}}$ are of absolute value $|\kappa|^{\frac{w}{2}}$. The sheaf $\mathcal{F}$ is said to be mixed of nonpositive weights if there exist a nonnegative integer $r$, 
nonpositive real numbers $w_1, \dots, w_r$, and a filtration of $\mathcal F$ by subsheaves
\begin{equation}
0 = \mathcal{F}^{(0)} \subseteq \mathcal{F}^{(1)} \subseteq \dots \subseteq \mathcal{F}^{(r)} = \mathcal{F}
\end{equation} 
such that the sheaf $\mathcal{F}^{(i)}/\mathcal{F}^{(i-1)}$ is punctually pure of weight $w_i$ for every $1 \leq i \leq r$.

Set $\kappa = \F_q[u]/(\pi)$, let $\eta$ be a generic point of $\mathbb{A}^1_\kappa$, and let 
\begin{equation}
j \colon \mathrm{Spec}(\kappa(X)) \to \mathrm{Spec}(\kappa[X]) = \mathbb{A}^1_{\kappa}
\end{equation}
be the map arising from the inclusion of $\kappa[X] \hookrightarrow \kappa(X)$.
Then $j^*\mathcal{F}$ equips the stalk $\mathcal{F}_{\overline{\eta}}$ with the structure of a continuous finite-dimensional representation of $\Gal(\kappa(X)^{\text{sep}}/\kappa(X))$ over $\overline{\mathbb{Q}_\ell}$.
We call $\dim_{\overline{\mathbb{Q}_\ell}} \mathcal{F}_{\overline{\eta}}$ the (generic) rank of $\mathcal F$, or the rank of $t$, and denote it by either $\rank(\mathcal F)$ or $r(t)$. 

Every closed point $x \in \mathbb{P}^1_\kappa$ defines a valuation on $\kappa(X)$,
which we can extend (non-uniquely) to a valuation $v_x$ on $\kappa(X)^{\text{sep}}$.
The closed subgroup
\begin{equation}
D_x = \{\sigma \in \Gal(\kappa(X)^{\text{sep}}/\kappa(X)) : v_x \circ \sigma= v_x\}
\end{equation}
fits into the exact sequence of profinite groups
\begin{equation}
1 \to I_x \to D_x \to \Gal(\overline{\kappa(x)}/\kappa(x)) \to 1.
\end{equation}

We call $I_x$ the inertia subgroup of $\Gal(\kappa(X)^{\text{sep}}/\kappa(X))$ at $x$, and note that it is well-defined up to conjugation. 
We let $P_x$ be a (unique) $p$-Sylow subgroup of $I_x$, and call it the wild inertia subgroup at $x$.
We say that $\mathcal{F}$ is unramified (respectively, tamely ramified) at $x$ if $I_x$ (respectively, $P_x$) acts trivially on 
$\mathcal{F}_{\overline{\eta}}$.
For $x \in \mathbb{P}^1_\kappa$, we denote by $\swan_x(\mathcal F)$ the swan conductor of $\mathcal F$ at $x$,
a nonnegative integer associated to the action of $P_x$ on $\mathcal{F}_{\overline{\eta}}$.
In particular, it is zero if and only if the action of $P_x$ is trivial.
For a thorough exposition of this notion see \cite[Section 4]{KR}.

We say that the trace function $t$ (or the sheaf $\mathcal{F}$) is \emph{infinitame} if $\mathcal F$ is tamely ramified at $\infty \in \mathbb{P}^1_\kappa$,
mixed of nonpositive weights, and has no finitely supported sections. 
The latter condition means that for every \'{e}tale map $e \colon U \to \mathbb{A}^1_{\kappa}$, and every section $s$ of $\mathcal{F}$ over $U$, the support of $s$
\begin{equation}
\mathrm{Supp}(s) = \{x \in U : s_{\overline{x}} \neq 0 \ \text{in} \ \mathcal{F}_{\overline{x}} \} 
\end{equation}
is infinite.
This is equivalent to the vanishing of the cohomology group $H^0_c(\mathbb{A}^1_{\overline{\kappa}}, \mathcal F)$, 
since any compactly supported global section can be decomposed into sections supported at individual points.

We define the conductor of an infinitame trace function $t$ (or of the sheaf giving rise to it) to be the nonnegative integer
\begin{equation} \label{InfinitameConductorDefEq}
c(t) = c(\mathcal{F}) = \sum_{ x \in |\mathbb A^1_{\kappa}|} [\kappa(x) : \kappa] (\rank(\mathcal F) - \dim(\mathcal F_{\overline{x}}) + \swan_x(\mathcal F) ) 
\end{equation} 
where the sum is taken over closed points, and the dimension is over $\overline{\mathbb{Q}_\ell}$.

\end{defi}

\begin{remark}

The assumption that $\mathcal{F}$ is mixed of nonpositive weights is merely a normalization condition capable of capturing all of the examples that are of interest. It implies that $|t(x)| \leq r(t)$ for every $x \in \F_q[u]/(\pi)$.
The technical assumption that $\mathcal F$ has no finitely supported sections guarantees that the conductor defined above has certain desirable properties.
This assumption could easily be removed since the finitely supported sections of a sheaf contribute to only finitely many values of the trace function, and these values can be handled separately for most purposes, but it would make the formulas involving the conductor more complicated.

On the other hand, the assumption that $\mathcal{F}$ is tamely ramified at infinity is a substantive restriction necessitated by our methods of proof,
and is (to some extent) suggested by the trace functions arising in the proofs of \cref{MainRes} and \cref{NewRes}.

\end{remark}

\begin{remark}

The definition of the conductor of $t$ almost matches the logarithm to base $|\kappa|$ of the (global) Artin conductor of the Galois representation $\mathcal F_{\overline{\eta}}$, defined as 
\begin{equation}
\prod_{x \in |\mathbb A^1_{\kappa}|} |\kappa(x)|^{ \rank(\mathcal F) - \dim(\mathcal F_{\overline{\eta}}^{I_x} ) + \swan_x(\mathcal F) }.
\end{equation} 
Note that there is a natural map $\mathcal F_{\overline{x}} \to \mathcal F_{\overline{\eta}}^{I_x}$ 
whose injectivity is equivalent to $\mathcal{F}$ having no sections supported at $x$.
Hence, if $\mathcal{F}$ has no finitely supported sections, all these maps are injections.
If $\mathcal{F}$ is moreover a middle extension sheaf, then these maps are isomorphisms.
Hence the conductor of $t$ is an adaptation of the Artin conductor to infinitame trace functions.

\end{remark}

\begin{example} \label{DirichletTraceFuncEx}
 
Let $\pi \in \F_q[u]$ be a prime, let
\begin{equation}
\chi \colon (\F_q[u]/(\pi))^\times \to \mathbb{C}^\times,
\end{equation} 
be a nonprincipal Dirichlet character, let $a \in (\F_q[u]/(\pi))^\times$ be a scalar, 
and let $b \in \F_q[u]/(\pi)$ be a shift.
After constructing the Kummer sheaf 
\begin{equation}
\mathcal{F} = \mathcal{L}_\chi(aT+b)
\end{equation}
on $\mathbb{A}^1_{\F_q[u]/(\pi)}$, we will see that the function 
\begin{equation} \label{DefDirTraceFuncEq}
t \colon \F_q[u]/(\pi) \to \mathbb{C}, \quad t(x) = 
\begin{cases}
\chi(ax + b) & x \neq -ba^{-1} \\
0 & x = -ba^{-1}
\end{cases}
\end{equation}
is an infinitame trace function with $r(t) = c(t) = 1$.
We call $t$ a \emph{Dirichlet trace function}.

\end{example}

\begin{defi} \label{TraceFuncDefi}


For a squarefree polynomial $g \in \F_q[u]$, we say that
\begin{equation}
t \colon \mathbb F_q[u]/(g) \to \mathbb{C}
\end{equation}
is a ($g$-periodic) \emph{trace function} if there exist trace functions 
\begin{equation}
t_\pi \colon \F_q[u]/(\pi) \to \mathbb{C}
\end{equation}
for each prime factor $\pi$ of $g$ such that
\begin{equation} \label{TraceFunctionsProdEq}
t(x) = \prod_{\pi \mid g} t_{\pi} (x \ \mathrm{mod} \ \pi), \quad x \in \F_q[u]/(g). 
\end{equation}
We say that $t$ is \emph{infinitame} if $t_\pi$ is for each $\pi \mid g$, and define 
\begin{equation}
r(t) = \max_{\pi \mid g}\{r(t_\pi)\}, \quad c(t) = \max_{\pi \mid g}\{c(t_\pi)\}.
\end{equation}
We will use the notation $\mathcal{F}_\pi$ for a sheaf giving rise to the trace function $t_\pi$ via \cref{DefTraceFuncFrobEq}.
This means that $\mathcal{F}_\pi$ is a sheaf with $t_{\mathcal{F}_\pi} = t_{\pi}$.

\end{defi}

The following trace functions show up in the proofs of \cref{MainRes} and \cref{NewRes}.



%
%

\begin{example} 

Let $g \in \F_q[u]$ be squarefree, let $\chi \colon (\mathbb F_q[u]/(g))^\times \to \mathbb C^\times$ be a multiplicative character,
and let $\psi \colon \mathbb F_q[u]/(g) \to \mathbb C^\times$ be an additive character. 
Let $a(T)$ be a nonconstant polynomial with coefficients in $\mathbb F_q[u]/(g)$, and define
\begin{equation*}
t \colon \F_q[u]/(g) \to \mathbb{C}, \quad
t(x) = 
\begin{cases} 
0 & a(x) \notin (\F_q[u]/(g))^\times \\ 
0 & x \notin (\F_q[u]/(g))^\times \\  
\chi (a(x)) \psi \left( \frac{1}{x} \right)  & \textrm{otherwise}. 
\end{cases}
\end{equation*}
The function $t$ is an infinitame trace function with
\begin{equation}
r(t) = 1, \quad c(t) \leq \deg(a) + 2.
\end{equation}

\end{example} 

The first problem about trace functions one usually studies is that of obtaining cancellation in the complete sum
\begin{equation}
\sum_{x \in \F_q[u]/(\pi)} t(x)
\end{equation}
for a trace function $t \colon \F_q[u]/(\pi) \to \mathbb{C}$. 
For infinitame trace functions this is carried out, using standard tools, in \cref{CompleteExponentialInfinitameSum}.

The following is our main result on trace functions,
a significant cancellation in very short sums of infinitame trace functions with a `Dirichlet component'.

\begin{thm} \label{intro-trace-interval-bound} 

Let $g \in \mathbb F_q[u]$ be a squarefree polynomial,
and let $t$ be an infinitame $g$-periodic trace function. 
Suppose that there exists a prime $\pi \mid g$ for which $t_\pi$ is a Dirichlet trace function.
Then
\begin{equation}\label{eq-csb-intro} 
\sum_{ \substack {f \in \mathbb F_q[u] \\ |f| < X}} t(f)   \ll      
X^{\frac{1}{2}}  |g|^{ \log_q ( 2r(t)+ c(t)) }, \quad X,|g| \to \infty 
\end{equation}
with the implied constant depending only on $q$.

\end{thm}

In applications, the quantities $r(t), c(t)$ are typically bounded, 
so for large (but fixed) $q$ we get arbitrarily close to square-root cancellation in intervals as short as $X = |g|^\epsilon$, for any fixed $\epsilon > 0$.
The reason for working with this kind of trace functions is that it seems to be the simplest family of functions to which we can reduce \cref{ShortCharSumEq} (and its twisted variants) under the additional assumption on $W$ and $M$ discussed earlier.
Indeed \cref{intro-trace-interval-bound} is a crucial input to our proofs of \cref{MainRes} and \cref{NewRes}. 
It would of course be desirable to treat trace functions of sheaves which are neither tamely ramified at infinity, 
nor necessarily related to Dirichlet characters.

A predecessor of \cref{intro-trace-interval-bound} is \cite[Theorem 2.1]{SS} proven under the assumption that $t_\pi$ is a Dirichlet trace function for every $\pi \mid g$, namely that $t$ is a shifted Dirichlet character.
The vanishing cycles argument used to prove that result produces comparable bounds, 
but its application beyond the (shifted) multiplicative scenario remains challenging. 

Over the integers, bounds for short sums of trace functions are in general not available beyond the P\'{o}lya-Vinogradov range $X \geq |g|^{\frac{1}{2}}$. 
We refer to \cite{FKMRRS} for recent developments in this direction.
For the function field version of the P\'{o}lya-Vinogradov argument see \cref{PolyaVinogradovLem}.


We now give some examples demonstrating that, even though the assumptions in \cref{intro-trace-interval-bound} are perhaps not strictly necessary,
some restrictions on the trace functions are required.

\begin{example}  

The constant function $t(x)=1$ for $x \in \F_q[u]/(\pi)$ is an infinitame trace function of rank $1$ and conductor $0$,
arising from the constant sheaves $\mathcal{F}_\pi = \overline{\mathbb{Q}_\ell}$.
This is not a Dirichlet trace function, and the conclusion of \cref{intro-trace-interval-bound} clearly fails in this case.

\end{example}

\begin{example} 

We use here exponentiation on $\F_q(u)$ as reviewed in \cref{REAC}.
Let $\tau, \pi \in \F_q[u]$ be distinct primes,
and let $\chi \colon (\mathbb F_q[u]/(\tau))^\times \to \mathbb C^\times$ be a nontrivial character. 
Take $\overline{\pi} \in \F_q[u]$ satisfying $\pi \overline{\pi} \equiv 1 \mod \tau$, and define the trace function
\begin{equation}
t \colon \F_q[u]/(\tau \pi) \to \mathbb{C}, \quad 
t(x) =  \chi(x) e \left(\frac{\pi \overline{\pi} - 1}{\tau} \cdot \frac{x}{\pi} \right).
\end{equation}
This trace function satisfies $r(t) = 1, c(t) = 1$, 
and all the assumptions of \cref{intro-trace-interval-bound} except that $\mathcal{F}_\pi$ is not tamely ramified at infinity.

For any $|\tau| \leq X \leq \frac{|\tau \pi|}{q} $ we use properties of the exponential function  compute
\begin{equation*}
\begin{split} 
\sum_{\substack{f \in \F_q[u] \\  |f|< X}} t(f) &= 
\sum_{\substack{f \in \F_q[u] \\  |f|< X}} \chi(f) e \left( \frac{\overline{\pi} f}{\tau} - \frac{f}{\tau \pi} \right) = 
\sum_{\substack{f \in \F_q[u] \\  |f|< X}} \chi(f) e \left( \frac{\overline{\pi} f}{\tau}\right) e \left( - \frac{f}{\tau \pi} \right) \\
&= \sum_{\substack{f \in \F_q[u] \\  |f|< X}} \chi(f) e \left( \frac{\overline{\pi} f}{\tau}\right)
= \frac{X}{|\tau|} \sum_{x \in \F_q[u]/(\tau)} \chi(x) e \left( \frac{\overline{\pi} x}{\tau}\right) .
\end{split} 
\end{equation*}
Since the Gauss sum appearing in the last formula has absolute value $|\tau|^{1/2}$, taking $|\tau|$ to be very small compared to $|\pi|$, 
we get barely any cancellation, so \cref{eq-csb-intro} does not hold. 

\end{example}

\subsubsection{The geometric strategy}

Our proof of \cref{intro-trace-interval-bound} relies on the theory of sheaves and trace functions on higher-dimensional varieties,
see \cite[11.11]{IK} for an exposition covering applications to analytic number theory. 
We view the set of polynomials $f \in \F_q[u]$ with $|f|<X$ as the $\mathbb F_q$-points of an affine space, with one coordinate for each coefficient of the polynomial. We then construct in \cref{trace-interval-prop} a sheaf $\overline{\mathcal F}$ on this space whose trace of Frobenius at each point is $t(f)$. Sheaves on higher-dimensional spaces are potentially much more complicated objects than the individual sheaves $\mathcal F_\pi$ (on $\mathbb{A}^1$) used to define $t$, 
but $\overline{\mathcal F}$ can be constructed from base changes of the $\mathcal F_\pi$ along $\F_q$-embeddings of $\F_{q^{\deg(\pi)}}$ into $\overline{\F_q}$,
as a tensor product of pullbacks along linear (evaluation) maps.
This tensor product decomposition is made possible by the factorizability into distinct linear factors over $\overline{\F_q}$ of the period $g$ of the trace function $t$.

This tensor product construction makes $\overline{\mathcal F}$ a lisse sheaf on $\mathbb{A}^n_{\overline{\F_q}}$ away from the inverse images of the finitely many points where each $\mathcal F_\pi$ fails to be lisse. In other words, $\overline{\mathcal F}$ is lisse away from an arrangement of hyperplanes.

The bound in \cref{intro-trace-interval-bound} follows from a strong cohomology vanishing result for $\overline{\mathcal F}$, which says that its cohomology with compact support vanishes in all degrees except for the middle degree and the next one, together with a bound for the dimensions of the potentially nonvanishing cohomology groups. These are proven by completely separate arguments.

The cohomology vanishing adapts a now-standard strategy to show vanishing for the cohomology of a sheaf on an affine variety by comparing its compactly supported cohomology, its usual cohomology, and the cohomology of a certain (derived) pushforward sheaf on the boundary of a well-chosen compactification of the affine variety. 
By Artin's affine theorem, the cohomology of any sheaf on an affine variety vanishes in high degrees, 
and by duality, the cohomology with compact support of a sufficiently nice sheaf on an affine variety vanishes in low degrees
(The sufficiently nice sheaves are called, perversely, ``perverse"). 
Thus, the more similar we can show the usual and compactly supported cohomologies are, the more vanishing we obtain, 
for both theories. The difference between the usual and compactly supported cohomology is controlled, unsurprisingly from the classical perspective, by the behavior ``near infinity" or, more productively in our setting, by the behavior near the boundary of any given compactification.

The affine space that $\overline{\mathcal F}$ lives on has a natural compactification, a projective space. This compactification is well-behaved but its boundary, the divisor at infinity, is unsuitable for our purposes. The \'{e}tale-local behavior near a point in that divisor depends in a subtle way on our individual sheaves $\mathcal F_\pi$, making it hard to compute the pushforward. Instead we make a change of perspective - in concrete terms, a projective change of coordinate system - where we view one of the hyperplanes where $\overline{\mathcal F}$ fails to be lisse as the boundary, and the projective space with this hyperplane removed as an affine variety. 
We must carefully choose the hyperplane in order to make the pushforward amenable to a local study. 
We choose a hyperplane arising as an inverse image of the singular (non-lisse) point $z$ of the Dirichlet trace function $\chi(x-z)$ that we assumed appears as a $t_\pi$ in \cref{intro-trace-interval-bound}.

The sheaf on $\mathbb{A}^1$ giving rise to the trace function $\chi(x-z)$ has a local monodromy representation around the point $z$ which is one-dimensional and nontrivial. It follows that the pushforward of this sheaf from the affine line with the point $z$ removed, to the whole line, vanishes at the point $z$. We want to use this vanishing to deduce that the pushforward of $\overline{\mathcal F}$ from the projective space minus our specially chosen hyperplane, to the whole projective space, vanishes at all but finitely many points of this hyperplane, i.e. is supported at those finitely many points. Using this general method, the number of cohomology groups that may be nonzero is equal to the dimension of the support of this pushforward plus two, so because we show the support of the pushforward is zero-dimensional, we can have nonzero cohomology only in two specific degrees.

To transfer the vanishing from the line to a higher-dimensional space we must find local coordinates near each point of our chosen hyperplane, except finitely many, in which the sheaf $\overline{\mathcal F}$ splits as a tensor product of our well-understood sheaf with trace function $\chi(x-z)$, depending on one coordinate $x$, and another sheaf, which depends on all the remaining coordinates, and may do so in an arbitrarily complicated way, but does not depend on $x$. This allows us to obtain the desired conclusion from the K\"unneth formula. 

One approach to the local tensor product decomposition would be to take one coordinate for each linear map which we pull back a sheaf on, but the number of linear maps is the degree of $g$, which is greater than the dimension $n = \log_q(X)$ of our variety, so this would be too many coordinates. Instead we must show that some of the sheaves are lisse (essentially, locally constant) and can be ignored in our (local) pushforward calculation. 
For points on our special hyperplane that do not lie on the original divisor at infinity, this requires controlling how many of the non-lisse hyperplanes can intersect at a point, which reduces to some simple algebra performed in \cref{affine-vanishing-lemma}. 
For points on both our chosen hyperplane and the original divisor at infinity, this doesn't quite work, as all the sheaves can have singularities at infinity. 
Instead, we use in \cref{projective-vanishing-lemma} our assumption that the local monodromy of the sheaves at infinity is tame, 
and employ properties of sheaves with tame ramification (ultimately, Abhyankar's Lemma) to separate variables locally.

Combined with an argument in \cref{perversity} to show that $\overline{\mathcal F}$ has the necessary perversity property, which requires a similar separation-of-variables argument but fewer explicit calculations, we have all the local properties needed to complete the global argument in \cref{geometric-vanishing},
which relies on a suitable form of the excision long exact sequence from \cref{excision} and properties of semiperverse sheaves.

The bound for dimensions of cohomology groups (that is, Betti numbers) follows a strategy loosely inspired by the Betti number bounds for cohomological transforms proved by Fouvry, Kowalski, and Michel in \cite{FKM13}. 
The basis of this strategy is to take as much advantage as possible of our understanding of Betti numbers of sheaves on curves. 
This understanding comes from the facts that all but one cohomology group of a sheaf on a curve has a simple global representation-theoretic description, 
and this remaining group can be controlled in terms of the Euler characteristic which can be expressed via local representation-theoretic information using the Grothendieck-Ogg-Shafarevich formula in \cref{EP} and \cref{EP2}.

At the heart of our strategy lies a procedure, introduced in the proof of \cref{betti-number-bound},
that replaces a sheaf $\mathcal{F}_\pi$ in the construction of $\overline{\mathcal{F}}$ with much simpler sheaves - skyscraper sheaves and Artin-Schreier sheaves, whose trace functions are indicators and additive characters.
We are able to bound the change in the sum of Betti numbers caused by such a replacement, in terms of the rank and conductor of $\mathcal{F}_\pi$.
Applying this procedure to $\mathcal{F}_\pi$ for each prime $\pi$ dividing $g$,
we eventually arrive at a sheaf cohomology problem that corresponds to a (possibly shorter) additive character sum.
Such sums can be evaluated explicitly, 
and indeed, we solve the corresponding sheaf cohomology problem by an explicit computation using \cref{linear-vanishing-lem}.

The aforementioned procedure starts by applying the projection formula which expresses the cohomology of the tensor product $\overline{\mathcal{F}}$ of the pullback of $n$ sheaves from $n$ curves as the cohomology of one sheaf (in our case, $\mathcal{F}_\pi$) on one of these curves (in our case, $\mathbb{A}^1$) tensored with the pushforward to that curve of the tensor product (of the pullbacks) of the remaining sheaves.
Our procedure then bounds in \cref{conductor-prod-lemma} the sum of Betti numbers for this tensor product in terms of the Betti numbers of the factors twisted by skyscraper and Artin-Schreier sheaves.  
This is done by calculating the tensor product sheaf cohomology on the curve in degrees $0$ and $2$ from the coinvariants of the global Galois representation associated to the sheaf,
applying the Grothendieck-Ogg-Shafarevich formula,
producing in \cref{SwanTensorProdBoundCor} an upper bound for the Swan conductor of a tensor product in terms of information available from the factors in the product,
and applying the Grothendieck-Ogg-Shafarevich formula once again in \cref{conductor-little-lemmas}(5).
The procedure culminates with invoking the projection formula as in the first step, 
and observing that the entire process is almost involutary in the sense that the final expression is reminiscent of the original one,
with the sheaf $\mathcal{F}_\pi$ replaced by simpler sheaves.

Using this argument, we are able to obtain Betti number bounds that are almost as strong as those obtained by \cite{SS} in a much more specialized situation, namely the one where $t_\pi$ is a Dirichlet trace function for \emph{every} prime $\pi$ dividing $g$.

Thanks to the power and generality of Deligne's Riemann hypothesis and theory of weights, 
the main difficulty left to convert these cohomology vanishing and Betti number bounds into a bound for the exponential sum is to verify that the trace function of the descent of the sheaf $\overline{\mathcal F}$ to $\mathbb{A}^n_{\F_q}$ agrees with the original function $t$. This involves relating the action of Frobenius on the stalk of $\overline{\mathcal F}$ at a point to the action of Frobenius on the stalks of the $\mathcal F_\pi$, the main subtlety of which is that these are not quite the same Frobenii unless all the primes $\pi$ dividing $g$ are of degree $1$. 
Nevertheless, the relation between the Frobenii is not too opaque, 
and gives a relation between their traces using a fact from linear algebra proven in \cref{LinearAlgebraLem}. Our approach is inspired by the construction of the tensor induction via descent in \cite{ARL}.

For the cohomology vanishing part of the argument, 
a strategy similar in some respects was taken by Cohen, Dimca, and Orlik in \cite{Nonresonance}. They gave a general method to show cohomology vanishing for sheaves on projective space, lisse away from an arrangement of hyperplanes in characteristic zero. 
We adopt from them the strategy of choosing one of these hyperplanes to play the role of the hyperplane at infinity, 
and showing vanishing of the pushforward. 

However, for them the greatest interest was to show vanishing of cohomology in every degree except the middle degree. For our purposes, it's just as good to show vanishing of cohomology in every degree except the middle two degrees. We could even allow more degrees, but this would not be helpful for the argument. 
This means that it is sufficient to show that the support of the pushforward is zero-dimensional, rather than empty as in \cite{Nonresonance}. 
The pushforward having empty support is a stronger condition that would not hold in our setting without additional assumptions. 

The second difference is that we work in characteristic $p$, where wild ramification can occur, while \cite{Nonresonance} works in characteristic zero, where it does not. 
This is one reason why it is so helpful for us that the sheaf $\overline{\mathcal F}$ arises from a certain explicit construction with tensor products of sheaves pulled back from curves. This allows us to control what types of wild ramification occur. Unlike in the characteristic zero setting, it would be difficult to come up with a formulation of the cohomology vanishing statement that applies to an arbitrary lisse sheaf on the complement of a hyperplane arrangement and is suitable for our purpose. 

The third, related, difference is that \cite{Nonresonance} uses an explicit resolution of singularities - this is an iterated blow-up of the projective space such that the inverse image of the hyperplane arrangement in question is a simple normal crossings divisor. This enables them to avoid working with perverse sheaves, because lisse sheaves with tame ramification on the complement of a normal crossings divisor have all the good properties of perverse sheaves (because they are, in fact, a special case of perverse sheaves). For sheaves with wild ramification, this description is not available, and so the machinery of perverse sheaves must be used.

We can also compare to the strategy of \cite{SS}, 
where we proved our cohomology vanishing statement in the special case where all the sheaves $\mathcal F_\pi$ are shifts of character sheaves, instead of just one. 
In that work, we considered a family of hyperplane arrangements, and studied the support of the vanishing cycles sheaf, rather than fixing a hyperplane at infinity and studying the support of the pushforward. 
The arguments needed to calculate the vanishing cycles and the pushforward are closely related. In both cases, the problem is entirely local at a given point, and a key strategy to study a sheaf constructed in a certain way, is to find a simpler construction which produces an equivalent sheaf locally (but not globally). 

The advantage of the pushforward over the vanishing cycles for our purposes is that we only need to do this local analysis for points on a single hyperplane. Indeed, vanishing cycles could appear at any singular point, i.e. on any hyperplane, unless dealt with. 
The fact that we only need to do difficult local calculations at a single hyperplane means that we need to make strong assumptions about only a single sheaf $\mathcal F_\pi$.  
However, abandoning the vanishing cycles method requires us to have an alternative strategy for Betti bounds, 
because the same vanishing cycles methods that proved cohomology vanishing statements in \cite{SS} was simultaneously used there to prove Betti number bounds. In this paper, it does not seem possible to derive Betti number bounds directly from considering the pushforward, so we instead obtain them from a separate argument.

\subsubsection{Trace functions vs Arithmetic functions}

Inspired by \cite{FKM14} and other works on orthogonality of trace functions and arithmetic functions over the integers,
we consider here the correlation between trace functions and von Mangoldt/M\"{o}bius over function fields.
We shall use the notation
\begin{equation}
\mathcal{M}_n = \{f \in \F_q[u] : \deg(f) = n, \ f \ \text{is monic} \}
\end{equation}
where $n$ is a nonnegative integer.

\begin{thm} \label{LinearMobiusVStraceFunctionThm}

Let $p$ be an odd prime, and let $q > 4e^2p^2$ be a power of $p$.
For a prime $\pi \in \F_q[u]$, an infinitame trace function $t \colon \F_q[u]/(\pi) \to \mathbb{C}$, 
and a nonnegative integer $n$ we have
\begin{equation}
\sum_{f \in \mathcal{M}_n} t(f)\mu(f) \ll |\mathcal{M}_n|^{1 - \frac{1}{2p} + \frac{\log_q(2ep)}{p}}
|\pi|^{\log_q \left( r(t) \left( 1 + \frac{1}{2p} \right) + \frac{c(t)}{2p} \right)}
\end{equation}
as $n,|\pi| \to \infty$, with the implied constant depending only on $q$.

\end{thm}

\cref{LinearMobiusVStraceFunctionThm} improves on the savings obtained in \cite[Theorem 1.8]{SS} 
for the Kloosterman fraction $t(f) = e(\overline{f}/\pi)$, in case $p$ is small enough and $q$ is large but fixed.
For larger $p$, the savings here are smaller, 
but apply to lengths of summation as short as $|\mathcal{M}_n| \sim |\pi|^\epsilon$ for any $\epsilon > 0$, once $q$ is chosen appropriately. 
As opposed to \cite{SS}, here we do not pursue the possible applications of a bound as in \cref{LinearMobiusVStraceFunctionThm} to the level of distribution of primes in arithmetic progressions.
Over the integers, different arguments have been given to obtain cancellation for sums longer than $|\pi|^{\frac{1}{2}}$ for more general trace functions, see \cite[Theorem 1.7, Remark 1.9]{FKM14}.
Using \cref{LinearMobiusVStraceFunctionThm} we are able to prove the following.

\begin{cor} \label{InfinitameVSprimesCor}

Let $p$ be an odd prime, let $q > 4e^2p^2$ be a power of $p$, let $\delta > 0$, and set
\begin{equation} \label{TheZetaConstantTraceEq}
\zeta = \frac{2\delta}{1 + 2\delta} \left( 1 + \frac{1}{2p} - \frac{\log_q(2ep)}{p} \right)^{-1}.
\end{equation}
Take a prime $\pi \in \F_q[u]$, an integer $n \geq (\frac{1}{2} + \delta)\deg(\pi)$, 
and an infinitame trace function $t \colon \F_q[u]/(\pi) \to \mathbb{C}$ arising from a sheaf $\mathcal F$ whose geometric monodromy representation does not admit the trivial representation $\overline{\mathbb{Q}_\ell}$ as a quotient.
Then for any $\epsilon > 0$ we have
\begin{equation*}
\sum_{f \in \mathcal{M}_n} t(f) \Lambda(f) = O\left( |\mathcal{M}_n|^{\frac{1}{1 + 2\delta} + \zeta + \epsilon}
|\pi|^{\log_q \left( r(t) \left( 1 + \frac{1}{2p} \right) + \frac{c(t)}{2p} \right)} + 
(r(t) + c(t))\frac{|\mathcal{M}_n|^{1 + \epsilon}}{|\pi|^{\frac{1}{2}}} \right) 
\end{equation*}
with the implied constant depending only on $q$ and $\epsilon$.

\end{cor}

This result gives very modest savings, and applies to fewer trace functions compared to \cite[Theorem 1.5]{FKM14}.
Nevertheless, \cref{InfinitameVSprimesCor} guarantees cancellation in intervals shorter than those treated over the integers,
see for instance \cite{Irv}.  
We obtain savings as long as $\deg (\pi) > \epsilon n$ and   
\begin{equation}
\delta  >  \frac{ (2p+1) \log \left(  r(t) \left( 1 + \frac{1}{2p} \right) + \frac{c(t)}{2p}  \right)  } { \log \left( \frac{q}{ 4 e^2 p^2 } \right) }, 
\end{equation}
so in particular we can take $\delta\to 0$ as $q \to\infty$ with fixed characteristic, rank, and conductor. The results of \cite{FKM14}, \cite{Irv} give savings only when (in our notation) $\delta> \frac{1}{4}$, though \cite{Irv} can handle any $\delta$ with an additional average over the modulus $\pi$.  

We have another application for \cref{LinearMobiusVStraceFunctionThm}, 
concerning very short sums over primes of shifted multiplicative characters.

\begin{cor} \label{ShiftedCharacterVsPrimesCor}

Let $p$ be an odd prime, let $q > 4e^2p^2$ be a power of $p$, set
\begin{equation} \label{TheZetaConstantEq}
\zeta = \left( 1 + \frac{1}{p} - \frac{\log_q(4e^2p^2)}{p} \right)^{-1} < 1,
\end{equation}
and let $\epsilon > 0$.
Then for a prime $\pi \in \F_q[u]$, a nontrivial Dirichlet character $\chi \colon (\F_q[u]/(\pi))^\times \to \mathbb{C}^\times$,
a polynomial $h \in \F_q[u]$, and a nonnegative integer $n$ we have
\begin{equation}
\sum_{f \in \mathcal{M}_n} \chi(f+h) \Lambda(f) = 
O \left( |\mathcal{M}_n|^{\frac{1 + \zeta}{2} + \epsilon} |\pi|^{\log_q(3)} + |\mathcal{M}_n|^{1 + \epsilon}|\pi|^{-1} \right)
\end{equation} 
as $|\pi|,n \to \infty$, with the implied constant depending only on $q$ and $\epsilon$.
We 
\end{cor}

As in \cref{LinearMobiusVStraceFunctionThm}, the strength of the result is in the shortness of the range of summation, the power saving being quite small. 
\cref{ShiftedCharacterVsPrimesCor} provides savings as long as 
\begin{equation}
\epsilon n < \deg (\pi) < \frac{   \log_3\left(\frac{q }{ 4 e^2 p^2 }\right)    }{2p +2    } n,
\end{equation} 
which as $q \to \infty$ with fixed $p$ allows us to take $\deg (\pi)$ an arbitrarily large multiple of $n$.  
For the state of the art on the analogous problem over the integers we refer to \cite{Rakh} and references therein.
In this result, and the previous two, we have worked for simplicity with trace functions to prime moduli only.
These results can be extended to trace functions with an arbitrary squarefree period.

\section{Sheaves}

One can speak of sheaves and trace functions not only on $\mathbb{A}^1$, as we did so far, 
but also on other curves and on more general varieties. 
Most of the notions from \cref{SheafFuncDictionaryDef} admit natural generalizations to this setting.
We start here by constructing the sheaves giving rise to the trace functions we have encountered, 
and their high-dimensional counterparts. These constructions are standard \cite[Sommes trig.]{sga4h}, but we provide here a detailed explanation including all the properties we need, for the reader's convenience.

\subsection{Kummer sheaves}

\begin{notation}\label{Kummer-sheaf} 

Let $\kappa$ be a finite field of characteristic $p$, 
let $\chi \colon \kappa^\times \to \overline{\mathbb{Q}_\ell}^\times$ be a multiplicative character (group homomorphism), 
and let $w \in \kappa[T]$ be a nonzero polynomial. 
We extend $\chi$ to a function on $\kappa$ by setting $\chi(0)=0$,
and construct a $\overline{\mathbb{Q}_\ell}$-sheaf $\mathcal L_\chi(w)$, on the affine line $\mathbb A^1_\kappa = \mathrm{Spec} \ \kappa[T]$, whose trace function is $\chi$, as follows.

Denote by $|\kappa|$ the number of elements in $\kappa$. Then the cover of $\mathbb A^1_\kappa$ defined by the equation \begin{equation}\label{Lang}
Y^{|\kappa|-1} = w(T) 
\end{equation} 
is finite \'{e}tale (see \cite[Example 2.5]{Mil}) away from the set
\begin{equation}
S = \{z \in \mathbb{A}^1_{\kappa} : w(z) = 0\}.
\end{equation}

The group $\kappa^\times$ acts on our cover (by automorphisms) via multiplication on $Y$, 
since every $\zeta \in \kappa^\times$ satisfies $\zeta^{|\kappa|-1} = 1$.
As all $\zeta \in \overline{\kappa}$ with $\zeta^{|\kappa|-1}=1$ lie in $\kappa$, 
we get a simply transitive action of $\kappa^\times$ on the (geometric) fiber of any geometric point $\overline{x}$ lying over a (not necessarily closed) point $x$ of
\begin{equation}
U = \mathbb{A}^1_{\kappa} - S.
\end{equation}  

From the definition of the \'{e}tale fundamental group as the automorphism group of the fiber functor (e.g. \cite[Theorem 5.4.2(2)]{Sz}), 
we get a continuous action of $\pi_1^{\text{\'{e}t}}(U,\overline{x})$ on the fiber of $\overline{x}$ in our \'{e}tale cover of $U$, commuting with the action of $\kappa^\times$.
Since the latter acts simply transitively, 
by picking a point $\overline{t}$ in the fiber over $\overline{x}$,
to each $g \in \pi_1^{\text{\'{e}t}}(U,\overline{x})$ we can associate a unique $\lambda \in \kappa^\times$ satisfying 
$\lambda (\overline t) = g(\overline{t})$.
This association is a continuous homomorphism as, if to $g_1,g_2 \in \pi_1^{\text{\'{e}t}}(U,\overline{x})$ we have associated $\lambda_1, \lambda_2 \in \kappa^\times$, then
\begin{equation*}
g_1g_2(\overline{t}) = g_1(g_2(\overline{t})) = g_1(\lambda_2(\overline{t})) = \lambda_2(g_1(\overline{t})) = 
\lambda_2 (\lambda_1 (\overline{t})) = \lambda_2\lambda_1(\overline{t}) = \lambda_1\lambda_2(\overline{t}). 
\end{equation*} 

Therefore, by composition with $\chi$, we get a continuous homomorphism $\pi_1^{\text{\'{e}t}}(U,\overline{x}) \to \overline{\mathbb Q_\ell}^\times$. This gives rise to a continuous one-dimensional representation of $\pi_1^{\text{\'{e}t}}(U,\overline{x})$ over $\overline{\mathbb{Q}_\ell}$, hence a rank one lisse sheaf on $U$ via the equivalence in \cite[2.0.2]{Ka}.
We define $\mathcal L_\chi(w)$ to be the extension by zero of this lisse sheaf from $U$ to $\mathbb A^1_{\kappa}$.
We call $\mathcal{L}_\chi(w)$ a Kummer sheaf.

As suggested by the notation, the construction is independent of the choice of $\overline{t}$. 
Indeed if $\overline{h}$ is another geometric point in the fiber over $\overline{x}$, then by transitivity there exists $\gamma \in \kappa^\times$ with $\gamma (\overline{t}) = \overline{h}$, so we have
\begin{equation*}
g(\overline{h}) = g(\gamma(\overline{t})) = \gamma(g(\overline{t})) = \gamma(\lambda(\overline{t})) = \gamma\lambda(\overline{t}) = \lambda\gamma(\overline{t}) = \lambda(\gamma(\overline{t})) = \lambda(\overline{h}),
\end{equation*}
where $\lambda \in \kappa^\times$ is associated to $g \in \pi_1^{\text{\'{e}t}}(U,\overline{x})$.
Moreover, by \cite[Proposition 5.5.1]{Sz} the fiber functors for different geometric points on the connected curve $U$ are isomorphic, so our construction is also independent of the choice of the point $x \in U$ (or the geometric point above it).

In case $\overline{x}$ is a geometric generic point of $U$, 
its fiber can be identified with the set of all homomorphisms of $\kappa(T)$-algebras from $\kappa(T)[Y]/(Y^{|\kappa| - 1} - w(T))$ to $\kappa(T)^{\mathrm{sep}}$. The group $\pi_1^{\text{\'{e}t}}(\overline{x}) = \mathrm{Gal}(\kappa(T)^{\mathrm{sep}} / \kappa(T))$ acts on this set by postcomposition,
and this action factors through the aforementioned action of $\pi_1^{\text{\'{e}t}}(U,\overline{x})$ on the fiber of $\overline{x}$ (via the map on fundamental groups induced from the inclusion of $\overline{x}$ in $U$).

\end{notation}

In the following lemma, among other things, we will see that the trace function $t_{\mathcal{L}_\chi(w)}$ arising from the sheaf $\mathcal{L}_\chi(w)$ is infinitame, and calculate its invariants.

\begin{lem}\label{Kummer-sheaf-properties} 

The sheaf $\mathcal L_\chi(w)$ on $\mathbb A^1_\kappa$ has the following properties.

\begin{enumerate}

\item For every $x \in \kappa$ we have $t_{\mathcal L_\chi(w) } (x) = \chi(w (x) )$;

\item the sheaf $\mathcal L_\chi(w)$ is lisse on $U$, and vanishes on its complement $S$;

\item the sheaf $\mathcal L_\chi(w)$ has tame local monodromy at every closed point $x \in \mathbb{P}^1_\kappa$, 
or in other words, it is tamely ramified (everywhere);

\item the sheaf $\mathcal L_\chi(w)$ is mixed of nonpositive weights;

\item the sheaf $\mathcal L_\chi(w)$ has no finitely supported sections;

\item the rank and conductor are given by
\begin{equation*}
\rank(\mathcal L_\chi(w))=1, \quad c(\mathcal L_\chi(w)) = \#\{a \in \overline{\kappa} : w(a) = 0\} \leq \deg(w);
\end{equation*} 

\item the sheaf $\mathcal L_\chi(w)$ is the extension by zero to $\mathbb{A}^1_\kappa$ of some one-dimensional representation of the tame arithmetic fundamental group of $\mathbb A^1_{\kappa } \setminus \{z\}$ for some $z \in \kappa$ if and only if there exists $c \in \kappa^\times$ and a positive integer $d$ such that
\begin{equation*}
w(T) = c(T - z)^d.
\end{equation*}
If this is the case, let $v \geq 1$ be the (multiplicative) order of $\chi$. 
Then the representation is trivial on the geometric fundamental group of $\mathbb{A}^1_\kappa \setminus \{z\}$ if and only if $v$ divides $d$. 

\end{enumerate}

\end{lem}

\begin{proof}

Visibly, (2) is immediate from our construction.

To verify (1), first note that because the sheaf is zero on $S$, its trace function is zero, which matches our convention
\begin{equation}
\chi(w(x))=\chi(0)=0, \quad x \in S.
\end{equation} 
For $x\in \kappa \setminus S$, we get from \cref{Lang} that $g = \Frob_{x,\kappa} \in \pi_1^{\text{\'{e}t}}(U,\overline{x})$ acts on the geometric fiber over $x$ by 
\begin{equation}
g(x,y) = (x^{|\kappa|}, y^{|\kappa|}) = (x,y^{|\kappa|}) = (x, w(x)y).
\end{equation}
Hence, by our definition of the representation giving rise to the sheaf $\mathcal{L}_{\chi}(w)$, 
the element $\lambda = w(x) \in \kappa^\times$ is associated to $g$,
so $g$ is mapped to $\chi(w(x))$ as desired.

For (3), note that since the monodromy (i.e. image) of the representation giving rise to $\mathcal L_{\chi}(w)$ is isomorphic to a quotient of $\kappa^\times$,
it has order prime to $p$. Therefore, by Lagrange's theorem, the image of an inertia group of any closed point $x \in \mathbb{P}^1_\kappa$ is of order prime to $p$ as well. It follows that  $\mathcal L_{\chi}(w)$ has tame local monodromy at $x$.

To get (4), note that for a closed point $x \in \mathbb{A}^1_\kappa$,
every eigenvalue of $\text{Frob}_{x, \kappa(x)}$ is a value of the finite order character $\chi$, 
hence a root of unity whose norm is thus $1= |\kappa(x)|^{0/2}$. 
This shows that $\mathcal{L}_\chi(w)$ is punctually pure of weight $0$, so in particular it is mixed of nonpositive weights. 

Observe that (5) is immediate from (2). Indeed, $\mathcal L_\chi(w)$ is lisse on $U$, so it has no finitely supported sections there, 
and it has no sections at all supported on $S$ as all of its stalks vanish there.

To get the first part of (6), recall from $(2)$ that $\mathcal L_\chi(w)$ is lisse on $U$,
hence it is lisse at a geometric generic point $\overline{\eta}$ of $U$ (and of $\mathbb{A}^1_\kappa$).
Hence the dimension of $\mathcal{L}_\chi(w)_\eta$ is the rank of the representation giving rise to it, which is $1$.
For the second part of (6), we get from (3) that $\mathcal L_\chi(w)$  has tame ramification everywhere so all the Swan conductors vanish. By the definition in \cref{InfinitameConductorDefEq} we therefore have
\begin{equation*}
\begin{split}
c(\mathcal{L}_\chi(w)) &= \sum_{x \in |\mathbb{A}^1_\kappa|} [\kappa(x) : \kappa](1 - \dim \mathcal{L}_\chi(w)_{\overline{x}}) \\
&= \sum_{x \in |U|} [\kappa(x) : \kappa](1 - 1) + \sum_{x \in |S|} [\kappa(x) : \kappa](1 - 0) = \#\{a \in \overline{\kappa} : w(a) = 0\}   
\end{split}
\end{equation*}
because the dimension of the stalk at every point where the sheaf is lisse equals the generic rank.


For (7), if $\mathcal L_\chi(w)$ is the extension by zero of a one-dimensional representation of $\pi_1^{\text{\'{e}t}}(\mathbb A^1_\kappa \setminus \{z\})$, then it is lisse away from $z$ and vanishes at $z$, making $z$ the unique root of $w$ by (2).
The uniqueness of the root $z$ allows us to write $w(T) = c (T-z)^d$ for a scalar $c \in \kappa^\times$ and a positive integer $d$. 
Conversely, if $z$ is the unique root of $w$, then by construction $\mathcal L_\chi(w)$ is the extension by zero of a one-dimensional representation, which is tame by (3).

Our representation is geometrically trivial if and only if the image of the map from the geometric fundamental group to $\kappa^\times$ is contained in $\Ker(\chi)$.
Since $\kappa^\times$ is cyclic of order $|\kappa|-1$, and $\chi$ is of order $v$, we see that
\begin{equation}
\Ker(\chi) = \{\zeta^v : \zeta \in \kappa^\times \} = \{\zeta \in \kappa^\times : \zeta^{n} = 1\}, \quad n = \frac{|\kappa|-1}{v}.
\end{equation}

Therefore, the aforementioned image is contained in the kernel above if and only if the geometric fundamental group acts on the (geometric) generic fiber via multiplication by $n$-th roots of unity. 
This is equivalent to the geometric fundamental group acting trivially on the generic fiber of the finite \'{e}tale subcover
\begin{equation}
{\widetilde{Y}}^v = w(T) = c(T-z)^d, \quad \widetilde{Y} = Y^n
\end{equation}
of $\mathbb{A}^1_{\overline{\kappa}}$.
Since the action of the fundamental group on the generic fiber is that of
$\Gal(\overline{\kappa}(T)^{\mathrm{sep}}/\overline{\kappa}(T))$,
the triviality of the action is tantamount to the existence of an $v$-th root for $w(T)$ in $\overline{\kappa}(T)$.
Such a root exists if and only if $d$ is a multiple of $v$, so we have finished the verification of (7).
\end{proof}

\subsection{Change of variable for sheaves}

For future use, we record some simple transformation rules of sheaves and their trace functions.

\begin{prop} \label{LinearCOVtraceFuncProp}

Let $g \in \F_q[u]$ be a squarefree polynomial,
let 
\begin{equation}
t \colon \F_q[u]/(g) \to \mathbb{C}
\end{equation}
be an infinitame trace function, and let $P,c \in \F_q[u]$. 
Then the function defined by
\begin{equation}
t'(x) = t(Px + c)
\end{equation}
is an infinitame trace function with rank and conductor satisfying
\begin{equation}
r(t') \leq r(t), \quad c(t') \leq c(t).
\end{equation}

\end{prop}

\begin{notation} \label{ShiftedPexpNotation}

For a finite field $\kappa$ of characteristic $p$ and $r \in \kappa$,
we define the map
\begin{equation} \label{ShiftedPexpFirstEq}
E_{r} \colon \mathbb A^1_{\kappa} \to \mathbb A^1_{\kappa}, \quad E_r(x) = r + x^p.
\end{equation}

\end{notation}

\begin{prop} \label{ShiftedFrobPullbackLem}

Let $\mathcal{F}$ be an infinitame sheaf on $\mathbb{A}^1_\kappa$. 
Then the sheaf $E_r^* \mathcal{F}$ and its trace function enjoy the following properties.

\begin{enumerate}

\item If $\mathcal{F}$ has no finitely supported sections, then neither does $E_r^* \mathcal F$.

\item If $\mathcal F$ is tamely ramified at infinity then so is $E_r^* \mathcal F$.

\item If $\mathcal{F}$ is mixed of nonpositive weights then so is $E_r^* \mathcal{F}$.

\item We have $c(E_r^* \mathcal F) = c(\mathcal{F})$ and $\rank(E_r^* \mathcal F) = \rank(\mathcal{F})$.

\item We have $t_{E_r^* \mathcal{F}}(x) = t_{\mathcal{F}}(r + x^p)$. 

\item If $t_{\mathcal {F}}$ is a Dirichlet trace function then so is $t_{E_r^* \mathcal{F}}$.

\end{enumerate} 

\end{prop}

\begin{proof}

The map $x \mapsto r+ x^p$ is an \'{e}tale homeomorphism and hence pullback under it preserves \'{e}tale topological invariants such as generic rank and conductor.
\end{proof}

\subsection{Local invariants}

Here we take a closer look at the local invariants of a sheaf $\mathcal{F}$ on a curve $C$ over a perfect field $\kappa$ of characteristic $p$. Some of these invariants (and their analogs) were mentioned in passing earlier.

\subsubsection{Drop, Slope, Swan}

\begin{defi} \label{DropDef}
For a sheaf $\mathcal F$ on a curve $C/\kappa$ and a closed point $x$ of  $C$, define the \emph{drop} 
\begin{equation} \label{DropDefEq}
\dro_x(\mathcal F) = \rank(\mathcal F) - \dim (\mathcal F_x).
\end{equation}
\end{defi}

This is the drop in the rank of $\mathcal F$ as we pass from a generic point to $x$. If $\mathcal F$ has no sections supported at $x$, then the $\dro_x(\mathcal F)\geq 0$. If $\mathcal F$ is a middle extension sheaf at $x$ in the sense that $\mathcal F$ is the (non-derived) pushforward from $C \setminus \{x\}$ to $C$ of some sheaf, then $\mathcal F_x$ is equal to the invariants of $\mathcal F_{\eta}$ under the inertia group $I_x$, and then $\dro_x(\mathcal F)$ is the codimension of the inertia invariants.

Next we introduce the `slope' of an irreducible inertia representation,
which is sometimes also called `break' or `jump', see \cite[Chapter 1]{Ka}.
For that we use the upper numbering filtration on an inertia group $I$ indexed by nonnegative real numbers.
That is, for $s \geq 0$ we denote by $I^s$ what is sometimes denoted by $\mathrm{Gal}(L^{\text{sep}}/L)^s$, 
where $L$ is the completion of the function field of $C$ at $x$,
see for instance \cite[Definition 3.54]{KR}.

\begin{defi} \label{SwanDef}

Let $C$ be an open subset of a proper curve $\overline{C}/\kappa$,
let $x$ be a closed point of $\overline{C}$.
For an irreducible (finite-dimensional, continuous) representation $V/\overline{\mathbb{Q}_\ell}$ of $I_x$ define
\begin{equation} \label{DefSlopeEq}
\text{slope}(V) = \inf \{s \geq 0 : I_x^s \ \text{acts trivially on} \ V\}.
\end{equation}

Let $V/\overline{\mathbb{Q}_\ell}$ be a representation of $I_x$,
and let $V_1, \dots, V_n$ be the (irreducible) Jordan-H\"{o}lder facotrs of $V$, listed with multiplicity. 
We define the Swan conductor of $V$ by
\begin{equation}
\swan(V) =  \sum_{i=1}^n \dim(V_i) \ \text{slope}(V_i),
\end{equation} 
and the slopes of $V$ to be
\begin{equation}
\text{slopes}(V) = \{\text{slope}(V_i) : 1 \leq i \leq n\}.
\end{equation}

For a sheaf $\mathcal F$ on $C$, 
we can view the stalk $\mathcal{F}_\eta$ at the generic point as a representation of $I_x$,
and define the Swan conductor of $\mathcal{F}$ at $x$ by 
\begin{equation}
\swan_x(\mathcal F) = \swan(\mathcal F_{\eta}).
\end{equation}
Similarly, if $I_x$ acts irreducibly on $\mathcal{F}_\eta$, we set
\begin{equation}
\text{slope}_x(\mathcal{F}) = \text{slope}(\mathcal{F}_\eta)
\end{equation}
and in general
\begin{equation}
\text{slopes}_x(\mathcal{F}) = \text{slopes}(\mathcal{F}_\eta).
\end{equation}
We further define the local conductor of $\mathcal{F}$ at $x$ as
\begin{equation} \label{LocalCondDefEq}
\cond_x(\mathcal F)=\dro_x(\mathcal F)+ \swan_x(\mathcal F).
\end{equation}

\end{defi}


Note that $\mathcal{F}$ is tamely ramified at $x$ if and only if 
$\text{slopes}_x(\mathcal{F}) = \{0\}$, or equivalently $\swan_x(\mathcal{F}) = 0$.
By our earlier remarks, if $\mathcal F$ is a middle extension sheaf then $\cond_x(\mathcal F)$ is the Swan conductor of the inertia representation of $\mathcal F$ at $x$ plus the codimension of the inertia invariants. By definition, this is the Artin conductor of the inertia representation. Thus, $\cond_x(\mathcal F)$ is an adaptation of the Artin conductor to the setting of sheaves.

For an alternative definition of the Swan conductor see \cite[Definition 4.72, Definition 4.82, Theorem 4.86]{KR}. 
 
\subsubsection{Euler characteristic}

\begin{defi} \label{EulerCharDef}

We define the Euler characteristic of a sheaf $\mathcal{F}$ on a curve $C/\overline{\kappa}$ by
\begin{equation}
\chi(C, \mathcal{F}) = \sum_{i=0}^2 (-1)^i \dim H^i_c(C, \mathcal{F}).
\end{equation}

\end{defi}

For the constant sheaf on a proper curve $C = \overline{C}$ of genus $g$ we have
\begin{equation}
\chi(\overline C) = \chi(\overline C, \overline{\mathbb{Q}_\ell}) 
= \sum_{i=0}^2 (-1)^i \dim H^i_c( \overline C, \overline{\mathbb{Q}_\ell}) = 1 - 2g + 1 = 2-2g,
\end{equation}
while in the affine case $C \subsetneq \overline C$ we have
\begin{equation} \label{AffineEulerChar}
\begin{split}
\chi(C) &= \chi(C, \overline{\mathbb{Q}_\ell}) 
= \dim H^2_c( C, \overline{\mathbb{Q}_\ell}) - \dim H^1_c( C, \overline{\mathbb{Q}_\ell}) \\
&= 1 - (2g  + |\overline{C} - C| - 1) = \chi(\overline{C}) - |\overline{C} - C|.
\end{split}
\end{equation}

\begin{lem}\label{EP} 
For a sheaf $\mathcal F$ on a proper curve $\overline{C}/\overline{\kappa}$, 
we have
\begin{equation}
\chi(\overline{C}, \mathcal F) = \chi(\overline{C}) \rank(\mathcal F) -  \sum_{x \in |\overline{C}|}\cond_x(\mathcal F).
\end{equation}
\end{lem}

Note that $\swan_x(\mathcal F)$ and $\dro_x(\mathcal F)$ both vanish at every point $x \in |\overline{C}|$ where $\mathcal{F}$ is lisse, so the sum above is finite.

\begin{proof} 

This is the Grothendieck-Ogg-Shafarevich formula \cite[X, Theorem 7.1]{sga5}, specialized to the case of sheaves (instead of complexes of sheaves). 
\end{proof}

\begin{lem}\label{EP2} 

For a sheaf $\mathcal F$ on an open subset $C$ of a compact curve $\overline{C}/\overline{\kappa}$, we have
\begin{equation}
\chi(C, \mathcal F) = \chi(C) \rank(\mathcal F) -  \sum_{x \in |C|}\cond_x(\mathcal F)-\sum_{ x \in \overline{C} - C} \swan_x(\mathcal F) 
\end{equation} 

\end{lem}

\begin{proof} 

Let $j \colon C \to \overline{C}$ be the open immersion. By \cref{EulerCharDef} and \cref{EP}, we have
\begin{equation} \label{CompactlySupportingEq}
\chi(C, \mathcal F) = \chi(\overline{C}, j_! \mathcal F) = \chi(\overline{C}) \rank(j_!\mathcal F) -  \sum_{x \in | \overline{C} |}\cond_x(j_!\mathcal F). 
\end{equation}
Extension by zero preserves all local invariants at points of $C$, so we have 
\begin{equation} \label{NoChangeLocInvExtZero}
\rank(j_! \mathcal F) = \rank(\mathcal F), \quad \cond_x(j_! \mathcal F) = \cond_x(\mathcal F), \quad x \in |C|.
\end{equation} 

For $x \in \overline{C}-C$, we have $(j_! \mathcal F)_x=0$ so from \cref{DropDef} we get
\begin{equation}
\dro_x(j_! \mathcal F) =\rank(j_!\mathcal F) - \dim (j_!\mathcal{F})_x =\rank(\mathcal F) 
\end{equation} 
and by \cref{SwanDef} we have 
\begin{equation}
\swan_x( j_! \mathcal F) = \swan (j_!\mathcal{F})_\eta = \swan(\mathcal{F}_\eta) = \swan_x(\mathcal F)
\end{equation} 
so by definition of the local conductor in \cref{LocalCondDefEq}
\begin{equation} \label{RankAndSwanEq}
\cond_x(j_!\mathcal F) = \rank(\mathcal F) + \swan_x(\mathcal F).
\end{equation}

Combining \cref{AffineEulerChar}, \cref{CompactlySupportingEq}, \cref{NoChangeLocInvExtZero}, and \cref{RankAndSwanEq} we get
\begin{equation*}
\begin{split}
\chi(C, \mathcal F) &= \chi(\overline{C}) \rank(j_!\mathcal F) -  \sum_{x \in | \overline{C} |}\cond_x(j_!\mathcal F) \\
&=(\chi(C) + |\overline{C} - C| )\rank(\mathcal F) -  \sum_{x \in | C |}\cond_x(\mathcal F) - 
\sum_{x \in \overline{C} - C} (\rank(\mathcal{F}) + \swan_x(\mathcal{F})) \\ 
&=\chi(C)\rank(\mathcal F) -  \sum_{x \in | C |}\cond_x(\mathcal F) - 
\sum_{x \in \overline{C} - C} \swan_x(\mathcal{F})
\end{split}
\end{equation*}
as desired.
\end{proof}



\subsubsection{Local invariants of tensor products}

\begin{prop}\label{slope-tensor-lemmas} 

For irreducible representations $V_1, V_2$ of an inertia group $I$ we have
\begin{equation} \label{SlopesInEq}
\max \ \textup{slopes}(V_1 \otimes V_2) \leq \max \{\slope(V_1),\slope(V_2) \}.
\end{equation}
Moreover, in case $\dim V_2 = 1$, the representation $V_1 \otimes V_2$ is irreducible, and
\begin{equation} \label{SlopesTensorFormula}
\slope(V_1 \otimes V_2) = \max \{\slope(V_1),\slope(V_2)\}
\end{equation}
unless $\slope(V_1)=\slope(V_2)$ and for every $g \in I^{\slope(V_1)}$ there exists $\lambda \in {\overline{\mathbb{Q}_\ell}}^\times$ such that for every $v_1 \in V_1$ and $v_2 \in V_2$ we have
\begin{equation} \label{ActionByScalarsEq}
g(v_1) = \lambda v_1, \quad g(v_2) = \lambda^{-1}v_2.
\end{equation}

\end{prop}

\begin{proof}

In order to establish \cref{SlopesInEq}, take $s>\max\{\slope(V_1),\slope(V_2)\}$.
By the definition in \cref{DefSlopeEq}, the subgroup $I^s$ acts trivially on both $V_1$ and $V_2$, 
so it acts trivially on $V_1 \otimes V_2$, hence on all its Jordan-H\"{o}lder factors.
It follows that $\max \text{slopes}(V_1 \otimes V_2) \leq s$, therefore \cref{SlopesInEq} holds.

That $V_1 \otimes V_2$ is irreducible if $\dim V_2=1$ is a general fact about representations, because a subspace of $V_1 \otimes V_2$ is invariant if and only if the corresponding subspace of $V_1$ is invariant. 

For the proof of \cref{SlopesTensorFormula}, assume first that $\slope(V_1)\neq\slope(V_2)$. 
For 
\begin{equation}
\min \{\slope(V_1),\slope(V_2) \} < s < \max \{ \slope(V_1),\slope(V_2) \},
\end{equation}
the subgroup $I^s$ of $I$ acts trivially on one of $V_1, V_2$ and nontrivially on the other, so it acts nontrivially on their tensor product. 
Since $I^{s} \subseteq I^{s'}$ if $s>s'$, 
we conclude that $I^s$ acts nontrivially on $V_1 \otimes V_2$ for any 
\begin{equation}
s < \max \{\slope(V_1),\slope(V_2)\},
\end{equation} 
hence $\slope(V_1 \otimes V_2) \geq \max\{\slope(V_1),\slope(V_2)\}$. 
Using \cref{SlopesInEq} and the irreducibility of $V_1 \otimes V_2$, we arrive at \cref{SlopesTensorFormula}.
 
Suppose now that $\slope(V_1)=\slope(V_2)=s$ but $\slope(V_1 \otimes V_2)<s$, so $I^s$ acts trivially on $V_1 \otimes V_2$. 
As $I^s$ acts by scalars on the one-dimensional representation $V_2$, 
it must act by the inverses of these scalars on $V_1$ for the action on $V_1 \otimes V_2$ to be trivial.
In other words, \cref{ActionByScalarsEq} holds. 
\end{proof}
 
\begin{cor} \label{SwanTensorProdBoundCor}

For representations $V_1, V_2$ of an inertia group we have
\begin{equation}
\swan(V_1 \otimes V_2) \leq \swan(V_1) \dim(V_2) + \swan(V_2) \dim(V_1).
\end{equation}

\end{cor}

\begin{proof}

By \cref{SwanDef}, the Swan conductor is additive in short exact sequences, 
so we are reduced to the case $V_1$ is irreducible, and then also to the case $V_2$ is irreducible.
If $W_1, \dots W_n$ are the Jordan-H\"{o}lder factors of $V_1 \otimes V_2$, then by \cref{slope-tensor-lemmas} we have
\begin{equation*}
\begin{split}
\swan(V_1 \otimes V_2) &= \sum_{i=1}^n \dim(W_i) \slope(W_i) \leq \sum_{i=1}^n \dim(W_i) \max \ \mathrm{slopes}(V_1 \otimes V_2)\\
&\leq \sum_{i=1}^n \dim(W_i) \max\{\slope(V_1), \slope(V_2)\} \\
&\leq (\slope(V_1) + \slope(V_2)) \dim(V_1 \otimes V_2) \\&= \swan(V_1) \dim(V_2) + \swan(V_2) \dim(V_1)
\end{split}
\end{equation*}
as required.
\end{proof}

\begin{lem}\label{general-tensor-product}  

Let $\mathcal F_1$ and $\mathcal F_2$  be sheaves on $\mathbb A^1_\kappa$.  
Then the tensor product $\mathcal F_1 \otimes \mathcal F_2 $ has the following properties.

\begin{enumerate}

\item For every $x \in \kappa$ we have
\[ t_{ \mathcal F_1 \otimes \mathcal F_2 } (x) = t_{ \mathcal F_1}(x) t_{\mathcal F_2} (x) ;\]

\item if $\mathcal F_1 $ and $\mathcal F_2$ have no finitely supported sections, 
then neither does $\mathcal F_1 \otimes \mathcal F_2$;

\item if $\mathcal F_1$ and $\mathcal F_2$ are tamely ramified at $\infty$, then so is $\mathcal F_1 \otimes \mathcal F_2$;

\item if $\mathcal F_1$ and $\mathcal{F}_2$ are mixed of nonpositive weights, then so is $\mathcal F_1 \otimes \mathcal F_2$;

\item the rank of the tensor product is given by $\rank(\mathcal F_1 \otimes \mathcal F_2) = \rank(\mathcal F_1) \rank(\mathcal F_2) ;$

\item if $\mathcal F_1$ and $\mathcal F_2$ are infinitame, then so is their tensor product, and its conductor satisfies 
\[c (\mathcal F_1 \otimes \mathcal F_2) \leq c(\mathcal F_1) \rank(\mathcal F_2) + \rank(\mathcal F_1) c(\mathcal F_2) . \]


\end{enumerate}

\end{lem}

\begin{proof} 

To verify (1), note that for every closed point $x \in \mathbb{A}^1_{\kappa}$ we have a $\Frob_{x,\kappa(x)}$-equivariant isomorphism
\begin{equation} \label{TensorProdStalkCOPY}
\left( \mathcal F_1 \otimes \mathcal F_2 \right)_{\overline x} \cong
\mathcal{F}_{1, \overline x} \otimes  \mathcal F_{2,\overline{x} } 
\end{equation}
so in case $x$ is $\kappa$-valued, from \cref{TensorProdStalkCOPY} we get
\begin{equation*}
\begin{split}
t_{ \mathcal F_1 \otimes \mathcal F_2 } (x) &= 
\tr(\Frob_{x,\kappa}, \left( \mathcal F_1 \otimes \mathcal F_2\right)_{\overline x}) = 
\tr(\Frob_{x,\kappa}, \mathcal{F}_{1, \overline x} \otimes  \mathcal F_{2,\overline{x}} ) \\
&= \tr(\Frob_{x,\kappa}, \mathcal{F}_{1, \overline x}) \tr(\Frob_{x,\kappa}, \mathcal{F}_{2, \overline x}) \\
&= t_{ \mathcal F_1}(x) t_{\mathcal F_2 }(x)  
\end{split}
\end{equation*}
so (1) is established.

We further see from \cref{TensorProdStalkCOPY} that the eigenvalues of $\Frob_{x,\kappa(x)}$ on the stalk of the tensor product are products of the eigenvalues on $\mathcal{F}_{1, \overline{x}}$ and $\mathcal{F}_{2, \overline{x}} $. Since the product of complex numbers of norm at most $|\kappa|^{\frac{0}{2}}$ has norm at most $|\kappa|^{\frac{0}{2}}$, 
this verifies (4).

To check (3), let $\eta \in  \mathbb{A}^1_\kappa$ be the generic point, and note that (as in \cref{TensorProdStalkCOPY}) we have an isomorphism
\begin{equation} \label{GenericTensorProdStalk}
\left( \mathcal{F}_1 \otimes \mathcal F_2 \right)_{\overline \eta} \cong
\mathcal{F}_{1, \overline \eta} \otimes \mathcal{F}_{2, \overline \eta} 
\end{equation}
of representations of $\Gal(\kappa(T)^{\text{sep}}/\kappa(T))$.
In particular this is an isomorphism of representations of the wild inertia subgroup $P_\infty$.
By the tameness assumption, the latter subgroup acts trivially on each of the factors in the right hand side of \cref{GenericTensorProdStalk}, so it acts trivially on their tensor product, hence it also acts trivially on the left hand side of \cref{GenericTensorProdStalk}.
This triviality of the action of $P_\infty$ is the desired tameness of the sheaf $\mathcal{F}_1 \otimes \mathcal F_2$ at $\infty$.

Let $i \in \{1,2\}$. If $\mathcal{F}_i$ has no finitely supported sections, 
then the natural map $\mathcal{F}_{i, \overline x} \to \mathcal{F}_{i, \overline \eta}^{I_x}$ is injective for every closed point $x \in \mathbb{A}^1_\kappa$. Since the tensor product of two injective maps of vector spaces is injective, 
we get from \cref{TensorProdStalkCOPY} and \cref{GenericTensorProdStalk} that the mappings
\begin{equation*}
\left( \mathcal{F}_1 \otimes \mathcal{F}_2  \right)_{\overline x} \cong 
\mathcal{F}_{1, \overline x} \otimes\mathcal{F}_{2, \overline x}  \to 
\mathcal{F}_{1, \overline \eta}^{I_x} \otimes\mathcal{F}_{2, \overline \eta}^{I_x}\to
\left( \mathcal{F}_{1, \overline \eta} \otimes  \mathcal{F}_{2 , \overline \eta} \right)^{I_x} \cong
\left( \mathcal{F}_1  \otimes  \mathcal F_2  \right)_{\overline \eta}^{I_x}
\end{equation*}
are all injective, hence $\mathcal{F}_1 \otimes \mathcal F_2 $ has no finitely supported sections, so (2) is established.

For (5), we use \cref{GenericTensorProdStalk}  to conclude that
\begin{equation} \label{TensoredRankCalculatedEq} 
\begin{split}
\rank(\mathcal{F}_1 \otimes \mathcal F_2 ) &= \dim \left( \mathcal{F}_1 \otimes \mathcal F_2 \right)_{\overline \eta} =
\dim (\mathcal{F}_{1, \overline \eta} \otimes \mathcal{F}_{2, \overline \eta}  ) \\
&= \dim(\mathcal{F}_{1, \overline \eta}) \dim(\mathcal{F}_{2, \overline \eta }) = \rank(\mathcal{F}_1 ) \rank(\mathcal F_2).
\end{split}
\end{equation}

Now we check (6).
For any closed point $x \in \mathbb{A}^1_\kappa$, we have by \cref{SwanTensorProdBoundCor}
\begin{equation} \label{SwanIgnoredTensorEq}
\swan_x(\mathcal{F}_1 \otimes \mathcal F_2) \leq \swan_x (\mathcal F_1) \rank(\mathcal F_2) + \rank(\mathcal F_1) \swan_x (\mathcal F_2)  .
\end{equation}
By definition of the conductor in \cref{InfinitameConductorDefEq} we have
\[c(\mathcal{F}_1 \otimes \mathcal F_2 ) = 
\sum_{ x \in |\mathbb A^1_{\kappa}|} [\kappa(x) : \kappa] 
(\rank(\mathcal{F}_1 \otimes \mathcal F_2 ) - \dim (\mathcal{F}_1 \otimes \mathcal F_2 )_{\overline{x}} 
\ + \swan_x(\mathcal{F}_1 \otimes \mathcal{F}_2 ) ) \] 
which in view of \cref{TensorProdStalkCOPY}, \cref{TensoredRankCalculatedEq}, and \cref{SwanIgnoredTensorEq}, is at most
\begin{equation} \label{InfinitameConductorTensorProductIneq}
\sum_{ x \in |\mathbb A^1_{\kappa}|} [\kappa(x) : \kappa] 
(\rank(\mathcal{F}_1 )  \rank(\mathcal F_2)  - \dim \mathcal{F}_{1, \overline{x}}  \dim \mathcal{F}_{2 , \overline{x}}  
\ + \swan_x (\mathcal F_1) \rank(\mathcal F_2) + \rank(\mathcal F_1) \swan_x (\mathcal F_2)   ). 
\end{equation}

On the other hand
\begin{equation*}
\begin{split}
&c(\mathcal F_1) \rank(\mathcal F_2) + \rank(\mathcal F_1) c(\mathcal F_2)  =  \\
&\sum_{ x \in |\mathbb A^1_{\kappa}|} [\kappa(x) : \kappa] \left(   ( \rank(\mathcal F_1) - \dim \mathcal{F}_{1, \overline{x}} + \swan_x (\mathcal F_1) )  \rank(\mathcal F_2) +\rank(\mathcal F_1)  ( \rank(\mathcal F_2) - \dim \mathcal{F}_{2, \overline{x}} + \swan_x (\mathcal F_2) ) \right)  
\end{split}
\end{equation*}
which comparing term-by-term, is larger than \cref{InfinitameConductorTensorProductIneq} by 
\[  \sum_{ x \in |\mathbb A^1_{\kappa}|} [\kappa(x) : \kappa]  (\rank ( \mathcal F_1)- \dim (\mathcal F_{1,x}) )  (\rank ( \mathcal F_2)- \dim (\mathcal F_{2,x}) )  \geq 0\] since $\mathcal F_1$ and $\mathcal F_2$ have no finitely supported sections. \end{proof}

\subsection{Artin-Schreier sheaves}

\subsubsection{Residues, exponentiation, additive characters} \label{REAC}
A variant of some of the material presented here can also be found in \cite{Hay}.

Each rational function $a \in \F_q(u)$ has a unique expansion
\begin{equation}
a(u) = \sum_{i = -\infty}^{\infty} a_i \cdot \left( \frac{1}{u} \right)^i
\end{equation}
as a Laurent series with $a_i \in \F_q$,
such that $a_i = 0$ for all but finitely many negative $i \in \mathbb{Z}$.
Using the $i=1$ coefficient $a_1$, we set
\begin{equation}
e(a) = \exp \left( \frac{2\pi i \cdot \mathrm{Tr}_{\F_q/\F_p}(a_1)}{p} \right)
\end{equation}
where we have identified $\F_p$ with $\{0,1, \dots, p-1\} \subseteq \mathbb{Z}$.
An alternative definition of $a_1$ in terms of the residue at infinity is
\begin{equation}
a_1 = - \mathrm{Res}_\infty(a).
\end{equation}

To get an explicit expression (or yet another equivalent definition) for $a_1$ write $a = M/N$ with $M,N \in \F_q[u]$,
and let $\widetilde M$ be the reduction of $M$ mod $N$ (represented by a unique polynomial of degree less than $\deg(N)$).
Then $a_1$ equals the coefficient of $u^{\deg(N)-1}$ in $\widetilde{M}$ (this is $0$ if there is no such coefficient) divided by the leading coefficient of $N$. 
In particular, for a polynomial $a \in \F_q[u]$ we have $a_1 = 0$ and thus $e(a) = 1$.
One also readily checks that $e(a+b) = e(a)e(b)$ for any $a,b \in \F_q(u)$.

We say that a function $\psi \colon \F_q[u]/(N) \to \mathbb{C}^\times$ is an additive character if 
\begin{equation}
\psi(f+g) = \psi(f)\psi(g), \quad f,g \in \F_q[u]/(N).
\end{equation}
Using the nondegeneracy of the bilinear map $(x,y) \mapsto \mathrm{Tr}_{\F_q/\F_p}(xy)$, we see that the additive characters are given by
\begin{equation} \label{AllAdditiveCharsEq}
\psi_h(M) = e \left( \frac{hM}{N} \right), \quad h \in \F_q[u]/(N).
\end{equation}

\subsubsection{Construction and Properties}

Our construction of Artin-Schreier sheaves will be analogous to that of Kummer sheaves.
Both constructions are special cases of the Lang isogeny construction.

\begin{notation} \label{AS-sheaf}

Let $\kappa$ be a finite field of characteristic $p$, 
let $\psi \colon \kappa \to \overline{\mathbb{Q}_\ell}^\times$ be a nontrivial additive character, 
and let $w \in \kappa(X)$ be a rational function. 
We construct an $\ell$-adic sheaf $\mathcal L_{\psi}(w)$, on the affine line $\mathbb A^1_\kappa = \mathrm{Spec} \ \kappa[X]$, as follows.

Write $w = \frac{a}{b}$ with $a,b \in \kappa[X]$ coprime, and $b \neq 0$.
Let
\begin{equation}
U = \{x \in \mathbb{A}^1_\kappa : b(x) \neq 0\}
\end{equation}
be the complement of the set $S$ of poles of $w$, 
and consider the finite \'{e}tale cover of $U$ defined by the equation
\begin{equation}\label{Lang2}
Y^{|\kappa|} - Y = w(X).
\end{equation}
The additive group of $\kappa$ acts on our cover (by automorphisms) via translation on $Y$, 
since for every $\lambda \in \kappa$ we have $(Y+ \lambda)^{|\kappa|} - (Y+ \lambda) = Y^{|\kappa|} - Y$ in the polynomial ring $\kappa[Y]$.
We thus get a simply transitive action of $\kappa$ on the fiber of any geometric point $\overline{x}$ lying over any point $x \in U$.

Arguing as in the construction of Kummer sheaves, we get a continuous homomorphism $ \pi_1^{\text{\'{e}t}}(U,\overline{x})\to \kappa$, so composing with $\psi$ gives rise to a lisse $\ell$-adic sheaf of rank one on $U$.
We define $\mathcal{L}_{\psi}(w)$ to be the extension by zero of this sheaf from $U$ to $\mathbb{A}^1_\kappa$.

\end{notation}

We shall now establish some properties of Artin-Schreier sheaves.
For the study of local invariants, we will use not only the upper numbering ramification filtration used so far, but also the lower numbering filtration,
as defined for instance in \cite[Definition 3.31]{KR}. 
For a comparison of these filtrations see \cite[Definition 3.52]{KR}

\begin{lem}\label{AS-sheaf-properties} 

The sheaf $\mathcal L_\psi(w)$ on $\mathbb A^1_\kappa$ has the following properties.

\begin{enumerate}

\item For every $x \in \kappa$ with $b(x)\neq 0$, we have $t_{\mathcal L_\psi(w) } (x) = \psi(w (x) )$;

\item The sheaf $\mathcal L_\psi(w)$ is lisse on $U$, and vanishes on $S = \mathbb{A}^1_\kappa - U$.

\item Unless $\deg(a)-\deg(b)$ is a positive multiple of $p$, we have
\begin{equation*}
\textup{slope}_{\infty}(\mathcal L_\psi(w)) = \max\{\deg(a) - \deg(b), 0\}.
\end{equation*}
In particular, if $\deg(a) \leq \deg (b)$, then the sheaf $\mathcal L_\psi(w)$ is tamely ramified at $\infty$. 

\item The sheaf $\mathcal L_\psi(w)$ is mixed of nonpositive weights.

\item The sheaf $\mathcal L_\psi(w)$ has no finitely supported sections;

\item We have $\rank(\mathcal L_{\psi(w)})=1$. 
In case $\deg (a) \leq \deg (b)$ and the multiplicity of every root of $b$ is prime to $p$, we also have
\begin{equation*}
c(\mathcal L_{\psi(w)}) = \#\{x \in \overline{\kappa} : b(x)  = 0\}  +\deg (b).
\end{equation*} 

\item There exists a unique $\alpha \in \kappa$ such that $\psi(x^p) = \psi(\alpha x)$ for every $x \in \kappa$,
and the sheaf $\mathcal{L}_\psi(w)$ is geometrically trivial on $U$ if and only if there exists an $f \in \kappa(X)$ such that 
$w = f^p - \alpha f$.



\end{enumerate}

\end{lem}

\begin{proof}
Property  (2) is immediate from our construction. As in \cref{Kummer-sheaf-properties}, (5) is immediate from (2). 

To verify (1), note first that $x \in U$.
We get from \cref{Lang2} that the Frobenius element $\Frob_{x,\kappa} \in \pi_1^{\text{\'{e}t}}(U,\overline{x})$ acts on the geometric fiber over $x$ by 
\begin{equation}
\Frob_{x,\kappa}(x,y) = (x^{|\kappa|}, y^{|\kappa|}) = (x,y^{|\kappa|}) = (x,w(x) + y).
\end{equation}
Hence, by our definition of the representation giving rise to the sheaf $\mathcal{L}_{\psi}(w)$, 
the element $w(x) \in \kappa$ is associated to $\Frob_{x,\kappa}$,
so $\Frob_{x,\kappa}$ is mapped to $\psi(w(x))$ as desired.

For (3), first consider the case when $\deg (a) \leq \deg (b)$. Then $w= \frac{a}{b}$ lies in the \'{e}tale local ring of $\mathbb P^1$ at $\infty$, so by the Henselian property all roots of 
\begin{equation}
Y^{|\kappa|}- Y = \frac{a(X)}{b(X)}
\end{equation}
lie in that ring.
Hence, the extension adjoining such a root is unramified, thus invariant under $I^s_{\infty}$ for all $s\geq 0$, and in particular has slope $0$.

Next consider the case when $\deg (a) - \deg (b)$ is positive and prime to $p$,
for which we use the argument of \cite[Example 1.1.7]{La81}.
The completion at $\infty$ of the function field of $\mathbb P^1$ admits a valuation $v$ satisfying
\begin{equation}
v(X)=-1, \quad v \left( \frac{a(X)}{b(X)} \right) = \deg (b) - \deg (a).
\end{equation}
Adjoining a root $y$ of \cref{Lang2}, we can extend our valuation by setting 
\begin{equation}
v(y) = \frac{ \deg (b) - \deg (a)}{|\kappa|}.
\end{equation} 

Since $\deg (b) - \deg (a)$ is prime to $|\kappa|$, there exist integers $j_1,j_2$ such that 
\begin{equation}
j_1   \frac{ \deg (b) - \deg (a)}{|\kappa|} - j_2 = \frac{1}{|\kappa|}. 
\end{equation}
Consequently $v ( y^{j_1} X^{j_2}) = \frac{1}{|\kappa|}$ so $y^{j_1} X^{j_2}$ is a uniformizer. 
Every nontrivial element $\sigma$ of the Galois group $G$ of our local extension sends $y$ to $y+c$ for some 
$c \in \kappa^{\times}$ so 
\begin{equation}
\sigma ( y^{j_1} X^{j_2}) = (y+c) ^{j_1} X^{j_2}= y^{j_1} X^{j_2} 
\left(1 +  \sum_{m=1}^{j_1} \frac{c^m {j_1 \choose m}}{y^m} \right). 
\end{equation}
Therefore 
\begin{equation*}
v (  \sigma ( y^{j_1} X^{j_2})  - y^{j_1} X^{j_2}) = 
v \left( y^{j_1} X^{j_2} \sum_{m=1}^{j_1} \frac{c^m {j_1 \choose m}}{y^m} \right) = 
\frac{1}{|\kappa|} - \frac{ \deg (b) - \deg (a)}{|\kappa|} 
\end{equation*} 
so by the definition of the lower numbering, 
$\sigma$ lies in $G_{ \deg (a) - \deg (b)}$ but not in $G_{\deg (a) -\deg (b) + 1}$. 
It follows that the slope is $\deg (a) - \deg(b)$.



To get (4), note that for a closed point $x \in U$,
every eigenvalue of $\text{Frob}_{x, \kappa(x)}$ is a value of the finite order character $\psi$, 
hence a root of unity whose norm is thus $1= |\kappa(x)|^{0/2}$, and for $x\notin U$, there are no Frobenius eigenvalues at all. 
Hence $\mathcal{L}_\psi(w)$ is punctually pure of weight $0$ and thus mixed of nonpositive weights.

To get the first part of (6), recall from $(2)$ that $\mathcal L_\psi(w)$ is lisse on $U$,
hence it is lisse at a geometric generic point $\overline{\eta}$ of $U$ (and of $\mathbb{A}^1_\kappa$).
Therefore the dimension of $\mathcal{L}_\psi(w)_{\overline{\eta}}$ is the rank of the representation giving rise to it, which is $1$.
For the second part of (6), because $\deg (a) \leq \deg (b)$, 
the sheaf $\mathcal L_\psi(w)$ is infinitame by (3), (4) and (5). 
Because $\mathcal L_\psi(w)$ has rank $1$, its Swan conductor at each point is equal to its slope. 

To  calculate $\text{slope}_x(\mathcal{L}_{\psi}(w))$ for a closed point $x \in \mathbb A^1_\kappa$, 
we choose a $\overline{\kappa}$-point $\overline{x}$ lying over $x$, 
and perform a change of variable that sends $\infty$ to $\overline{x}$, 
replacing $X$ with $\overline{x} + \frac{1}{X}$. 
The degree in $X$ of the rational function 
\begin{equation}
\frac{ a(\overline{x} + \frac{1}{X}) }{ b( \overline{x}+ \frac{1}{X} )}
\end{equation}
is equal to the order of vanishing of $b$ at $\overline{x}$ minus the order of vanishing of $a$ at $\overline{x}$. By (3), the slope of $\mathcal L_\psi(w)$ at $x$ is the maximum of this degree and $0$, which is the order of vanishing of $b$ at $\overline{x}$ since $a$ and $b$ are coprime.

By the definition in \cref{InfinitameConductorDefEq} we therefore have
\begin{equation}
\begin{split}
c(\mathcal{L}_\psi (w)) &= \sum_{x \in |\mathbb{A}^1_\kappa|} [\kappa(x) : \kappa](1 - \dim \mathcal{L}_\chi(w)_{\overline{x}} + \operatorname{ord}_{\overline{x}}(b) ) \\
&=   \sum_{x \in |U|} [\kappa(x) : \kappa](1 - 1) + \sum_{x \in |S|} [\kappa(x) : \kappa](1  + \operatorname{ord}_{\overline{x}}(b) ) \\
&= \#\{\overline{x} \in \overline{\kappa} : b(\overline{x}) = 0\}   + \deg (b)
\end{split}
\end{equation}
where $\operatorname{ord}_{\overline{x}}(b)$ denotes the order of vanishing of $b$ at $\overline{x}$.
\end{proof}

\begin{notation} \label{TheAnalyticBarNotation}

For a nonzero polynomial $g \in \F_q[u]$ and a polynomial $x$ in $\F_q[u]$ coprime to $g$, 
we denote by $\overline{x} \in \F_q[u]$ the unique polynomial of degree less than $\deg(g)$ satisfying
\begin{equation}
x \overline{x} \equiv 1 \mod g.
\end{equation}

\end{notation}

\begin{prop} \label{ExpsToTracesProp}

Let $g \in \F_q[u]$ be squarefree, and let $h \in \F_q[u]$.
Then there exists an infinitame $g$-periodic trace function $t \colon \F_q[u] \to \mathbb{C}$ with
\begin{equation}
r(t) \leq 1, \quad c(t) \leq 2,
\end{equation}
and
\begin{equation}
t(x) = e \left( \frac{h\overline{x}}{g} \right)
\end{equation}
for every $x \in \F_q[u]$ that is coprime to $g$.

\end{prop}

\begin{proof}

We induct on the number of distinct prime factors of $g$.
In the base case, where $g$ is prime, by \cref{AllAdditiveCharsEq} and \cref{AS-sheaf-properties}(1) we have
\begin{equation}
e \left( \frac{h\overline{x}}{g} \right) = \psi_1(hx^{-1}) = t_{\mathcal{L}_{\psi_1}(hX^{-1})}(x).
\end{equation}
The fact that this is an infinitame trace function,
and the requisite bounds on the rank and conductor follow from \cref{AS-sheaf-properties}(3,4,5,6).

Suppose now that $g = g_1g_2$ is a nontrivial factorization. 
Since $g$ is squarefree, the polynomilas $g_1,g_2$ are coprime, so there exist $a,b \in \F_q[u]$ with
\begin{equation}
ag_1 + bg_2 = 1.
\end{equation}
We then have
\begin{equation}
e \left( \frac{h\overline{x}}{g} \right) = e \left( \frac{h\overline{x}}{g_1g_2} \right) = 
e \left( \frac{b h\overline{x}}{g_1} +  \frac{a h \overline{x}}{g_2} \right) =
e \left( \frac{bh\overline{x}}{g_1} \right) e \left( \frac{a h\overline{x}}{g_2} \right). 
\end{equation}
By the induction hypothesis, we have a product of an infinitame $g_1$-periodic trace function and an infinitame $g_2$-periodic trace function with ranks at most $1$ and conductors at most $2$.
By \cref{TraceFuncDefi}, we get an infinitame $g$-periodic trace function satisfying the rank and conductor bounds.
\end{proof}

\subsection{Abhyankar's Lemma}

The following version of Abhyankar's lemma, which follows from \cite{SGA1}[XIII 5.2], will be of use to us.

\begin{lem} \label{AbhLem}

Let $X$ be a smooth scheme over $\overline{\F_q}$, let $D$ be a smooth divisor on $X$, 
let $\mathcal F$ be a lisse sheaf on $X - D$ with tame monodromy around $D$, 
let $f \colon X \to \mathbb{A}^1$ be a map whose zero locus is $D$ with restriction 
$f_0 \colon X - D \to \mathbb{G}_m$, and let $x \in D$. 
Suppose that $f$ vanishes to order one on $D$.

Then there exists a lisse sheaf $\mathcal L$ on $\mathbb{G}_m$ such that $\mathcal{F}$ and $f_0^* \mathcal {L}$ 
become isomorphic upon restriction to some punctured \'{e}tale neighborhood of $x$.

\end{lem}

\begin{proof} 

Let $R$ be the \'{e}tale local ring of $X$ at $x$, which contains the function $f$. We can pull $\mathcal F$ back to $\operatorname{Spec} R [ f^{-1}]$, where it becomes a representation of the tame fundamental group of $\operatorname{Spec} R [ f^{-1}]$. By \cite{SGA1}[XIII 5.3], the tame fundamental group is $\prod_{\ell \neq p } \mathbb Z_\ell$, with the isomorphism obtained from the covers taking prime-to-$p$ power roots of $f$.  On the other hand, the tame fundamental group of $\mathbb G_m$ is  also $\prod_{\ell \neq p} \mathbb Z_\ell$, with the isomorphism obtained from the covers taking prime-to-$p$ power roots of the coordinate. So we can view this representation of $\prod_{\ell \neq p} \mathbb Z_\ell$ as a lisse sheaf $\mathcal L$ on $\mathbb G_m$, whose pullback to $\operatorname{Spec} R [ f^{-1}]$ is isomorphic to the pullback of $\mathcal F$. This isomorphism must then be witnessed on some particular \'{e}tale cover.
\end{proof}

\section{Short sums of trace functions}

This section is devoted to proving \cref{intro-trace-interval-bound}.

\subsection{Vanishing of cohomology}

Here we obtain a vanishing of cohomology result, which is a key input to the Grothendieck-Lefschetz trace formula. 


\begin{notation}

For a variety $X$ we will be working with the bounded derived category $D^b_c(X,  \overline{\mathbb Q_\ell})$. 
We use notation such as $f_*, f^*$, for the derived pushforward and pullback, never the operations on individual sheaves. 
This is to avoid continually writing $R f_* , R f^*, Rf_!$, etc. to refer to these operations.

Note that $f^*$ always sends sheaves to sheaves, as does $\otimes$, and $f_!$ sends sheaves to sheaves if $f$ is an open immersion, so when only these operations have been applied, we will be working with usual sheaves (as opposed to complexes).
For brevity of notation, we also occasionally denote the stalk of a sheaf $\mathcal{F}$ at a geometric point $\overline{x}$ lying over a point $x$ by $\mathcal{F}_x$.

\end{notation}

\begin{notation}\label{geometric-setup}

Let $g \in \overline{\mathbb F_q}[u]$ be a squarefree polynomial of degree $m \geq 1$, 
and let $x_1,\dots, x_m \in \overline{\F_q}$ be the roots of $g$. 
For each $1 \leq i \leq m$, let $\mathcal F_i$ be a sheaf on $\mathbb{A}^1 = \mathbb A^1_{\overline{\mathbb F_q}}$. 
Our interest here is in infinitame trace functions, so we assume for all $i$ that

\begin{itemize}

\item the sheaf $\mathcal F_i$ has no finitely supported sections; 

\item the sheaf $\mathcal F_i$ has tame local monodromy at $\infty$, or in other words, it is tamely ramified at infinity.

\end{itemize}

For one of the sheaves, say $\mathcal F_1$, we make a more stringent assumption.
Assume that $\mathcal F_1$ is the extension by zero of some nontrivial (continuous) one-dimensional $\overline{\mathbb{Q}_\ell}$-representation of the tame \'{e}tale fundamental group
\begin{equation}
\pi_1^{\text{tame}} \left( \mathbb A^1  \setminus \{z\} \right) 
\cong \prod_{\ell \neq p} \mathbb Z_\ell
\end{equation}
for some $z \in \overline{\mathbb F_q}$.
This is a geometric form of the assumption in \cref{intro-trace-interval-bound} that for some prime we have a Dirichlet trace function.
The formulation of this assumption is motivated in part by \cref{Kummer-sheaf-properties}(7).

Let $n \leq m$ be a nonnegative integer, 
view $\mathbb A^n = \mathbb A^n_{\overline{\F_q}}$ as the space of polynomials over $\overline{\F_q}$ of degree less than $n$, 
and for every $1 \leq i \leq m$ let 
\begin{equation} \label{EvalEq}
e_i \colon \mathbb A^n \to \mathbb A^1, \ \quad e_i( a_0,\dots, a_{n-1}) = a_0 + a_1 x_i + \dots + a_{n-1} x_i^{n-1}, 
\end{equation}
be the (linear) map that evaluates a polynomial at $x_i$.

Our first goal is to prove (in Corollary \ref{geometric-vanishing}) a vanishing statement for the compactly supported cohomology groups
$H^j_c \Bigl( \mathbb A^n, \bigotimes_{i=1}^m e_i^* \mathcal F_i \Bigr) $.
To do this, view $\mathbb A^n$ as the complement of a hyperplane $H_\infty$ in $\mathbb P^n = \mathbb{P}^n_{\overline{\F_q}}$. 
Let $H_1$ be the hyperplane in $\mathbb P^n$ obtained as the closure of the hyperplane $e_1^{-1}( z)$ in $\mathbb A^n$.  
Let 
\begin{equation}
\begin{split}
&w \colon \mathbb P^n \setminus (H_1 \cup H_\infty) \to \mathbb P^n \setminus H_\infty = \mathbb A^n, 
\quad v \colon \mathbb P^n \setminus H_1 \to  \mathbb P^n \\
&u \colon \mathbb P^n \setminus (H_1 \cup H_\infty)  \to \mathbb P^n \setminus H_1
\end{split}
\end{equation}
be the natural open immersions and
\begin{equation} 
d \colon H_1 \to \mathbb P^n 
\end{equation}
the closed immersion.
\end{notation}

Since $\mathcal{F}_1$ is an extension by zero, 
its stalk at $z$ vanishes, so the stalk of $e_1^* \mathcal F_1$ vanishes for every $x \in e_1^{-1}(z) \subseteq \mathbb{A}^n$.
Hence
\begin{equation}
\bigotimes_{i=1}^m e_i^* \mathcal F_i = w_! w^* \bigotimes_{i=1}^m e_i^* \mathcal F_i.
\end{equation} 
The above implies that for every nonnegative integer $j$ we have 
\begin{equation} \label{letter-switching} 
\begin{split}
H^j_c \Bigl( \mathbb A^n, \bigotimes_{i=1}^m e_i^* \mathcal F_i \Bigr) &= H^j_c \Bigl(\mathbb P^n \setminus (H_1 \cup H_\infty) ,  w^* \bigotimes_{i=1}^m e_i^* \mathcal F_i \Bigr) \\  
&= H^j_c \Bigl(\mathbb P^n \setminus H_1 ,  u_! w^* \bigotimes_{i=1}^m e_i^* \mathcal F_i \Bigr). 
\end{split}
\end{equation}
Our strategy will be focused on the following excision long exact sequence.

\begin{lem} \label{excision}
We have the long exact sequence
\begin{equation*}
\begin{split} 
\dots \to H^*_c \Bigl(\mathbb P^n \setminus H_1 ,  u_! w^* \bigotimes_{i=1}^m e_i^* \mathcal F_i \Bigr) &\to 
H^*  \Bigl(\mathbb P^n \setminus H_1 ,  u_! w^* \bigotimes_{i=1}^m e_i^* \mathcal F_i \Bigr) \\
&\to H^* \Bigl(H_1  ,   d^* v_* u_! w^* \bigotimes_{i=1}^m e_i^* \mathcal F_i \Bigr) \to \dots
\end{split}
\end{equation*}
\end{lem}

\begin{proof} 
By \cite[XVII, (5.1.16.2)]{sga4-3} we have a long exact sequence 
\begin{equation} 
\begin{split} 
H^* \Bigl(\mathbb P^n  ,   v_! v^* v_*  u_! w^* \bigotimes_{i=1}^m e_i^* \mathcal F_i \Bigr) &\to 
H^*  \Bigl(\mathbb P^n  , v_*  u_! w^* \bigotimes_{i=1}^m e_i^* \mathcal F_i \Bigr) \\
&\to H^* \Bigl(\mathbb P^n  , d_*  d^* v_* u_! w^* \bigotimes_{i=1}^m e_i^* \mathcal F_i \Bigr). 
\end{split}
\end{equation} 
combined with (in the first term) the Leray spectral sequence for $v_!$ and the fact that $v^* v_* = \mathrm{id}$, (in the second term) the Leray spectral sequence for $v_*$, and (in the third term) the Leray spectral sequence for $d_*$.
\end{proof}

To that end, our first order of business will be understanding the complex 
\begin{equation}
d^* v_* u_! w^* \bigotimes_{i=1}^m e_i^* \mathcal F_i
\end{equation}
of sheaves on $H_1$ appearing in the exact sequence above.

Given a sheaf $\mathcal{F}$ on a variety $X$, and a point $x \in X$, we say that $\mathcal{F}$ is lisse at $x$ if there is a neighborhood $N$ of $x$ such that the restriction $\mathcal{F}|_N$ is a lisse sheaf. If $\mathcal{F}$ is not lisse at $x$, we say that $x$ is a singular point of $\mathcal{F}$. 

\begin{lem}\label{affine-vanishing-lemma}

Let $Z \subseteq \mathbb{A}^n$ be the set of those polynomials $f$ for which
\begin{equation}
\#\{ 1 \leq i \leq m : f(x_i) \ \text{is a singular point of} \ \mathcal{F}_i \} > n.
\end{equation}
The set $Z$ is finite, 
and the restriction of the complex 
\begin{equation}
v_* u_! w^*  \bigotimes_{i=1}^m e_i^* \mathcal F_i
\end{equation} 
to $H_1 \setminus ( H_{\infty}  \cap H_1)$ 
vanishes away from $Z$.

\end{lem}

\begin{proof} 

There are only finitely many subsets of $\{1,\dots , m \}$ of size at least $n$, 
and for each subset, only finitely many choices of a singular point of each $\mathcal F_i$.
Since there is at most one polynomial of degree less than $n$ that takes prescribed values for (at least) $n$ given points $x_i \in \overline{\F_q}$, it follows that the set $Z$ is finite.

Let $j \colon \mathbb P^n \setminus H_{\infty} \to \mathbb P^n$ be the natural inclusion. 
The Cartesian square
\begin{equation*}
\begin{tikzpicture}[scale=1.5]
\node (A) at (0,1) {$\mathbb{P}^n \setminus (H_1 \cup H_{\infty})$};
\node (B) at (2.2,1) {$\mathbb{P}^n \setminus H_{\infty}$};
\node (C) at (0,0) {$\mathbb{P}^n \setminus H_1$};
\node (D) at (2.2,0) {$\mathbb{P}^n$};
\path[->,font=\scriptsize,>=angle 90]
(A) edge node[above]{$w$} (B)
(A) edge node[right]{$u$} (C)
(B) edge node[right]{$j$} (D)
(C) edge node[above]{$v$} (D);
\end{tikzpicture}
\end{equation*}
gives $j^* v_* = w_* u^*$ by the smooth base change theorem, 
and $u$ is an open immersion so $u^* u_!$ is the identity, therefore 
\begin{equation}
j^* v_* u_! w^* = w_* u^* u_! w^* = w_* w^*.
\end{equation} 
Thus it suffices to show that the stalk of the complex
\begin{equation} \label{NewwSheaf}
w_* w^*  \bigotimes_{i=1}^m e_i^* \mathcal F_i 
\end{equation} 
vanishes for every point in $H_1$ which is neither in $H_\infty$ nor in $Z$. 

Fix a point $f_0 \in H_1 \setminus (H_1 \cap H_\infty)$, so that $f_0 \in \mathbb A^n = \mathbb{P}^n \setminus H_\infty$ and $f_0(x_1)=z$.
Let 
\begin{equation}
S = \{1 \leq i \leq m : f_0(x_i) \ \text{is a singular point of} \ \mathcal F_i\}
\end{equation}
and assume that $|S| \leq n$, so that $f_0 \notin Z$.
We show that the stalk of the complex in \cref{NewwSheaf} vanishes at $f_0$.  
For each $1 \leq i \leq m$ that is not in $S$, 
the sheaf $\mathcal{F}_i$ is lisse in a neighborhood of the point $f_0(x_i) = e_i(f_0) \in \mathbb{A}^1$,
so the sheaf $e_i^* \mathcal F_i$ is lisse in a neighborhood of $f_0$. 
Because the pushforward and pullback along an open immersion can be computed locally, 
and both commute with tensoring by a lisse sheaf, it suffices to prove that the complex
\begin{equation} \label{DoubleDoubleuSheafEq}
w_* w^* \bigotimes_{i \in S}  e_i^* \mathcal F_i.
\end{equation} 
vanishes at $f_0$.

Since $|S| \leq n$, the maps $\{e_i\}_{i \in S}$ are linearly independent,
so we can identify $\mathbb A^n$ with $(\mathbb A^1 )^{|S|} \times \mathbb A^{n- |S|}$ by using the $\{f(x_i)\}_{i \in S}$ as the coordinates of the first $|S|$ copies of $\mathbb A^1$.
Using this identification we can write
\begin{equation}
\bigotimes_{i \in S}  e_i^* \mathcal F_i \cong  \Bigl( \boxtimes_{i \in S } \mathcal F_i \Bigr ) 
\boxtimes \overline{\mathbb Q_\ell}
\end{equation} 
where $\overline{\mathbb{Q}_\ell}$ stands for a constant rank one sheaf on $\mathbb{A}^{n-|S|}$. 
It follows from our assumptions that $1 \in S$,
so we make a further identification of $\mathbb{A}^n$ with 
$\mathbb A^1 \times (\mathbb A^1)^{|S|-1} \times \mathbb A^{n-|S|}$, giving 
\begin{equation} \label{SheafIdentifEq}
\bigotimes_{i \in S}  e_i^* \mathcal F_i \cong \mathcal F_1\boxtimes  \Bigl( \boxtimes_{i \in S \setminus \{1\}} \mathcal F_i \Bigr ) 
\boxtimes  \overline{\mathbb Q_\ell}.
\end{equation} 

Taking $\overline{w} \colon \mathbb A^1 \setminus \{z\} \to \mathbb A^1$ to be the open immersion, 
our identifications give 
$w= \overline{w} \times \mathrm{id}^{ S \setminus \{1\}} \times \mathrm{id}$,
so by \cref{SheafIdentifEq}, the complex from \cref{DoubleDoubleuSheafEq} becomes
\begin{equation}
\left(  \overline{w} \times \mathrm{id}^{ S \setminus \{1\}} \times \mathrm{id} \right)_* 
\left(  \overline{w} \times \mathrm{id}^{ S \setminus \{1\}} \times \mathrm{id} \right)^* 
\Bigl(   F_1\boxtimes  \Bigl( \boxtimes_{i \in S \setminus \{1\}} \mathcal F_i \Bigr ) 
\boxtimes  \overline{\mathbb Q_\ell}  \Bigr).
\end{equation}
By the K\"{u}nneth formula, the above equals
\begin{equation*}
\begin{split}
\left(  \overline{w} \times \mathrm{id}^{ S \setminus \{1\}} \times \mathrm{id} \right)_* 
\Bigl(  \overline{w}^* \mathcal F_1 \boxtimes   \Bigl( \boxtimes_{i \in S\setminus \{1\} } \mathcal F_i \Bigr ) \boxtimes \overline{\mathbb Q_\ell}  \Bigr) = 
\overline{w}_* \overline{w}^* \mathcal F_1 \boxtimes  \Bigl( \boxtimes_{i \in S\setminus \{1\} } \mathcal F_i \Bigr ) \boxtimes \overline{\mathbb Q_\ell}.
\end{split}
\end{equation*}

As $\mathcal F_1$ has rank one with nontrivial monodromy around $z$, the stalk of $\overline{w}_* \overline{w}^* \mathcal F_1$ vanishes at $z$, so the stalk of our external tensor product vanishes at any $f \in \mathbb{A}^n$ with $f(x_1)=z$. 
In particular, it vanishes at $f_0$.
\end{proof}

\begin{lem}\label{projective-vanishing-lemma}

Keep \cref{geometric-setup}.
The restriction of the complex 
\begin{equation} \label{OurCmPlXeq}
v_* u_! w^*  \bigotimes_{i=1}^m e_i^* \mathcal F_i
\end{equation} 
to $H_{\infty}  \cap H_1$ vanishes.

\end{lem}

\begin{proof} 

We view the points of $\mathbb P^n$ as pairs $(f:t)$ of a polynomial $f$ over $\overline{\F_q}$ of degree less than $n$ and a scalar $t \in \overline{\F_q}$, 
not both zero, up to scaling.
Then
\begin{equation}
H_{\infty} = \{(f:t) \in \mathbb{P}^n \ | \ t=0\}, \quad H_1 = \{ (f,t) : f(x_1) - zt = 0\},
\end{equation}
and the map $e_i$ is given by the formula
\begin{equation} \label{NewFormulaForEi}
{e_i}(f:t) = \frac{f(x_i)}{t}. 
\end{equation}
Let $(f_0: 0) \in H_\infty \cap H_1$, 
put 
\begin{equation}
S = \{1 \leq i \leq m \ | \ f_0(x_i)=0 \},
\end{equation}
note that $1 \in S$, and that $|S| < n$ since $f_0 \neq 0$.

Since our goal is to establish the vanishing of the stalk of the complex in \cref{OurCmPlXeq} at $(f_0,0)$, we are free to restrict to an \'{e}tale neighborhood of $(f_0,0)$. We will first restrict to a Zariski open neighborhood with a convenient coordinate system, and then further restrict to an \'{e}tale neighborhood where the sheaves $e_i^* \mathcal F_i$ for $i \notin S$ become simpler.

Since $n \leq m$ there exists a subset 
\begin{equation}
S \subsetneq T \subseteq \{1, \dots, m\}, \quad |T| = n.
\end{equation}
As any polynomial of degree less than $n$ is uniquely determined by its values at $n$ distinct points, 
the set $\{f(x_i)\}_{i \in T} \cup \{t\}$ forms a projective coordinate system for $\mathbb P^n$. 

Fix $j \in T \setminus S$, and define the coordinates
\begin{equation}
c_1 = \frac{ f(x_1) - zt } { f(x_j)}, \ c_j = \frac{t}{ f(x_j)}, \ c_i  = \frac{ f(x_i)}{ f(x_j)}, \ i \in T \setminus \{1,j\}.
\end{equation}
This system of coordinates is obtained from the previous one by dividing all the coordinates by the coordinate $f(x_j)$ 
and then applying the linear translation $c_1 \mapsto c_1 - z c_j$.
Since $f_0(x_j) \neq 0$, it follows that $\{c_i\}_{i\in T}$ forms a coordinate system for the affine neighborhood 
\begin{equation}
U = \{(f : t) \in \mathbb{P}^n \ | \ f(x_j) \neq 0\}
\end{equation} 
of $(f_0,0)$ in $\mathbb P^n$. 
We let $\nu \colon U \to \mathbb P^n$ be the open immersion.

In our new coordinates, 
\begin{equation} \label{OldFunctionsNewCoordinates}
e_1 = \frac{c_1}{c_j} + z, \quad  e_i = \frac{c_i}{c_j}, \ i \in T \setminus \{1,j\}.
\end{equation}
Let $U_{1,j}$ be the locus in $U$ where $c_1$ and $c_j$ are nonzero, and let 
\begin{equation}
\xi \colon U_{1,j} \to \mathbb P^n \setminus ( H_1 \cup H_{\infty})
\end{equation}
be the open immersion, well-defined because $\nu^{-1}(H_1)$ is the vanishing locus of $c_1$ and $\nu^{-1}(H_\infty)$ is the vanishing locus of $c_j$. 
We can write 
\begin{equation} \label{FacturUprimeEq}
U_{1,j} =  (\mathbb A^1 \setminus \{0\}) \times (\mathbb A^{n-1} - \mathbb A^{n-2}) \subseteq U \cong \mathbb{A}^n
\end{equation}
where the coordinate in the first factor is $c_1$, 
and the coordinates in the second factor are $\{c_i\}_{i \in T\setminus \{1\}} $.
We will now express the restriction of the sheaf
\begin{equation}
\xi^* w^* \bigotimes_{i=1}^m e_i^* \mathcal F_i
\end{equation}
to a certain punctured \'{e}tale neighborhood of $(f_0 : 0)$ as the restriction to that neighborhood of the external tensor product of sheaves on each of the two factors of $U_{1,j}$ in \cref{FacturUprimeEq}.
We do this for each $1 \leq i \leq m$ separately, distinguishing between the cases $i \in S \setminus \{1\}$, $i \notin S$, and $i = 1$.

Following \cref{OldFunctionsNewCoordinates}, for $i \in S \setminus \{1\}$
we define the (restricted) map 
\begin{equation}
\overline{e_i} \colon \mathbb A^{n-1} - \mathbb A^{n-2} \to \mathbb A^1, \quad
\overline{e_i} \left( (c_k)_{k \in T\setminus\{1\}} \right) = \frac{c_{i} }{c_j},
\end{equation}
so that we have
\begin{equation} \label{FirstExternalDecomposition}
\xi^* w^* e_i^* \mathcal F_i \cong \overline{\mathbb Q_\ell} \boxtimes \overline{e}_i ^* \mathcal F_i, \quad i \in S \setminus \{1\}.
\end{equation}

Let $1 \leq i \leq m$ which is not in $S$. Let $Z_i$ be the intersection of $H_{\infty}$ with the vanishing locus of $f(x_i)$. 
Then, following \cref{NewFormulaForEi} we can extend $e_i$ to a map
\begin{equation}
\widehat{e_i} \colon \mathbb{P}^n \setminus Z_i \to \mathbb{P}^1, \quad \widehat{e_i}(f:t) = \frac{f(x_i)}{t}. 
\end{equation}
Geometrically, we can see that $H_{\infty}$ is a pole of $\widehat{e_i}$, the vanishing locus of $f(x_i)$ is the zero locus of $\widehat{e_i}$, and the intersection $Z_i$ is the indeterminacy locus. 

By our definition of $S$ and choice of $i$, the point $(f_0:0) \in \mathbb{P}^n \setminus Z_i$ lies in the pole and not in the indeterminacy locus, 
so the map $\widehat{e_i}$ is defined at that point, sending it to $\infty \in \mathbb P^1$. 
On some punctured Zariski neighborhood of $\infty$ in $\mathbb P^1$ the sheaf $\mathcal F_i$ is lisse and tamely ramified around the puncture $\infty$, 
so it follows that on some punctured neighborhood of $(f_0,0)$ in $\mathbb P^n$ the sheaf $e_i^* \mathcal F_i$ is lisse and tamely ramified at $H_\infty$. 
Hence, there exists a punctured neighborhood of $(f_0,0)$ on which the sheaf 
\begin{equation} \label{AnImportantSheafEq}
\mathcal{F}_{\not S} = \bigotimes_{i\notin S} e_i^* \mathcal F_i
\end{equation} 
is lisse and tamely ramified at $H_\infty$. 

More precisely, we puncture neighborhoods of $(f_0,0)$ by removing $H_{\infty}$, which is the inverse image of $\infty$ under $\widehat{e_i}$, so if $U^*$ is the Zariski neighborhood above, then its punctured form is $U' = U^* \setminus ( U^* \cap H_{\infty})$.  We can assume, by puncturing further if necessary, that $U^* \subseteq U$. 
We let $\alpha \colon U' \to \mathbb P^n \setminus H_{\infty}$ be the inclusion,
let $\delta \colon U^* \to \mathbb A^1$ be the restriction of $c_j$ to $U^*$,
and let $\delta_0: U' \to \mathbb G_m$ be the restriction of $c_j$ to $U'$. 

Note that $U^* \cap H_{\infty}$ is a smooth divisor where $\delta$ vanishes to order one.
We now apply Abhyankar's lemma, as stated in \cref{AbhLem}, to the scheme $U^*$, the divisor $U^* \cap H_\infty$, the sheaf $\alpha^*  \mathcal{F}_{\not S}$ on $U'$, and the map $\delta$. 
We conclude that there exists an \'{e}tale neighborhood $V \to U^*$ of $(f_0,0)$, 
giving rise to a map $\beta \colon V \times_{U^*} U' \to U'$, and a lisse sheaf $\mathcal L_{\not S}$ on $\mathbb G_m$, such that 
\begin{equation}
\beta^* \alpha^*  \mathcal{F}_{\not S} \cong  \beta^* \delta_0^* \mathcal L_{\not S}.
\end{equation}

Let $\tau \colon V \to U$ be the map induced by the composition $V \to U^* \to U$,
and let 
\begin{equation}
\zeta \colon  V \times_U U_{1,j} \to U_{1,j}, \quad \gamma \colon V \times_U U_{1,j} \to  V \times_{U} (U \setminus H_{\infty}) =
V \times_{U^*} U'
\end{equation}
so that we have
\begin{equation} \label{AppliedAbhyankarEq}
\gamma^* \beta^* \alpha^*  \mathcal{F}_{\not S} \cong  \gamma^*  \beta^* \delta_0^* \mathcal L_{\not S}
\end{equation}
and 
\begin{equation} \label{TwoSidedTripleCompositionEq}
w \circ \xi \circ \zeta = \alpha \circ \beta \circ \gamma.
\end{equation}

We further define
\begin{equation}
\pi \colon \mathbb A^{n-1} - \mathbb A^{n-2} \to \mathbb G_m = \mathbb{A}^1 \setminus \{0\}, \quad
\pi \left( (c_i)_{i \in T\setminus\{1\}} \right) = c_j,
\end{equation} 
and let $\mathrm{pr}_2 \colon U_{1,j} \to  \mathbb A^{n-1} - \mathbb A^{n-2} $ be the projection on the second factor in \cref{FacturUprimeEq}. 
Then 
\begin{equation} \label{AnotherTwoSidedTripleCompositionEq}
\delta_0 \circ \beta \circ \gamma = \pi \circ \mathrm{pr}_2 \circ \zeta
\end{equation}
because both compositions are given by the coordinate $c_j$. 
It follows from \cref{AnImportantSheafEq}, \cref{TwoSidedTripleCompositionEq}, \cref{AppliedAbhyankarEq}, and \cref{AnotherTwoSidedTripleCompositionEq} that
\begin{equation} \label{SecondExternalDecomposition}
\zeta^* \xi^* w^*  \bigotimes_{i\notin S} e_i^* \mathcal F_i  \cong \gamma^* \beta^* \alpha^*  \mathcal{F}_{\not S} \cong  \gamma^*  \beta^* \delta_0^* \mathcal L_{\not S} \cong \zeta^* \mathrm{pr}_2^* \pi^* \mathcal L_{\not S} \cong  \zeta^*   (  \overline{\mathbb Q_\ell}  \boxtimes \pi ^*\mathcal L_{\not S}  )  .
\end{equation}

We turn to the case $i=1$. 
Let $\mathcal L_\chi$ be the sheaf on $\mathbb G_m$ obtained by translating $\mathcal F_1$ by $z$. 
Then the sheaf $\xi^* w^* e_1^ * \mathcal F_1$ is the pullback of $\mathcal L_\chi$ by the map $\frac{c_1}{c_j}$. 
By the multiplicative properties of tame rank one lisse sheaves on $\mathbb G_m$, 
this is the tensor product of the pullback of $\mathcal L_\chi$ by $c_1$ with the pullback of its dual $\mathcal L_\chi^\vee$ by $c_j$, 
so
\begin{equation} \label{ThirdExternalDecomposition}
\xi^* w^* e_1^* \mathcal F_1 \cong \mathcal L_\chi \boxtimes \pi^* \mathcal L_\chi^\vee .
\end{equation} 
Combining \cref{FirstExternalDecomposition}, \cref{SecondExternalDecomposition}, and \cref{ThirdExternalDecomposition}, 
we see that
\begin{equation}\label{CombinedExternalDecomposition}
\zeta^* \xi^* w^*  \bigotimes_{i=1}^m  e_i^* \mathcal F_i   \cong \zeta^* \left(  \mathcal L_\chi \boxtimes \Bigl( \pi^* \mathcal L_\chi^\vee  \otimes \pi^* \mathcal L_{\not S } \otimes \bigotimes_{i\in S\setminus\{1\} } \overline{e}_i ^* \mathcal F_i \Bigr)  \right) .
\end{equation}

Let
\begin{equation}
\overline{v} \colon \mathbb A^1 \setminus \{0\} \to \mathbb A^1,
\quad \overline{u} \colon \mathbb A^{n-1} - \mathbb A^{n-2} \to \mathbb A^{n-1}
\end{equation}
be the inclusions of the loci where $c_1$ and $c_j$ do not vanish.  
We have the commutative diagram
\begin{equation*}
\begin{tikzpicture}[scale=1.5]
\node (A) at (0,0.8) {$\mathbb{P}^n \setminus (H_1 \cup H_{\infty})$};
\node (B) at (3.5,0.8) {$\mathbb{P}^n \setminus H_{1}$};
\node (C) at (0,0) {$(\mathbb A^1 \setminus \{0\})  \times (\mathbb{A}^{n-1} - \mathbb{A}^{n-2})$};
\node (D) at (3.5,0) {$(\mathbb{A}^1 \setminus \{0\}) \times \mathbb{A}^{n-1}$};
\node (E) at (6,0.8) {$\mathbb{P}^n$};
\node (F) at (6,0) {$ \mathbb{A}^1 \times \mathbb{A}^{n-1}$};
\node (G) at (0,-0.8) {$V \times_U U_{1,j}$};
\node (H) at (3.5,-0.8) {$V \times_U ( U \setminus \nu^{-1} H_1)$};
\node (I) at (6,-0.8) {$V$};
\node (J) at (-2,0.8) {$\mathbb{P}^n \setminus H_{\infty}$};
\path[->,font=\scriptsize,>=angle 90]
(A) edge node[above]{$u$} (B)
(C) edge node[right]{$\xi$} (A)
(D) edge node[right]{$$} (B)
(C) edge node[above]{$\text{id} \times \overline{u}$} (D)
(B) edge node[above]{$v$} (E)
(F) edge node[right]{$\nu$} (E)
(D) edge node[above]{$\overline{v} \times \text{id}$} (F)
(G) edge node[right]{$\zeta$} (C)
(G) edge node[above]{$\tilde{u}$} (H)
(H) edge node[right]{$$} (D)
(H) edge node[above]{$\tilde{v}$} (I)
(A) edge node[above]{$w$} (J)
(I) edge node[right]{$\tau$} (F);
\end{tikzpicture}
\end{equation*}
where all arrows represent \'{e}tale maps and all squares are Cartesian. 

We claim that we can make the following series of identifications
\begin{equation} \label{MultipleBaseChangeEq}
\begin{split} 
&\tau^* \nu^* v_* u_! w^*  \bigotimes_{i=1}^m  e_i^* \mathcal F_i  = 
\tilde{v}_* \tilde{u}_! \zeta^* \xi^* w^*  \bigotimes_{i=1}^m  e_i^* \mathcal F_i  = \\
&\tilde{v}_* \tilde{u}_! \zeta^*  \Bigl( \mathcal L_\chi \boxtimes \Bigl( \pi^* \mathcal L_\chi^\vee  \otimes \pi^* \mathcal L_{\not S } \otimes \bigotimes_{i\in S\setminus\{1\} } \overline{e}_i ^* \mathcal F_i \Bigr)\Bigr) = \\
&\tau^* (\overline{v} \times \text{id})_*  (\text{id} \times \overline{u} )_! \Bigl( \mathcal L_\chi \boxtimes \Bigl( \pi^* \mathcal L_\chi^\vee  \otimes \pi^* \mathcal L_{\not S } \otimes \bigotimes_{i\in S\setminus\{1\} } \overline{e}_i ^* \mathcal F_i \Bigr)\Bigr).  
\end{split}
 \end{equation}
 The first equality requires base change over all four squares of the above commutative diagram. For the top-right and bottom-right squares, we are base-changing a pushforward by a smooth map, and so we may apply the smooth base change theorem. For the top-left and bottom-left, we are base-chaning a compactly supported pushforward (also by a smooth map), and so we may apply the proper base change theorem. The second equality follows from \cref{CombinedExternalDecomposition}. The third equality requires base change along the bottom-left and bottom-right squares, which again uses the smooth and proper base change theorems.

Since the pullbacks $\tau^*$ and $\nu^*$ are compatible with stalks, 
it follows from \cref{MultipleBaseChangeEq} that the stalk of $v_* u_! w^*  \bigotimes_{i=1}^m  e_i^* \mathcal F_i   $ at $(f_0,0)$ is isomorphic to the stalk of 
\begin{equation}
(\overline{v} \times \text{id})_* (\text{id} \times \overline{u} )_! \Bigl( \mathcal L_\chi \boxtimes \Bigl( \pi^* \mathcal L_\chi^\vee  \otimes \pi^* \mathcal L_{\not S } \otimes \bigotimes_{i\in S\setminus\{1\} } \overline{e}_i ^* \mathcal F_i \Bigr)\Bigr)
\end{equation}
at $(f_0,0)$. To show that the latter stalk vanishes, we invoke K\"{u}nneth to get
\begin{equation} 
\begin{split} 
&(\overline{v} \times \text{id})_* (\text{id} \times \overline{u} )_! \Bigl( \mathcal L_\chi \boxtimes \Bigl( \pi^* \mathcal L_\chi^\vee  \otimes \pi^* \mathcal L_{\not S } \otimes \bigotimes_{i\in S\setminus\{1\} } \overline{e}_i ^* \mathcal F_i \Bigr)\Bigr) = \\
&(\overline{v} \times \text{id})_* \Bigl( \mathcal L_\chi \boxtimes \overline{u}_! \Bigl( \pi^* \mathcal L_\chi^\vee  \otimes \pi^* \mathcal L_{\not S } \otimes \bigotimes_{i\in S\setminus\{1\} } \overline{e}_i ^* \mathcal F_i \Bigr) \Bigr) = \\
&\overline{v}_*\mathcal L_\chi \boxtimes \overline{u}_!  \Bigl( \pi^* \mathcal L_\chi^\vee  \otimes \pi^* \mathcal L_{\not S } \otimes \bigotimes_{i\in S\setminus\{1\} } \overline{e}_i ^* \mathcal F_i \Bigr).
\end{split}
\end{equation}
Since we assumed that $\mathcal F_1$ has nontrivial local monodromy at $z$, 
the rank one sheaf $\mathcal L_\chi$ has nontrivial local monodromy at $0$, 
so the stalk of $\overline{v}_ * \mathcal L_\chi$ at $0$ vanishes.
We conclude that the stalk of the external tensor product above vanishes at every point $U$ with $c_1=0$.
In particular, it vanishes at $(f_0:0)$.

\end{proof}

%
%
%
%
%
%

\begin{lem}\label{perversity}  

Keep \cref{geometric-setup}.
The shifted sheaf 
\begin{equation}
\bigotimes_{i=1}^m e_i^* \mathcal F_i [n]
\end{equation}
is a perverse sheaf. 

\end{lem} 

\begin{proof} 

Perversity is an \'{e}tale-local condition, so it suffices to show that each polynomial $f \in \mathbb A^n$ has an \'{e}tale neighborhood $\nu \colon U \to \mathbb A^n$ such that $\nu ^* \bigotimes_{i=1}^m e_i^* \mathcal F_i [n]$ is perverse.
For each $1 \leq i \leq m$, we will choose a suitable \'{e}tale neighborhood $U_i$ of $e_i(f)$ in $\mathbb A^1$ and then take 
\begin{equation} \label{UfiberedProdEq}
U =(  (( \mathbb A^n \times_{\mathbb A^1} U_1 ) \times_{\mathbb A^1} U_2 ) \dots ) \times_{ \mathbb A^1} U_m ,
\end{equation}
which will be an \'{e}tale neighborhood of $f$. 

Fix $1 \leq i \leq m$.
By definition, every section of the stalk of $\mathcal F_i$ at $e_i(f)$ is defined over an \'{e}tale neighborhood of $e_i(f) \in \mathbb{A}^1$. 
By constructibility,
\begin{equation}
r_i = \dim (\mathcal{F}_i)_{e_i(f)} < \infty
\end{equation}
so there exists an \'{e}tale neighborhood $U_i'$ of $e_i(f)$ in $\mathbb{A}^1$ over which all the sections of $\mathcal{F}_i$
at $e_i(f)$ are defined.
We then have a natural map
\begin{equation}
\psi \colon \overline{\mathbb Q_\ell}^{r_i}  \to \mathcal F_i \mid_{U_i'}
\end{equation}
inducing an isomorphism on the stalks at $e_i(f)$.
We denote the cokernel of $\psi$ by $\mathcal{Q}_i$,
and note that its stalk at $e_i(f)$ vanishes.

Observe that $\psi$ is injective. 
Indeed if any nontrivial section of $\overline{\mathbb Q_\ell}^{r_i} $ has image vanishing on some Zariski open set containing $e_i(f)$, then the corresponding nontrivial section of $\mathcal F_i$ is supported in the complement of that open set, which is finite, contradicting the assumption that $\mathcal F_i$ has no finitely supported sections.
It follows that the restriction of $\mathcal F_i$ to $U_i'$ is the extension of $\mathcal Q_i$ by $\overline{\mathbb Q_\ell}^{r_i}$.

Let $U_i$ be the union with $\{e_i(f)\}$ of the largest open subset of $U_i'$ where $\mathcal Q_i$ is lisse. 
Restricted to $U_i$, the sheaf $\mathcal Q_i$ is lisse on $U_i \setminus \{e_i(f)\}$ extended by zero to $U_i$.
We define $U$ using \cref{UfiberedProdEq}, 
and let $\widetilde{e_i} \colon U \to U_i$ for $1 \leq i \leq m$ and $\nu \colon U \to \mathbb{A}^n$ be the projections. 

As $\mathcal F_i$ restricted to $U_i$ is the extension of $\mathcal Q_i$ by $\overline{\mathbb Q_\ell}^{r_i}$,  
the sheaf $\nu^* \bigotimes_{i=1}^m e_i^* \mathcal F_i$ is the iterated extension of $2^m$ sheaves, each of the form 
\begin{equation}
\bigotimes_{i \not \in S} \widetilde{e_i}^{*} \overline{\mathbb Q_\ell}^{r_i} \otimes 
\bigotimes_{i \in S} \widetilde{e_i}^{*} \mathcal Q_i, 
\quad S \subseteq \{1, \dots, m\}. 
\end{equation}
Since an extension of perverse sheaves is perverse, it suffices to prove that 
\begin{equation} 
\left( \bigotimes_{i \not \in S} \widetilde{e_i}^{*} \overline{\mathbb Q_\ell}^{r_i} \otimes 
\bigotimes_{i \in S} \widetilde{e_i}^{*} \mathcal Q_i \right) [n] 
\end{equation}
is perverse. 
Tensoring with the pullback of $\overline{\mathbb Q_\ell}^{r_i}$ is equivalent to taking a direct sum of $r_i$ copies, 
so it suffices to show that
\begin{equation} \label{LastPerverseSheaf}
\bigotimes_{i \in S} \widetilde{e_i}^{*} \mathcal Q_i[n] 
\end{equation}
is perverse. 
 
Since each $\mathcal Q_i$ is the extension by zero from $U_i \setminus \{e_i(f)\}$  to $U_i$ of a lisse sheaf, 
the sheaf $\widetilde{e_i}^{*} \mathcal Q_i$ is the extension by zero from $U \setminus \widetilde{e_i}^{-1} \{ e_i(f) \}$ to $U$ of a lisse sheaf, 
and thus the sheaf
\begin{equation}
\bigotimes_{i \in S} \widetilde{e_i}^{*} \mathcal Q_i
\end{equation}
is the extension by zero from the complement in $U$ of
\begin{equation}
D = \bigcup_{i \in S} \widetilde{e_i}^{-1} \{e_i(f)\}
\end{equation} 
to $U$ of a lisse sheaf.
 
Because $D$ is a divisor in $U$, 
the inclusion of its complement in $U$ is an affine open immersion. 
Lisse sheaves shifted by $\dim(U) = \dim(\mathbb{A}^n) = n$ are perverse, and by \cite[Corollary 4.1.3]{bbdg}, extensions by zero along affine open immersions of perverse sheaves are perverse, so indeed \cref{LastPerverseSheaf} is perverse.
We can thus conclude that $\bigotimes_{i=1}^m e_i^* \mathcal F_i [n]$ is perverse.
\end{proof}

\begin{cor}\label{geometric-vanishing} 

Keep \cref{geometric-setup}. We have
\begin{equation}
H^j_c \Bigl( \mathbb A^n, \bigotimes_{i=1}^m e_i^* \mathcal F_i \Bigr) =0 
\end{equation}
for every integer $j \notin \{n, n + 1\}$.
\end{cor}

\begin{proof}  

The vanishing for $j < n $ follows from Artin's affine theorem \cite[Exp. XIV, Corollaire 3.2]{art} and the fact, 
from \cref{perversity}, that $\bigotimes_{i=1}^m e_i^* \mathcal F_i[n]$ is perverse.

By the excision long exact sequence
\begin{equation*}
\begin{split} 
\dots \to H^*_c \Bigl(\mathbb P^n \setminus H_1 ,  u_! w^* \bigotimes_{i=1}^m e_i^* \mathcal F_i \Bigr) &\to 
H^*  \Bigl(\mathbb P^n \setminus H_1 ,  u_! w^* \bigotimes_{i=1}^m e_i^* \mathcal F_i \Bigr) \\
&\to H^* \Bigl(H_1  ,   d^* v_* u_! w^* \bigotimes_{i=1}^m e_i^* \mathcal F_i \Bigr) \to \dots
\end{split}
\end{equation*}
from \cref{excision}, 
and \cref{letter-switching}, to prove vanishing for $j >n+1$, it suffices to show that for $j>n$ we have
\begin{equation} \label{TwoCohomologiesVanishingEq}
H^j \Bigl(\mathbb P^n \setminus H_1 ,  u_! w^* \bigotimes_{i=1}^m e_i^* \mathcal F_i \Bigr)= 0, \quad
H^j \Bigl(H_1  ,   d^* v_* u_! w^* \bigotimes_{i=1}^m e_i^* \mathcal F_i \Bigr)=0. 
\end{equation} 
 
For the first, note that $\mathbb P^n \setminus H_1$ is an affine variety of dimension $n$, and 
\begin{equation}
u_! w^* \bigotimes_{i=1}^m e_i^* \mathcal F_i
\end{equation}
is a sheaf, so we can invoke Artin's affine theorem again.

We shall now prove the second vanishing statement in \cref{TwoCohomologiesVanishingEq}.
Since $u_! w^* \bigotimes_{i=1}^m e_i^* \mathcal F_i $ is a sheaf on an $n$-dimensional variety, the complex
\begin{equation}
u_! w^* \bigotimes_{i=1}^m e_i^* \mathcal F_i [n]
\end{equation}
is semiperverse, so $v_* u_! w^* \bigotimes_{i=1}^m e_i^* \mathcal F_i[n]$ is semiperverse as $v$ is affine, and 
\begin{equation}
d^* v_* u_! w^* \bigotimes_{i=1}^m e_i^* \mathcal F_i[n]
\end{equation} 
is semiperverse because $d$ is a closed immersion. 
It follows that the stalks of the complex above are supported in nonpositive degrees, 
hence the stalks of $d^* v_* u_! w^* \bigotimes_{i=1}^m e_i^* \mathcal F_i$ are supported in degrees not exceeding $n$.

We know from \cref{affine-vanishing-lemma} and \cref{projective-vanishing-lemma} that the complex
\begin{equation}
d^* v_* u_! w^* \bigotimes_{i=1}^m e_i^* \mathcal F_i
\end{equation} 
is supported at only finitely many points, 
so its cohomology is simply the direct sum of its stalks.
We have seen that these stalks are supported in degrees not exceeding $n$,
so the cohomology indeed vanishes in degrees greater than $n$.
\end{proof}

\subsection{Betti numbers bound} 

Here we bound the dimension of the cohomology groups that are not known to vanish by our previous arguments. We let $\kappa$ be a perfect field of characteristic $p$.

\begin{defi} \label{ComplexesInvarsDef}

For $K \in D^b_c( \mathbb A^1_{\kappa},\overline{ \mathbb Q_\ell})$ define the \emph{rank}
\begin{equation}
\rank(K) = \sum_{j =-\infty}^{\infty} \dim \mathcal H^j (K)_{\eta} 
\end{equation}
where $\eta$ is the geometric generic point of $\mathbb A^1$, and the \emph{Fourier conductor}  
\begin{equation}
c_F(K) = \sum_{j=-\infty}^{\infty} \dim H^j ( \mathbb A^1_{\overline{\kappa (\alpha)}} , K \otimes \mathcal L_\psi(\alpha x)) 
\end{equation}
where $\kappa(\alpha)$ is the field of rational functions over $\kappa$ in a variable $\alpha$, $\psi$ is an additive character of $\mathbb F_p$, and $\mathcal L_\psi(\alpha x)$ is the Artin-Schreier sheaf. \end{defi}

Note that the rank agrees with the usual notion of the generic rank when $K$ is a sheaf or a perverse sheaf. 
We call $c_F$ the Fourier conductor because it is equal to the rank of the Fourier transform. 
However, we will not use the characterization in terms of Fourier transform here.
The Fourier conductor can also be expressed in terms of local invariants, see \cref{conductor-little-lemmas}(5).
In order to write this expression, we need the following ad-hoc modification of the Swan conductor.

\begin{defi} \label{ModifiedSwan}

As in \cref{SwanDef}, let $V/\overline{\mathbb{Q}_\ell}$ be a representation of an inertia group,
and let $V_1, \dots, V_n$ be the Jordan-H\"{o}lder factors of $V$. We set
\begin{equation*}
\swan'(V) = \sum_{i=1}^n \max \{\swan(V_i)-\dim V_i, 0 \} =  \sum_{i=1}^n \max \{\slope(V_i)-1,0\} \dim V_i.
\end{equation*}
For a sheaf $\mathcal{F}$ on an open subset $C$ of a proper curve $\overline{C}/\kappa$
and a closed point $x$ of $\overline{C}$, put
\begin{equation}
\swan_x'(\mathcal F) = \swan'(\mathcal{F}_\eta)
\end{equation}
with $\mathcal{F}_\eta$ viewed as a representation of $I_x$.

\end{defi}

Recall that every complex $K \in D^b_c( \mathbb A^1_{\kappa},\overline{ \mathbb Q_\ell})$ has a filtration, the perverse filtration, whose associated graded objects are shifts of $ \{ {}^p\mathcal H^j (K) \}_{j \in \mathbb{Z}}$, which are perverse sheaves.

\begin{lem}\label{conductor-little-lemmas} \begin{enumerate}

\item For $K \in D^b_c( \mathbb A^1_{\kappa},\overline{ \mathbb Q_\ell})$ we have 
\begin{equation*}
\rank(K) = \sum_{j=-\infty}^{\infty} \rank ({}^p\mathcal H^j (K)), \quad
c_F(K) = \sum_{j=-\infty}^{\infty} c_F ({}^p\mathcal H^j (K)).
\end{equation*}

\item For a short exact sequence of perverse sheaves 
\[0 \to P_1 \to P_2 \to P_3 \to 0, \] 
on $\mathbb{A}^1_\kappa$, we have 
\[ \rank(P_2) = \rank(P_1) + \rank(P_3),\quad c_F(P_2) = c_F(P_1) + c_F(P_3).\]

\item For a skyscraper sheaf $K$ on $\mathbb{A}^1_\kappa$, we have 
\[\rank(K)=0, \quad c_F(K) = 1.\]

\item For any $\beta \in \overline{\kappa}$, we have 
\[\rank(\mathcal L_\psi(\beta x)) = 1, \quad c_F( \mathcal L_\psi(\beta x))=0.\]

\item For a sheaf $\mathcal F$ on $\mathbb A^1_\kappa$ with no finitely supported sections, 
we have  
\[ c_F (  \mathcal F) = \sum_{ x\in | \mathbb A^1_{\overline{\kappa}} | } \cond_x (\mathcal F)   + \swan'_{\infty} (\mathcal F ). \] 

\item Suppose that $\kappa$ is finite.
Then for an infinitame sheaf $\mathcal{F}$ on $\mathbb A^1_{\kappa}$ we have $c_F(\mathcal{F}) = c(\mathcal{F})$.

\end{enumerate}

\end{lem}

\begin{proof} 

For (1) and (2), the key point will be that for a perverse sheaf $P$ on $\mathbb{A}^1_\kappa$,
the stalk $\mathcal H^i(P)_{\eta}$ vanishes for $i \neq -1$ and
\begin{equation} \label{PtensorSchreierVanishEq}
H^i \left( \mathbb A^1_{\overline{\kappa(\alpha)}} , P \otimes \mathcal L_\psi(\alpha x) \right) = 0, \quad i \neq 0.
\end{equation} 

The vanishing of $\mathcal H^i(P)_{\eta}$  is due to the fact that perverse sheaves on a curve are lisse and supported in degree $-1$ on an open set. The vanishing in \cref{PtensorSchreierVanishEq} follows from the fact that $P \otimes \mathcal L_\psi(\alpha x)$ is perverse on a curve, hence has compactly-supported cohomology in degrees $0$ and $1$ only,
and the cohomology in degree $1$ equals the monodromy coinvariants, which vanish for $\alpha$ generic since the representation 
$\mathcal{L}_\psi(\alpha x)^\vee = \mathcal L_\psi(-\alpha x)$ can occur as a quotient of the monodromy representation of $P$ for only finitely many specializations of $\alpha$.

Because of the vanishing above, 
the spectral sequence computing $\mathcal H^i(K)_{\eta}$ from 
\begin{equation}
\mathcal H^i(P_j)_{\eta}, \quad P_j = {}^p \mathcal H^j(K)
\end{equation}
degenerates on the first page, as does the spectral sequence computing $H^i ( \mathbb A^1_{\overline{\kappa(\alpha)}} , K \otimes \mathcal L_\psi(\alpha x)) $ from $H^i ( \mathbb A^1_{\overline{\kappa(\alpha)}} , P_j \otimes \mathcal L_\psi(\alpha x)) $, giving (1).
We also deduce from the vanishing above that the functors $P \mapsto \mathcal H^{-1} (P)_{\eta}$ and $P \mapsto H^0 ( \mathbb A^1_{\overline{\kappa(\alpha)}}, P \otimes \mathcal L_\psi(\alpha x)) $ are exact,
and that composing these functors with dimension gives $\rank(P)$ and $c_F(P)$ respectively.
This proves (2).


For a skyscraper sheaf, its stalk at the generic point vanishes, while its twist by an Artin-Schreier sheaf is again a skyscraper sheaf, so has one-dimensional cohomology in degree zero and no cohomology in all other degrees, verifying (3).

For an Artin-Schreier sheaf, its stalk at the generic point has rank one in degree $0$ and none in all other degrees, while its cohomology twisted by any Artin-Schreier sheaf but its dual vanishes, verifying (4).

Now we check (5). Certainly the stalk of $\mathcal F$ at the generic point has rank $\rank(\mathcal F)$ in degree $0$ and rank zero in other degrees. Since $\mathcal{F}$ has no finitely supported sections, we get that $\mathcal F[1]$ is perverse, 
so $\mathcal F \otimes \mathcal L_\psi(\alpha x)$ has no cohomology in degrees other than $1$, hence
\begin{equation}
c_F( \mathcal F) = -\chi (\mathbb A^1_{\overline{\kappa(\alpha)}},  \mathcal F \otimes \mathcal L_\psi(\alpha x) ).
\end{equation}

From \cref{EP2} we get that
\[ \chi(\mathbb A^1_{\overline{\kappa(\alpha)}}, \mathcal F \otimes \mathcal L_\psi(\alpha x)) =  
\chi(\mathbb A^1) \rank(\mathcal F \otimes \mathcal L_\psi(\alpha x) ) -  \sum_{y \in | \mathbb A^1_{\overline{\kappa(\alpha)}} |} \cond_y(\mathcal F \otimes \mathcal L_\psi(\alpha x) ) - \swan_{\infty} (\mathcal F \otimes \mathcal L_\psi(\alpha x)) \] 
so to establish (5), it suffices to check that for all $y \in  \left| \mathbb A^1_{\overline{\kappa(\alpha)}}\right|$, we have
\begin{equation} \label{firstIdentityLocalConuctor}
\cond_y( \mathcal F \otimes \mathcal L_\psi(\alpha x ))  = \cond_y (\mathcal F) 
\end{equation}
and that  
\begin{equation} \label{IntricateSwanEq}
\swan_{\infty} (\mathcal F \otimes \mathcal L_\psi(\alpha  x )) - \chi(\mathbb A^1) \rank(\mathcal F \otimes \mathcal L_\psi(\alpha x)) = \swan'_\infty (\mathcal F). 
\end{equation}

\cref{firstIdentityLocalConuctor} is straightforward from \cref{LocalCondDefEq} since neither the Swan conductor at $y$ nor the drop at $y$ can be changed by tensoring with a lisse sheaf of rank one in a neighborhood of $y$. 

For \cref{IntricateSwanEq} note that $\chi(\mathbb A^1)=1$ by \cref{AffineEulerChar}, 
and that tensoring with a lisse sheaf of rank one does not affect the rank, so it suffices to prove that
\begin{equation}
\swan_{\infty} (\mathcal F \otimes \mathcal L_\psi(\alpha x))- \rank(\mathcal F) = \swan'_{\infty}(\mathcal F).
\end{equation}
Every term above can be expressed in terms of the representation $V = \mathcal F_{\eta}$ of $I_{\infty}$, 
so it suffices to show that
\begin{equation}
\swan ( V \otimes \mathcal L_\psi(\alpha x) ) - \dim (V) = \swan'(V) 
\end{equation} 
where we have abused notation by using $\mathcal L_\psi(\alpha x)$ for both a sheaf and its inertia representation at $\infty$. 

Since all terms above are additive in extensions of irreducible representations, we may assume $V$ is irreducible, in which case it suffices by \cref{SwanDef} and \cref{ModifiedSwan} to prove that
\begin{equation}
\slope ( V \otimes \mathcal L_\psi(\alpha x)) \dim (V) - \dim(V) = \max \{ \slope(V)-1, 0 \} \dim (V) 
\end{equation}
or, equivalently, that
\begin{equation}
\slope(V \otimes \mathcal L_\psi(\alpha x))=\max \{ \slope(V), 1 \}.
\end{equation}

The above follows from \cref{slope-tensor-lemmas} once we check that $I^1_\infty$ does not act on $V$ by scalars via the character $\mathcal L_\psi(- \alpha x)$. If it were to act by scalars, the character defined by those scalars would be unique,
but $\mathcal L_\psi(- \alpha x)$ gives distinct characters of $I^1_\infty$ for different specializations of $\alpha$, 
so such an action by scalars is impossible for generic $\alpha$.  

At last we deduce (6) from (5).
Since $\mathcal{F}$ has no finitely supported sections and is tamely ramified at infinity,
it follows from (5), \cref{ModifiedSwan}, \cref{SwanDef}, \cref{DropDef}, and the definition in \cref{InfinitameConductorDefEq} that
\begin{equation}
c_F (  \mathcal F) = \sum_{ x\in | \mathbb A^1_{\overline{\kappa}} | } \cond_x (\mathcal F)  =
\sum_{ x\in | \mathbb A^1_{\kappa} | } [\kappa(x) : \kappa] (\dro_x (\mathcal F) + \swan_x(\mathcal{F})) = c(\mathcal{F})
\end{equation} 
as required for (6).
\end{proof}

\subsubsection{Betti bounds for tensor products}

\begin{lem}\label{conductor-prod-lemma} 

For $K_1 \in D^b_c( \mathbb A^1_\kappa,\overline{\mathbb Q_\ell})$ and a sheaf $K_2$ on $\mathbb A^1_\kappa$ with no finitely supported sections, we have 
\begin{equation}\label{eq-conductor-prod-lemma} 
\sum_{j=-\infty}^{\infty} \dim H^j_c( \mathbb A^1_{\overline{\kappa}}, K_1 \otimes K_2 )  \leq c_F( K_1) \rank(K_2) + \rank(K_1) c_F(K_2) + \rank(K_1) \rank(K_2) .
\end{equation} 

\end{lem}

\begin{proof} 

First, let us reduce to the case when $K_1$ is perverse. The perverse filtration on $K_1$, whose $j$'th associated graded by definition is ${}^p \mathcal H^j(K_1)$, induces a filtration on $K_1 \otimes K_2$ whose $j$'th associated graded is ${}^p \mathcal H^j( K_1) \otimes K_2$ and thus a filtration on  $H^* ( \mathbb A^1_{\overline{\kappa}}, K_1 \otimes K_2)$ whose $j$'th associated graded is $H^*_c( \mathbb A^1_{\overline{\kappa}}, {}^p \mathcal H^j( K_1) \otimes K_2 ) $. The spectral sequence associated to this filtration computes  $ H^n_c( \mathbb A^1_{\overline{\kappa}}, K_1 \otimes K_2 )  $ in terms of all the $ H^i_c( \mathbb A^1_{\overline{\kappa}}, {}^p \mathcal H^j( K_1) \otimes K_2 ) $. This spectral sequence gives the inequality 
\begin{equation}
\dim H^n_c( \mathbb A^1_{\overline{\kappa}}, K_1 \otimes K_2 ) \leq \sum_{j=-\infty}^{\infty} \dim H^{n-j}_c (\mathbb A^1_{\overline{\kappa} }, {}^p \mathcal H^j(K_1) \otimes K_2) 
\end{equation}
which implies
\begin{equation}
\sum_{i=-\infty}^{\infty} \dim H^i_c( \mathbb A^1_{\overline{\kappa}}, K_1 \otimes K_2 )  \leq   \sum_{i=-\infty}^{\infty}\sum_{j=-\infty}^{\infty}  \dim H^i_c( \mathbb A^1_{\overline{\kappa}}, {}^p\mathcal H^j (K_1 )\otimes K_2 ) .
\end{equation} 

Thus the left hand side of \cref{eq-conductor-prod-lemma} is subadditive when we pass to perverse cohomology. 
By \cref{conductor-little-lemmas}(1),  the right hand side of \cref{eq-conductor-prod-lemma} is additive when we pass to perverse cohomology. It is therefore sufficient to handle the case when $K_1$ is perverse.
By the same argument, except using \cref{conductor-little-lemmas}(2), it suffices to handle the case when $K_1$ is an irreducible perverse sheaf.

As $K_1$ is an irreducible perverse sheaf, it is either a skyscraper sheaf or the shift of a middle extension sheaf,
which in particular will have no finitely supported sections. Since both sides of \cref{eq-conductor-prod-lemma} are invariant under shifts, it suffices to handle the case when $K_1$ is either a skyscraper sheaf or a sheaf with no finitely supported sections.

If $K_1 = \delta_x$ for some $x \in \mathbb{A}^1_\kappa$ is a skyscraper sheaf  then by \cref{conductor-little-lemmas}(3) we have $c_F(K_1) =1$ and 
\begin{equation}
H^i_c( \mathbb A^1_{\overline{\kappa} }, K_1 \otimes K_2)= \mathcal H^i ( K_2)_{x} = 
\begin{cases} 
K_{2,x} & i= 0 \\
0 &i \neq 0
\end{cases} 
\end{equation} 
so since $K_2$ has no finitely supported sections, we get
\begin{equation}
\sum_{i=-\infty}^{\infty} H^i_c( \mathbb A^1_{\overline{\kappa} }, K_1 \otimes K_2) = \dim K_{2,x} \leq \rank(K_2) = c_F(K_1) \rank(K_2) 
\end{equation}
so the required inequality in \cref{eq-conductor-prod-lemma} is satisfied.

Finally, we must check the case when $K_1$ and $K_2$ are sheaves with no finitely supported sections. 
Thus $K_1 \otimes K_2$ has no finitely supported sections either, so $H^0_c( \mathbb A^1_{\overline{\kappa}}, K_1 \otimes K_2) =0$.
Hence we have 
\begin{equation*}
\begin{split}
\sum_{i=-\infty}^{\infty} \dim H^i_c( \mathbb A^1_{\overline{\kappa}}, K_1 \otimes K_2 ) &= 
\dim H^1_c( \mathbb A^1_{\overline{\kappa}}, K_1 \otimes K_2 ) + \dim H^2_c( \mathbb A^1_{\overline{\kappa}}, K_1 \otimes K_2 ) \\
&= 2 \dim H^2_c( \mathbb A^1_{\overline{\kappa}}, K_1 \otimes K_2 ) - \chi( \mathbb A^1_{\overline{\kappa}}, K_1 \otimes K_2).
\end{split}
\end{equation*}

By \cite[Proof of Lemma 4.7]{FKM13}, we can identify $H^2_c( \mathbb A^1_{\overline{\kappa}}, K_1 \otimes K_2 )$ with the coinvariants of $K_{1,\eta} \otimes K_{2,\eta}$ under the action of $\pi_1^{\text{\'{e}t}}(U)$, 
for some open $U$ in $\mathbb{A}^1_{\overline{\kappa}}$ where $K_1, K_2$ are lisse. 
By \cref{EP2} we have 
\begin{equation*}
\begin{split}
&\chi(\mathbb A^1_{\overline{\kappa}}, K_1 \otimes K_2) = \\
&\rank(K_1 \otimes K_2)  -  \sum_{ x \in | \mathbb A^1_{\overline{\kappa} } | }  (\rank(K_1 \otimes K_2)  - \dim (K_1 \otimes K_2)_x + \swan_x(K_1 \otimes K_2 ) )   - \swan_{\infty} (K_1\otimes K_2 ).
\end{split}
\end{equation*}

We will check that
\begin{equation}\label{prod-local-finite} 
\rank(K_1 \otimes K_2)  - \dim (K_1 \otimes K_2)_x + \swan_x(K_1 \otimes K_2 ) \leq \cond_x (K_1)  \rank(K_2) + \rank (K_1) \cond_x (K_2)
\end{equation}
for every $x \in |\mathbb{A}^1_{\overline{\kappa}}|$, and that
\begin{equation} \label{prod-local-infinite}  
\begin{split}
&\swan_{\infty} (K_1\otimes K_2 )- \rank(K_1 \otimes K_2)  + 2 \dim ( K_{1,\eta} \otimes K_{2,\eta} )_{\pi_1^{\text{\'{e}t}}(U)} \leq \\
&\swan'_{\infty} (K_1) \rank(K_2) + \rank(K_1) \swan'_{\infty}(K_2) + \rank(K_1) \rank(K_2).
\end{split}
\end{equation}
The bound in \cref{eq-conductor-prod-lemma} will then follow upon summing \cref{prod-local-finite} over all $x \in |\mathbb A^1_{\overline{\kappa}}| $, adding \cref{prod-local-infinite}, and using \cref{conductor-little-lemmas}(5).

For \cref{prod-local-finite}, first observe that 
\begin{equation}
\dim (K_1 \otimes K_2)_x  = \dim(K_{1,x}\otimes K_{2,x})=\dim(K_{1,x})\dim(K_{2,x}) 
\end{equation} 
so that 
\begin{equation*}
\begin{split}
&\rank(K_1\otimes K_2)  -  \dim (K_1 \otimes K_2)_x = \rank(K_1) \rank (K_2) - \dim ( K_{1,x}) \dim(K_{2,x} ) \leq \\
&\rank(K_1) \rank (K_2) - \dim ( K_{1,x}) \dim(K_{2,x} ) + (\rank(K_1) -\dim(K_{1,x}))  (\rank(K_2) -\dim (K_{2,x})) = \\
&\rank(K_1)  (\rank(K_2) -\dim (K_{2,x}))  +  (\rank(K_1) -\dim(K_{1,x})) \rank(K_2).
\end{split}
\end{equation*}

Next we apply \cref{SwanTensorProdBoundCor} to obtain
\begin{equation}
\swan_x(K_1 \otimes K_2) \leq \swan_x(K_1) \rank(K_2) + \rank(K_1) \swan_x (K_2).
\end{equation} 
\cref{prod-local-finite} now follows from the definition of $\cond_x(K_1)$ and  $ \cond_x(K_2)$.

We turn to  \cref{prod-local-infinite}.
As every global monodromy coinvariant is a coinvariant of the local monodromy at $\infty$,
it suffices to prove that
\[  \swan_{\infty} (K_1\otimes K_2 ) + 2 \dim ( K_{1,\eta} \otimes K_{2,\eta} )_{I_\infty} \leq  \swan'_{\infty} (K_1) \rank(K_2) + \rank(K_1) \swan'_{\infty}(K_2) +2  \rank(K_1) \rank(K_2). 
 \] 
Both sides above depend only on $V_1 = K_{1,\eta}$ and $V_2 = K_{2,\eta}$ viewed as representations of $I_{\infty}$. 
Writing $V_1$ and $V_2$ as iterated extensions of irreducible representations, the swan conductor, rank, and $\swan'$ are all additive, while the dimension of the inertia coinvariants is subadditive, so it suffices to handle the case when $V_1$ and $V_2$ are irreducible. 

In case $V_1$ and $V_2$ are irreducible, by \cref{ModifiedSwan}, we must prove that
\begin{equation*}
\begin{split}
&\swan (V_1 \otimes V_2) + 2 \dim (V_1 \otimes V_2)_{I_{\infty}}  \leq \\
&\dim(V_1) \dim(V_2) ( \max \{ \slope(V_1) -1,0\} + \max \{\slope(V_2)-1,0\} + 2)  .
\end{split}
\end{equation*}
Since trivial representations have Swan conductor zero, and
\begin{equation*}
\dim (V_1 \otimes V_2)_{I_\infty} = \dim \mathrm{Hom}_{I_\infty} (V_1 \otimes V_2, \overline{\mathbb{Q}_\ell}) = \dim \mathrm{Hom}_{I_{\infty}} (V_1, V_2^\vee) \leq 1
\end{equation*}
in view of irreducibility, it follows from \cref{slope-tensor-lemmas}(1) that
\begin{equation*}
\begin{split}
\swan (V_1 \otimes V_2) &\leq ( \dim (V_1 ) \dim (V_2) - \dim (V_1 \otimes V_2)_{I_\infty}) \max \{\slope(V_1),\slope(V_2)\} \\ 
&\leq ( \dim (V_1 ) \dim (V_2) - \dim (V_1 \otimes V_2)_{I_\infty}) \max \{\slope(V_1),\slope(V_2),2\} \\
&\leq \dim (V_1 ) \dim(V_2) \max \{\slope(V_1),\slope(V_2),2\} - 2\dim (V_1 \otimes V_2)_{I_\infty}
\end{split}
\end{equation*}
and this is at most 
\[ \dim(V_1) \dim (V_2)   ( \max \{\slope(V_1) -1,0\} + \max\{\slope(V_2)-1,0\} + 2) - 2\dim (V_1 \otimes V_2)_{I_\infty} \]
so \cref{prod-local-infinite} is established.
\end{proof}

We shall need an auxiliary vanishing statement for the cohomology of Artin-Schreier sheaves.

\begin{lem}
\label{linear-vanishing-lem} 
Let $n,m$ be a positive integers, let $\mathbb A^n$ be an affine space over $\kappa$, 
let $e_1,\dots, e_m\colon \mathbb A^n \to \mathbb A^1$ be affine maps, 
and let $\alpha_1, \dots, \alpha_m \in \kappa$ be scalars such that the map 
\begin{equation}
e = \sum_{i=1}^m \alpha_i e_i
\end{equation} 
is nonconstant.
Then
\begin{equation}
H^*_c \Bigg(\mathbb A^n_{\overline{\kappa}}, \bigotimes_{i =1}^m e_i^* \mathcal L_\psi(\alpha_i x) \Bigg)  = 0. 
\end{equation}
\end{lem}

\begin{proof} 

We start with the special case $n = m = 1, \ \alpha_1 = 1, \ e_1 = \text{id}$, where we need to show that
\begin{equation} \label{CohomologyASSonLine}
H^0_c(\mathbb A^1_{\overline{\kappa}}, \mathcal L_\psi(x) )= H^1_c(\mathbb A^1_{\overline{\kappa}}, \mathcal L_\psi(x) )=  H^2_c(\mathbb A^1_{\overline{\kappa}}, \mathcal L_\psi(x) ) = 0.
\end{equation}
Vanishing in degree $0$ follows from \cref{AS-sheaf-properties}(5).
For degree $2$ we have
\begin{equation}
H^2_c(\mathbb A^1_{\overline{\kappa}}, \mathcal L_\psi(x) ) = 
(\mathcal L_\psi(x)_{\overline{\eta}})_{\pi_1^{\text{\'{e}t}} (U)}
\end{equation}
for some open $U \subseteq \mathbb{A}^1_{\overline{\kappa}}$ where $\mathcal L_\psi(x)$ is lisse as in \cref{AS-sheaf-properties}(2).
Since $\dim \mathcal L_\psi(x)_{\overline{\eta}} = 1$ by \cref{AS-sheaf-properties}(6) and $\pi_1^{\text{\'{e}t}} (U)$ acts nontrivially,
the dimension of the coinvariants is less than $1$, so we have the desired vanishing of cohomology in degree $2$.
In view of the vanishing in degrees $0$ and $2$, we get from \cref{EP2}, \cref{AS-sheaf-properties}(2,6), and \cref{AffineEulerChar} that
\begin{equation*}
\begin{split}
\dim H^1_c(\mathbb A^1_{\overline{\kappa}}, \mathcal L_\psi(x) ) = -\chi(\mathbb A^1_{\overline{\kappa}}, \mathcal L_\psi(x) )
&= \swan_{\infty} \mathcal{L}_\psi(x) - \chi(\mathbb{A}^1_{\overline{\kappa}}) \rank(\mathcal{L}_\psi(x))  \\ &= 1 \cdot 1 -1 = 0.
\end{split}
\end{equation*}

For the general case, we shall begin by checking that  
\begin{equation}
\bigotimes_{i =1}^m e_i^* \mathcal L_\psi(\alpha_i x)   \cong e^{*} \mathcal L_\psi(x).
\end{equation} 
Both sheaves are lisse of rank $1$ on $\mathbb{A}^n$, 
so it suffices to check that each $\sigma$ in $\pi_1^{\text{\'{e}t}} (\mathbb A^n_{\kappa})$ acts on their generic fibers by the same scalar. 
For each $1 \leq i \leq m$, 
the action of $\sigma$ on the generic fiber of $e_i^* \mathcal L_\psi(\alpha_i x) $ arises from its action on the finite \'{e}tale cover $y_i^p- y_i = \alpha_i e_i$ of $\mathbb{A}^n$ (by translation on $y_i$) composed with $\psi$, 
so the action of $\sigma$ on the generic fiber of the tensor product arises from its action on the product of all these covers, 
composed with $\psi$, and multiplying. That is, $\sigma$ acts by the scalar
\begin{equation*}
\prod_{i=1}^m \psi ( \sigma(y_i)-y_i) = \psi \left( \sum_{i=1}^m (\sigma(y_i)-y_i) \right) = 
\psi \left(  \sigma \left( \sum_{i=1}^m y_i \right) - \sum_{i=1}^m y_i \right).
\end{equation*}

Setting $y = \sum_{i=1}^m  y_i$, 
we see that 
\begin{equation}
y^p-y = \sum_{i=1}^m y_i^p - \sum_{i=1}^m y_i = \sum_{i=1}^m \alpha_i e_i= e,
\end{equation} 
so $\sigma$ acts by the same scalar on the generic fiber of $e^{*} \mathcal L_\psi (x)$.

If $e$ is nonconstant, we can use it as a coordinate of $\mathbb A^n$, 
namely write $\mathbb A^n = \mathbb A^1 \times \mathbb A^{n-1}$ with $e$ projecting onto the first factor. 
From the K\"{u}nneth formula and \cref{CohomologyASSonLine} we get that 
\begin{equation*}
\begin{split} 
H^*_c \Bigg(\mathbb A^n_{\overline{\kappa}}, \bigotimes_{i =1}^m e_i^* \mathcal L_\psi(\alpha_i x)  \Bigg) &=  
H^*_c (\mathbb A^n_{\overline{\kappa}}, e^{*} \mathcal L_\psi(x)  ) = 
H^*_c( \mathbb A^{1}_{\overline{\kappa}} \times \mathbb A^{n-1}_{\overline{\kappa}}, \mathcal L_\psi(x) \boxtimes \overline{\mathbb Q_\ell}) \\ 
&= 
H^*_c(\mathbb A^1_{\overline{\kappa}}, \mathcal L_\psi(x) ) \otimes H^*_c(\mathbb A^{n-1}_{\overline{\kappa}}, \overline{\mathbb Q_\ell}) \\ &= 0 \otimes  H^*_c(\mathbb A^{n-1}_{\overline{\kappa}} , \overline{\mathbb{Q}_\ell}) = 0.
\end{split}
\end{equation*}
\end{proof}

\begin{lem} \label{betti-number-bound} 

Let $\kappa$ be an algebraically closed field of characteristic $p$,
let $x_1,\dots, x_m \in \kappa$ be distinct elements, 
and let $K_1, \dots, K_m$ be sheaves on $\mathbb A^1_\kappa$ with no finitely supported sections.
For a nonnegative integer $n \leq m$, 
view $\mathbb A^n_\kappa$ as the space of polynomials of degree less than $n$, 
and for $1 \leq i \leq m$ let $e_i \colon \mathbb A^n \to \mathbb A^1$ be the map that evaluates a polynomial at $x_i$. 
We then have
\begin{equation}
\sum_{j=-\infty}^{\infty} \dim H^j_c\Bigl(\mathbb A^n, \bigotimes_{i=1}^m e_i^* K_i \Bigr) \leq 
\Bigl( \prod_{i=1}^m( \rank(K_i) (1+Z) + c_F(K_i) Z) \Bigr) [Z^n] 
\end{equation}
where $Z$ is a formal variable and $[Z^n]$ is the operator extracting the coefficient of $Z^n$ from a polynomial. 

\end{lem}

\begin{proof} We will prove this by inductively replacing each $K_i$ with either a skyscraper sheaf $\delta_{\alpha_i}$ or an Artin-Schreier sheaf $\mathcal L_\psi( \alpha_i x)$. To that end, let us formulate a more general statement, depending on a parameter $d$, which we will prove by induction.

Fix $0\leq d\leq m$. 
Let $\kappa(\alpha_{d+1},\dots ,\alpha_m)$ be the field of rational functions in $m-d$ variables. 
Let $S \subseteq \{d+1,\dots, m\}$ be a subset, and denote its complement by $S^c$.
Our more general statement is that
\[ \sum_{j=-\infty}^{\infty} \dim H^j_c\Bigl(\mathbb A^n_{\overline{\kappa (\alpha_{d+1}, \dots, \alpha_m)}}, \Bigl(\bigotimes_{i=1}^d e_i^* K_i\Bigr) \otimes \Bigl(\bigotimes_{i\in S} e_i^* \delta_{\alpha_i}  \Bigr)  \otimes \Bigl( \bigotimes_{i \in S^c} e_i^* \mathcal L_\psi(\alpha_i x)  \Bigr) \Bigr) \]
is at most
\begin{equation} \label{ZcoefficientGeneralizedboundEq}
\Bigl( Z^{|S|}  \prod_{i=1}^d( \rank(K_i) (1+Z) + c_F(K_i) Z) \Bigr) [Z^n].
\end{equation}
Our lemma follows by taking $d=m$.

We prove the above by induction on $d$.
Let us first check the base case, when $d=0$, so no $K_i$ appear, and \cref{ZcoefficientGeneralizedboundEq} is simply $1$ if $|S|=n$ and $0$ otherwise. In this case, observe that $\bigotimes_{i\in S} e_i^* \delta_{\alpha_i}$ is the constant sheaf on 
\begin{equation}
L = \{f \in \mathbb A^n : f(x_i)= \alpha_i, \ i \in S \}.
\end{equation}
As the $\alpha_i$ are independent transcendentals, 
$L$ is empty in case $|S| > n$, 
and then the sheaf $\bigotimes_{i\in S} e_i^* \delta_{\alpha_i}$ is zero,
so the cohomology is vanishing in all degrees hence the zero bound in \cref{ZcoefficientGeneralizedboundEq} is confirmed.
In case $|S| \leq n$, the locus $L$ is an affine space of dimension $n- |S|$. By Lemma \ref{linear-vanishing-lem}, the cohomology of this affine space with coefficients in  $ \bigotimes_{i \in S^c} e_i^* \mathcal L_\psi(\alpha_i x) $ vanishes as long as $\sum_{i \in S^c} \alpha_i  e_i$ is nonconstant. Since the $\alpha_i$ are independent transcendentals this sum is nonconstant as soon as one of the $e_i$ is (i.e. $f \mapsto f(\alpha_i)$) is nonconstant on $L$. 
If $|S| < n$ then all of these forms are nonconstant on $L$, and because $n \leq m$ by assumption,
the set $S^c$ parametrizing these forms is nonempty, so indeed one are nonconstant, and 
the zero bound in \cref{ZcoefficientGeneralizedboundEq} is valid also in case $|S|<n$. 
If $|S|=n$, we are taking the cohomology of a point with coefficients in a (constant) sheaf of rank $1$, 
hence the cohomology is $1$-dimensional (concentrated in degree $j = 0$). 
This verifies the base case.

For the induction step, assume that the statement is known for $d-1$ - we will verify it for $d$. 
By the projection formula, for every integer $j$ we have
\begin{equation*}
\begin{split}
&H^j_c\Bigl(\mathbb A^n_{\overline{\kappa (\alpha_{d+1}, \dots, \alpha_m)}}, \Bigl(\bigotimes_{i=1}^d e_i^* K_i\Bigr) \otimes \Bigl(\bigotimes_{i\in S} e_i^* \delta_{\alpha_i}  \Bigr)  \otimes \Bigl( \bigotimes_{i \in S^c} e_i^* \mathcal L_\psi(\alpha_i x)  \Bigr) \Bigr) = \\
&H^j_c\Bigl(\mathbb A^1_{\overline{\kappa (\alpha_{d+1}, \dots, \alpha_m)}},  K_d \otimes e_{d,!} \Bigl ( \Bigl(\bigotimes_{i=1}^{d-1}  e_i^* K_i\Bigr) \otimes \Bigl(\bigotimes_{i\in S} e_i^* \delta_{\alpha_i}  \Bigr)  \otimes \Bigl( \bigotimes_{i \in S^c} e_i^* \mathcal L_\psi(\alpha_i x)  \Bigr)  \Bigr)\Bigr).
\end{split}
\end{equation*}
It follows from \cref{conductor-prod-lemma} that
\begin{equation*}
\begin{split}
&\sum_{j=-\infty}^{\infty} \dim H^j_c\Bigl(\mathbb A^n_{\overline{\kappa (\alpha_{d+1}, \dots, \alpha_m)}}, \Bigl(\bigotimes_{i=1}^d e_i^* K_i\Bigr) \otimes \Bigl(\bigotimes_{i\in S} e_i^* \delta_{\alpha_i}  \Bigr)  \otimes \Bigl( \bigotimes_{i \in S^c} e_i^* \mathcal L_\psi(\alpha_i x)  \Bigr) \Bigr) \leq \\
&\rank(K_d) c_F \Bigl( e_{d,!} \Bigl ( \Bigl(\bigotimes_{i=1}^{d-1}  e_i^* K_i\Bigr) \otimes \Bigl(\bigotimes_{i\in S} e_i^* \delta_{\alpha_i}  \Bigr)  \otimes \Bigl( \bigotimes_{i \in S^c} e_i^* \mathcal L_\psi(\alpha_i x)  \Bigr)  \Bigr) \Bigr) + \\  
&(\rank(K_d) +  c_F(K_d)) \rank \Bigl ( e_{d,!} \Bigl( \Bigl(\bigotimes_{i=1}^{d-1}  e_i^* K_i\Bigr) \otimes \Bigl(\bigotimes_{i\in S} e_i^* \delta_{\alpha_i}  \Bigr)  \otimes \Bigl( \bigotimes_{i \in S^c} e_i^* \mathcal L_\psi(\alpha_i x)  \Bigr)  \Bigr) \Bigr).
\end{split}
\end{equation*}

Using the fact that the algebraic closure of $\overline{\kappa (\alpha_{d+1}, \dots, \alpha_m)}(\alpha_d)$ is $\overline{\kappa(\alpha_d,\dots ,\alpha_m)}$, 
we get from \cref{ComplexesInvarsDef}, the projection formula, and the inductive hypothesis that
\begin{equation*}
\begin{split}
&c_F \Bigl( e_{d,!} \Bigl ( \Bigl(\bigotimes_{i=1}^{d-1}  e_i^* K_i\Bigr) \otimes \Bigl(\bigotimes_{i\in S} e_i^* \delta_{\alpha_i}  \Bigr)  \otimes \Bigl( \bigotimes_{i \in S^c} e_i^* \mathcal L_\psi(\alpha_i x)  \Bigr)  \Bigr) \Bigr)  = \\
&\sum_{j=-\infty}^{\infty}\dim H^j_c\Bigl(\mathbb A^1_{\overline{\kappa (\alpha_{d}, \dots, \alpha_m)}},  \mathcal L_\psi(\alpha_d x)  \otimes e_{d,!} \Bigl ( \Bigl(\bigotimes_{i=1}^{d-1}  e_i^* K_i\Bigr) \otimes \Bigl(\bigotimes_{i\in S} e_i^* \delta_{\alpha_i}  \Bigr)  \otimes \Bigl( \bigotimes_{i \in S^c} e_i^* \mathcal L_\psi(\alpha_i x)  \Bigr)  \Bigr)\Bigr) = \\
&\sum_{j=-\infty}^{\infty}\dim H^j_c\Bigl(\mathbb A^n_{\overline{\kappa (\alpha_{d}, \dots, \alpha_m)}}, \Bigl(\bigotimes_{i=1}^{d-1} e_i^* K_i\Bigr) \otimes \Bigl(\bigotimes_{i\in S} e_i^* \delta_{\alpha_i}  \Bigr)  \otimes \Bigl( \bigotimes_{i \in S^c \cup \{d\} } e_i^* \mathcal L_\psi(\alpha_i x)  \Bigr) \Bigr) \leq \\
&\Bigl( Z^{|S|}  \prod_{i=1}^{d-1} ( \rank(K_i) (1+Z) +c_F(K_i) Z) \Bigr) [Z^n].
\end{split}
\end{equation*}

We make a similar argument for the rank. 
To do so, observe that taking the stalk at the generic point is equivalent to taking the stalk, over the field extension of the base field adjoining a new variable $\alpha_d$, at the point $\alpha_d$, and this is equivalent to taking the tensor product with the skyscraper sheaf $\delta_{\alpha_d}$ and taking cohomology in degree zero. This gives

\begin{equation*}
\begin{split}
&\rank \Bigl( e_{d,!} \Bigl ( \Bigl(\bigotimes_{i=1}^{d-1}  e_i^* K_i\Bigr) \otimes \Bigl(\bigotimes_{i\in S} e_i^* \delta_{\alpha_i}  \Bigr)  \otimes \Bigl( \bigotimes_{i \in S^c} e_i^* \mathcal L_\psi(\alpha_i x)  \Bigr)  \Bigr) \Bigr) = \\
&\sum_{j=-\infty}^{\infty}\dim H^j_c\Bigl(\mathbb A^1_{\overline{\kappa (\alpha_{d}, \dots, \alpha_m)}},  \delta_{\alpha_d}   \otimes e_{d,!} \Bigl ( \Bigl(\bigotimes_{i=1}^{d-1}  e_i^* K_i\Bigr) \otimes \Bigl(\bigotimes_{i\in S} e_i^* \delta_{\alpha_i}  \Bigr)  \otimes \Bigl( \bigotimes_{i \in S^c} e_i^* \mathcal L_\psi(\alpha_i x)  \Bigr)  \Bigr)\Bigr) = \\
&\sum_{j=-\infty}^{\infty}\dim H^j_c\Bigl(\mathbb A^n_{\overline{\kappa (\alpha_{d}, \dots, \alpha_m)}}, \Bigl(\bigotimes_{i=1}^{d-1} e_i^* K_i\Bigr) \otimes \Bigl(\bigotimes_{i\in S\cup \{d\} } e_i^* \delta_{\alpha_i}  \Bigr)  \otimes \Bigl( \bigotimes_{i \in S^c } e_i^* \mathcal L_\psi(\alpha_i x)  \Bigr) \Bigr) \leq \\
&\Bigl( Z^{|S|+1}  \prod_{i=1}^{d-1} ( \rank(K_i) (1+Z) + c_F(K_i) Z) \Bigr) [Z^n].
\end{split}
\end{equation*}

Combining all the bounds above, we obtain
\begin{equation*}
\begin{split}
&\sum_{j=-\infty}^{\infty}\dim H^j_c\Bigl(\mathbb A^n_{\overline{\kappa (\alpha_{d+1}, \dots, \alpha_m)}}, \Bigl(\bigotimes_{i=1}^d e_i^* K_i\Bigr) \otimes \Bigl(\bigotimes_{i\in S} e_i^* \delta_{\alpha_i}  \Bigr)  \otimes \Bigl( \bigotimes_{i \in S^c} e_i^* \mathcal L_\psi(\alpha_i x)  \Bigr) \Bigr) \leq \\
&\Bigl( Z^{|S|} (\rank(K_d) + (\rank(K_d) + c_F(K_d) )Z )   \prod_{i=1}^{d-1} ( \rank(K_i) (1+Z) + c_F(K_i) Z) \Bigr) [Z^n] = \\
&\Bigl( Z^{|S|}  \prod_{i=1}^{d} ( \rank(K_i) (1+Z) + c_F(K_i) Z) \Bigr) [Z^n],
\end{split}
\end{equation*}
completing the induction step.
\end{proof} 
 

 
%
%
%
%
%
%
%
%
%

\subsubsection{Short trace sum bound}

The following Lemma is a variant of \cite[Lemma 6.12]{Fe20}.

\begin{lem} \label{LinearAlgebraLem}

Let $k \geq 2$ be an integer, let $V_1, \dots, V_k$ be finite-dimensional vector spaces, 
and let 
\begin{equation}
A_1 \colon V_1 \to V_2, \ A_2 \colon V_2 \to V_3, \ \dots, \ A_{k-1} \colon V_{k-1} \to V_k, \ A_k \colon V_k \to V_1
\end{equation}
be linear maps. 
Denote by 
\begin{equation}
R \colon V_2 \otimes V_3 \otimes \dots \otimes V_{k} \otimes V_1 \to V_1 \otimes V_2 \otimes \dots \otimes V_{k-1} \otimes V_k
\end{equation}
the cyclic right shift, and set 
\begin{equation*}
V  = V_1 \otimes V_2 \otimes \dots \otimes V_{k-1} \otimes V_k, 
\quad A = R \circ (A_1 \otimes A_{2} \otimes \dots \otimes A_{k-1}\otimes A_k), 
\ A \colon V \to V.
\end{equation*}
Then
\begin{equation} \label{LinAlgToProveEq}
\textup{tr}(A, V) = \textup{tr}(A_k \circ A_{k-1} \circ \dots \circ A_2 \circ A_1, V_1).
\end{equation}

\end{lem}

\begin{proof}

We have the commutative diagram
\begin{equation*}
\begin{tikzpicture}[scale=1.5]
\node (A) at (-2,1) {$\mathrm{Hom}(V_1, V_2) \otimes \dots \otimes \mathrm{Hom}(V_k, V_{k+1})$};
\node (B) at (2.2,1) {$\mathrm{Hom}(V_1, V_{k+1})$};
\node (C) at (-2,0) {$V_1^\vee \otimes V_2 \otimes \dots \otimes V_k^\vee \otimes V_{k+1}$};
\node (D) at (2.2,0) {$V_1^\vee \otimes V_{k+1}$};
\path[->,font=\scriptsize,>=angle 90]
(A) edge node[above]{\text{composition}} (B)
(A) edge node[right]{} (C)
(B) edge node[right]{} (D)
(C) edge node[above]{\text{evaluation}} (D);
\end{tikzpicture}
\end{equation*}
where $V_{k+1}$ is a finite-dimensional vector space, 
the vertical arrows are isomorphisms arising from the canonical isomorphism of vector spaces 
\begin{equation} \label{CanonIsomVecSpEq}
\Gamma^\vee \otimes \Theta \cong \mathrm{Hom}(\Gamma, \Theta), \quad \xi \otimes \theta \mapsto (\gamma \mapsto \xi(\gamma)\theta),
\end{equation}
and the lower horizontal (evaluation) map is given on pure tensors by
\begin{equation}
\xi_1 \otimes v_2 \otimes \xi_2 \otimes v_3 \otimes \dots \otimes \xi_k \otimes v_{k+1} \mapsto 
\xi_2(v_2) \xi_3(v_3) \dots \xi_k(v_k) \xi_1 \otimes v_{k+1}.
\end{equation}

Now we set $V_{k+1} = V_1$.
Note that if in \cref{CanonIsomVecSpEq} we put $\Theta = \Gamma$, 
then under our identification the trace map on $\mathrm{Hom}(\Gamma, \Gamma)$ corresponds to the evaluation map on $\Gamma^\vee \otimes \Gamma$ given on pure tensors by $\xi \otimes \gamma \mapsto \xi(\gamma)$.
Therefore, if we take the element $A_1 \otimes \dots \otimes A_k$ in the upper left corner of our diagram, 
and apply composition followed by the vertical identification and then the evaluation map on $V_1^\vee \otimes V_1$, we get
$
\tr(A_k \circ \dots \circ A_1, V_1)
$
which is the right hand side of \cref{LinAlgToProveEq}.

We can identify the upper left corner of our diagram with $\mathrm{Hom}(V,V)$ by
\begin{equation}
A_1 \otimes \dots \otimes A_k \mapsto A = R \circ (A_1 \otimes \dots \otimes A_k),
\end{equation}
and identify the lower left corner with $V_1^\vee \otimes \dots \otimes V_k^\vee \otimes V_1 \otimes \dots \otimes V_k$ via a reordering of the vectors in a tensor product. 
Under these identifications, 
the left vertical arrow becomes the isomorphism in \cref{CanonIsomVecSpEq} composed with the natural isomorphism
\begin{equation}
V^\vee \cong V_1^\vee \otimes \dots \otimes V_k^\vee.
\end{equation}
Hence if we start with $A_1 \otimes \dots \otimes A_k$ and traverse our diagram counter-clockwise, 
and then apply the evaluation map on $V_1^\vee \otimes V_1$, we get $\tr(A,V)$ which is the left hand side of \cref{LinAlgToProveEq}.
Therefore, the desired equality follows from the commutativity of our diagram.
\end{proof}

\begin{cor} \label{trace-interval-prop} 

Let $g \in \F_q[u]$ be a squarefree polynomial, and let
\begin{equation}
t \colon \F_q[u]/(g) \to \mathbb{C}
\end{equation} 
be an infinitame trace function, such that for some prime factor $\tau$ of $g$,
the function $t_\tau$ is a Dirichlet trace function.
Then for $n < \deg(g)$ we have 
\begin{equation*}
\Bigl| \sum_{\substack{f \in \mathbb F_q[u] \\ \deg (f) < n} }  t(f) \Bigr|  \leq   q^{ \frac{n}{2} + \frac{1}{2} }   
\Bigl( \prod_{\pi \mid g} \left( r(t_{\pi} )(1+Z)  + c(t_{\pi}) Z \right)^{\deg (\pi)}\Bigr ) [Z^n]
\end{equation*}
while for $n \geq \deg(g)$ we have
\begin{equation*}
\sum_{\substack{f \in \F_q[u] \\ \deg(f) < n}} t(f) = 0.
\end{equation*}

\end{cor}

%

\begin{proof} 

Suppose first that $n < \deg(g)$,
set $m = \deg(g)$, let $x_1,\dots, x_m \in \overline{\mathbb F_q}$ be the roots of $g$ ordered in such a way that $\tau(x_1) = 0$, 
and for each prime $\pi$ dividing $g$ let $\mathcal{F}_\pi$ be a sheaf giving rise to the trace function $t_\pi$.
Since $t_\tau$ is a Dirichlet trace function, 
in view of \cref{DirichletTraceFuncEx} and \cref{Kummer-sheaf}, we can take 
\begin{equation}
\mathcal{F}_\tau = \mathcal{L}_\chi(c(T-z))
\end{equation}
where $c \in (\F_q[u]/(\tau))^\times$, $z \in \F_q[u]/(\tau)$, and $\chi \colon (\F_q[u]/(\tau))^\times \to \mathbb{C}^\times$ is a character of order greater than $1$.

Fix $1 \leq i \leq m$. 
Since $g$ is squarefree, there exists a unique prime factor $\pi$ of $g$ such that $\pi(x_i) = 0$.
We define a sheaf $\mathcal{F}_i$ on $\mathbb{A}^1_{\overline{\F_q}}$ to be the base change of $\mathcal{F}_\pi$ along the embedding
$\mathbb F_q[u]/(\pi) \hookrightarrow \overline{\mathbb F_q}$ mapping $u$ to $x_i$.
Since $\tau(x_1) = 0$, for the case $i=1$ we have $\pi = \tau$, hence $\mathcal{F}_1$ is geometrically isomorphic to the Kummer sheaf $\mathcal{L}_\chi(c(T-z))$, so we conclude from \cref{Kummer-sheaf-properties}(7), 
that all the assumptions made in \cref{geometric-setup} are satisfied here.
From that notation we borrow the evaluation maps $e_i \colon \mathbb A^n \to \mathbb A^1$ defined in \cref{EvalEq}. 

Let $\sigma\in S_m$ be the unique permutation with $\Frob_q(x_i) =x_{\sigma(i) } $. Then 
\begin{equation}
\Frob_q \circ \ e_i = e_{\sigma(i)} \circ \Frob_q.
\end{equation}
As the embeddings defining the sheaves $\mathcal F_i$ and $\mathcal F_{\sigma(i)}$ differ by an application of $\Frob_q$,
we have an isomorphism $\Frob_q^*  \mathcal F_{\sigma(i)} \cong \mathcal F_i$ of sheaves. 
Consequently
\begin{equation}
\begin{split}
\Frob_q^* \bigotimes_{i=1}^m e_i^* \mathcal F_i &\cong 
\Frob_q^* \bigotimes_{i=1}^m e_{\sigma(i)}^* \mathcal F_{\sigma(i)} \cong  
\bigotimes_{i=1}^m (e_{\sigma(i)} \circ \Frob_q)^* \mathcal F_{\sigma(i)} \\
&\cong \bigotimes_{i=1}^m e_i^* \Frob_q^* \mathcal F_{\sigma(i)} \ \ \cong 
\bigotimes_{i=1}^m e_i^* \mathcal F_i,
\end{split}
\end{equation}
so we can descend the sheaf 
\begin{equation} \label{GeomSheafDef}
\overline{\mathcal{F}} = \bigotimes_{i=1}^m e_i^* \mathcal F_i
\end{equation}
from $\mathbb A^n_{\overline{\mathbb F_q}}$ to $\mathbb A^n_{\mathbb F_q}$, 
producing a sheaf $\mathcal F$ on $\mathbb A^n_{\mathbb F_q}$. 

Equivalently, we can choose a finite field extension $\mathbb F_{q'}$ of $\F_q$ over which $g$ splits completely, 
observe that $ e_i, $  $ \mathcal F_i$, and thus $\overline{\mathcal F}$  are all defined over $\mathbb F_{q'}$, 
and then descend $\overline{\mathcal F}$ from $\mathbb A^n_{\mathbb F_{q'}}$ to $\mathbb A^n_{\mathbb F_q}$. 
Since the sheaves $\{\mathcal F_i\}_{i=1}^m$ are mixed of nonpositive weights, 
so are their pullbacks and tensor products, 
hence the descent $\mathcal F$ is mixed of nonpositive weights.
A similar descent argument applies to the sheaf
\begin{equation} \label{SmallDescentArgumentEq}
\bigotimes_{\substack{1 \leq i \leq m \\ \pi(x_i) = 0}} e_i^* \mathcal{F}_i
\end{equation}
for every prime $\pi$ dividing $g$.

We claim that for every $f \in \F_q[u]$ with $\deg(f) < n$ we have 
\begin{equation} \label{GeomArithTraceFuncComparisonEq}
t_{\mathcal F } ( f ) = t( f).
\end{equation}
To see this, note that we have the isomorphisms 
\begin{equation} \label{VectorSpacesFrobqIsomEq}
\mathcal{F}_f \cong \overline{\mathcal F}_f \cong \Bigl( \bigotimes_{i=1}^m e_i^* \mathcal F_i \Bigr)_f 
\cong \bigotimes_{\pi \mid g} \Biggl( \bigotimes_{\substack{ 1 \leq i \leq m \\  \pi(x_i) = 0 }} e_i^* \mathcal F_i \Biggr)_f
\end{equation}
of vector spaces with an action of $\Frob_q$.
In view of our descent argument for the sheaf in \cref{SmallDescentArgumentEq}, 
for every $\pi$ dividing $g$ the linear map $\Frob_q$ on the vector space
\begin{equation}
V_\pi = \bigotimes_{i \in I_\pi} (e_i^* \mathcal F_i)_f, \quad I_\pi = \{1 \leq i \leq m : \pi(x_i) = 0\}
\end{equation}
is the tensor product of the linear maps
\begin{equation}
\Frob_q \colon \left( e_i^* \mathcal{F}_i \right)_f \to \left( e_{\sigma(i)}^*\mathcal{F}_{\sigma(i)} \right)_f, \quad i \in I_\pi.
\end{equation}

Since $\sigma$ permutes the sets $I_\pi$ cyclically, each $I_\pi$ contains $\deg(\pi)$ elements,
and for every $i \in I_\pi$ we have an isomorphism 
\begin{equation}
(e_i^*\mathcal{F}_i)_f \cong (\mathcal{F}_i)_{e_i(f)} =  (\mathcal{F}_i)_{f(x_i)} \cong (\mathcal{F}_\pi)_f
\end{equation}
of $\Frob_{q^{\deg(\pi)}}-$modules, we get from \cref{VectorSpacesFrobqIsomEq} and \cref{LinearAlgebraLem} that

\begin{equation*}
\begin{split}
t_{\mathcal{F}}(f) = \tr(\Frob_q, \mathcal{F}_f) &= \prod_{\pi \mid g} \tr(\Frob_q, V_\pi) \\
&= \prod_{\pi \mid g} \tr(\Frob_{q^{\deg(\pi)}}, (\mathcal{F}_{\pi})_{f}) = \prod_{\pi \mid g} t_{\pi}(f) = t(f).
\end{split}
\end{equation*}


It follows from \cref{GeomArithTraceFuncComparisonEq} and the Grothendieck-Lefschetz trace formula that
\begin{equation}
\Bigl| \sum_{\substack{f \in \mathbb F_q[u] \\ \deg f < n} }  t( f) \Bigr| = 
\Bigl| \sum_{x \in \mathbb{A}^n(\F_q) }   t_{\mathcal F} (x)\Bigr|  \leq
\sum_{j=-\infty}^{\infty} \left| \tr (\Frob_q, H^j_c( \mathbb A^n_{\overline{\mathbb F_q}}, \overline{\mathcal F })) \right|.
\end{equation}
Since $\mathcal{F}$ is mixed of nonpositive weights, Deligne's Riemann Hypothesis and \cref{GeomSheafDef} bound the above by
\begin{equation}
\sum_{j=-\infty}^{\infty} q^{j/2} 
\dim H^j_c \Bigl( \mathbb A^n_{\overline{\mathbb F_q}},\bigotimes_{i=1}^m e_i^* \mathcal F_i \Bigr). 
\end{equation} 
\cref{geometric-vanishing} allows us to bound the sum above by
\begin{equation}
q^{\frac{n+1}{2}} \dim H^n_c\Bigl( \mathbb A^n_{\overline{\mathbb F_q}},\bigotimes_{i=1}^m e_i^* \mathcal F_i \Bigr)+ q^{ \frac{n+1}{ 2}} \dim H^{n+1} _c\Bigl( \mathbb A^n_{\overline{\mathbb F_q}},\bigotimes_{i=1}^m e_i^* \mathcal F_i \Bigr).
\end{equation}

Since the sheaves $\mathcal{F}_i$ have no finitely supported sections, and $n \leq m$,
we get from \cref{betti-number-bound} that the above is at most 
\begin{equation}
q^{\frac{n+1}{2}} \left( \prod_{i=1}^m( \rank(\mathcal F_i) (1+Z) + c_F(\mathcal F_i) Z) \right) [Z^n]
\end{equation}
and since each $\mathcal F_{\pi}$ occurs with multiplicity $\deg (\pi)$ among the $\mathcal F_i$, we get
\begin{equation}
q^{ \frac{n}{2} + \frac{1}{2} }   \Biggl( \prod_{\pi \mid g} \left( \rank(\mathcal{F}_{\pi} )(1+Z)  + c_F (\mathcal F_{\pi}) Z \right)^{\deg (\pi)} \Biggr) [Z^n].
\end{equation}
By \cref{SheafFuncDictionaryDef} we have $\rank(\mathcal{F}_\pi) = r(t_\pi)$, 
and since the sheaves $\mathcal{F}_\pi$ are infinitame, we get from \cref{conductor-little-lemmas}(6) and \cref{InfinitameConductorDefEq} that $c_F(\mathcal{F}_\pi) = c(t_\pi)$, so the above equals
\begin{equation}
q^{ \frac{n}{2} + \frac{1}{2} }   \Biggl( \prod_{\pi \mid g} \left( r(t_{\pi} )(1+Z)  + c (t_{\pi}) Z \right)^{\deg (\pi)} \Biggr) [Z^n]
\end{equation}
as required.

Suppose now that $n \geq m$.
By the Chinese Remainder Theorem, and the fact that each residue class mod $g$ contains $q^{n-m}$ polynomials of degree less than $n$,
we have
\begin{equation*}
\sum_{\substack{f \in \F_q[u] \\ \deg(f) < n}} t(f) = 
q^{n-m} \sum_{f \in \F_q[u]/(g)} \prod_{\pi \mid g} t_\pi(f) = 
q^{n-m} \prod_{\pi \mid g} \sum_{f \in \F_q[u]/(\pi)} t_\pi(f).
\end{equation*}
For $\pi = \tau$ we are summing a Dirichlet trace function over all residue classes, so this sum vanishes,
hence the product is zero.
\end{proof}

We deduce \cref{intro-trace-interval-bound}.

\begin{proof}

Applying \cref{trace-interval-prop} with $n = \lceil \log_q(X) \rceil$, and recalling \cref{TraceFuncDefi}, we get the nound
\begin{equation}
\begin{split}
\sum_{ \substack {f \in \mathbb F_q[u] \\ |f| < X}} t(f) &\ll
q^{ \frac{n}{2} + \frac{1}{2} } \left( \prod_{\pi \mid g} \left( \rank(t_{\pi} )(1+Z)  + c_F (t_{\pi}) Z \right)^{\deg (\pi)} \right) [Z^n] \\
&\ll X^{\frac{1}{2}} \left( \prod_{\pi \mid g} \left( r(t)(1+Z)  + c(t) Z \right)^{\deg (\pi)} \right) [Z^n].
\end{split}
\end{equation}
The coefficients of the polynomial above are nonnegative,
so the coefficient of $Z^n$ is bounded by the sum of all the coefficients.
This sum is the value of the polynomial at $Z=1$ which equals
\begin{equation*}
\begin{split}
X^{\frac{1}{2}} \prod_{\pi \mid g} \left( 2r(t)  + c(t) \right)^{\deg (\pi)} &=
X^{\frac{1}{2}} \left( 2r(t)  + c(t) \right)^{\sum_{\pi \mid g} \deg (\pi)} \\
&= X^{\frac{1}{2}} \left( 2r(t)  + c(t) \right)^{\deg(g)} = X^{\frac{1}{2}} |g|^{\log_q( 2r(t)  + c(t))}
\end{split}
\end{equation*}
as required.
\end{proof}

\section{M\"{o}bius, discriminants, resultants}

\begin{notation} \label{IntervalNotation}

Define an interval $\mathcal I$ in $\mathbb F_q[u]$ to be a set of the form
\begin{equation}
\mathcal I_{f,d} = \{ f+ g : g \in \mathbb F_q[u], \deg(g) < d\}
\end{equation}
for any $f \in \mathbb F_q[u]$ and a nonnegative integer $d$. 
Define the dimension, length, and degree of the interval $\mathcal I = \mathcal I_{f,d}$ to be 
\begin{equation}
\dim(\mathcal I) = d, \quad \operatorname{len}(\mathcal I) = q^d, \quad \deg(\mathcal{I}) = \max\{d,\deg(f)\}.
\end{equation}
For instance, the set of monic polynomials of degree $d$ is an interval of dimension $d$, which we can see by taking $f=u^d$. 

Associated to an interval $\mathcal I = \mathcal I_{f, d}$, we have the subset $\mathcal I_{\overline{\mathbb F_q}} $ of $\overline{\mathbb F_q}[u]$, similarly defined as 
\begin{equation}
\mathcal I_{\overline {\mathbb{F}_q}} = \{ f+ g : g \in \overline{\mathbb F_q}[u], \deg(g) < d\}.
\end{equation} 
Writing $f = \tau_0u^0 + \dots + \tau_{\deg(f)}u^{\deg(f)}$, we get that
\begin{equation}
\mathcal I_{\overline {\mathbb{F}_q}} = \left\{ \sum_{i = 0}^{d-1} \theta_iu^i + \sum_{j= d}^{\deg(f)} \tau_ju^j: \theta_i \in \overline{\F_q} \right\},
\end{equation} 
and say that $\{\theta_i\}_{i=0}^{d-1}$ are the coordinates of $\mathcal{I}_{\overline{\F_q}}$, or by abuse of notation,
the coordinates of $g \in \mathcal{I}_{\overline{\F_q}}$.
We call $\theta_0$ the lowest coordinate of $\mathcal I_{\overline{\F_q}}$.
Note that the number of coordinates of $\mathcal I_{\overline{\F_q}}$ is $\dim(\mathcal I)$.

\end{notation}

\subsection{Relating M\"{o}bius to Dirichlet characters}

The main goal of this section, generalizing \cite[Section 3]{SS}, is to prove \cref{formula-for-Mobius}, which gives a formula for the quantity $\mu(F(u,f(u))$ from \cref{NewRes}, when restricted to special subsets of the form $f(u) = r(u)+ s(u)^p$ for fixed $r(u)$ and varying $s(u)$. Later, we will use this to control the average of $\mu(F(u,f(u)))$ by averaging over each special subset separately. 

\subsubsection{Zeuthen's rule}

We recall Zeuthen's rule from \cite[Lemma 4.6]{CCG} in a slightly generalized form.

\begin{notation} \label{CCGNotation}

Let $f_1(u,T)$ and $f_2(u,T)$ be two polynomials in $\mathbb F_q[u,T]$. Set
\begin{equation}
Z_{f_i} = \{(a,b) \in \overline{\F_q}^2 : f_i(a,b) = 0 \}, \quad i \in \{1,2\}.
\end{equation}
In case $Z_{f_1} \cap Z_{f_2}$ is finite, for any  
\begin{equation} \label{xuxtxEq}
x = (u_x, t_x) \in \overline{\F_q}^2
\end{equation}
we denote by 
\begin{equation} \label{DefIntrsctionNbrEq}
i_x( Z_{f_1},Z_{f_2}) = \dim_{\overline{\F_q}} \overline{\F_q}[u,T]_{(u-u_x, T-t_x)} / (f_1,f_2) 
\end{equation}
the intersection number of $Z_{f_1}$ and $Z_{f_2}$ at $x$.
One readily checks that the quantity above is positive if and only if $x \in Z_{f_1} \cap Z_{f_2}$.


%
%
%
%
%
%

Let $d,d'$ be nonnegative integers. As in \cite[Section 3]{CCG}, we denote by
\begin{equation}
R_{d,d'}(\alpha(u), \beta(u))
\end{equation}
the resultant defined by the universal formula for a polynomial $\alpha(u) \in \F_q[u]$ of degree at most $d$,
and a polynomial $\beta(u) \in \F_q[u]$ of degree at most $d'$ in terms of the coefficients of these polynomials. 
Omitting $d,d'$ we set
\begin{equation}
R(\alpha(u), \beta(u)) = R_{\deg(\alpha), \deg(\beta)}(\alpha(u), \beta(u)).
\end{equation}

In this work, every time we write $R_{d,d'}(\alpha(u), \beta(u))$ we will in fact have $d = \deg(\alpha)$, 
in which case \cite[(3.2)]{CCG} says that
\begin{equation} \label{CCGThreePointTwoEq}
R_{d,d'}(\alpha(u), \beta(u)) = \alpha_{d}^{d' - \deg(\beta)} R(\alpha(u), \beta(u))
\end{equation}
where $\alpha_d$ is the coefficient of $u^d$ in $\alpha(u)$.
We conclude that
\begin{equation} \label{CCGThreePointTwoConclusionEq}
R_{d,d''}(\alpha(u), \beta(u)) =
\alpha_d^{d'' - d'} R_{d,d'}(\alpha(u), \beta(u))
\end{equation}
for any integer $d'' \geq \deg(\beta)$.
We also recall from \cite[(3.1)]{CCG} that
\begin{equation} \label{DiscriminantProductOverRootsEq}
R(\alpha(u), \beta(u)) = \alpha_d^{\deg(\beta)} \prod_{\substack{z \in \overline{\F_q} \\ \alpha(z) = 0}} \beta(z).
\end{equation}

Occasionally, we will think of $f_1, f_2 \in \F_q[u][T]$ as polynomials in $T$ with coefficients from $\F_q[u]$.
For example, the leading coefficient of $f_1$ is the coefficient of the highest power of $T$.
Moreover, we use the notation $R(f_1,f_2)$ for the resultant of $f_1$ and $f_2$, 
always to be taken with respect to the variable $T$,
producing a polynomial in $\F_q[u]$.

For $\gamma \in \overline{\F_q}[u]$ and $u_0 \in \overline{\F_q}$, we denote by 
\begin{equation}
\mathrm{ord}_{u = u_0} \gamma(u) = \sup \{m \geq 0 : (u-u_0)^m \mid \gamma(u)\}
\end{equation}
the order of vanishing of $\gamma(u)$ at $u = u_0$.
All of the above is in fact valid for an arbitrary field in place of $\F_q$.

\end{notation}

\begin{lem} \label{modified-Zeuthen} 
Keep \cref{CCGNotation}, and suppose that $Z_{f_1} \cap Z_{f_2}$ is finite.
Then
\begin{equation} \label{mzr} 
\operatorname{ord}_{u=u_0} R (f_1,f_2) \geq \sum_{ c\in \overline{\F_q} } i_{(u_0, c)} (f_1, f_2) 
\end{equation} 
for every $u_0 \in \overline{\F_q}$, 
with equality if the leading coefficient of one of the polynomials $f_1, f_2$ does not vanish at $u_0$.

\end{lem}

\begin{proof} 

The case where one leading coefficient does not vanish is \cite[Lemma 4.6]{CCG}, so we only prove the inequality above.

Since $Z_{f_1} \cap Z_{f_2}$ is finite, we can find $\lambda \in \overline{\F_q}$ with $(u_0,\lambda) \notin Z_{f_1} \cap Z_{f_2}$.
Making the change of variable $T \mapsto T + \lambda$, which preserves both sides of the inequality above,
we can assume that $(u_0,0) \notin Z_{f_1} \cap Z_{f_2}$. 
In other words, the constant term of one of the polynomials $f_1, f_2$ does not vanish at $u_0$.

Let $d_1$ and $d_2$ be the degrees of $f_1$ and $f_2$ respectively, and set 
\begin{equation}
f_1'(u, T) = f_1( u , T^{-1}) T^{d_1}, \quad f_2'(u,T)= f_2( u, T^{-1} )T^{d_2}
\end{equation}
exchanging the constant and leading terms.
Then 
\begin{equation} 
R ( f_1', f_2')= (-1)^{d_1 d_2} R( f_1, f_2)
\end{equation} 
so 
\begin{equation}\label{resultant-inverse-equation}
 \operatorname{ord}_{u=u_0} R (f_1, f_2) = \operatorname{ord}_{u=u_0} R  (f_1', f_2').
\end{equation}

Since the leading coefficient of one of the polynomials $f_1', f_2'$ does not vanish at $u_0$, 
by the previous case we have equality in \cref{mzr}, namely
\begin{equation}\label{resultant-previous-case}
\operatorname{ord}_{u = u_0} R(f_1', f_2') = \sum_{ c \in \overline{\F_q} } i_{(u_0, c)} (f_1', f_2' ).
\end{equation}
Removing $c = 0$, we get
\begin{equation} \label{DroppingZeroEq}
\sum_{ c \in \overline{\F_q} } i_{(u_0, c)} (f_1', f_2' ) \geq \sum_{ c \in \overline{\F_q}^\times } i_{(u_0, c)} (f_1', f_2' ).
\end{equation}

For every $c \in \overline{\F_q}^\times$, mapping $T$ to $T^{-1}$ induces an isomorphism
\begin{equation}
\overline{\F_q}[u,T]_{(u-u_0, T- c^{-1})} \cong \overline{\F_q}[u,T]_{(u-u_0, T- c)}
\end{equation}
hence also the isomorphisms
\begin{equation}
\begin{split}
&\overline{\F_q}[u,T]_{(u-u_0, T- c^{-1})}/(f_1(u,T), f_2(u,T)) \cong \\
&\overline{\F_q}[u,T]_{(u-u_0, T- c)}/(f_1(u,T^{-1}), f_2(u, T^{-1})) \cong \\
&\overline{\F_q}[u,T]_{(u-u_0, T- c)}/(f_1'(u,T), f_2'(u, T)).
\end{split}
\end{equation} 
Therefore, by definition of intersection numbers in \cref{DefIntrsctionNbrEq}, we get
\begin{equation}
i_{(u_0, c^{-1})}(f_1, f_2) = i_{(u_0, c)}(f_1', f_2').
\end{equation}

 Inverting $c \in \overline{\F_q}^\times$, we get
\begin{equation}\label{AddingZeroBack}
\sum_{ c \in \overline{\F_q}^\times } i_{(u_0, c)} (f_1', f_2' ). = \sum_{ c \in \overline{\F_q}^\times } i_{(u_0, c)} (f_1, f_2) = \sum_{c \in \overline{\F_q}}  i_{(u_0, c)} (f_1, f_2)  .
\end{equation}
with the last equality because the constant term of one of the polynomials $f_1, f_2$ does not vanish at $u_0$ and so $i_{(u_0, 0)} (f_1, f_2) $ vanishes.  Combining \cref{resultant-inverse-equation,resultant-previous-case,DroppingZeroEq,AddingZeroBack}, we get \cref{mzr}.
\end{proof}


%

\begin{remark}

The proof above is valid for every algebraically closed field in place of $\overline{\F_q}$.

\end{remark}

\subsubsection{Resultant formula}

Using Zeuthen's rule,
we prove a formula for a resultant, which we later apply to the resultant of a polynomial and its derivative, namely the discriminant.
The latter is related to the value of M\"{o}bius by Pellet's formula.
Our formula is a variant of \cite[Theorem 4.5]{CCG}, 
with our condition on the interval $\mathcal I$ replacing the assumption on the degree therein. 

\begin{lem}\label{ccg-like}

Keep \cref{IntervalNotation} and \cref{CCGNotation}.
Suppose that $Z_{f_1} \cap Z_{f_2}$ is finite,
and let $\mathcal I$ be an interval in $\mathbb F_q[u]$. 
Assume that the degree of
$
f_1(u,g(u))
$
is independent of $g(u) \in \mathcal{I}_{\overline{\mathbb F_q}}$ and nonnegative.
Denote this degree by $d$, and let $d'$ be an integer satisfying
\begin{equation} \label{Definingd'condition}
\deg(f_2(u, g(u))) \le d', \quad g(u) \in \mathcal{I}_{\overline{\F_q}}.
\end{equation}
Then there exists $c \in \mathbb F_q^{\times}$ (depending on $\mathcal I, f_1,f_2$) such that
\begin{equation}\label{eq-ccg-like} 
R_{d,d'} ( f_1(u,g(u)), f_2(u,g(u))) = c \prod_{x \in Z_{f_1} \cap Z_{f_2} } (g(u_x)-t_x) ^{i_x \left( Z_{f_1},Z_{f_2} \right)}
\end{equation}
for any $g(u) \in \mathcal{I}_{\overline{\F_q}}$.

\end{lem}

\begin{proof} 

We assume first that $\dim(\mathcal I) \geq 2$, namely that $\mathcal{I}_{\overline{\F_q}}$ has at least two coordinates $\theta_0, \theta_1$.

As in the proof of \cite[Theorem 4.5]{CCG}, the first step is to prove that there exists $c \in \F_q^\times$ and an assignment of a positive integer $e_x$ to each $x \in Z_{f_1} \cap Z_{f_2}$ such that 
\begin{equation}\label{approximate-formula} 
R_{d,d'}( f_1(u,g), f_2(u,g)) = c \prod_{x \in Z_{f_1} \cap Z_{f_2} } (g(u_x)-t_x) ^{e_x} \end{equation}
for every $g \in \mathcal{I}_{\overline{\F_q}}$.
In the proof of this factorization we essentially follow the proof of \cite[Lemma 4.4]{CCG},
and the first paragraph in the proof of \cite[Theorem 4.5]{CCG}, with some modifications to account for the fact that we range over all polynomials $g$ in a base-changed interval rather than over all monic polynomials of a given degree.


Note that the left hand side of \cref{approximate-formula} is a polynomial in the coordinates of $g$,
and that for every $x \in Z_{f_1} \cap Z_{f_2}$, the polynomial $g(u_x) - t_x$ is linear in the coordinates of $g$, and thus geometrically irreducible.
Moreover, using our assumption that $g$ has at least two coordinates, one readily checks that for any $y \in Z_{f_1} \cap Z_{f_2}$ different from $x$, the polynomial $g(u_y) - t_y$ in the coordinates of $g$, is not a multiple of $g(u_x) - t_x$ by a scalar from $\overline{\F_q}$. 
Hence, by the Nullstellensatz, in order to establish \cref{approximate-formula} with some $c \in \overline{\F_q}^\times$,
it suffices to show that our resultant vanishes if and only if 
\begin{equation} \label{RHSvanishEq}
g(u_x) - t_x = 0
\end{equation}
for some $x \in Z_{f_1} \cap Z_{f_2}$. 

Our resultant vanishes if and only if $f_1(u, g(u))$ and $f_2(u, g(u)))$ share a root $u_0 \in \overline{\F_q}$, or the coefficients of $u^d$ in $f_1(u, g(u))$ and of $u^{d'}$ in $f_2(u, g(u))$ both vanish.
The latter possibility is excluded by our definition of $d$,
and the former is equivalent to the existence of an
\begin{equation}
x = (u_0, g(u_0)) \in Z_{f_1} \cap Z_{f_2}
\end{equation}
for which \cref{RHSvanishEq} is satisfied.
Hence, \cref{approximate-formula} is established with $c \in \overline{\F_q}^\times$.

To check that $c$ is in $\F_q^\times$ (and not merely in $\overline{\F_q}^\times$),
we note that (each linear factor, and thus) the product on the right hand side of \cref{approximate-formula} is monic when viewed as a polynomial in the lowest coordinate $\theta_0$ of $g$, so $c$ is a coefficient of the polynomial on the left hand side of \cref{approximate-formula}. The latter is clearly a polynomial over $\F_q$, so indeed $c \in \F_q^\times$.

Next, in order to establish \cref{eq-ccg-like}, we fix $y \in Z_{f_1} \cap Z_{f_2}$, 
and check that $e_y  = i_y( Z_{f_1},Z_{f_2})$.
Since 
$
\dim(\mathcal I) \geq 2, 
$
we can find $g_0 \in \mathcal I_{\overline{\mathbb F_q }} $ such that 
\begin{equation} \label{Constructg0Eq}
\{x \in Z_{f_1} \cap Z_{f_2} : g_0(u_x) = t_x \} = \{y\}.
\end{equation}
This choice of $g_0$ is such that
\begin{equation} \label{OrderOfVanEq}
\mathrm{ord}_{z = 0} (g_0(u_x)+z-t_x)^{e_x} = \begin{cases}
e_y,\  x=y \\
0, \ \ x \neq y
\end{cases}
\end{equation}
for any $x \in Z_{f_1} \cap Z_{f_2}$.
We conclude from \cref{approximate-formula} and \cref{OrderOfVanEq} that
\begin{equation} \label{ResultantBeforeTildesEq}
\mathrm{ord}_{z = 0} R_{d, d'}( f_1(u,g_0+z), f_2(u,g_0+z)) = e_y
\end{equation} 
and set
\begin{equation} \label{DefTwoTildesEq}
\widetilde{f}_1(u,z) = f_1(u,g_0+z), \quad \widetilde{f}_2(u,z) = f_2(u,g_0+z).
\end{equation} 

Since the degree $d$ of $f_1(u,g)$ is independent of $g \in \mathcal{I}_{\overline{\F_q}}$ by assumption, we see that the coefficient of $u^d$ in $\widetilde{f}_1$ does not vanish for any $z \in \overline{\F_q}$, in particular for $z = 0$.
We therefore get from \cref{CCGThreePointTwoEq} that
\begin{equation} \label{ResultantAfterTildesEq}
\mathrm{ord}_{z = 0} R_{d,d'} \left( \widetilde{f}_1(u,z), \widetilde{f}_2(u,z) \right) = 
\mathrm{ord}_{z = 0} R \left( \widetilde{f}_1(u,z), \widetilde{f}_2(u,z) \right). 
\end{equation}

We apply the case of equality in \cref{modified-Zeuthen} to the above. 
This requires checking that 
\begin{equation}
\left| Z_{\widetilde{f_1}} \cap Z_{\widetilde{f_2}} \right| < \infty,
\end{equation} 
and that the coefficients of the highest powers of $u$ in $\widetilde{f}_1$ and $\widetilde{f}_2$ do not have a common zero at $z=0$.
The former follows from our assumption that $Z_{f_1} \cap Z_{f_2}$ is finite,
and the latter was deduced above from our assumption that the degree $f_1(u,g)$ is independent of $g$. 
It then follows from \cref{modified-Zeuthen}, \cref{ResultantBeforeTildesEq}, and \cref{ResultantAfterTildesEq} that 
\begin{equation}
e_y = \sum_{\lambda \in \overline{\mathbb F_q}} i_{(\lambda, 0)}( Z_{ \widetilde{f}_1},Z_{\widetilde{f}_2}).
\end{equation}

Using the definition of $\widetilde{f_1}, \widetilde{f_2}$ in \cref{DefTwoTildesEq} we get from the above that
\begin{equation}
e_y = \sum_{\lambda \in \overline{\mathbb F_q}} i_{(\lambda, g_0(\lambda))} (Z_{f_1},Z_{f_2}).
\end{equation}
By construction of $g_0$ in \cref{Constructg0Eq}, the summands with $(\lambda, g_0(\lambda)) \neq y$ vanish, 
so our sum reduces to $i_{y} (Z_{f_1}, Z_{f_2})$,
and it follows that $e_y = i_y(Z_{f_1}, Z_{f_2})$ as required. 

Assume now that $\dim(\mathcal I) = 1$,
so that there exists some $h \in \F_q[u]$ such that
\begin{equation}
\mathcal{I}_{\overline{\F_q}} = \{h +z : z \in \overline{\mathbb F_q} \}.
\end{equation}
In this case, \cref{eq-ccg-like} can be rewritten as
\begin{equation} \label{SecondCaseApproxEq}
R_{d,d'} ( f_1(u, h(u) + z), f_2(u, h(u) + z)) = c \prod_{\alpha \in \overline{\F_q}} (z - \alpha)^{m_\alpha}
\end{equation}
where
\begin{equation}
m_\alpha = \sum_{\substack{x \in Z_{f_1} \cap Z_{f_2} \\ t_x - h(u_x) = \alpha}} i_{x}(Z_{f_1}, Z_{f_2}).
\end{equation}

As in the previous case, applying \cref{CCGThreePointTwoEq} and \cref{modified-Zeuthen} we get that
\begin{equation}
\mathrm{ord}_{z = \alpha} \ R_{d,d'} ( f_1(u, h + z), f_2(u, h + z))
\end{equation}
equals
\begin{equation*}
\begin{split}
\mathrm{ord}_{z = \alpha} \ R ( f_1(u, h + z), f_2(u, h + z)) &= 
\sum_{\lambda \in \overline{\F_q}} i_{(\lambda, \alpha)}\left(Z_{f_1(u, h + z)}, Z_{f_2(u, h + z)} \right) \\
&= \sum_{\lambda \in \overline{\F_q}} i_{(\lambda, \alpha + h(\lambda))}\left(Z_{f_1}, Z_{f_2} \right).
\end{split}
\end{equation*}

We can restrict the sum above to those $\lambda \in \overline{\F_q}$ with 
\begin{equation}
(\lambda, \alpha + h(\lambda)) = (u_x, t_x)
\end{equation}
for some $x \in Z_{f_1} \cap Z_{f_2}$, since the other terms vanish.
We then see that our sum equals $m_\alpha$, so \cref{SecondCaseApproxEq} holds with some $c \in \overline{\F_q}^\times$.
To show that in fact $c \in \F_q^\times$, one can argue as in the previous case.

Suppose at last that $\dim(\mathcal{I}) = 0$, or equivalently that $\operatorname{len}(\mathcal{I}) = 1$.
Note that the left hand side of \cref{eq-ccg-like} is in $\F_q$,
and by invariance under the action of $\mathrm{Gal}(\overline{\F_q}/\F_q)$,
the same is true for the product on the right hand side of \cref{eq-ccg-like}.
Hence, \cref{eq-ccg-like} boils down to the fact, proven earlier, 
that our resultant vanishes if and only if $g(u_x)-t_x = 0$ for some $x \in Z_{f_1} \cap Z_{f_2}$.
\end{proof}

\begin{remark}
It is possible to extract from the proof an explicit expression for the constant $c$.
\end{remark}

\begin{notation} \label{TvarNotation}

Keep \cref{CCGNotation}.
Define the polynomial
\begin{equation} \label{DefMf1f2Eq}
M(f_1,f_2) = \rad(R(f_1,f_2))= \prod_{\pi \mid R(f_1,f_2)} \pi
\end{equation}
in $\F_q[u]$, and let
\begin{equation} \label{Leq}
L(f_1, f_2) \in \F_q[u]
\end{equation}
be the greatest common divisor of the leading coefficients of $f_1$ and $f_2$.

Assume from now on that $q$ is odd, and denote the unique multiplicative quadratic character of $\F_q$ by $\chi_2$.
For every $x \in \F_q$ we have 
\begin{equation} \label{AbusedCongruenceEq}
\chi_{2}(x) = \begin{cases}
1 & x \in {\F_q^\times}^2\\
-1 & x \in \F_q^\times \setminus {\F_q^\times}^2 \\
0 & x= 0.
\end{cases}
\end{equation}

For $a \in \F_q[u]$ and a nonzero $b \in \F_q[u]$ we denote by 
\begin{equation}
\left( \frac{a}{b} \right) = \left( \frac{a}{b} \right) _2
\end{equation} 
the Jacobi symbol (quadratic residue symbol) in $\F_q[u]$, 
studied for instance in \cite[Chapter 3]{Ros}.
For a nonzero $M \in \F_q[u]$, we denote by
\begin{equation}
N_{\F_q[u]/(M)}^{\F_q} \colon \F_q[u]/(M) \to \F_q
\end{equation}
the norm map defined by 
\begin{equation} \label{NormDefEq}
N_{\F_q[u]/(M)}^{\F_q}(f) = \prod_{\substack{a \in \overline{\F_q} \\ M(a) = 0}} f(a)
\end{equation}
where $f(a)$ stands for the image of $f \in \F_q[u]/(M)$ in $\overline{\F_q}$ under the map sending $u$ to $a$.
This map is surjective, and we have
\begin{equation} \label{NormQuadCharEq}
\chi_2 \left (N^{\F_q}_{\F_q[u]/(M)}(f) \right) = \left( \frac{f}{M} \right).
\end{equation} 

\end{notation}

The following proposition, whose proof builds on \cref{ccg-like}, is the key to deducing \cref{formula-for-Mobius}.
It is the generalization of \cite[Lemma 3.1]{SS} needed here.

\begin{prop}\label{SaSh-like} 
Keep \cref{TvarNotation}, and the assumptions of \cref{ccg-like}.
Suppose that $\deg_T(f_1) \geq 1$.
Then there exists a polynomial 
\begin{equation}
W(u,T) \in (\F_q[u] / (M(f_1,f_2)))[T]
\end{equation}
that satisfies the following two properties.
\begin{itemize}

\item For each root $a \in \overline{\F_q}$ of $M(f_1,f_2)$, 
the image $W(a,T)$ of $W(u,T)$ in $\overline{\F_q}[T]$ under the map sending $u$ to $a$ satisfies
\begin{equation} \label{OrderOfVanishingForWEq}
\mathrm{ord}_{T = b}W(a,T) = i_{(a,b)}(Z_{f_1},Z_{f_2})
\end{equation}
for every $b \in \overline{\F_q}$;

\item for all $g \in \mathcal I$ we have
\begin{equation} \label{Resultant-QuadCharEq}
\chi_2( R_{d,d'} ( f_1(u,g), f_2(u,g) )) = \left( \frac{W(u, g)}{M(f_1,f_2)} \right)
\end{equation}
the right hand side being the Jacobi symbol.

\end{itemize}
\end{prop} 

\begin{remark}

The first property above satisfied by $W(u,T)$ determines it up to multiplication by an element of $(\F_q[u]/(M(f_1,f_2)) )^\times$.  

\end{remark}

\begin{proof} 


Fix a prime $\pi \mid M(f_1,f_2)$. 
For a root $a \in \overline{\mathbb F_q}$ of $\pi$ define the polynomial 
\begin{equation}\label{W-over-one-root} 
W^{(\pi,a)}(T) = \prod_{ \substack{ x \in \overline{\mathbb F_q}^2 \\ u_x = a }}  
(T - t_x)^{ i_{x}(Z_{f_1}, Z_{f_2})} \in \overline{\mathbb F_q}[T].
\end{equation}
That the above is indeed a polynomial follows from the assumption, made in \cref{ccg-like}, that $Z_{f_1} \cap Z_{f_2}$ is finite.

We claim that $W^{(\pi,a)}(T)$ belongs to
$
\mathbb F_{q^{\deg (\pi)} }[T],
$
and that its pullback 
\begin{equation}
W_0^{(\pi,a)}(u,T) \in (\mathbb F_q[u]/(\pi))[T] 
\end{equation}
under the isomorphism from $\mathbb F_q[u]/(\pi)$ to $\mathbb F_{q^{\deg (\pi)} }$ sending $u$ to $a$, 
is independent of the root $a$. 

To prove the claim, note that the function $x \mapsto i_x (Z_{f_1}, Z_{f_2})$ is constant on each orbit of the natural action 
\begin{equation}
\operatorname{Gal}(\overline{\mathbb F_q } /\mathbb F_q) \curvearrowright \overline{\mathbb{F}_q}^2,
\end{equation}
so it is also constant on orbits of the stabilizer of $a$ in $\Gal(\overline{\F_q}/\F_q)$, namely the subgroup
\begin{equation}
\operatorname{Gal}(\overline{\mathbb F_q } /\mathbb F_{q}(a)) = 
\operatorname{Gal}(\overline{\mathbb F_q } /\mathbb F_{q^{\deg (\pi)}}).
\end{equation}
It follows that $W^{(\pi,a)}(T)$ is invariant under $\operatorname{Gal}(\overline{\mathbb F_q } /\mathbb F_{q^{\deg (\pi)}})$,
hence 
\begin{equation}
W^{(\pi,a)}(T) \in \F_{q^{\deg(\pi)}}[T].
\end{equation}
We also conclude that for every $\sigma \in \Gal(\overline{\F_q}/\F_q)$ we have
\begin{equation}
\sigma \left( W^{(\pi,a)}(T) \right) = W^{(\pi,\sigma(a))}(T).
\end{equation}

Since the isomorphism from $\mathbb F_q[u]/(\pi)$ to $\mathbb F_{q^{\deg (\pi)} }$  sending $u$ to $\sigma(a)$ is the composition of $\sigma$ on the isomorphism sending $u$ to $a$, we get that $W_0^{(\pi,a)}(u,T)$ is indeed independent of the chosen root $a$ of $\pi$. We denote this polynomial by $W_0^{(\pi)}(u,T)$,
and use the Chinese remainder theorem to define a polynomial 
\begin{equation}
W_0(u,T) \in (\F_q[u]/(M(f_1,f_2)))[T]
\end{equation}
that reduces mod $\pi$ to $W_0^{(\pi)}(u,T)$ for every $\pi \mid M(f_1, f_2)$.


Next we claim that
\begin{equation} \label{NormIntersectionEq}
N_{ \mathbb F_q[u]/ (M(f_1,f_2))}^{\mathbb F_q} ( W_0(u,g)) = \prod_{x \in Z_{f_1} \cap Z_{f_2} } (g(u_x)-t_x) ^{i_x \left( Z_{f_1},Z_{f_2} \right)}
\end{equation}
for any $g \in \mathcal{I}$.
By definition of the norm map from \cref{NormDefEq}, we have
\begin{equation}
\begin{split}
N_{ \mathbb F_q[u]/ (M(f_1,f_2))}^{\mathbb F_q} ( W_0(u,g)) = 
\prod_{\substack{a \in \overline{\F_q} \\ M(f_1,f_2)(a) = 0}} W_0(a,g(a)).
\end{split}
\end{equation}
By definition of $W_0(u,T)$, independece of $a$, and the fact that reduction mod $\pi$ commutes with plugging $g$ in $T$, 
the above equals
\begin{equation}
\prod_{\pi \mid M(f_1,f_2)} \prod_{\substack{a \in \overline{\F_q} \\ \pi(a) = 0}} W_0^{(\pi,a)}(a, g(a))  = 
\prod_{\pi \mid M(f_1,f_2)} \prod_{\substack{a \in \overline{\F_q} \\ \pi(a) = 0}} W^{(\pi,a)}(g(a)).
\end{equation}
From the definition of $W^{(\pi,a)}(T)$ in \cref{W-over-one-root} we get
\begin{equation} \label{PartialProdEq}
N_{ \mathbb F_q[u]/ (M(f_1,f_2))}^{\mathbb F_q} ( W_0(u,g)) = \prod_{ \substack{ x \in Z_{f_1} \cap Z_{f_2} \\ M(f_1, f_2)(u_x) = 0}}  
(g(u_x) - t_x)^{ i_{x}(Z_{f_1}, Z_{f_2})}.
\end{equation}

For every $x \in Z_{f_1} \cap Z_{f_2}$, the polynomials $f_1(u_x,T), f_2(u_x,T)$ vanish at $t_x$.
Hence $R(f_1(u_x,T), f_2(u_x,T)) = 0$, so $R(f_1, f_2)$ vanishes at $u_x$.
From the definition of $M(f_1,f_2)$ in \cref{DefMf1f2Eq} we conclude that $M(f_1,f_2)(u_x) = 0$,
so \cref{PartialProdEq} coincides with the right hand side of \cref{NormIntersectionEq} as required for our claim. 



Finally, we take $c \in \F_q^\times$ from \cref{ccg-like} that satisfies 
\begin{equation} \label{ReversedCCGLikeEq}
c \prod_{x \in Z_{f_1} \cap Z_{f_2} } (g(u_x)-t_x) ^{i_x \left( Z_{f_1},Z_{f_2} \right)} =
R_{d,d'} ( f_1(u,g(u)), f_2(u,g(u))) 
\end{equation} 
and choose $c' \in (\mathbb F_q[u]/(M(f_1,f_2)))^\times$ such that $N_{ \mathbb F_q[u]/ (M(f_1,f_2))}^{\mathbb F_q} (c') = c$.
Define $W(u,T) = c' W_0(u,T)$, 
so that from \cref{NormIntersectionEq} and \cref{ReversedCCGLikeEq} we get 
\begin{equation}
N_{ \mathbb F_q[u]/ M(f_1,f_2)}^{\mathbb F_q} ( W(u,g)) = R_{d,d'} ( f_1(u,g(u)), f_2(u,g(u))).
\end{equation}
Applying $\chi_2$ to the above, it follows from \cref{NormQuadCharEq} that
\begin{equation*}
\left( \frac{W(u,g)}{M(f_1,f_2)} \right) =  \chi_2 \left(N_{ \mathbb F_q[u]/ M(f_1,f_2)}^{\mathbb F_q} ( W(u,g))\right) = \chi_2( R_{d,d'} ( f_1(u,g), f_2(u,g)) )
\end{equation*}
so \cref{Resultant-QuadCharEq} holds.
For $a,b \in \overline{\F_q}$ we have
\begin{equation*}
\begin{split}
\mathrm{ord}_{T=b} W(a,T) &= \mathrm{ord}_{T=b} W_0(a,T) =  \mathrm{ord}_{T = b} W_0^{(\pi)}(a,T) \\
&= \mathrm{ord}_{T = b} W_0^{(\pi,a)}(a,T) = \mathrm{ord}_{T = b} W^{(\pi,a)}(T) = i_{(a,b)}(Z_{f_1}, Z_{f_2})
\end{split}
\end{equation*}
so \cref{OrderOfVanishingForWEq} holds.
\end{proof}

\subsubsection{M\"{o}bius formula}

We set up much of the notation needed to state, prove, and apply \cref{formula-for-Mobius}.

\begin{notation} \label{DerivativeNotation}

Keep \cref{TvarNotation}.
Let $k \geq 1$ be an integer, let
\begin{equation} \label{FcoeffsEq}
F(u,T)  = \sum_{i=0}^k a_i(u)T^i \in \F_q[u,T], \quad a_k(u) \neq 0,
\end{equation} 
and let $c_1,c_2 \in \mathbb{R}$ such that
\begin{equation} \label{aBound}
c_1 \geq 0, \quad c_2 \leq 0, \quad \deg(a_i(u)) \leq c_1 + c_2 i, \quad 0 \leq i \leq k.
\end{equation}
We introduce the auxiliary function
\begin{equation} \label{AuxFuncEq}
E(c_1,c_2, x) = 2kc_1 + k\max\{0, c_2 + x\} - k + c_2k^2, \quad x \in \mathbb{R}.
\end{equation}

Let $\overline{\F_q(u)}$ be an algebraic closure of $\F_q(u)$,
and let $\alpha_1, \dots, \alpha_k \in \overline{\F_q(u)}$ be such that
\begin{equation} \label{FrootsEq}
F(u,T) = a_k\prod_{i=1}^k (T- \alpha_i). 
\end{equation}
We assume that $F$ is separable as a polynomial in $T$, namely that the roots $\alpha_1, \dots, \alpha_k$ are distinct.
Equivalently, $\frac{\partial F}{\partial T}(u, \alpha_i) \neq 0$ for every $1 \leq i \leq k$.

Set
\begin{equation} \label{DefiningFv}
F_v(u,T) = \frac{\partial F}{\partial u} (u, T) + v\frac{\partial F}{\partial T}  (u, T) \in \F_q[u,T,v]
\end{equation}
and for $r \in \F_q[u]$ put
\begin{equation}
F_{[r]}(u,T) = F_{\frac{dr}{du}}(u,T) = \sum_{j = 0}^k b_j(u)T^j \in \F_q[u,T]
\end{equation}
where
\begin{equation} \label{beq}
b_j = \frac{d a_j}{du} + (j+1) a_{j+1}\frac{dr}{du}, \quad 1 \leq j \leq k-1, \quad b_k = \frac{da_k}{du}.
\end{equation}

Following \cref{Leq}, in case $F_{[r]} \neq 0$, we further set
\begin{equation} \label{DefiningLFr}
k' = \deg_T(F_{[r]}), \ L_{F,r} = L(F, F_{[r]}) = \gcd(a_k,b_{k'}) \in \F_q[u],
\end{equation}
and denote by
\begin{equation}
R(F, F_v) \in \F_q[u,v], \quad R(F, F_{[r]}) \in \F_q[u]
\end{equation} 
the resultants in the variable $T$.
By \cref{DiscriminantProductOverRootsEq} and \cref{DefiningFv}, we have
\begin{equation} \label{RFFvFormula}
\begin{split}
R ( F, F_v) &= a_k^{ \deg_T F_v} \prod_{i=1}^k F_v (u, \alpha_i) \\
&= a_k^{ \deg_T F_v} \prod_{i=1}^k \left( \frac{\partial F}{\partial u} ( u, \alpha_i) + v \frac{\partial F}{\partial T} (u, \alpha_i) \right).
\end{split}
\end{equation}
Following \cref{DefMf1f2Eq}, define 
\begin{equation} \label{DefMFrEq}
M_{F,r} = M(F, F_{[r]}) = \rad (R(F,F_{[r]})) = \prod_{\pi \mid R(F,F_{[r]})} \pi.
\end{equation}

Let $\mathcal{I}$ be an interval in $\F_q[u]$ such that for some $r \in \mathcal{I}$ the assumptions of \cref{SaSh-like} are satisfied for $f_1 = F, \ f_2 = F_{[r]}$, with $d'$ the least even integer satisfying \cref{Definingd'condition}.
We can then fix a polynomial
\begin{equation}
W_{F,r}(u,T) \in (\F_q[u]/(M_{F,r}))[T]
\end{equation}
such that for each root $a \in \overline{\F_q}$ of $M_{F,r}$ and $b \in \overline{\F_q}$ we have
\begin{equation} \label{OrderVanishWFratbEq}
\mathrm{ord}_{T = b}W_{F,r}(a,T) = i_{(a,b)}(Z_{F}, Z_{F_{[r]}}),
\end{equation}
and for all $g \in \mathcal I$ we have
\begin{equation} \label{LoadedJacobiEq}
\chi_2( R_{d,d'} ( F(u,g), F_{[r]}(u,g) )) = \left( \frac{W_{F,r}(u, g)}{M_{F,r}} \right).
\end{equation}
For a prime $\pi \mid M_{F,r}$, we denote by 
\begin{equation} \label{DefWFrEq}
W_{F,r}^{(\pi)}(T) \in (\F_q[u]/(\pi))[T]
\end{equation}
the reduction of $W_{F,r}$ mod $\pi$.

For a polynomial $f \in \F_q[u]$ we define its discriminant, following \cite[(2.3)]{CCG}, to be
\begin{equation}
\Delta(f) = \prod_{i < j} (\gamma_i - \gamma_j)^2.
\end{equation}
where $\gamma_1 , \dots, \gamma_{\deg(f)}$ are the roots of $f$ in $\overline{\F_q}$.
Denoting the leading coefficient of $f$ by $f_0$, and the degree of $f$ by $d$, we learn from \cite[(3.3)]{CCG} that
\begin{equation} \label{FromDiscriminantToResultantEq}
\Delta(f) = \frac{(-1)^{\frac{d(d - 1)}{2}} R_{d,d-1} (f,\frac{df}{du})}{f_0^{2d - 1}}.
\end{equation}
If we want to emphasize that the discriminant is taken with respect to the variable $u$, we write $\Delta_u(f)$.
For instance,
\begin{equation} \label{DiscVDefEq}
\Delta_v(R(F, F_v) ) \in \F_q[u]
\end{equation}
stands for the discriminant with respect to $v$ of the resultant in the variable $T$ of the polynomials $F$ and $F_v$ as above.

\end{notation}

The following corollary is the generalization of \cite[Lemma 3.2]{SS} needed to prove \cref{NewRes}.
The proof mainly rests on \cref{SaSh-like} and Pellet's formula 
\begin{equation} \label{PelletEq}
\mu(f) = (-1)^{\deg(f)} \chi_2(\Delta(f)), \quad f \in \F_q[u]
\end{equation}
as given in \cite[(2.5)]{CCG}.

\begin{cor}\label{formula-for-Mobius} 

Keep \cref{IntervalNotation} and \cref{DerivativeNotation}.
Let $\mathcal{I}$ be an interval in $\F_q[u]$,
and fix $r \in \mathcal{I}$.
Suppose that $Z_F \cap Z_{F_{[r]}}$ is finite,
and that the leading term $au^d$ of $F(u,g(u))$ is independent of $g(u) \in  \mathcal I_{ \overline{\mathbb F_q}}$. 
Then for any $s \in \mathbb F_q[u]$ with $\deg(s) <  \frac{\dim(\mathcal I)}{p}$ we have
\begin{equation} \label{TheMobiusFormulaEq}
\mu( F ( u, r+s^p) )= 
(-1) ^{d}   \chi_2(-1)^{\frac{d(d-1)}{2}}\chi_2(a)^{d} \left(\frac{W_{F,r}(u, r + s^p)}{M_{F,r}} \right).
\end{equation}

\end{cor}

\begin{proof} 

 


By Pellet's formula above, we have
\begin{equation} \label{PelletEq}
\mu(F(u,r+s^p)) =( -1)^{ d } \chi_2 ( \Delta ( F(u,r+s^p)) ),
\end{equation}
so applying \cref{FromDiscriminantToResultantEq}, we see that the above equals
\begin{equation*}
(-1)^d \chi_2(-1)^{\frac{d(d-1)}{2}}\chi_2(a) \chi_2 \left( R_{d,d-1} \left(F (u, r + s^p), \frac{d}{du} F(u, r+ s^p) \right)\right).
\end{equation*}
Using the Leibniz derivative product rule, the chain rule, and the fact that derivatives of $p$-th powers vanish, we arrive at
\begin{equation} 
(-1) ^{d} \chi_2(-1)^{\frac{d(d-1)}{2}}\chi_2(a) 
\chi_2 \left( R_{d, d-1} \left( F(u, r+s^p), F_{[r]}(u, r + s^p) \right) \right).
\end{equation} 

Applying \cref{CCGThreePointTwoConclusionEq}, the above becomes
\begin{equation} \label{MobConvertedToResEq}
(-1) ^{d}   \chi_2(-1)^{\frac{d(d-1)}{2}}\chi_2(a)^{d - d'} 
\chi_2 \left( R_{d, d'} \left( F(u, r+s^p), F_{[r]}(u, r + s^p) \right) \right)
\end{equation} 
with $d'$ defined in \cref{DerivativeNotation}, so using \cref{LoadedJacobiEq} we get
\begin{equation} 
(-1) ^{d}   \chi_2(-1)^{\frac{d(d-1)}{2}}\chi_2(a)^{d - d'} 
\left( \frac{W(u, r + s^p)}{M_{F,r}} \right).  
\end{equation} 
Since $d'$ is even we have $\chi_2(a)^{d-d'} = \chi_2(a)^d$, so we arrive at the right hand side of \cref{TheMobiusFormulaEq}.\end{proof}

The assumption from \cref{formula-for-Mobius} that the leading term of $F(u,g(u))$ is independent of $g(u) \in  \mathcal I_{ \overline{\mathbb F_q}}$ seems to be strictly stronger than the assumption from \cref{ccg-like} that $\deg(F(u,g(u)))$ is independent of $g(u) \in \mathcal{I}_{\overline{\F_q}}$.
However, as the next proposition shows, these are equivalent.

\begin{prop}\label{leading-term-degree} 

Keep \cref{IntervalNotation} and \cref{DerivativeNotation}.
Let $\mathcal I$ be an interval in $\mathbb F_q[u]$.
Assume that the degree of $F(u,g(u))$ is independent of $g(u) \in {\mathcal I}_{\overline{\mathbb F_q}}$ and nonnegative.
Then the leading term of $F(u,g(u))$ is also independent of $g(u) \in {\mathcal I}_{\overline{\mathbb F_q}}$. 

\end{prop}

\begin{proof} 

The coefficient of the highest power $u^d$ of $u$ in $F(u,g(u))$ is a polynomial function $P$ of the coordinates of $g(u) \in {\mathcal I}_{\overline{\mathbb F_q}}$. 
Since the degree $d$ of $F(u,g(u))$ is independent of $g(u) \in {\mathcal I}_{\overline{\mathbb F_q}}$,
this polynomial function $P$ vanishes nowhere, so by the Nullstellensatz, it is constant.
\end{proof}

\subsection{Tools for applying the M\"{o}bius formula}

Here we prove several claims that help verify the hypotheses of \cref{formula-for-Mobius},
deal with the cases when these fail,
make the application of \cref{formula-for-Mobius} more effective, and relate it to the trace function bounds we proved earlier.

\subsubsection{Infinite intersection}

We show that on special subsets for which the finite intersection condition in \cref{formula-for-Mobius} fails, 
the M\"{o}bius function vanishes almost everywhere.

\begin{prop} \label{ScarcityInfiniteIntersectionLem}

Keep \cref{DerivativeNotation}.
Let $r \in \mathbb F_q[u]$ for which $Z_F \cap Z_{F_{[r]}}$ is infinite.
Then 
\begin{equation}
\# \{s \in \F_q[u] : \mu( F ( u, r+s^p)) \neq 0 \} \leq k(q-1).
\end{equation}
\end{prop}

\begin{proof} 

Since the two zero loci have infinite intersection,
it follows from Bezout's Theorem that $F$ and $F_{[r]}$ share a common irreducible factor $P(u,T)$.
It follows that both $F(u, r+s^p)$ and
\begin{equation}
F_{[r]}(u, r + s^p) = \frac{d}{du} F(u,r+s^p)
\end{equation}
are divisible by $P(u, r + s^p)$.
We conclude that $\mu(F(u, r+s^p))=0$ once $P(u,r+s^p) \notin \mathbb F_q^\times$. 
Because $P(u,r+s^p)$ is a polynomial in $s^p$ of degree at most $k$, 
there are at most $(q-1) k$ choices of $s^p \in \F_q(u)$ for which 
\begin{equation}
P(u, r + s^p) \in \F_q^\times.
\end{equation} 
The proposition follows since in a field of characteristic $p$, the map $s \mapsto s^p$ is injective.
\end{proof}

\subsubsection{Partitioning an interval}

In order to prove (a generalized form of) \cref{NewRes} we need to control suns of the form $\sum_{ g \in \mathcal I } \mu(F(u,g))$.
For the M\"{o}bius formula from \cref{formula-for-Mobius} to apply, we need the leading term of $F(u,g)$ to be independent of $g$.
Since this is not always the case, we introduce the following lemma partitioning $\mathcal I$ into well-behaved subintervals.
This will allow us not to impose unnecessary monicity conditions and certain inequalities on degrees as in \cite[Theorem 4.5]{SS}. 

\begin{lem} \label{PartitionIntervalLem}

Keep \cref{IntervalNotation} and \cref{DerivativeNotation}.
For every interval $\mathcal I$ in $\mathbb F_q[u]$,
there exists a collection $\mathcal{P}$ of intervals in $\F_q[u]$ such that
\begin{enumerate}

\item every $\mathcal{J} \in \mathcal{P}$ is contained in $\mathcal{I}$.

\item for every $f \in \mathcal{I}$ there exists $\mathcal{J} \in \mathcal{P}$ with $f \in \mathcal{J}$;

\item for every two distinct intervals $\mathcal{J}, \mathcal{K} \in \mathcal{P}$ we have $\mathcal{J} \cap \mathcal{K} = \emptyset$;

\item for every $\mathcal{J} \in \mathcal{P}$, the leading term of $F(u, g(u))$ is independent of $g(u) \in \mathcal{J}_{\overline{\F_q}}$;

\item for each $0 \leq j \leq \dim (\mathcal I) - 1$, we have $\# \{\mathcal{J} \in \mathcal{P} : \dim(\mathcal{J}) = j \} \leq kq$;

\end{enumerate}

\end{lem}

\begin{proof} 

Fix an extension of the norm on $\mathbb F_q[u]$ to $\overline{\mathbb F_q(u)}$,
and denote by 
\begin{equation}
\omega(z) = -\log_q|z|
\end{equation}
the associated valuation of $z \in \overline{\mathbb F_q(u)}$.
Note that if $z \in \overline{\F_q}[u]$, then 
\begin{equation} \label{OmegaDegEq}
\omega(z) = -\deg(z).
\end{equation}

For $f \in \mathcal I$, if $f \notin \{\alpha_1, \dots, \alpha_k\}$, let 
\begin{equation}  \label{fsIntervalEq}
\mathcal{J}_f = \left\{z \in \mathcal I : |z - f| < \min_{1 \leq i \leq k} |f - \alpha_i| \right\},
\end{equation}
and for $f \in \mathcal I \cap \{\alpha_1, \dots, \alpha_k\}$ set
$
\mathcal{J}_f = \{f\}.
$
Put
\begin{equation}
\mathcal{P} = \left\{ \mathcal J_{f} : f \in \mathcal I\right \},
\end{equation}
and note that (1) and (2) above are satisfied.

To check (3), 
suppose that $\mathcal{J}_f \cap \mathcal{J}_g \neq \emptyset$ for some $f,g \in \mathcal{I}$. 
Since our intervals are nonarchimedean, this implies (without loss of generality) that $\mathcal{J}_g \subseteq \mathcal{J}_f$,
so in particular $g \in \mathcal{J}_f$.
If $g \in \{\alpha_1, \dots, \alpha_k\}$, we see from \cref{fsIntervalEq} that $g \in \mathcal{J}_f$ implies $f \in \{\alpha_1, \dots, \alpha_k\}$.
It follows that $\mathrm{len}(\mathcal{J}_g) = \mathrm{len}(\mathcal{J}_f) = 1$ and thus that $\mathcal{J}_g = \mathcal{J}_f$ as required.
If $g \notin \{\alpha_1, \dots, \alpha_k\}$, then this is also the case for $f$,
so from \cref{fsIntervalEq} we get that $|g-f| < |f - \alpha_i|$ for all $1 \leq i \leq k$.
Our norm is nonarchimedean, so $|g - \alpha_i| = |f - \alpha_i|$ by the above.
It follows from \cref{fsIntervalEq} that $\dim \mathcal{J}_g = \dim \mathcal{J}_f$ so $\mathcal{J}_g = \mathcal{J}_f$ as required. 

To check (4) for some $\mathcal{J} \in \mathcal{P}$, by \cref{leading-term-degree}, 
it suffices to check that $\deg(F(u, g(u)))$ is independent of $g(u) \in \mathcal J_{\overline{\mathbb F_q}}$.
Equivalently, by \cref{OmegaDegEq}, we need to check the independence of $\omega(F(u,g))$ on $g \in \mathcal{J}_{\overline{\F_q}}$.  
For that, pick an $f \in \mathcal I \setminus \{\alpha_1, \dots, \alpha_k\}$ with $\mathcal{J} = \mathcal{J}_f$. 
For $g \in \mathcal{J}_{\overline{\F_q}}$ we get as in the above paragraph that $\omega(g - \alpha_i) = \omega(f - \alpha_i)$ so \cref{FrootsEq} implies that
\begin{equation} 
\omega(F(u,g))= \omega(a_k) + \sum_{i=1}^k \omega (g-\alpha_i)  =  \omega(a_k) + \sum_{i=1}^k \omega(f-\alpha_i)
\end{equation}
is indeed independent of $g$.

At last we check (5).
For that, fix $0 \leq j \leq \dim(\mathcal I) - 1$, and let $f \in \mathcal I$ with $\dim(\mathcal{J}_f) = j$.
It follows from our definition of $\mathcal{J}_f$ that there exists some $1 \leq i \leq k$ such that $\omega(f - \alpha_i) \geq -j$.
Therefore, it suffices to check that for a given $i$ we have 
\begin{equation}
\# \{\mathcal{J}_g : g \in \mathcal{I}, \ \dim(\mathcal{J}_g) = j, \ \omega(g - \alpha_i) \geq -j \} \leq q.
\end{equation}

To establish the above inequality we show that $\mathcal{J}_g$ (as above) is determined by the coefficient of $u^j$ in $g$. Let $\mathcal J_g, \mathcal J_{g'}$ be two such intervals. 
We have $\omega(g - \alpha_i), \omega(g' - \alpha_i) \geq -j$,
so we get from \cref{OmegaDegEq} that 
\begin{equation}
\deg(g-g') = -\omega(g-g') \leq j
\end{equation}
since $\omega$ is nonarchimedean.
Hence, if the coefficient of $u^j$ in $g$ coincides with the coefficient of $u^j$ in $g'$, we get that $\deg(g-g') \leq j-1$ and thus $\mathcal{J}_g= \mathcal{J}_{g'}$ since $\dim(\mathcal{J}_g)  = \dim(\mathcal{J}_{g'})= j$.
\end{proof}

%

\subsubsection{Sheaf-theoretic setup}

We set up some of the notation needed to prove \cref{NewRes} and to state its `trace-twisted' variant.

\begin{notation}\label{trace-with-Mobius} 

Keep \cref{DerivativeNotation}, \cref{Kummer-sheaf}, and \cref{ShiftedPexpNotation}.
Let $r \in \F_q[u]$ be a polynomial for which $Z_F \cap Z_{F_{[r]}}$ is finite.
Let $g \in \mathbb F_q[u]$ be a squarefree polynomial, let 
\begin{equation}
t \colon \F_q[u]/(g) \to \mathbb{C} 
\end{equation}
be an infinitame trace function, let $\mathcal{F}_\pi$ be a sheaf giving rise to the trace function $t_\pi$, and set 
\begin{equation} \label{DefgFrEq}
g_{F,r} = \lcm ( g, M_{F,r} ).
\end{equation} 
Fix a prime factor $\pi$ of $g_{F,r}$, 
let $\kappa = \F_q[u]/(\pi)$, and let $\chi \colon \kappa^\times \to {\overline{\mathbb{Q}_\ell}}^\times$ be the unique quadratic character. 
In other words, the character $\chi$ is the Legendre symbol mod $\pi$, that is 
\begin{equation} \label{chiDefQuadEq}
\chi(f) = \left( \frac{f}{\pi} \right), \quad f \in \kappa^\times = (\F_q[u]/(\pi))^\times.
\end{equation}

We reduce $r$ mod $\pi$, and recall from \cref{ShiftedPexpNotation} the map
\begin{equation} \label{ShiftedPexpEq}
E_{r} \colon \mathbb A^1_{\kappa} \to \mathbb A^1_{\kappa}, \quad E_r(x) = r + x^p.
\end{equation}
Using this map, we define a sheaf on $\mathbb A^1_{\kappa}$ by
\begin{equation} \label{ThreeCasesSheafEq}
\mathcal{F}_{F,r,\pi} = E_r^* 
\begin{cases}
\mathcal L_{\chi} \left(W^{(\pi)}_{F,r} \right) & \pi \nmid g, \ \pi \mid M_{F,r}  \\
\mathcal F_{\pi} & \pi \mid g, \ \pi \nmid M_{F,r} \\
\mathcal L_{\chi} \left( W^{(\pi)}_{F,r} \right) \otimes \mathcal F_{\pi} & \pi \mid g, \ \pi \mid M_{F,r}
\end{cases}
\end{equation}
and use the shorthand notation $t_{F,r,\pi}$ for the associated trace function $t_{\mathcal{F}_{F,r,\pi}}$.
At last, define the trace function
\begin{equation} \label{TheBigTraceFunctionEq}
t_{F,r} = \prod_{\pi \mid g_{F,r}} t_{F,r,\pi}.
\end{equation}
By \cref{Kummer-sheaf-properties} and \cref{general-tensor-product}(6), this is an infinitame trace function.

\end{notation}

In the proof of \cref{NewRes} and its variants, we will be tasked with applying \cref{trace-interval-prop} to $g_{F,r}$ and $t_{F,r}$. 
In order to make \cref{trace-interval-prop} a useful bound,
we need to have some control on the rank and conductor.

%
%
%

\subsubsection{Bounding rank and conductor}

%
%
%

In order to control $|M_{F,r}|$ and $|g_{F,r}|$, we recall from \cref{DefMFrEq} that $|M_{F,r}|$ is bounded by $|R(F, F_{[r]})|$, so it suffices to control the latter.
The following is the variant of \cite[(4.24)]{SS} needed here.

\begin{prop} \label{BoundDegResProp}

Keep \cref{DerivativeNotation}. 
For $r \in \F_q[u]$ we have 
\begin{equation}
\deg(R(F, F_{[r]})) \leq E(c_1, c_2, \deg(r)).
\end{equation}

\end{prop}

\begin{proof}

The quasi-homogeneity of the resultant from \cite[p. 399, (1.6)]{GKZ}, and Sylvester's formula as given in \cite[p. 400, (1.12)]{GKZ}, imply that
$R(F,F_{[r]})$ is a linear combination (over $\F_q$) of subproducts of 
\begin{equation}
a_{i_1}(u) \dots a_{i_k}(u) b_{j_1}(u) \dots b_{j_k}(u), \quad i_1 + \dots + i_k + j_1 + \dots + j_k = k^2. 
\end{equation}



By \cref{beq}, we have the bound
\begin{equation} \label{bBound}
\begin{split}
\deg(b_i) &\leq\max\left\{ \deg(a_i)-1 , \deg(r) + \deg(a_{i+1}) - 1 \right\} \\
&\leq \max\{  c_1 -1 + c_2i ,\deg(r) - 1 + c_1 + c_2(i+1)\} \\
&= \max\{0, c_2 + \deg(r)\} + c_1 - 1 + c_2i 
\end{split}
\end{equation}
for the degrees of the coefficients of $F_{[r]}$.
As a result we get that 
\begin{equation}
\begin{split}
\deg(R(F, F_{[r]})) &\leq \max_{i_1 + \dots + i_k + j_1 + \dots + j_k = k^2} \deg(a_{i_1} \dots a_{i_k} b_{j_1} \dots b_{j_k}) \\
&= \max_{i_1 + \dots + i_k + j_1 + \dots + j_k = k^2} \sum_{\ell = 1}^k \deg(a_{i_\ell}) + \deg(b_{j_\ell}).
\end{split}
\end{equation}

Using \cref{aBound} and \cref{bBound} we see that the above is at most 
\begin{equation}
2kc_1 + k\max\{0, c_2 + \deg(r)\} - k + \max_{i_1 + \dots + i_k + j_1 + \dots + j_k = k^2}  \sum_{\ell = 1}^k c_2(i_\ell + j_\ell) 
\end{equation}
which evaluates to
\begin{equation}
2kc_1 + k\max\{0, c_2 + \deg(r)\} - k + c_2k^2 .
\end{equation}
By the notation in \cref{AuxFuncEq}, the above equals $E(c_1, c_2, \deg(r))$.
\end{proof}

\begin{prop} \label{gammaLemma}

Keep \cref{trace-with-Mobius}. 
For any positive $\gamma \in \mathbb R$ we have 
\begin{equation*}
\begin{split} 
\prod_{\pi \mid g_{F,r}} (  r ( t_{F,r,\pi} ) (1+\gamma) + c (t_{F,r,\pi} ) \gamma )^{\deg (\pi)} \leq 
(1+ 2 \gamma)^{ E(c_1,c_2, \deg(r)) }  
\prod_{ \pi \mid g} (  r( t_{\pi} ) (1+\gamma) + c ( t_{\pi} ) \gamma )^{\deg (\pi)}.
\end{split}
\end{equation*}

\end{prop}

\begin{proof} 

Let $\pi$ be a prime dividing $g_{F,r}$.
In case $\pi$ divides $g$ and does not divide $M_{F,r}$, from the definition of $\mathcal{F}_{F,r,\pi}$ in \cref{ThreeCasesSheafEq}(2), and the invariance of rank and conductor in \cref{ShiftedFrobPullbackLem}(4), we get
\begin{equation} \label{gammaBoundCaseTwo}
\begin{split}
\rank ( \mathcal F_{F,r,\pi} ) (1+\gamma) + c ( \mathcal F_{F,r,\pi} ) \gamma &= 
\rank ( E_r^* \mathcal F_{\pi} ) (1+\gamma) + c ( E_r^* \mathcal F_{\pi} ) \gamma \\
&= \rank ( \mathcal F_{\pi} ) (1+\gamma) + c ( \mathcal F_{\pi} ) \gamma.
\end{split}
\end{equation}

In case $\pi$ divides both $g$ and $M_{F,r}$, from \cref{ThreeCasesSheafEq}(3), \cref{ShiftedFrobPullbackLem}(4), and \cref{general-tensor-product}(5) we get
\begin{equation} \label{Case3rboundEq}
\begin{split}
\rank (\mathcal F_{F,r,\pi}) &= \rank \left( E_r^* \left(\mathcal L_{\chi} \left(W^{(\pi )}_{F,r}\right) \otimes \mathcal F_{\pi}\right)\right) = 
\rank \left( \mathcal L_{\chi} \left(W^{(\pi )}_{F,r} \right) \otimes \mathcal F_{\pi} \right) = \rank ( \mathcal F_{\pi } ).
\end{split}
\end{equation} 
Similarly, by \cref{ThreeCasesSheafEq}(3), \cref{ShiftedFrobPullbackLem}(4), \cref{general-tensor-product}(6),
and \cref{Kummer-sheaf-properties}(6), we have
\begin{equation}  \label{Case3cboundEq}
\begin{split}
c (\mathcal F_{F,r,\pi})  &= c \left( E_r^* \left(\mathcal L_{\chi} \left(W^{(\pi )}_{F,r}\right) \otimes \mathcal F_{\pi} \right)\right) =
c \left( \mathcal L_{\chi} \left(W^{(\pi )}_{F,r} \right) \otimes \mathcal F_{\pi} \right) \\
&\leq  c (\mathcal F_{\pi} )+ c \left( \mathcal L_{\chi} \left(W_{F,r}^{(\pi)}\right)\right) \rank ( \mathcal F_{\pi} ) 
\leq c (\mathcal F_{\pi} )+  \deg \left(W_{F,r}^{(\pi)}\right) \rank ( \mathcal F_{\pi} ). 
\end{split}
\end{equation} 

Let $a \in \overline{\F_q}$ be a root of $\pi$.
From the definition of $W_{F,r}^{(\pi)}$ after \cref{DefWFrEq}, 
the information on multiplicities in \cref{OrderVanishWFratbEq}, and \cref{modified-Zeuthen} we obtain
\begin{equation} \label{Case3degboundEq}
\begin{split}
\deg \left(W_{F,r}^{(\pi)}\right) &= \deg(W_{F,r}(a,T)) = \sum_{b \in \overline{\F_q}} \operatorname{ord}_{T=b} W_{F,r}(a,T) \\
&= \sum_{\substack{x \in Z_F \cap Z_{F_{[r]}} \\ u_x = a}} i_x(Z_F, Z_{F_{[r]}}) 
\leq \operatorname{ord}_{u=a} R(F,F_{[r]}) = v_\pi ( R( F, F_{[r]}))
\end{split}
\end{equation}
where $v_\pi$ is the $\pi$-adic valuation on $\F_q[u]$.
From \cref{Case3cboundEq} and \cref{Case3degboundEq} we conclude that 
\begin{equation} \label{Case3cFboundEq}
c (\mathcal F_{F,r,\pi}) \leq c ( \mathcal F_{\pi} ) + v_\pi ( R( F, F_{[r]})) \rank ( \mathcal F_{\pi }).
\end{equation}

Combining \cref{Case3rboundEq} with \cref{Case3cFboundEq}, 
and using Bernoulli's inequality, we get
\begin{equation} \label{gammaBoundCaseThree}
\begin{split} 
\rank ( \mathcal F_{F,r,\pi} ) (1+\gamma) + c ( \mathcal F_{F,r,\pi} ) \gamma &\leq 
\rank(\mathcal F_{\pi} )(1+\gamma) + c ( \mathcal F_{\pi}) \gamma +  v_\pi ( R( F, F_{[r]})) \rank (\mathcal F_{\pi})  \gamma \\
&\leq (1 +v_\pi ( R( F, F_{[r]}))  \gamma)   ( \rank(\mathcal F_{\pi} )(1+\gamma) + c ( \mathcal F_{\pi}) \gamma ) \\
&\leq 
(1 +\gamma) ^{ v_\pi ( R( F, F_{[r]}))  }   ( \rank(\mathcal F_{\pi} )(1+\gamma) + c ( \mathcal F_{\pi}) \gamma ).
\end{split}
\end{equation}
 

In case $\pi$ divides $M_{F,r}$ and does not divide $g$, by \cref{ThreeCasesSheafEq}(1), \cref{ShiftedFrobPullbackLem}(4), 
and \cref{Kummer-sheaf-properties}(6) we have 
\begin{equation} \label{rankCase1Eq}
\rank(\mathcal{F}_{F,r,\pi}) = \rank(E_r^* \mathcal L_{\chi}(W_{F,r}^{(\pi)}) ) =  \rank( \mathcal L_{\chi}(W_{F,r}^{(\pi)}) ) = 1.
\end{equation}
Similarly, from \cref{ThreeCasesSheafEq}(1), \cref{ShiftedFrobPullbackLem}(4), \cref{Kummer-sheaf-properties}(6), 
and \cref{Case3degboundEq} we get
\begin{equation} \label{cFCase1Eq}
\begin{split}
c(\mathcal{F}_{F,r,\pi}) &= c \left( E_r^* \mathcal L_{\chi} \left(W_{F,r}^{(\pi)}\right) \right) \\
&=  c \left( \mathcal L_{\chi} \left(W_{F,r}^{(\pi)} \right) \right) \leq
\deg \left( W_{F,r}^{(\pi)} \right) \leq v_\pi(R(F,F_{[r]})).
\end{split}
\end{equation}

Since $\pi$ divides $M_{F,r}$, it follows from the definition of the latter in \cref{DefMFrEq} that $\pi$ divides $R(F, F_{[r]})$, 
or equivalently $v_\pi(R(F, F_{[r]})) \geq 1$.
Therefore, from \cref{rankCase1Eq}, \cref{cFCase1Eq}, and Bernoulli's inequality we have 
\begin{equation} \label{gammaBoundCaseOne}
\begin{split}
\rank ( \mathcal F_{F,r,\pi} ) (1+\gamma) + c ( \mathcal F_{F,r,\pi} ) \gamma  &\leq 1 + \gamma + v_\pi ( R( F, F_{[r]})) \gamma \\
&\leq 1 +2 v_\pi ( R( F, F_{[r]})) \gamma \leq (1+2\gamma)^ {v_\pi ( R( F, F_{[r]})) }.
\end{split}
\end{equation}

At last, combining \cref{SheafFuncDictionaryDef}, \cref{gammaBoundCaseOne}, \cref{gammaBoundCaseTwo}, \cref{gammaBoundCaseThree}, and \cref{BoundDegResProp} we get
\begin{equation*}
\begin{split}
&\prod_{\pi \mid g_{F,r}} (  r ( t_{F,r,\pi} ) (1+\gamma) + c (t_{F,r,\pi} ) \gamma )^{\deg (\pi)} = \\
&\prod_{\pi \mid g_{F,r}} (  \rank ( \mathcal F_{F,r,\pi} ) (1+\gamma) + c ( \mathcal F_{F,r,\pi} ) \gamma )^{\deg (\pi)} = \\
&\prod_{\substack{\pi \mid M_{F,r} \\ \pi \nmid g}} (  \rank ( \mathcal F_{F,r,\pi} ) (1+\gamma) + c ( \mathcal F_{F,r,\pi} ) \gamma )^{\deg (\pi)}
\prod_{\pi \mid g} (  \rank ( \mathcal F_{F,r,\pi} ) (1+\gamma) + c ( \mathcal F_{F,r,\pi} ) \gamma )^{\deg (\pi)} \leq \\ 
&\prod_{\substack{\pi \mid M_{F,r} \\ \pi \nmid g}} (1+2\gamma)^{\deg(\pi) v_\pi ( R(F, F_{[r]} ))}
\prod_{\pi \mid g} (1+2\gamma)^{ \deg(\pi) v_\pi ( R(F, F_{[r]} ))} (  \rank ( \mathcal F_{\pi} ) (1+\gamma) + c ( \mathcal F_{\pi} ) \gamma )^{\deg(\pi)} = \\
&\prod_{\pi \mid M_{F,r}} (1+2\gamma)^{ \deg(\pi) v_\pi ( R(F, F_{[r]} ))}
\prod_{\pi \mid g} (  \rank ( \mathcal F_{\pi} ) (1+\gamma) + c ( \mathcal F_{\pi} ) \gamma )^{\deg(\pi)} = \\
&(1+2\gamma)^{ \sum_{\pi \mid M_{F,r}} \deg(\pi) v_\pi ( R(F, F_{[r]} ))}
\prod_{\pi \mid g} (  \rank ( \mathcal F_{\pi} ) (1+\gamma) + c ( \mathcal F_{\pi} ) \gamma )^{\deg(\pi)} = \\
&(1+2\gamma)^{ \deg(R(F, F_{[r]}))}
\prod_{\pi \mid g} (  \rank ( \mathcal F_{\pi} ) (1+\gamma) + c ( \mathcal F_{\pi} ) \gamma )^{\deg(\pi)} \leq \\
&(1+2\gamma)^{ E(c_1, c_2, \deg(r))}
\prod_{\pi \mid g} (  \rank ( \mathcal F_{\pi} ) (1+\gamma) + c ( \mathcal F_{\pi} ) \gamma )^{\deg(\pi)} = \\
&(1+ 2 \gamma)^{ E(c_1,c_2, \deg(r)) }  
\prod_{ \pi \mid g} (  r( t_{\pi} ) (1+\gamma) + c ( t_{\pi} ) \gamma )^{\deg (\pi)}.
\end{split}
\end{equation*}
\end{proof}


\subsubsection{Finding a good prime}

Our goal here is to give a sufficient condition for the existence of a prime $\tau$ as in \cref{trace-interval-prop} for the trace function $t_{F,r}$ from \cref{trace-with-Mobius}.

\begin{prop}\label{unique-root} 

Keep \cref{DerivativeNotation}.
Suppose that $a \in \overline{\mathbb F_q}$ is not a root of the polynomial
\begin{equation}
\Delta_v (R (F,F_v))
\end{equation}
introduced in \cref{DiscVDefEq}.
Then for any $r \in \mathbb F_q[u]$ for which $Z_F \cap Z_{F_{[r]}}$ is finite, 
there exists at most  one $b \in \overline{\mathbb F_q}$ such that $(a,b) \in Z_{F} \cap Z_{F_{[r]}}$. 

\end{prop}

\begin{proof}

Suppose toward a contradiction that there exist distinct $b_1, b_2 \in \overline{\F_q}$ with 
$
(a,b_1), (a,b_2) \in Z_F \cap Z_{F_{[r]}}
$
and let $\pi \in \F_q[u]$ be the minimal polynomial of $a$ over $\F_q$.
We will arrive at a contradiction to $a$ not being a root of $\Delta_v(R(F,F_v))$ by showing that $\pi$ divides $\Delta_v(R(F,F_v))$ in the ring $\F_q[u]$.

Denote by $v_0$ the residue class of $\frac{dr}{du}$ in $\F_q[u]/(\pi)$. 
To check that 
\begin{equation} \label{PiDivEq}
\Delta_v(R(F, F_v)) \equiv 0 \mod \pi,
\end{equation} 
it suffices to prove that 
\begin{equation} \label{OrdVan2Eq}
\mathrm{ord}_{v = v_0} \left( R(F, F_v) \ \mathrm{mod} \ \pi \right) \geq 2.
\end{equation}
By \cref{RFFvFormula}, we have
\begin{equation} \label{ResultantProdEq}
R ( F, F_v) = 
a_k^{ \deg_T F_v} \prod_{i=1}^k \left( \frac{\partial F}{\partial u} ( u, \alpha_i) + v \frac{\partial F}{\partial T} (u, \alpha_i) \right).
\end{equation}

Suppose first that $\deg_T(F_v) = 0$.
Then 
\begin{equation}
\deg_T(F_{[r]} \ \mathrm{mod} \ \pi) \leq \deg_T(F_{[r]}) \leq \deg_T(F_v) = 0, 
\end{equation}
and since $\pi(a)=0$,
the polynomial $F_{[r]} \ \mathrm{mod} \ \pi$ has a zero, so it is the zero polynomial.
It follows from finiteness of $Z_F \cap Z_{F_{[r]}}$ that $F \ \mathrm{mod} \ \pi$ is not zero,
and since $\pi(a)=0$, we conclude that $F \ \mathrm{mod} \ \pi$ has at least two zeros so
\begin{equation}
k = \deg_T(F) \geq \deg_T(F \ \mathrm{mod} \ \pi) \geq 2.
\end{equation} 

Our assumption that $F_v$ is constant as a polynomial in $T$, the separability of $F$ which implies that $\deg_v(F_v) = 1$, and the fact that $k \geq 2$ established above, imply that
\begin{equation}
\Delta_v(R(F,F_v)) = \Delta_v(F_v^k) = 0
\end{equation}
so \cref{PiDivEq} holds in this case.

Suppose now that $\deg_T(F_v) \geq 1$.
We see from \cref{ResultantProdEq} that if $\pi \mid a_k$ then \cref{PiDivEq} is satisfied, so we assume from now on that $\pi \nmid a_k$.
Since the $\alpha_i$ are roots of a polynomial with leading coefficient not divisible by $\pi$,
we can reduce \cref{ResultantProdEq} mod a prime of $\overline{\F_q(u)}$ lying over $\pi$.
Since $\pi(a)=0$, 
it follows that after the reduction, at least two of the factors on the right hand side of \cref{ResultantProdEq} vanish at $v = v_0$.
Hence, \cref{OrdVan2Eq} holds.
\end{proof} 

The following is the generalization of \cite[Remark 3.3]{SS} needed here. 

\begin{prop}\label{discriminant-definer}  

Keep \cref{trace-with-Mobius}.
Suppose that $R (F, F_{[r]})$ is not of the form $A^2 B$ for any $A \in \F_q[u]$ and any $B \in \F_q[u]$ that divides the polynomial
\begin{equation} \label{TripProdPolEq}
g \cdot a_k \cdot \Delta_v ( R(F, F_v)) \in \F_q[u]. 
\end{equation}
Then there exists a prime $\pi \in \F_q[u]$ dividing $g_{F,r}$ such that $t_{F,r,\pi}$ is a Dirichlet trace function.


\end{prop}

\begin{proof} 

Our assumption on $R(F, F_{[r]})$ is equivalent to the existence of a prime $\pi$ not dividing the polynomial in \cref{TripProdPolEq} such that $v_\pi(R(F,F_{[r]}))$ is odd.
Since the latter valuation is nonzero, our prime $\pi$ divides $R(F, F_{[r]})$, so by the definition in \cref{DefMFrEq}, $\pi$ divides $M_{F,r}$. We conclude from the definition of $g_{F,r}$ in \cref{DefgFrEq}, and from \cref{TripProdPolEq}, 
that $\pi$ divides $g_{F,r}$ and does not divide $g$.
By \cref{ThreeCasesSheafEq}(1) we have
\begin{equation}
\mathcal{F}_{F,r,\pi} = E_r^* \mathcal{L}_\chi \Big( W_{F,r}^{(\pi)} \Big).
\end{equation}

In order to show that the associated trace function $t_{F,r,\pi}$ is a Dirichlet trace function, 
by the permanence property in \cref{ShiftedFrobPullbackLem}(6), it suffices to show that the function
\begin{equation}
t_{\mathcal{L}_\chi \left(W_{F,r}^{(\pi)} \right)} (x) = \chi \left(W_{F,r}^{(\pi)} (x) \right), \quad x \in \F_q[u]/(\pi),
\end{equation}
is a Dirichlet trace function.
Recall from \cref{chiDefQuadEq} that the character $\chi$ is quadratic, 
so by the definition in \cref{DefDirTraceFuncEq} it is enough to show that $W_{F,r}^{(\pi)}$ is an odd power of a monic linear polynomial, up to a constant from $(\F_q[u]/(\pi))^\times$.
In other words, we want to show that $W_{F,r}^{(\pi)}$ vanishes at no more than one point in $\overline{\F_q[u]/(\pi)}$, 
and its order of vanishing there is odd.


Fix a root $a \in\overline{\mathbb F_q}$ of $\pi$.
We are tasked with showing that $W_{F,r}(a,T)$ has a unique zero in $\overline{\F_q}$, 
and the multiplicity of this zero is odd.
Since $\pi$ does not divide the polynomial in \cref{TripProdPolEq}, it does not divide $\Delta_v(R(F,F_v))$, so $\Delta_v(R(F,F_v))(a) \neq 0$.
The desired uniqueness of the zero of $W_{F,r}(a,T)$ follows from \cref{OrderVanishWFratbEq} and \cref{unique-root}.
From \cref{OrderVanishWFratbEq} we moreover conclude that the order of vanishing of $W_{F,r}(a,T)$ at its unique vanishing point is
\begin{equation}
\sum_{b \in \overline{\F_q}} i_{(a,b)}(Z_F, Z_{F_{[r]}}).
\end{equation}

Since $\pi$ does not divide the polynomial in \cref{TripProdPolEq}, it does not divide the leading coefficient of $F$. 
In other words, the leading coefficient of $F$ does not vanish at $a$, 
so by \cref{modified-Zeuthen} the sum above equals
\begin{equation}
\operatorname{ord}_{u=a} R(F, F_{[r]}).
\end{equation}
This order of vanishing equals $v_\pi (R(F, F_{[r]}))$ which is odd by assumption. 
\end{proof}



\cref{discriminant-definer} is not useful in case $\Delta_v(R(F, F_v)) = 0$.
The next proposition characterizes the cases in which this vanishing occurs.
The arising condition is our generalization of the `distinct derivatives' assumption in \cite[Proposition 4.3]{SS}.

\begin{prop} \label{DistinctValuesNonvanishingRes}

Keep \cref{DerivativeNotation}.
In $\overline{\F_q(u)}$ we have
\begin{equation} \label{kDistinctEq}
\frac{ \frac{\partial F}{\partial u}}{ \frac{\partial F}{\partial T}}(u, \alpha_i) \neq
\frac{ \frac{\partial F}{\partial u}}{ \frac{\partial F}{\partial T}}(u, \alpha_j), \quad 1 \leq i < j \leq k,
\end{equation}
if and only if $\Delta_v ( R(F, F_v) )$ is not the zero polynomial.

\end{prop}

%
%
%
%
%
%
%
%

\begin{proof} 

Our discriminant is nonzero if and only if $R(F,F_v)$
does not have a double root in $\overline{\F_q(u)}$ as a polynomial in $v$.
By \cref{RFFvFormula}, we have
\begin{equation}
R ( F, F_v) = a_k^{ \deg_T F_v} \prod_{i=1}^k \left( \frac{\partial F}{\partial u} ( u, \alpha_i) + v \frac{\partial F}{\partial T} (u, \alpha_i) \right)
\end{equation}
so the nonexistence of a double root among the $k$ roots 
\begin{equation}
v_i = -\frac{ \frac{\partial F}{\partial u} }{  \frac{\partial F}{\partial T} } (u, \alpha_i), \quad 1 \leq i \leq k,
\end{equation}
of $R(F, F_v)$ in $\overline{\F_q(u)}$ is equivalent to our assumption in \cref{kDistinctEq}.
\end{proof}




Using \cref{DistinctValuesNonvanishingRes}, we show in the next proposition that we can always arrive at a situation where $\Delta_v(R(F,F_v)) \neq 0$ by performing a linear change of variable.
This is our generalization of the main argument in the proof of \cite[Theorem 4.5]{SS}.

\begin{prop} \label{ChangeOfVariable}

Keep \cref{DerivativeNotation}.
There exists a (monic) polynomial $P(u) \in \mathbb F_q[u]$ with $|P(u)| \leq  q{k \choose 2} $ such that for all $c(u) \in \F_q[u]$,
the polynomial
\begin{equation}
G(u,T) = F(u, P(u)T + c(u))
\end{equation}
is separable in $T$ and satisfies 
\begin{equation}
\Delta_v ( R(G, G_v) ) \neq 0.
\end{equation}

\end{prop} 

\begin{proof} 

In view of \cref{FrootsEq}, we have 
\begin{equation}
G \left(u, \frac{\alpha_i - c(u)}{P(u)} \right) = 0, \quad 1 \leq i \leq k,
\end{equation} 
so these are all the roots of  $G$ in $\overline{\F_q(u)}$ since $\deg_T(G) = \deg_T(F) = k$.
By \cref{DistinctValuesNonvanishingRes}, it suffices to choose $P(u)$ in such a way that 
\begin{equation} \label{DistinctValsEq}
\frac{ \frac{\partial G}{\partial u}}{ \frac{\partial G}{\partial T}}\left(u, \frac{\alpha_i-c(u)}{P(u)}\right) \neq
\frac{ \frac{\partial G}{\partial u}}{ \frac{\partial G}{\partial T}}\left(u, \frac{\alpha_j - c(u)}{P(u)}\right), \quad 1 \leq i < j \leq k.
\end{equation}
Using the chain rule, we get
\begin{equation*}
\begin{split}
\frac{ \frac{\partial G }{\partial u}}{ \frac{\partial G}{\partial T}} \left( u, \frac{\alpha_i-c(u)}{P(u)} \right) &=
\frac{  \frac{\partial F}{ \partial u} (u,\alpha_i) + \frac{\alpha_i-c(u)}{P(u)} \frac{d P}{du} \frac{\partial F}{ \partial T} (u, \alpha_i) +  
\frac{dc}{du} \frac{\partial F}{ \partial T} (u, \alpha_i)      }{  P(u)\frac{\partial F }{\partial T}(u, \alpha_i)  } \\
&= \frac{  \frac{\partial F}{ \partial u}(u,\alpha_i) }{  P(u)\frac{\partial F}{\partial T}(u,\alpha_i)  }  + \frac{(\alpha_i -c(u))\frac{dP}{du}}{P(u)^2}  + \frac{   \frac{d c}{du} }{  P(u)}. 
\end{split}
\end{equation*}
Hence, \cref{DistinctValsEq} holds unless for some $1 \leq i < j \leq k$ we have
\begin{equation}
\frac{  \frac{\partial F}{ \partial u} }{  \frac{\partial F}{\partial T}  }(u,\alpha_i)  + \frac{(\alpha_i -c(u))\frac{dP}{du}}{P(u)} =
\frac{  \frac{\partial F}{ \partial u} }{  \frac{\partial F}{\partial T}  }(u,\alpha_j)  + \frac{(\alpha_j -c(u))\frac{dP}{du}}{P(u)}. 
\end{equation}
Since $\alpha_i - \alpha_j \neq 0$ by separability, the above is equivalent to
\begin{equation}
\frac{ \frac{ dP }{ du} } {P(u)} = \frac{  \frac{ \frac{\partial F}{ \partial u}  }{  \frac{\partial F}{\partial T}}(u, \alpha_j)  -    
\frac{ \frac{\partial F}{ \partial u}}{  \frac{\partial F }{\partial T}}(u, \alpha_i)  }{ \alpha_i - \alpha_j} 
\end{equation}
so \cref{DistinctValsEq} holds if (and only if)  
\begin{equation}
\frac{ \frac{ dP }{ du} } {P(u)}
\end{equation}
does not belong to a specific set of at most ${k \choose 2}$ elements of $\overline{\F_q(u)}$.

The `logarithmic derivation' map
\begin{equation}
P \mapsto \frac{ \frac{ dP }{ du} } {P}, \quad P \in \F_q[u],
\end{equation}
sends monic polynomials $P,Q$ to the same rational function if and only if their quotient $\frac{P}{Q}$ is a $p$-th power in $\F_q(u)$.
In particular, the restriction of the logarithmic derivation map to monic squarefree polynomials is injective.
By \cite[Proposition 2.3]{Ros}, the number of such polynomials of degree at most $d$ exceeds $q^{d}$,
so we need that $q^d \geq {k \choose 2}$.
We thus take
\begin{equation}
d = \left\lceil \log_q {k \choose 2} \right\rceil \leq  \log_q {k \choose 2} + 1
\end{equation}
so we can choose $P$ satisfying \cref{DistinctValsEq} with $|P| \leq q^d \leq q {k \choose 2}$.
\end{proof}

In order to bound the number of possible $B$ in \cref{discriminant-definer}, 
we bound the degree of the polynomial $\Delta_v(R(F,F_v))$.

\begin{prop} \label{BoundDegDelta}

Keep \cref{DerivativeNotation}. Then 
\begin{equation} \label{DegBoundGoal}
\deg(\Delta_v(R(F,F_v))) \leq 4 k (k-1)  (c_1 + k \max\{c_2,0\} ).
\end{equation}

\end{prop}

\begin{proof}  

Since $\deg_T(F), \deg_T(F_v) \leq k$, 
it follows from Sylvester's formula as given in \cite[p. 400, (1.12)]{GKZ} that $R(F, F_v)$ is a linear combination (over $\F_q$) of products of at most $k$ coefficients of $F$ and at most $k$ coefficients of $F_v$.
By \cref{aBound}, the degree of every coefficient $a_i$ of $F$ is at most
\begin{equation}
c_1 + c_2 i \leq c_1 + k \max\{c_2,0\}
\end{equation}
thus the degree in $u$ of (every coefficient of) $F_v$ is also bounded by the right hand side above.
We conclude that
\begin{equation} \label{uDegEq}
\deg_u(R(F, F_v)) \leq 2k (c_1 + k \max\{c_2,0\} ). 
\end{equation}

Since $F$ is separable in $T$, it follows from \cref{RFFvFormula} that
\begin{equation} \label{vDegEq}
\deg_v(R(F, F_v))) = k.
\end{equation}
We then infer from \cite[p. 404]{GKZ} that $\Delta_v(R(F,F_v))$ is a linear combination (over $\F_q$) of products of $2(k-1)$ coefficients of $R(F, F_v)$.
Using the bound on the degree of a coefficient from \cref{uDegEq}, we get that
\begin{equation}
\deg(\Delta_v(R(F,F_v))) \leq 2(k-1) \cdot 2k (c_1 + k \max\{c_2,0\} )
\end{equation}
and the right hand side above matches the right hand side of \cref{DegBoundGoal}.
 \end{proof}

Now that we have control over the number of possible $B$, 
we need to know how often $R(F, F_{[r]}) = A^2 B$ for a particular $B$.
For that, we have the following lemma which is a consequence of Cohen's quantitative Hilbert's irreducibility theorem as stated in \cite[Theorem 2.3]{Coh}.
We refer to \cite{BSE} of Bary-Soroker--Entin extending Cohen's work to function fields.

\begin{lem} \label{CohenLem}

Let $H(u,v) \in \F_q[u][v]$ be a polynomial which is not a perfect square in $\overline{\F_q(u)}[v]$, 
and let $B \in \mathbb F_q[u]$.
Then for $X \geq \max\{\deg_u(H),\deg(B)\}^4$ we have
\begin{equation}
\#\{g \in \F_q[u] : |g| < X, \ H(u,g(u))= B \cdot \square\} \ll \sqrt{X} \log X
\end{equation}
as $X \to \infty$, with the implied constant depending only on $\deg_v(H)$.

\end{lem}

The lemma above is the generalization of \cite[Proposition 4.2]{SS} needed here.
More specifically, we need the following corollary.

\begin{cor} \label{SquareScarcity}

Keep \cref{IntervalNotation}, \cref{DerivativeNotation}, and suppose that 
\begin{equation}
\Delta_v(R(F, F_v)) \neq 0.
\end{equation}
Let $B \in \mathbb F_q[u]$, and let $\mathcal I$ be an interval in $\F_q[u]$ with 
\begin{equation} \label{TheLengthOfIassumptionEq}
\mathrm{len}(\mathcal{I}) \geq  \max\{E(c_1, c_2, \deg(\mathcal{I})), \deg(a_k^kB) \}^4.
\end{equation}
Take $\mathcal{R} \subseteq \mathcal{I}$ such that for every $f \in \mathcal I$ there exists a unique $r \in \mathcal{R}$ with 
\begin{equation}
\frac{df}{du} = \frac{dr}{du}.
\end{equation}
Then as $\operatorname{len}(\mathcal{I}) \to \infty$ we have
\begin{equation}
\# \{r \in \mathcal{R} : R(F, F_{[r]}) = B \cdot \square \} \ll \sqrt{\mathrm{len}(\mathcal{I})} \log \mathrm{len}(\mathcal{I})
\end{equation}
with the implied constant depending only on $k$.

\end{cor}

\begin{proof}

Fix $f \in \mathcal I$, and note that
\begin{equation*}
\begin{split}
\#\{r \in \mathcal{R} : R(F, F_{[r]}) = B \cdot \square \} &= 
\#\{g \in \F_q[u] : |g| < \mathrm{len}(\mathcal I), \ R(F, F_{\frac{d(f+g)}{du}}) = B \cdot \square \} \\
&\leq \#\{g \in \F_q[u] : |g| < \mathrm{len}(\mathcal I), \ R(F, F_{\frac{df}{du} + g}) = B \cdot \square \}. \\
\end{split}
\end{equation*}
By \cref{CCGThreePointTwoEq}, the above is at most
\begin{equation*}
\# \bigcup_{i = 0}^j \{g \in \F_q[u]: |g| < \mathrm{len}(\mathcal I), \ R_{k,j}(F, F_{\frac{df}{du} + g}) = a_k^{i}B \cdot \square \}, \ 
j = \deg_T \left(F_{\frac{df}{du} + v} \right)
\end{equation*}
so setting 
\begin{equation}
H(u,v) = R_{k,j}(F, F_{\frac{df}{du} + v}),
\end{equation} 
and noting that $j \leq k$, we get the bound
\begin{equation}
\sum_{i=0}^k \#\{g \in \F_q[u] : |g| < \mathrm{len}(\mathcal I), \ H(u,g) = a_k^i B \cdot \square \}.
\end{equation}

Therefore, in order to conclude by applying \cref{CohenLem}, one thing we need to check is that
$\mathrm{len}(\mathcal I) \geq \max\{\deg_u(H),\deg(a_k^kB)\}^4$.
By our assumption in \cref{TheLengthOfIassumptionEq}, this amounts to showing that 
\begin{equation} \label{AmountsToForDegHEq}
\deg_u(H) \leq E(c_1, c_2, \deg(\mathcal{I})).
\end{equation}
We claim that there exists $\lambda \in \overline{\F_q}$ for which
\begin{equation} \label{ClaimOnDegHEq}
\deg_u(H) = \deg R(F, F_{[f + \lambda u]}).
\end{equation}

Since $F$ is separable, the coefficient of the highest power of $u$ in $H(u,v)$ is a nonzero polynomial $P \in \F_q[v]$,
and the coefficient of the highest power of $T$ in $F_{\frac{df}{du} + v}$ is a nonzero polynomial $Q \in \F_q[u,v]$.
Hence, there exists $\lambda \in \overline{\F_q}$ such that $P(\lambda) \neq 0$ and $Q(u, \lambda) \neq 0$.
It follows that 
\begin{equation*}
\deg_u(H(u,v)) = \deg(H(u, \lambda)) = \deg(R_{k,j}(F,F_{\frac{df}{du} + \lambda})) =  \deg R(F, F_{[f + \lambda u]})
\end{equation*} 
so our claim from \cref{ClaimOnDegHEq} is established.
From \cref{BoundDegResProp}, \cref{IntervalNotation}, and the fact that $\mathrm{len}(\mathcal{I}) > 1$, we get that 
\begin{equation}
\deg(R(F, F_{[f + \lambda u]})) \leq E(c_1, c_2, \deg(f + \lambda u)) \leq  E(c_1, c_2, \deg(\mathcal{I}) )
\end{equation}
so the two equations above imply \cref{AmountsToForDegHEq}.

The other thing we need to check is that $R_{k,j}(F, F_{\frac{df}{du} + v})$ is not a perfect square in $\overline{\F_q(u)}[v]$.
For that we use \cref{RFFvFormula} to write
\begin{equation*}
R_{k,j} ( F, F_{\frac{df}{du} + v}) = 
a_k^{j} \prod_{i=1}^k \left( \frac{\partial F}{\partial u} ( u, \alpha_i) + \left(\frac{df}{du} + v \right) \frac{\partial F}{\partial T} (u, \alpha_i) \right).
\end{equation*}
Viewed as a polynomial in $v$, the roots in $\overline{\F_q(u)}$ of the polynomial above are
\begin{equation}
v_i = -\frac{ \frac{\partial F}{\partial u} }{  \frac{\partial F}{\partial T} } (u, \alpha_i) - \frac{df}{du}, \quad 1 \leq i \leq k.
\end{equation}

Since $\Delta_v(R(F, F_v)) \neq 0$ by assumption, it follows from \cref{DistinctValuesNonvanishingRes} that the roots above are pairwise distinct,
so $R_{k,j}(F, F_{\frac{df}{du} + v})$ is not a square of a polynomial in $v$ over $\overline{\F_q(u)}$.

\end{proof}

\section{Trace functions vs M\"{o}bius}

This section is devoted to proving \cref{NewRes} and its twisted variants.
The most general form is the following theorem.
We give an essentially self-contained statement, recalling some of \cref{DerivativeNotation}. 

\begin{thm} \label{MobiusVStraceThm}

Fix an odd prime $p$, a power $q$ of $p$, and a positive integer $k$. 
Let $0 < \gamma \leq 1$ and
\begin{equation} \label{alpha-gamma-assumption}
0 < \alpha< \frac{1}{2p} +  \frac{ \log_q  \gamma}{p}- k \log_q (1 + 2 \gamma)  
\end{equation} 
be real numbers, and set $\beta = (1+2\gamma)^k$.
Take a separable polynomial 
\begin{equation}
F(u,T)= \sum_{i=0}^k a_i(u) T^i \in \mathbb F_q[u,T]
\end{equation}
of degree $k$ in $T$. 
Pick $c_1,c_2 \in \mathbb R$ with $c_1 \geq 0 \geq c_2$  such that 
\begin{equation} \label{CoeffaBoundinTheorem}
\deg (a_i(u)) \leq c_1 + c_2 i, \quad 0 \leq i \leq k.
\end{equation}
Let $g \in \F_q[u]$ be a squarefree polynomial, let $t$ be an infinitame $g$-periodic trace function, 
and let $\mathcal I$ be an interval in $\F_q[u]$ as in \cref{IntervalNotation}.
Then
\begin{equation} \label{TheScaryBoundEq}
\begin{split}
&\sum_{ f \in \mathcal I } \mu(F(u,f)) t(f) \ll \\
& q^{\dim(\mathcal I) (1 - \alpha ) }  \beta^{ 2c_1 + (k+1)c_2}  
\left( \beta^{ -c_2 - \dim(\mathcal I) } + \beta^{\deg(\mathcal I) - \dim(\mathcal I)} \right)   
\prod_{ \pi \mid g} (  r ( t_{\pi} ) (1+\gamma) + c ( t_{\pi} ) \gamma )^{\deg (\pi)} 
\end{split} 
\end{equation} 
as $\dim(\mathcal I) \to \infty$, with the implied constant depending only on $q, k, \alpha, \gamma$. 

\end{thm}

The trivial bound here is $q^{ \dim(\mathcal I ) } \prod_{ \pi \mid g} r ( t_\pi)$, 
where $q^{ \dim(\mathcal I ) } = \operatorname{len}(\mathcal I)$ is the length of the sum and $\prod_{ \pi \mid g} r ( t_\pi)$ is a bound for each term. 
If we think of $c_1, c_2, t,$ and $\deg(\mathcal I) - \dim(\mathcal I)$ as fixed, then the bound in the theorem describes a power savings of $\alpha$, with the other terms describing the quality of the uniformity in $F,t$ and $\mathcal{I}$.
Our proof builds on the strategy of proving \cite[Proposition 4.3]{SS}.

\begin{proof}

We first reduce to the case of a polynomial $F$ with 
\begin{equation} \label{ReduceToCase}
\Delta_v(R(F, F_v)) \neq 0.
\end{equation}
By \cref{ChangeOfVariable}, there exists (a nonzero) polynomial $P \in \F_q[u]$ with $|P| \ll 1$ such that for every $c \in \F_q[u]$ with $\deg(c) < \deg(P)$ the polynomial 
$
G(u,T) = F(u, PT + c)
$
is separable in $T$, and satisfies 
\begin{equation} 
\Delta_v ( R(G, G_v) ) \neq 0.
\end{equation}

Define the intervals
\begin{equation}
\mathcal{I}_c = \left\{\frac{f-c}{P} : f \in \mathcal{I}, \ f \equiv c \ \mathrm{mod} \ P \right\}, \quad c \in \F_q[u], \ \deg(c) < \deg(P),
\end{equation}
and note that 
\begin{equation}
\begin{split}
\sum_{ f \in \mathcal I } \mu(F(u,f)) t(f) &= 
\sum_{\substack{c \in \F_q[u] \\ \deg(c) < \deg(P)}} \sum_{h \in \mathcal{I}_c} \mu(G(u,h)) t(Ph + c) \\
&\ll |P| \sum_{h \in \mathcal{J}} \mu(G(u,h)) t(Ph+c)
\end{split}
\end{equation}
where $\mathcal{J} = \mathcal{I}_c$ for some $c$ as above. 
In view of \cref{LinearCOVtraceFuncProp}, the change of the trace function does not increase neither $r(t_\pi)$ nor $c(t_\pi)$,
and the change of the polynomial $F$ can be handled by increasing $c_1$ by $k\deg(P)$.
The overall loss in the change of variable $T \mapsto PT + c$ is therefore a factor of $O(1)$,
so we can assume throughout that \cref{ReduceToCase} is satisfied. 

By \cref{PartitionIntervalLem} there exists a partition $\mathcal P$ of $\mathcal{I}$ into subintervals such that the leading term of $F(u, f(u))$ is independent of $f(u) \in \mathcal{J}_{\overline{\F_q}}$ for every $\mathcal J \in \mathcal P$,
and the number of $\mathcal J \in \mathcal P$ of any given dimension is $O(1)$. 
As a result, for 
\begin{equation} \label{TheChoiceOfParameterXiEq} 
\xi = 4 \log_q \max\{ E(c_1, c_2, \deg(\mathcal{I})), \deg(g \cdot a_k^{k+1} \cdot \Delta_v(R(F, F_v))) \}
\end{equation}
we have
\begin{equation} \label{TheXiErrorTerm}
\begin{split}
\sum_{f \in \mathcal{I}} \mu(F(u,f)) t(f) &=  
\sum_{\mathcal{J} \in \mathcal{P}} \sum_{f \in \mathcal J} \mu(F(u,f)) t(f) \\
&= \sum_{\substack{\mathcal{J} \in \mathcal{P} \\ \dim(\mathcal J) \geq \xi}} \sum_{f \in \mathcal J} \mu(F(u,f))t(f) + 
O \left(q^\xi \prod_{\pi \mid g} r(t_\pi) \right).
\end{split}
\end{equation}

Fix an interval $\mathcal{J}$ as above, and set $n = \frac{\dim(\mathcal J)}{p}$.
Pick a subset $\mathcal{R} \subseteq \mathcal{J}$ in a way that for every $f \in \mathcal J$ there exists a unique $r \in \mathcal{R}$ and a unique $s \in \F_q[u]$ with $\deg(s) < n$ such that $f = r + s^p$.
We can then write
\begin{equation} \label{SameDerivativeEq}
\sum_{f \in \mathcal J} \mu(F(u,f))t(f) = 
\sum_{r \in \mathcal{R}} \sum_{\substack{s \in \F_q[u] \\ \deg(s) < n}} \mu(F(u, r+s^p))t(r + s^p).
\end{equation}
Our choice of $\mathcal{R}$ is such that for each $f \in \mathcal J$ there is a unique $r \in \mathcal{R}$ with 
\begin{equation}
\frac{df}{du} = \frac{dr}{du}.
\end{equation} 


Fix $r \in \mathcal{R}$, and suppose first that $Z_F \cap Z_{F_{[r]}}$ is infinite.
Then from \cref{ScarcityInfiniteIntersectionLem} we get that
\begin{equation} \label{ContributionFromrInfiniteIntersectionEq}
\sum_{\substack{s \in \F_q[u] \\ \deg(s) < n}} \mu(F(u, r+s^p)) t(r + s^p) \ll \prod_{\pi \mid g} r(t_\pi).
\end{equation}
Therefore, the contribution of such $r$ to \cref{SameDerivativeEq} is $\ll$
\begin{equation}
|\mathcal{R}| \prod_{\pi \mid g} r(t_\pi) \ll q^{\dim(\mathcal{J})(1 - \frac{1}{p})} \prod_{\pi \mid g} r(t_\pi).
\end{equation}
Summing over the intervals $\mathcal{J}$ that partition $\mathcal{I}$, we get a contribution of $\ll$
\begin{equation} \label{TotalContributionFromrInfiniteIntersectionEq}
q^{\dim(\mathcal{I})(1 - \frac{1}{p})} \prod_{\pi \mid g} r(t_\pi)
\end{equation}
to \cref{TheXiErrorTerm}.

We shall now see that \cref{TotalContributionFromrInfiniteIntersectionEq} is dominated by our final bound from \cref{TheScaryBoundEq}. 
Using our assumption that $0 < \gamma \leq 1$, we get that 
\begin{equation}
\alpha <  \frac{1}{2p} +  \frac{ \log_q  \gamma}{p}- k \log_q (1 + 2 \gamma)  \leq \frac{1}{2p} < \frac{1}{p}
\end{equation}
and that $\beta = (1+2\gamma)^k \geq 1$.
From \cref{IntervalNotation} we recall that $\deg(\mathcal{I})$ is at least $\dim(\mathcal{I})$,
and from \cref{CoeffaBoundinTheorem} we deduce that
\begin{equation}\label{SoftPositivityByTwoCoeffEq}
2c_1 + (k+1)c_2 = (c_1 + c_2) + (c_1 + kc_2) \geq \deg(a_1) + \deg(a_k) \geq 0.
\end{equation}
It is now visible that \cref{TotalContributionFromrInfiniteIntersectionEq} is smaller than \cref{TheScaryBoundEq}.


From now on we assume that $Z_F \cap Z_{F_{[r]}}$ is finite, so that we can use \cref{trace-with-Mobius}.
By \cref{formula-for-Mobius} we have
\begin{equation*}
\sum_{\substack{s \in \F_q[u] \\ \deg(s) < n}} \mu(F(u, r+s^p))t(r + s^p) \ll 
\sum_{\substack{s \in \F_q[u] \\ \deg(s) < n}}  \left(\frac{W_{F,r}(u, r + s^p)}{M_{F,r}} \right) t(r + s^p).
\end{equation*}
By definition of the Jacobi symbol, and the definition of a trace function in \cref{TraceFunctionsProdEq}, the above equals
\begin{equation} 
\sum_{\substack{s \in \F_q[u] \\ \deg(s) < n}}  \prod_{\pi \mid M_{F,r}} \left(\frac{W_{F,r}^{(\pi)}(r + s^p)}{\pi} \right) 
\prod_{\pi' \mid g} t_{\pi'}(r + s^p).
\end{equation}
Using \cref{ShiftedFrobPullbackLem}(5) and \cref{Kummer-sheaf-properties}(1) we can rewrite the above as
\begin{equation} 
\sum_{\substack{s \in \F_q[u] \\ \deg(s) < n}}  \prod_{\pi \mid M_{F,r}} t_{E_r^*\mathcal{L}_{\chi} \left(W_{F,r}^{(\pi)} \right)}(s) 
\prod_{\pi' \mid g} t_{E_r^*\mathcal{F}_{\pi'}}(s).
\end{equation}
With \cref{general-tensor-product}(1), the notation of \cref{ThreeCasesSheafEq}, and \cref{TheBigTraceFunctionEq} we arrive at
\begin{equation} \label{UltimateTraceSum} 
\sum_{\substack{s \in \F_q[u] \\ \deg(s) < n}} \prod_{\pi \mid g_{F,r}} t_{F,r,\pi}(s) = 
\sum_{\substack{s \in \F_q[u] \\ \deg(s) < n}} t_{F,r}(s).
\end{equation}


For all those $r \in \mathcal{R}$ for which $t_{F,r,\pi}$ is not a Dirichlet trace function for any $\pi \mid g_{F,r}$, we bound the sum above trivially.
Since $Z_F \cap Z_{F_{[r]}}$ is finite, \cref{discriminant-definer} tells us that for every such $r$ there exist $A,B \in \F_q[u]$ such that
\begin{equation} \label{BdividesThisEq}
R(F,F_{[r]}) = A^2B, \quad B \mid g \cdot a_k \cdot \Delta_v(R(F, F_v)).
\end{equation}
Let us now check that \cref{SquareScarcity} applies here.

First, recall that we have $\Delta_v(R(F, F_v)) \neq 0$. 
Second, we use \cref{IntervalNotation}, \cref{TheChoiceOfParameterXiEq}, 
and \cref{BdividesThisEq} to get that that
\begin{equation*}
\begin{split}
\mathrm{len}(\mathcal{J}) = q^{\dim(\mathcal{J})}\geq q^\xi &= 
\max\{ E(c_1, c_2, \deg(\mathcal{I})), \deg(a_k^k \cdot g \cdot a_k \cdot \Delta_v(R(F, F_v))) \}^4 \\
&\geq \max\{ E(c_1, c_2, \deg(\mathcal{J})), \deg(a_k^k B)\}^4. 
\end{split}
\end{equation*}
This verifies the assumption made in \cref{TheLengthOfIassumptionEq}, so we can indeed invoke \cref{SquareScarcity}.

It follows from \cref{SquareScarcity} applied to each $B$ in \cref{BdividesThisEq}, and the function field version of the divisor bound in \cite[Eq. (1.81)]{IK} that the number of $r \in \mathcal{R}$ for which $t_{F,r,\pi}$ is not a Dirichlet trace function for any $\pi \mid g_{F,r}$ is $\ll$
\begin{equation}
\operatorname{len}(\mathcal{J})^{\frac{1}{2} + \epsilon} | g \cdot a_k \cdot \Delta_v(R(F, F_v))|^{\epsilon}
\end{equation}
for any $\epsilon > 0$.
Now we use \cref{BoundDegDelta}, and conclude that the contribution of these $r$ to \cref{UltimateTraceSum} is $\ll$
\begin{equation} \label{NoTauBound}
\operatorname{len}(\mathcal{J})^{\frac{1}{2} + \epsilon} |g|^{\epsilon} \cdot |a_k|^{\epsilon} \cdot 
 q^{4 \epsilon k (k-1)  (c_1 + k \max \{c_2,0 \}) + n} \cdot \prod_{\pi \mid g_{F,r}} r(t_{F,r,\pi}).
\end{equation}


For all those $r \in \mathcal{R}$ for which there exists a prime $\tau$ dividing $g_{F,r}$ such that $t_{F,r,\tau}$ is a Dirichlet trace function, 
we bound the sum on the right hand side of \cref{UltimateTraceSum} by invoking \cref{trace-interval-prop} 
and get
\begin{equation*}
\sum_{\substack{s \in \F_q[u] \\ \deg(s) < n}} t_{F,r}(s) \ll
q^{ \frac{n}{2}}   
\left( \prod_{\pi \mid g_{F,r}} \left( r(t_{F,r,\pi} )(1+Z)  + c (t_{F,r,\pi}) Z \right)^{\deg (\pi)} \right) [Z^n].
\end{equation*}
As the coefficients of powers of $Z$ in the polynomial above are nonnegative, 
for any $ \gamma > 0$ the coefficient of $Z^n$ is at most
\begin{equation}
q^{ \frac{n}{2}} \gamma^{-n}   
\left( \prod_{\pi \mid g_{F,r}} \left( r(t_{F,r,\pi} )(1+\gamma)  + c(t_{F,r,\pi}) \gamma \right)^{\deg (\pi)} \right).
\end{equation}
Applying \cref{gammaLemma} we get that the above is at most
\begin{equation} \label{BoundForDirichletablerEq}
q^{\frac{n}{2}} \gamma^{-n} (1+ 2 \gamma)^{ E(c_1,c_2, \deg(r)) }  
\prod_{ \pi \mid g} (  r(t_{\pi} ) (1+\gamma) + c (t_{\pi} ) \gamma )^{\deg (\pi)} .
\end{equation}

Using the definition of $E$ in \cref{AuxFuncEq}, 
and the inequalities 
\begin{equation}
\deg(r) \leq \deg (\mathcal J)\leq \deg(\mathcal I),
\end{equation} 
we see that \cref{BoundForDirichletablerEq} is at most
\begin{equation*}
q^{\frac{n}{2}} \gamma^{-n} (1+ 2 \gamma)^{ 2 k c_1 +k \max\{0, c_2+ \deg (\mathcal I) \}  -k + c_2 k^2 }  
\prod_{ \pi \mid g} (  r ( t_{\pi} ) (1+\gamma) + c ( t_{\pi} ) \gamma )^{\deg (\pi)} .
\end{equation*}
We can ignore the factor $(1+2\gamma)^{-k}$ as it is at most $1$. 
Summing trivially over $\mathcal{R}$ in \cref{SameDerivativeEq}, 
the above is multiplied by $|\mathcal{R}| \ll q^{\dim(\mathcal{J}) - n}$,
so recalling that $n = \frac{ \dim(\mathcal J)}{p}$, we obtain
\begin{equation}\label{more-interesting-contribution}
q^{\dim(\mathcal J) \left(1 - \frac{1}{2p} \right) } \gamma^{- \frac{ \dim(\mathcal J)}{ p} } (1+ 2 \gamma)^{ 2 k c_1 +k \max\{0, c_2+ \deg (\mathcal I) \}  + c_2 k^2 }  
\prod_{ \pi \mid g} (  r ( t_{\pi} ) (1+\gamma) + c( t_{\pi} ) \gamma )^{\deg (\pi)} .
\end{equation}


Let us now check that the contribution from \cref{NoTauBound} is smaller than that of \cref{more-interesting-contribution}, 
and can thus be neglected. To do this, observe that \cref{more-interesting-contribution} is at least
\begin{equation}
q^{\dim(\mathcal J) \left(1 - \frac{1}{2p} \right) } (1+ 2 \gamma)^{ 2 k c_1 +k ( c_2+ \deg (\mathcal I) )  + c_2 k^2 }  
\prod_{ \pi \mid g} ( r( t_{\pi} ) (1+\gamma)  )^{\deg (\pi)} 
\end{equation}
since conductors are nonnegative, and $\gamma \leq 1$.
As $g$ is squarefree, its degree is the sum of the degrees of its prime factors, so the above is at least
\begin{equation*}
q^{\dim(\mathcal J) \left(1 - \frac{1}{2p} \right) } (1+ 2 \gamma)^{ 2 k c_1 +\left( k-\frac{1}{p} \right)  \deg (\mathcal I) + c_2 k (k+1) }  (1+\gamma)^{\deg (g)}   
\prod_{ \pi \mid g}  r( t_{\pi} )^{\deg (\pi)}.
\end{equation*}
Since $\deg(\pi) \geq 1$ and $n = \frac{\dim(\mathcal{J})}{p}$, the above is at least 
\begin{equation*} 
\begin{split}
\left( q^{\frac{1}{2} \dim (\mathcal J) + n }  \prod_{ \pi \mid g}  r(t_{\pi})   \right)  \cdot \left ( q^{\dim(\mathcal J) \left( \frac{p-3}{2p} \right) } (1+ 2 \gamma)^{ 2 k c_1 +\left( k-\frac{1}{p} \right)  \deg (\mathcal I) + c_2 k (k+1) }  (1+\gamma)^{\deg (g)}   
 \right). 
\end{split}
\end{equation*} 

By \cref{ThreeCasesSheafEq}, \cref{Kummer-sheaf-properties}(6), \cref{general-tensor-product}(5), and \cref{ShiftedFrobPullbackLem}(4),
for primes $\pi$ dividing $g_{F,r}$ but not dividing $g$, we have $r(t_{F,r,\pi}) = 1$, 
and for primes $\pi \mid g$ we have $r(t_{F,r,\pi}) = r(t_\pi)$.
Therefore, \cref{NoTauBound} equals 
\begin{equation} \label{AlternateNoTauBound}
\left( q^{ \frac{1}{2}\dim (\mathcal J)  + n}  \prod_{ \pi \mid g}  r(t_{\pi} )  \right)  \cdot  \left( \operatorname{len}(\mathcal{J}) |g|  |a_k|  q^{4  k (k-1)  (c_1 + k \max \{c_2,0 \}) } \right)^{\epsilon}  
\end{equation} 
so it suffices to show that
\begin{equation*} 
\left( \operatorname{len}(\mathcal{J}) |g|  |a_k|  q^{4  k (k-1)  (c_1 + k \max \{c_2,0 \}) } \right)^{\epsilon} \leq  q^{\dim(\mathcal J) \left( \frac{p-3}{2p} \right) } (1+ 2 \gamma)^{ 2 k c_1 +\left( k-\frac{1}{p} \right)  \deg (\mathcal I) + c_2 k (k+1) }  (1+\gamma)^{\deg (g)}.   
\end{equation*}

By assumption, $\gamma>0$ and $p \geq 3$, so by taking logarithms to base $q$, we see that the above reduces to
\begin{equation} \label{OnOurWayToComparingErrorTermsEq}
\begin{split}  
&\dim(\mathcal{J}) + \deg(g) + \deg(a_k)+ 4 k (k-1) (c_1 + k \max \{c_2,0\} ) \ll \\
&2 k c_1 +\left( k-\frac{1}{p} \right)  \deg (\mathcal I) + c_2 k (k+1)  + \deg(g).
\end{split}
\end{equation}
By \cref{IntervalNotation} and \cref{CoeffaBoundinTheorem} we have
\begin{equation}
\dim (\mathcal J ) \leq \dim(\mathcal I) \leq  \deg( \mathcal I), \quad \deg(a_k) \leq c_1 + c_2k
\end{equation}
so \cref{OnOurWayToComparingErrorTermsEq} would follow once we check that 
\begin{equation} \label{WhatWeNeedFromcandkEq}
c_1 + kc_2 + 4 k (k-1) (c_1 + k \max \{c_2,0\} ) \ll 2k  c_1 + k (k+1) c_2.
\end{equation} 

If $k=1$ the above is obvious. 
Otherwise, because 
\begin{equation*}
c_1 \geq 0, \ c_1 + k c_2 \geq \deg a_k \geq 0, \ 2k  c_1 + k (k+1) c_2 = (k+1) (c_1+ kc_2) + (k-1) c_1,
\end{equation*} 
we have
\begin{equation}
0 \leq c_1 +k c_2 \leq \frac{ 2k  c_1 + k (k+1) c_2}{k+1}, \quad
0 \leq c_1 \leq \frac{ 2k  c_1 + k (k+1) c_2}{k-1},
\end{equation} 
so any linear combination of $c_1$ and $c_2$ is $O( 2k  c_1 + k (k+1) c_2)$, 
which establishes \cref{WhatWeNeedFromcandkEq} and thus concludes the argument that \cref{NoTauBound} is smaller than \cref{more-interesting-contribution}.


Since \cref{more-interesting-contribution} is exponential in $\dim (\mathcal J)$,
and there are $O(1)$ intervals $\mathcal J$ of any given dimension in our partition of $\mathcal{I}$,
summing \cref{more-interesting-contribution} over the intervals $\mathcal J$ that make up $\mathcal I$,
we get a bound for the sum in \cref{TheXiErrorTerm} of
 \begin{equation}\label{eq-main-thm-almost-final-bound}
q^{\dim(\mathcal I) \left(1 - \frac{1}{2p} \right) } \gamma^{- \frac{ \dim(\mathcal I)}{ p} } 
(1+ 2 \gamma)^{ 2 k c_1 +k \max\{0, c_2+ \deg (\mathcal I) \}  + c_2 k^2 }  
\prod_{ \pi \mid g} (  r( t_{\pi} ) (1+\gamma) + c( t_{\pi} ) \gamma )^{\deg (\pi)} 
\end{equation}
because the highest possible value of $\dim(\mathcal J)$ is $\dim(\mathcal J) = \dim(\mathcal I)$.

After exponentiating, the second inequality in \cref{alpha-gamma-assumption} translates to 
\begin{equation}
q^\alpha < q^{ \frac{1}{2p}} \gamma^{ \frac{1}{p} } (1+2 \gamma)^{-k} 
\end{equation} 
so multiplying both sides by $q$ and rearranging we get
\begin{equation}
q^{1 - \frac{1}{2p}} \gamma^{- \frac{1}{ p} } <  q^{ 1- \alpha}  (1+ 2 \gamma) ^{ -k   }.
\end{equation}
Raising to power $\dim(\mathcal{I})$ gives
\begin{equation*}
q^{\dim(\mathcal I) \left(1 - \frac{1}{2p} \right) } \gamma^{- \frac{ \dim(\mathcal I)}{ p} } \leq  q^{\dim(\mathcal I) (1- \alpha) }  (1+ 2 \gamma) ^{ -k  \dim(\mathcal I )  }    = q^{\dim(\mathcal I) (1- \alpha) }  (1+ 2 \gamma) ^{ -k (c_2 + \dim(\mathcal I ) ) + c_2 k } \end{equation*}
which implies that \cref{eq-main-thm-almost-final-bound} is $\ll$ 
\begin{equation*}
q^{\dim(\mathcal I) (1- \alpha) }  (1+ 2 \gamma)^{ 2 k c_1 +k \max\{-c_2 - \dim(\mathcal I) ,  \deg (\mathcal I)  - \dim(\mathcal I) \}  + c_2 k(k+1) }  
\prod_{ \pi \mid g} (  r( t_{\pi} ) (1+\gamma) + c( t_{\pi} ) \gamma )^{\deg(\pi)}.
\end{equation*}
Recalling that $\beta = (1 + 2\gamma)^k$, and bounding the maximum of powers of $\beta$ by thier sum, 
we arrive at \cref{TheScaryBoundEq}.


All that remains is to control the error term in \cref{TheXiErrorTerm}, which is
\begin{equation}
\max\{ E(c_1, c_2, \deg(\mathcal{I})), \deg(g \cdot a_k^{k+1} \cdot \Delta_v(R(F, F_v))) \}^4
\prod_{\pi \mid g} r(t_\pi).
\end{equation}
For every $\epsilon > 0$ we have
\begin{equation*}
\deg(g \cdot a_k^{k+1} \cdot \Delta_v(R(F, F_v)))^4 \prod_{\pi \mid g} r(t_\pi) \ll 
|g|^\epsilon |a_k|^\epsilon |\Delta_v(R(F, F_v))|^\epsilon \prod_{\pi \mid g} r(t_\pi)
\end{equation*}
so by \cref{BoundDegDelta}, the above is bounded by \cref{AlternateNoTauBound}.
We have seen that the latter is bounded by \cref{more-interesting-contribution} which led us to \cref{TheScaryBoundEq}, so this term is controlled.


By the definition of $E$ in \cref{AuxFuncEq} we have
\begin{equation} \label{TwoSummandExpBoundForEEq}
E(c_1, c_2, \deg(\mathcal{I}))^4 \prod_{\pi \mid g} r(t_\pi) \ll 
q^{\epsilon(2c_1 + \max\{0, c_2 + \deg(\mathcal{I})\} + c_2k)} \prod_{\pi \mid g} r(t_\pi)
\end{equation}
for every $\epsilon > 0$.
In case the maximum is attained at $0$, we have
\begin{equation}
c_2 + \dim(\mathcal{I}) \leq c_2 + \deg(\mathcal{I}) \leq 0
\end{equation}
so by \cref{SoftPositivityByTwoCoeffEq}, the right hand side of \cref{TwoSummandExpBoundForEEq} is
\begin{equation*}
\begin{split}
q^{ \epsilon (2 c_1 + (k+1) c_2)} q^{ - \epsilon c_2}  \prod_{\pi \mid g} r(t_\pi) &= 
q^{ \epsilon \dim (\mathcal{I}) }  q^{ \epsilon ( 2c_1 + (k+1) c_2 )} q^{ \epsilon (-c_2 - \dim(\mathcal{I} )) }  \prod_{\pi \mid g} r(t_\pi) \\
&\ll q^{ (1 - \alpha) \dim (\mathcal{I}) }  \beta^{ 2c_1 + (k+1) c_2} \beta^{ -c_2 - \dim(\mathcal{I} ) }  \prod_{\pi \mid g} r(t_\pi)
\end{split}
\end{equation*}
which is bounded by \cref{TheScaryBoundEq}.

If the maximum in \cref{TwoSummandExpBoundForEEq} is attained at $c_2 + \deg(\mathcal{I})$,
then the right hand side of \cref{TwoSummandExpBoundForEEq} is
\begin{equation}
\begin{split}
&q^{\epsilon \dim (\mathcal{I})} q^{\epsilon(2c_1 + c_2(k+1))} q^{\epsilon(\deg(\mathcal{I}) - \dim(\mathcal{I}) )} \prod_{\pi \mid g} r(t_\pi) \ll \\
&q^{(1 - \alpha) \dim (\mathcal{I})} \beta^{2c_1 + c_2(k+1)} \beta^{\deg(\mathcal{I}) - \dim(\mathcal{I})} \prod_{\pi \mid g} r(t_\pi)
\end{split}
\end{equation}
again bounded by \cref{TheScaryBoundEq}.
\end{proof}

\begin{cor} \label{MobiusVStraceCor}

Let $p$ be an odd prime, let $k$ be a positive integer, and let 
\begin{equation} \label{qAssumption}
q > 4e^2k^2p^2
\end{equation}
be a power of $p$.
Take a nonnegative integer $n$, a scalar $\lambda \in \F_q$, and define the interval
\begin{equation}
\mathcal{I} = \{f_n u^n + f_{n-1}u^{n-1} + \dots + f_0 u^0 \in \F_q[u] : f_n = \lambda \}.
\end{equation}
Then for a separable polynomial $F(T) \in \F_q[u][T]$ with $\deg_T(F) = k$, 
and an infinitame trace function $t$ to a squarefree modulus $g \in \F_q[u]$ we have
\begin{equation*}
\sum_{f \in \mathcal{I}} \mu(F(f))t(f) \ll 
q^{n\left(1 - \frac{1}{2p} + \frac{\log_q(2ekp)}{p} \right)}
\prod_{\pi \mid g}\left( r(t_\pi) \left(1 + \frac{1}{2kp} \right) + \frac{c(t_\pi)}{2kp} \right)^{\deg(\pi)}
\end{equation*}
as $n \to \infty$, with the implied constant depending only on $q$ and $F$.

\end{cor}

\begin{proof}

We invoke \cref{MobiusVStraceThm} with
\begin{equation}
\gamma = \frac{1}{2kp}, \quad \alpha = \frac{1}{2p} - \frac{\log_q(2ekp)}{p}, \quad c_1 = \deg_u(F), \quad c_2 = 0,
\end{equation}
and note that the positivity of $\alpha$ follows from \cref{qAssumption} by taking logarithms and dividing by $2p$.
Moreover we have
\begin{equation}
\begin{split}
\alpha = \frac{1}{2p} + \frac{\log_q(\gamma)}{p} - \frac{\log_q(e)}{p} &<
\frac{1}{2p} + \frac{\log_q(\gamma)}{p} -  \frac{\log_q \left(1 + \frac{1}{kp} \right)^{kp}}{p} \\
&= \frac{1}{2p} + \frac{\log_q(\gamma)}{p} -  k\log_q \left(1 + 2\gamma \right)
\end{split}
\end{equation}
so the assumptions on $\gamma$ and $\alpha$ in \cref{MobiusVStraceThm} hold.
The result follows by absorbing into the implied constant all the factors in \cref{TheScaryBoundEq} that depend only on $q, F$, and checking that $\dim(\mathcal{I}) = n = \deg(\mathcal{I})$.  
\end{proof}

Now we deduce \cref{NewRes}.

\begin{proof}

We invoke \cref{MobiusVStraceCor} with
\begin{equation}
n = \lfloor \log_q(X) \rfloor + 1, \quad \lambda = 0, \quad g=1, \quad t=1, 
\end{equation}
and since $q > 4e^2k^2p^2$, get that
\begin{equation}
\sum_{\substack{f \in \F_q[u] \\ |f| \leq X}} \mu(F(f)) \ll X^{1 - \frac{1}{2p} + \frac{\log_q(2ekp)}{p} } = o(X)
\end{equation}
as required.
\end{proof}

We similarly deduce \cref{LinearMobiusVStraceFunctionThm}.

\begin{proof}

We invoke \cref{MobiusVStraceCor} with
\begin{equation}
k=1, \quad \lambda = 1, \quad F(u,T) = T, \quad g=\pi,
\end{equation}
and get that 
\begin{equation*}
\begin{split}
\sum_{f \in \mathcal{M}_n} \mu(f)t(f) &\ll 
|\mathcal{M}_n|^{1 - \frac{1}{2p} + \frac{\log_q(2ep)}{p}}
\left( r(t) \left(1 + \frac{1}{2p} \right) + \frac{c(t)}{2p} \right)^{\deg(\pi)}\\
&= 
|\mathcal{M}_n|^{1 - \frac{1}{2p} + \frac{\log_q(2ep)}{p}}
|\pi|^{\log_q \left( r(t) \left( 1 + \frac{1}{2p} \right) + \frac{c(t)}{2p} \right)}.
\end{split}
\end{equation*}

\end{proof}

We will need the following consequence of \cref{MobiusVStraceThm} in the proof of \cref{MainRes}.

\begin{cor} \label{MobiusBeatsKloostermanCor}
Keep \cref{TheAnalyticBarNotation} and \cref{IntervalNotation}.
Fix an odd prime $p$, and a power $q$ of $p$. 
Let $0 < \gamma \leq 1$  and $\alpha$ be real numbers satisfying
\begin{equation} \label{weird-alpha-gamma-assumption} 
0 < \alpha < \min \left\{\frac{1}{2} - 10\log_q ( 1+ 2 \gamma)  + \log_q (1+ 3 \gamma), 
\frac{1}{2p} +  \frac{ \log_q  \gamma}{p}- 2 \log_q (1 + 2 \gamma)  \right\}.
\end{equation}
Set $\beta = (1+2\gamma)^2$. 

Let $n$ be a nonnegative integer, pick $c_1,c_2, c_3 \in \mathbb R$ with $c_1 \geq 0 \geq c_2$, 
and let $a,b,c \in \F_q[u]$ be polynomials satisfying
\begin{equation*}
b^2 - 4ac \neq 0, \quad \deg(a) \leq c_1 + 2c_2, \ \deg(b) \leq c_1 + c_2 - n, \ \deg(c) \leq c_1 - 2n.
\end{equation*}
For every nonzero polynomial $y \in \F_q[u]$ of degree at most $n$ put 
\begin{equation}
F_y(T) = aT^2 + byT + cy^2 \in \F_q[u][T] 
\end{equation}
and let $\mathcal{I}_y$ be an interval in $\F_q[u]$ of degree at most $c_3$.
Then for $h \in \F_q[u]$ we have
\begin{equation} \label{TheScaryBoundEq2}
\begin{split} \sum_{\substack{y \in \F_q[u] \setminus \{0\} \\ \deg(y) \leq n }} \left| \sum_{\substack{x \in \mathcal I_y \\ \gcd(x,y) = 1}} 
\mu (F_y(x)) e \left( \frac{ h \overline{x} }{y} \right)  \right| \ll  q^{n + c_3  (1 - \alpha ) }  \beta^{ 2c_1 + 3c_2}  
\left( \beta^{ -c_2 -  c_3} + 1 \right)   (1+3\gamma)^n  \end{split}  \end{equation} 
as $n  \to \infty$, with the implied constant depending only on $q, \alpha, \gamma$. 

\end{cor}

%
%

\begin{proof}

Every monic polynomial $y \in \F_q[u]$ can be decomposed uniquely as $y = y_1y_2$ with $y_1$ a squareful monic polynomial, 
and $y_2$ a squarefree monic polynomial coprime to $y_1$. 
Explicitly, the polynomial $y_2$ is the product of all primes $\pi \in \F_q[u]$ with $\pi$, but not $\pi^2$. dividing $y$.
We can therefore bound our sum by
\begin{equation}
\sum_{m=0}^n
\sum_{y_1 \in \mathcal{S}_{m}}
\sum_{\substack{y_2 \in \mathcal{H}_{n-m} \\ \gcd(y_1,y_2) = 1}}
\left| \sum_{\substack{x \in \mathcal{I}_{y_1y_2} \\ \gcd(x, y_1y_2) = 1}} \mu(F_{y_1y_2}(x)) e \left( \frac{h \overline{x}}{y_1y_2} \right)  \right|
\end{equation}
where $\mathcal{S}_m, \mathcal{H}_{n-m} \subseteq \F_q[u]$ are the sets of squareful polynomials of degree $m$ and squarefree polynomials of degree $n-m$ respectively. 
Since every polynomial in $\mathcal{S}_m$ is the product of a square and a cube, we have
\begin{equation} \label{SquarefulCountEq}
|\mathcal{S}_m| \ll q^{\frac{m}{2}},
\end{equation}
see \cite[(2.7)]{RG}.

For every two coprime polynomials $y_1,y_2 \in \F_q[u]$, we can find polynomials $A,B \in \F_q[u]$ with $Ay_1 + By_2 = 1$, 
so we can rewrite the above as
\begin{equation} \label{Sumy1y2Eq}
\sum_{m=0}^n
\sum_{y_1 \in \mathcal{S}_{m}}
\sum_{\substack{y_2 \in \mathcal{H}_{n-m}  \\ \gcd(y_1, y_2) = 1}}
\left| \sum_{\substack{x \in \mathcal{I}_{y_1y_2} \\ \gcd(x, y_1y_2) = 1}} \mu(F_{y_1y_2}(x)) e \left( \frac{Bh \overline{x}}{y_1} \right) e \left( \frac{Ah \overline{x}}{y_2} \right)  \right|.
\end{equation}
We use the trivial bound for those pairs $(y_1, y_2)$ with $\deg(y_1) > \dim(\mathcal{I}_{y_1y_2})$, which is 
\begin{equation*}
\begin{split}
&\sum_{m=0}^n \sum_{y_1 \in \mathcal S_m}\sum_{\substack{y_2 \in \mathcal{H}_{n-m}  \\ \gcd(y_1, y_2) = 1}} 
q^{\dim(\mathcal{I}_{y_1y_2})} \leq
\sum_{m=0}^n \sum_{y_1 \in \mathcal S_m}\sum_{\substack{y_2 \in \mathcal{H}_{n-m}}} 
q^{\min\{\deg(\mathcal{I}_{y_1y_2}), \deg(y_1)\}} \leq \\
&\sum_{m=0}^n \sum_{y_1 \in \mathcal S_m}\sum_{\substack{y_2 \in \mathcal{H}_{n-m}}}
q^{\frac{c_3 + m}{2}} \ll \sum_{m=0}^n q^{\frac{m}{2}}q^{n-m} q^{\frac{c_3 + m}{2}} \ll nq^{n+\frac{c_3}{2}}
\ll (1 + 3 \gamma)^n q^{n + (1 - \alpha)c_3}
\end{split}
\end{equation*}
and that is bounded by the right hand side of \cref{TheScaryBoundEq2}.

For the other pairs $(y_1, y_2)$ in \cref{Sumy1y2Eq}, those with $\dim(\mathcal{I}_{y_1y_2}) \geq \deg(y_1)$, we define the intervals
\begin{equation*}
\mathcal{I}_{y_1y_2}^r = \left\{\frac{f-r}{y_1} : f \in \mathcal{I}_{y_1y_2}, \ f \equiv r \ \mathrm{mod} \ y_1 \right\}, \quad r \in \F_q[u], \ \deg(r) < \deg(y_1),
\end{equation*}
so that our sum can be bounded, using the triangle inequality, by
\begin{equation*}
\sum_{m=0}^n
\sum_{y_1 \in \mathcal{S}_{m}}
\sum_{\substack{r \in \F_q[u] \\ \deg(r) < m \\ \gcd(r, y_1) = 1}} 
\sum_{\substack{y_2 \in \mathcal{H}_{n-m} \\ \gcd(y_1, y_2) = 1 \\ \dim(\mathcal{I}_{y_1y_2} ) \geq m}}
\left| \sum_{\substack{z \in \mathcal{I}^r_{y_1y_2} \\ \gcd(y_1z+r, y_2) = 1}} \mu(F_{y_1y_2}(y_1z + r)) e \left( \frac{Ah \overline{(y_1z + r)}}{y_2} \right)  \right|
\end{equation*}
where $x = y_1z + r$.

Since $y_2$ is squarefree, from \cref{ExpsToTracesProp} we get that the above is at most
\begin{equation*}
\sum_{m=0}^n
\sum_{y_1 \in \mathcal{S}_{m}}
\sum_{\substack{r \in \F_q[u] \\ \deg(r) < m}} 
\sum_{\substack{y_2 \in \mathcal{H}_{n-m} \\ \dim(\mathcal{I}_{y_1y_2} ) \geq m}}
\left| \sum_{\substack{z \in \mathcal{I}^r_{y_1y_2} \\ \gcd(y_1z+r, y_2) = 1}} \mu(F_{y_1y_2}(y_1z + r)) 
t(y_1z + r)  \right|
\end{equation*}
where $t$ is an infinitame trace function with 
\begin{equation} \label{RankOneConductorTwoBoundEq}
r(t) \leq 1, \quad c(t) \leq 2.
\end{equation}
We have
\begin{equation*}
\begin{split}
F_{y_1y_2}(y_1T + r) &= a(y_1T+r)^2 + by_1y_2(y_1T+r) + cy_1^2y_2^2 \\
&= ay_1^2T^2 + (2ary_1 + by_1^2y_2)T + ar^2 + bry_1y_2+ cy_1^2 y_2^2. 
\end{split}
\end{equation*} 

It follows from our initial assumptions on $a,b,c,$ and $c_2$ that the degrees of the coefficients of $F_{y_1y_2}(y_1T + r)$ satisfy 
\begin{equation*}
\deg(ar^2 + bry_1y_2 + cy_1^2y_2^2) \leq \max\{c_1 + 2c_2 + 2m, c_1 + c_2 - n + m + n, c_1 - 2n + 2n\} \leq c_1 + 2m,
\end{equation*}\begin{equation*}
\deg(2ary_1 + by_1^2y_2) \leq \max\{c_1 + 2c_2 + m + m, c_1 + c_2 - n + 2m + n - m\} \leq c_1 + c_2 + 2m,
\end{equation*}
and
\begin{equation*}
\deg(ay_1^2) \leq c_1 + 2c_2 + 2m.
\end{equation*}

We can drop the condition $\gcd(y_1z + r, y_2) = 1$ in the sum above since for any nonconstant common divisor $D \in \F_q[u]$ of $y_1z+r$ and $y_2$,
we see that $D^2$ divides $F_{y_1y_2}(y_1z + r)$ so $\mu(F_{y_1y_2}(y_1z + r)) = 0$.
Since $b^2 - 4ac \neq 0$, the polynomial  $F_{y_1y_2}(y_1T + r)$ is separable,
so we can invoke \cref{MobiusVStraceThm} with 
\begin{equation}
p, \ q, \ k=2, \ \gamma, \ \alpha, \ c_1 + 2m, \ c_2, \ g=y_2, \ t,
\end{equation} 
and get from \cref{RankOneConductorTwoBoundEq} that the sum above is $\ll$
\begin{equation}\label{sieve-sum}\begin{split} & 
\sum_{m=0}^n
\sum_{y_1 \in \mathcal{S}_{m}}
\sum_{\substack{r \in \F_q[u] \\ \deg(r) < m}} 
\sum_{\substack{y_2 \in \mathcal{H}_{n-m} \\ \dim(\mathcal{I}_{y_1y_2}) \geq m}} \\ &  q^{\dim(\mathcal I^r_{y_1y_2} ) (1 - \alpha ) }  
\beta^{ 2c_1 + 4 m + 3c_2}  
\left( \beta^{ -c_2 - \dim(\mathcal I^r_{y_1y_2}) } + \beta^{\deg(\mathcal I^r_{y_1y_2}) - \dim(\mathcal I^r_{y_1y_2})} \right)   
\prod_{ \pi \mid  y_2 } (  1+ 3  \gamma )^{\deg (\pi)}.  
\end{split}  
\end{equation} 

Since $y_2$ is squarefree we have
\begin{equation} \label{DegrrePolDegreeFactorsSquareFreeEq}
\prod_{ \pi \mid  y_2 } (  1+ 3  \gamma )^{\deg (\pi)} = (  1+ 3  \gamma )^{\sum_{ \pi \mid  y_2 } \deg (\pi)} = (1 + 3\gamma)^{\deg(y_2)}.  
\end{equation}
By our assumptions we have $\gamma \leq 1$ hence
\begin{equation}
\alpha < \frac{1}{2p} +  \frac{ \log_q  \gamma}{p}- 2 \log_q (1 + 2 \gamma) \leq 1 - 2 \log_q (1 + 2 \gamma)
\end{equation}
so from our choice of $\beta$ we get
\begin{equation}
\log_q \beta = 2 \log_q (1 + 2\gamma) \leq 1 - \alpha
\end{equation}
or equivalently $\beta \leq q^{1 - \alpha}$.

Since $\dim(\mathcal{I}_{y_1y_2} ) \geq m$ we have
\begin{equation} 
\dim(\mathcal I^r_{y_1y_2} ) = \dim ( \mathcal I_{y_1 y_2}  ) - m  \leq \deg ( \mathcal I_{y_1y_2})-m \leq c_3 -m 
\end{equation}
so
\begin{equation} \label{QbetaInequalityTakeOne}
q^{\dim(\mathcal I^r_{y_1y_2} ) (1 - \alpha ) } \beta^{ -c_2 - \dim(\mathcal I^r_{y_1y_2}) } \leq 
q^{(c_3 - m)(1-\alpha)} \beta^{-c_2 - (c_3-m)}
\end{equation}
because $\beta \leq q^{1 - \alpha}$, and similarly
\begin{equation} \label{QbetaInequalityTakeTwo}
q^{\dim(\mathcal I^r_{y_1y_2} ) (1 - \alpha ) }\beta^{\deg(\mathcal I^r_{y_1y_2}) - \dim(\mathcal I^r_{y_1y_2})} \leq 
q^{(c_3 - m)(1-\alpha)} \beta^{c_3 - (c_3 - m)}.
\end{equation}

It follows from \cref{SquarefulCountEq}, \cref{DegrrePolDegreeFactorsSquareFreeEq}, \cref{QbetaInequalityTakeOne}, 
and \cref{QbetaInequalityTakeTwo} that \eqref{sieve-sum} is $\ll$
\begin{equation}
\sum_{m=0}^n q^{ \frac{m}{2} } q^m q^{n-m} q^{(c_3-m)(1 - \alpha)}\beta^{2c_1 + 4m + 3c_2}
(\beta^{m-c_2-c_3} + \beta^m)(1 + 3\gamma)^{n-m}
\end{equation}
which simplifies to
\begin{equation}
q^{n + c_3(1 - \alpha)} \beta^{2c_1 + 3c_2}(\beta^{-c_2-c_3} + 1) (1 + 3\gamma)^{n}
\sum_{m=0}^n \left(\frac{q^{\alpha - \frac{1}{2}}  \beta^{5}}{1+3\gamma} \right)^m
\end{equation}
so to obtain the bound in \cref{TheScaryBoundEq2}, 
it suffices to check that 
\begin{equation}  
\frac{q^{\alpha - \frac{1}{2}}  \beta^{5}}{1+3\gamma} < 1.
\end{equation} 

After taking logarithms in the above, rearranging, and recalling that by definition $\beta = (1 + 2\gamma)^2$, the above becomes
\begin{equation}
\alpha  < \frac{1}{2} - 10\log_q (1 + 2\gamma) + \log_q(1 + 3\gamma)
\end{equation}
which is part of our initial assumptions.
\end{proof}

\section{Quadratic congruences}

\begin{notation}

Let $N \in \F_q[u]$ be a nonzero polynomial, and set $n = \deg(N)$.
We identify $\F_q[u]/(N)$ with the set of representatives
\begin{equation}
\mathcal{P}_{<n} = \{f \in \F_q[u] : \deg(f) < n \}
\end{equation}
for the residue classes.
As in \cref{REAC}, for a polynomial $M \in \F_q[u]$ we denote by $\widetilde{M}$ the unique representative of its residue class in $\mathcal{P}_{<n}$.

\end{notation}

\begin{prop} \label{IndicatorAdditiveCharsProp}

For an integer $0 \leq d \leq n$,
the indicator function of the degree of the reduction of $M$ mod $N$ being less than $d$ can be expressed as
\begin{equation}
{\bf{1}}_{\deg(\widetilde{M}) < d} = q^{d-n}\sum_{\substack{ h \in \F_q[u] \\ \deg(h) < n-d}} e \left(\frac{hM}{N} \right).
\end{equation}

\end{prop}

\begin{proof}

We claim first that the indicator function of the $\F_p$-subspace $\mathcal{P}_{<d}$ of $\F_q[u]/(N)$ equals the average over all additive characters of $\F_q[u]/(N)$ that are identically $1$ on $\mathcal{P}_{<d}$.
Clearly, this average is $1$ on $\mathcal{P}_{<d}$, so the claim follows in case $d = n$.
In case $d < n$, we restrict to the (nonempty) complementof $\mathcal{P}_{<d}$ in $\F_q[u]/(N)$, 
and view our average as the average over all characters of the nontrivial quotient group
\begin{equation}
Q_d = \frac{\F_q[u]/(N)}{\mathcal{P}_{<d}}.
\end{equation}
By orthogonality of characters, this average vanishes, so our claim is verified.

The number of characters we are averaging over is
\begin{equation}
|Q_d| = \left| \frac{\F_q[u]/(N)}{\mathcal{P}_{<d}} \right| =  \frac{\left| \F_q[u]/(N) \right|}{\left| \mathcal{P}_{<d} \right|}  =
 \frac{q^n}{q^d} = q^{n-d} 
\end{equation}
so by \cref{REAC}, these characters are
\begin{equation}
\psi_h(M) = e \left( \frac{hM}{N} \right), \quad h \in \F_q[u], \quad \deg(h) < n-d,
\end{equation}
hence the proposition follows.
\end{proof}

\begin{notation} \label{QuadraticCongNot}

Let $d,k$ be nonnegative integers, let $A \in \mathcal{M}_k$, and let $D$ in $\F_q[u]$ be a polynomial for which 
the polynomial 
\begin{equation}
F(T) = T^2 + D \in \F_q[u][T]
\end{equation}
is irreducible over $\F_q[u]$.
We set
\begin{equation}
\rho_d(A;F) = \# \{f \in \mathcal{M}_d : F(f) \equiv 0 \ \mathrm{mod} \ A \}, \quad \rho(A;F) = \rho_k(A;F).
\end{equation}
In case $d \geq k$ we clearly have
\begin{equation} \label{RhoFormula}
\rho_d(A;F) = q^{d-k}\rho(A;F).
\end{equation}

\end{notation}

\begin{cor} \label{QuadCongCor}

Suppose that $d < k$. Then
\begin{equation*}
\rho_d(A;F) = q^{d-k}\rho(A;F) + q^{d-k} \sum_{\substack{h \in \F_q[u] \setminus\{0\} \\ \deg(h) <k-d }} e \left( \frac{-hu^d}{A} \right)
\sum_{\substack{f \in \F_q[u]/(A) \\ F(f) \equiv 0 \ \mathrm{mod} \ A}} e \left( \frac{hf}{A} \right).
\end{equation*}  
  
\end{cor}

\begin{proof}
We have
\begin{equation}
\rho_d(A;F) = \sum_{\substack{f \in \F_q[u]/(A) \\ F(f) \equiv 0 \ \mathrm{mod} \ A}} {\bf{1}}_{f \in \mathcal{M}_d} = 
\sum_{\substack{f \in \F_q[u]/(A) \\ F(f) \equiv 0 \ \mathrm{mod} \ A}} {\bf{1}}_{\deg(f - u^d) < d}
\end{equation}  
which by \cref{IndicatorAdditiveCharsProp} equals
\begin{equation}
q^{d-k} \sum_{\substack{h \in \F_q[u] \\ \deg(h) <k-d }} e \left( \frac{-hu^d}{A} \right) \sum_{\substack{f \in \F_q[u]/(A) \\ F(f) \equiv 0 \ \mathrm{mod} \ A}} e \left( \frac{hf}{A} \right).
\end{equation}  
Separating the contribution of $h=0$ gives the corollary.
\end{proof}


\begin{prop} \label{AverageNumberOfSolsQuadCongProp}

For every positive integer $k$ we have
\begin{equation}
\sum_{A \in \mathcal{M}_k} \rho(A;F) \ll  |\mathcal{M}_k| \cdot |D|^{\epsilon} 
\end{equation}
with the implied constant depending only on $q$ and $\epsilon$.

\end{prop}

\begin{proof} 

We can decompose uniquely $D = D_1 D_2^2$ where 
\begin{equation} \label{DOneDefEq}
D_1 = \prod_{\substack{\pi \mid D \\ v_\pi(D) \equiv 1 \ \mathrm{mod} \ 2}} \pi
\end{equation} 
is squarefree. We define a character on $\F_q[u]$ by
\begin{equation}
\chi(f) = \left( \frac{-D_1}{f}\right), \quad f \in \F_q[u] \setminus \{0\}.
\end{equation}

If $\pi \in \F_q[u]$ is a prime that does not divide $D$, for every positive integer $r$ we can use Legendre symbols to write
\begin{equation} \label{CasepiDoesntDivideDs}
\rho( \pi^r; F) = 1 + \left( \frac{-D}{\pi} \right) = 1 + \left( \frac{- D_1D_2^2}{\pi} \right)  = 1 + \chi(\pi) 
\end{equation}
in view of Hensel's Lemma.

If $\pi \in \F_q[u]$ is a prime that divides $D_1$, for every integer $r \geq 1$ we have
\begin{equation} \label{CasepidividesDone}
\rho (\pi^r; F ) = \begin{cases} |\pi|^{\lfloor \frac{r}{2 } \rfloor} &  r \leq  v_\pi (D) \\ 0 & r > v_\pi (D).
\end{cases}
\end{equation}
Indeed, when $D \equiv 0$ mod $\pi^r$, we are counting the elements in $\F_q[u]/(\pi^r)$ which square to zero, 
or equivalently are zero mod $\pi^{\lceil{\frac{r}{2}} \rceil}$, so their number is
\begin{equation}
\frac{|\pi|^r}{|\pi|^{\lceil{\frac{r}{2}} \rceil}} = |\pi|^{r - \lceil{\frac{r}{2}} \rceil} = |\pi|^{\lfloor{\frac{r}{2}} \rfloor}.
\end{equation}
Since $v_\pi(D)$ is odd by our definition of $D_1$ in \cref{DOneDefEq}, nothing squares to $-D$ mod $\pi^r$ for $r > v_{\pi}(D)$. 

Finally, if $\pi$ divides $D_2$ and does not divide $D_1$, for $r \geq 1$ we have
\begin{equation} \label{CasepiDividesDtwo}
\rho (\pi^r; F ) = \begin{cases} |\pi|^{\lfloor \frac{r}{2 } \rfloor} &  r \leq v_\pi (D) \\  (1 + \chi(\pi)) |\pi|^{ v_\pi(D_2) } & r >  v_\pi (D). \end{cases}
\end{equation}
Indeed, the first case is established as in \cref{CasepidividesDone}.
For the second case we note that every element in $\F_q[u]/(\pi^r)$ is of the form $\pi^i \alpha$ for a unique choice of $0 \leq i \leq r$ and $\alpha \in (\F_q[u]/(\pi^{r-i}))^\times$.
The elements whose square is $-D$ are those that have 
\begin{equation}
i = \frac{v_{\pi}(D)}{2} = \frac{v_{\pi}(D_1 D_2^2)}{2} = \frac{v_{\pi}(D_1) + 2v_{\pi}(D_2)}{2} = v_{\pi}(D_2)
\end{equation} 
and 
\begin{equation}
\alpha^2 \equiv -D \pi^{-v_{\pi}(D)} \equiv -D_1D_2^2 \pi^{-2v_{\pi}(D_2)} \mod \pi^{r - 2v_{\pi}(D_2)}.
\end{equation}
The number of such $\alpha \in (\F_q[u]/(\pi^{r - v_{\pi}(D_2)}))^\times$ is
\begin{equation}
\left(1 + \left( \frac{-D_1D_2^2 \pi^{-2v_{\pi}(D_2)}}{\pi} \right)\right) \frac{|\pi|^{r - v_{\pi}(D_2)}}{|\pi|^{r - 2v_{\pi}(D_2)}} = 
(1 + \chi(\pi))|\pi|^{v_{\pi}(D_2)}
\end{equation}
as stated in the second case of \cref{CasepiDividesDtwo}.

The function $A \mapsto \rho(A,F)$ is multiplicative so in view of \cref{CasepiDoesntDivideDs}, \cref{CasepidividesDone}, and \cref{CasepiDividesDtwo}, 
we have the Euler product
\begin{equation*}
\begin{split}
&H(t) = \sum_{k=0}^\infty t^{k} \sum_{A \in \mathcal{M}_k} \rho(A;F) = 
\prod_{\pi} \left( \sum_{r=0}^{\infty} t^{r \deg(\pi)}\rho(\pi^r; F) \right) = \\
&\prod_{\pi \nmid D}\left( 1 + (1 + \chi(\pi))\sum_{r=1}^\infty t^{r \deg(\pi)} \right) \cdot
\prod_{\pi \mid D_1 } \left( \sum_{r=0}^{v_\pi(D)} |\pi|^{\lfloor \frac{r}{2} \rfloor} t^{r \deg(\pi)} \right) \cdot \\ &\prod_{\substack{ \pi \mid D_2 \\ \pi \nmid D_1}}
\left( (1 + \chi(\pi)) |\pi|^{v_\pi (D_2)}\sum_{r > 2v_\pi(D_2)} t^{r \deg(\pi)}+
 \sum_{r=0}^{2v_\pi(D_2)} |\pi|^ {\lfloor \frac{r}{2} \rfloor} t^{r \deg(\pi)} \right) .
\end{split}
\end{equation*}

We will now express the above as the product of
\begin{equation} 
L(t; \chi) \zeta_{\mathbb F_q[u]} (t) = 
\prod_{ \pi \nmid D_1} \frac{1}{ 1 - \chi(\pi) t^{\deg(\pi)}   } \cdot \prod_{\pi}  \frac{1}{ 1 - t^{\deg(\pi)} }
\end{equation} 
with a rapidly converging Euler product.
To do this, note that for primes $\pi$ not dividing $D$ we have
\begin{equation*} 
\begin{split}
& ( 1 - \chi(\pi) t^{\deg(\pi)}) (1- t^{\deg(\pi)} ) \left( 1 + (1 + \chi(\pi))\sum_{r=1}^\infty t^{r\deg(\pi)} \right) = \\ 
&( 1 - \chi(\pi) t^{\deg(\pi)}) (1+  \chi(\pi) t^{\deg(\pi)} ) = 1 - \chi^2(\pi) t^{2\deg(\pi)} = 1 - t^{2\deg(\pi)}.
\end{split} 
\end{equation*} 
Similarly, for primes $\pi$ dividing $D_2$ but not $D_1$ we get
\begin{equation*} 
\begin{split} 
& ( 1 - \chi(\pi) t^{\deg(\pi)}) (1- t^{\deg(\pi)}  )\left( 
(1 + \chi(\pi))|\pi|^{v_\pi (D_2)}\sum_{r > 2v_\pi(D_2)} t^{r\deg(\pi)} +
\sum_{r=0}^{2v_\pi(D_2)} |\pi|^ {\lfloor \frac{r}{2} \rfloor}t^{r\deg(\pi)}
\right) 
=\\ &|\pi|^{v_{\pi}(D_2)}t^{2\deg(\pi)v_{\pi}(D_2)}(1 - t^{2\deg(\pi)} )  + 
( 1 - \chi(\pi) t^{\deg(\pi)}) (1- t^{\deg(\pi)}  )\sum_{r=0}^{2v_\pi(D_2)-1 } |\pi|^ {\lfloor \frac{r}{2} \rfloor} t^{r\deg(\pi)}.
\end{split} 
\end{equation*} 
Combining these, we obtain
\begin{equation} \label{LdecompositionEq}
\begin{split}
&H(t) = 
L(t;\chi) \zeta_{\mathbb F_q[u]} (t) \prod_{\pi \nmid D}\left( 1 - t^{2\deg(\pi)}   \right)  \prod_{\pi \mid D_1 }\left(   (1- t^{\deg(\pi)}  ) \sum_{r=0}^{v_\pi(D)} |\pi|^ {\lfloor \frac{r}{2} \rfloor}t^{r\deg(\pi)}  \right)   \\ 
&\prod_{\substack{ \pi \mid D_2 \\ \pi \nmid D_1}}  
\left(  
|\pi|^{v_{\pi}(D_2)}t^{2\deg(\pi)v_{\pi}(D_2)}(1 - t^{2\deg(\pi)} )  +
( 1 - \chi(\pi) t^{\deg(\pi)}) (1- t^{\deg(\pi)}  )\sum_{r=0}^{2v_\pi(D_2)-1 } |\pi|^ {\lfloor \frac{r}{2} \rfloor}t^{r\deg(\pi)}
\right). 
\end{split}
\end{equation}

Let us now show that the Euler product terms are $\ll |D|^{\epsilon} $ for $|t| \leq q^{-\frac{3}{4}}$
(any cutoff strictly between $\frac{1}{2}$ and $1$ in place of $\frac{3}{4}$ would work equally well for our purposes here). 
We have
\begin{equation*} 
\left|\prod_{\pi \nmid D}\left( 1 - t^{2\deg(\pi)}   \right) \right| \leq \prod_{\pi} \left( 1 + |\pi|^{ - 2 \cdot \frac{3}{4} } \right) \leq \prod_\pi \frac{1}{ 1 - |\pi|^{ - \frac{3}{2}} } = \frac{1}{1- q^{ - \frac{1}{2} }} \ll 1. 
\end{equation*}
The contribution of each prime $\pi$ dividing $D_1$ is 
\begin{equation*} 
\left|  (1- t^{\deg(\pi)}  ) \sum_{r=0}^{v_\pi(D)} |\pi|^ {\lfloor \frac{r}{2} \rfloor}t^{r\deg(\pi)}  \right| \leq  (1 + |\pi|^{-\frac{3}{4}} ) \sum_{r=0}^{\infty} |\pi|^{ \frac{r}{2} - \frac{3r}{4} } = (1 + |\pi|^{- \frac{3}{4}}) \frac{1}{ 1 - |\pi|^{ - \frac{1}{4} } } \leq |\pi|^{\epsilon} \end{equation*} 
for all but finitely many primes $\pi$ in $\F_q[u]$. 
Similarly, every prime $\pi$ dividing $D_2$ but not $D_1$ gives
\begin{equation*} 
\begin{split} 
&\left| 
|\pi|^{v_{\pi}(D_2)}t^{2\deg(\pi)v_{\pi}(D_2)} (1 - t^{2\deg(\pi)} ) + ( 1 - \chi(\pi) t^{\deg(\pi)}) (1- t^{\deg(\pi)}  )\sum_{r=0}^{2v_\pi(D_2)-1 } |\pi|^ {\lfloor \frac{r}{2} \rfloor}t^{r\deg(\pi)} \right| \leq\\ &|\pi|^{ - \frac{1}{2} v_\pi(D_2) } (1+ |\pi|^{- \frac{3}{2} } ) + (1 + |\pi|^{ - \frac{3}{4} } )^2 \sum_{r=0}^{\infty} |\pi|^{ \frac{r}{2} - \frac{3r}{4}} \leq \\ &|\pi|^{ - \frac{1}{2} }  (1+ |\pi|^{- \frac{3}{2} } )  + (1 + |\pi|^{ - \frac{3}{4} } )^2 \frac{1}{ 1 - |\pi|^{-\frac{1}{4} }}  \leq |\pi|^{\epsilon}\end{split} \end{equation*} 
for all but finitely many $\pi$. Combining these, we obtain 
\begin{equation}   
\frac{ H(t)} {  L(t;\chi) \zeta_{\mathbb F_q[u]} (t) }  \ll |D|^\epsilon, \quad |t| \leq q^{-\frac{3}{4}},
\end{equation} 
and recall that $\zeta_{\mathbb F_q[u]}(t) = (1 - qt)^{-1}$.

We now split into two cases. 
In the first case, $D_1$ is nonconstant, so $L(t; \chi)$ is a polynomial in $t$ which satisfies 
\begin{equation}
L(t;\chi) \ll |D|^\epsilon, \quad |t| \leq q^{-\frac{3}{4}},
\end{equation}
by Weil's Riemann Hypothesis.
Therefore, the only pole of $H$ with $|t| \leq  q^{-\frac{3}{4}} $ is at $t = q^{-1} $, and this pole is simple.  
By Cauchy's residue theorem, we thus have
\begin{equation*} 
\begin{split} 
\sum_{A \in \mathcal{M}_k} \rho(A;F) &\ll   
\left| \oint_{ |t| =q^{-\frac{3}{4}} } \frac{ H(t)  } {t^{k+1} } \right| + \left| q^{k+1}  \operatorname{Res}_{ t= q^{-1} }  H(t) \right| \\ &  \ll  \oint_{ |t| =q^{-\frac{3}{4}} } \frac{ |D|^\epsilon \left| 1- qt  \right|^{-1}      } {|t| ^{k+1} }  + q^{k+1}  |D|^\epsilon \left| \operatorname{Res}_{t=q^{-1}} (1-qt)^{-1} \right|  \\ &  \ll  |D|^\epsilon \left|1 - q^{ \frac{1}{4} } \right|^{-1}  q^{ \frac{3k}{4}}  + q^k |D|^\epsilon \ll  q^k |D|^\epsilon.
\end{split} 
\end{equation*} 

In the second case $D_1$ is constant, so 
\begin{equation}
L(t;\chi) = \sum_{k=0}^{\infty} t^k \sum_{f \in \mathcal{M}_k} \chi(f) = 
\sum_{k=0}^{\infty} t^k q^k (-1)^k = \frac{1}{1 +qt}. 
\end{equation}
Therefore, the only poles of $H$ with $|t|  \leq q^{- \frac{3}{4}}$ are $t = \pm q^{-1}$, and these poles are simple.
From Cauchy's residue theorem we similarly get
\begin{equation*} 
\begin{split} 
\sum_{A \in \mathcal{M}_k} \rho(A;F) &\ll   
\left| \oint_{ |t| =q^{-\frac{3}{4}} } \frac{ H(t)  } {t^{k+1} } \right| + \left| q^{k+1}  \operatorname{Res}_{ t= \pm q^{-1} }  H(t) \right| \\ &  
\ll  \oint_{ |t| =q^{-\frac{3}{4}} } \frac{ |D|^\epsilon \left| 1- qt  \right|^{-1}   \left| 1 + qt \right|^{-1}   } {|t| ^{k+1} }  + q^{k+1}  |D|^\epsilon \left| \operatorname{Res}_{t= \pm q^{-1}} (1 \mp qt)^{-1} \right|  \\ 
&\ll  |D|^\epsilon \left|1 - q^{ \frac{1}{2} } \right|^{-1}  q^{ \frac{3k}{4}}  + q^k |D|^\epsilon \ll  q^k |D|^\epsilon.
\end{split} 
\end{equation*}

\end{proof}

\begin{notation}

Keep \cref{QuadraticCongNot}.
For a prime $\pi \in \F_q[u]$ set
\begin{equation*}
\chi_F(\pi) = \rho(\pi ; F) - 1 = \#\{f \in \F_q[u]/(\pi) : F(f) \equiv 0 \ \mathrm{mod} \ \pi \} - 1 = \left( \frac{-D}{\pi} \right)
\end{equation*} 
and define the singular series
\begin{equation} \label{SingSerDef}
\mathfrak{S}_q(F) = \prod_{\pi} \left( 1 - \left(1 + \chi_F(\pi) \right)|\pi|^{-1} \right) \left(1 -|\pi|^{-1} \right)^{-1}.
\end{equation}
Define also the $L$-function of $\chi_F$ (in the variable $t = q^{-s}$) to be
\begin{equation}
L(t; \chi_F) = \prod_{\pi} \frac{1}{1 - \chi_F(\pi)t^{\deg(\pi)}}.
\end{equation}

\end{notation}

\begin{prop} \label{SingularSeriesProp}

For a positive integer $n$ we have
\begin{equation} \label{SingularSeriesEq}
\sum_{k=1}^n kq^{-k}\sum_{\substack{A \in \mathcal{M}_k}} \mu(A) \rho(A;F) = 
-\mathfrak{S}_q(F) + q^{-\frac{n}{2} + o(n)}, \quad n \to \infty,
\end{equation}
as soon as $\deg_u(F) \ll n$.
\end{prop}

\begin{proof}
We define
\begin{equation}
Z(t) = t\frac{d}{dt} \sum_{k=0}^\infty t^k \sum_{A \in \mathcal{M}_k} \mu(A)\rho(A; F)  
\end{equation}
and use the multiplicativity in $A$ of $\mu(A)$ and $\rho(A;F)$ to write the above as an Euler product, getting
\begin{equation}
Z(t) = 
t \frac{d}{dt} \prod_{\pi} \left( 1 - \left(1 + \chi_F(\pi) \right)t^{\deg(\pi)}\right) 
= t\frac{d}{dt} \left( (1 - qt)G(t) \right)
\end{equation}
where
\begin{equation} \label{DefGEq}
G(t) = \prod_{\pi} \left( 1 - \left(1 + \chi_F(\pi)  \right)t^{\deg(\pi)}\right) \left(1 - t^{\deg(\pi)}\right)^{-1}.
\end{equation}
An alternative expression for $G(t)$ is
\begin{equation}\label{G-L-equation}
G(t) = L(t; \chi_F)^{-1} \prod_{\pi} 
\left( 1 - \frac{ t^{2 \deg (\pi)} \chi_F( \pi) } { (1 - t^{\deg(\pi)}) (1- \chi_F( \pi)  t^{ \deg (\pi) } ) } \right).
\end{equation}

By the derivative product rule and \cref{SingSerDef}, we have
\begin{equation}
Z \left( \frac{1}{q} \right) = -G\left( \frac{1}{q} \right) = - \mathfrak{S}_q(F).
\end{equation}
By Cauchy's differentiation formula, the left hand side of \cref{SingularSeriesEq} differs from the above by $\ll$ 
\begin{equation} \label{CauchyEq} 
\sum_{k > n} kq^{-k} \ \left| \ \oint\limits_{|t| = r} \frac{G(t)}{t^{k+1}} \ \right|, 
\quad r = q^{-\frac{1}{2} - \epsilon} < \frac{1}{\sqrt{q}}
\end{equation} 
where we take $\epsilon = \epsilon(n) > 0$ to satisfy
\begin{equation} \label{epsilonDefEq}
\epsilon = o(1), \quad \epsilon^{-1} = o(\log n).
\end{equation}

To bound the integral in \cref{CauchyEq}, we prove a pointwise bound on $G$.  
We first handle the case where $D$ is not a constant (an element of $\F_q^\times$) times a square in $\F_q[u]$. 
To do that, (assuming none of the factors in \cref{DefGEq} is zero) we write
\begin{equation}
\log |G(t)| = \sum_\pi \log \left|  1 -  \frac{\chi_F(\pi)t^{\deg (\pi)}}{1 - t^{\deg (\pi)}}  \right| 
\end{equation}
and using the bound $\log |1- z| = \frac{1}{2}\log |1-z|^2 \ll \log(1 - z - \overline{z} + |z|^2)$ get
\begin{equation*}
\sum_\pi \log \left( 1 + \frac{(\chi_F^2(\pi) + 2\chi_F(\pi)) |t|^{2\deg(\pi)} -\chi_F(\pi)\left(t^{\deg(\pi)} + \overline{t^{\deg (\pi)}}\right)}{|1 - t^{\deg(\pi)}|^2} \right).
\end{equation*}

Since $\log(1+x) \le x$ for any real $x > -1$, and $|1 - t^{\deg(\pi)}|^{-2} \ll 1$, the above is $\ll$
\begin{equation}
\sum_\pi r^{2\deg(\pi)}+ \left| \sum_\pi \frac{\chi_F(\pi)t^{\deg(\pi)}}{|1 - t^{\deg(\pi)}|^2} \right| 
\end{equation}
so summing separately over each degree we get at most
\begin{equation} \label{TwoCaseCharSum}
\sum_{\ell=1}^\infty q^\ell r^{2\ell} + 
\left| \sum_{\ell = 1}^{\infty} \sum_{\deg(\pi) = \ell} \frac{\chi_F(\pi)t^{\ell}}{|1 - t^{\ell}|^2} \right| = 
\frac{1}{1 - qr^2} + 
\left| \sum_{\ell = 1}^{\infty} \sum_{\deg(\pi) = \ell} \frac{\chi_F(\pi)t^{\ell}}{|1 - t^{\ell}|^2} \right|.
\end{equation}

%

Using the triangle inequality, and the bound $|1 - t^{\ell}|^{-2} \ll 1$, we arrive at
\begin{equation}
\sum_{\ell=1}^\infty q^\ell r^{2\ell} + 
\sum_{\ell = 1}^{\infty}  r^{\ell} \left|  \sum_{\deg(\pi) = \ell}  \chi_F(\pi) \right|.
\end{equation}
For the second sum we use a trivial bound for $\ell \leq 2 \log_q \deg(D)$,
and invoke Weil's Riemann Hypothesis (see \cite[(2.5)]{Rud}) for all other $\ell$ to get $\ll$
\begin{equation}
\frac{1}{1 - qr^2} + \sum_{\ell \leq 2 \log_q \deg(D) } r^\ell q^\ell  + \sum_{\ell > 2\log_q \deg(D)} \deg(D)r^\ell q^{\frac{\ell}{2}}. 
\end{equation}
Evaluating the geometric series, and using the bound $(1- q^{-2\epsilon})^{-1} \ll \epsilon^{-1}$
we finally get 
\begin{equation} \label{PwiseBoundEq}
\log |G(t)| \ll \epsilon^{-1} \deg(D)^{1 -2\epsilon}.
\end{equation}

It follows from our assumption that $\deg(D) \ll n$, \cref{CauchyEq} and \cref{PwiseBoundEq} that our error term is 
\begin{equation}
e^{O(\deg(D)^{1 - 2\epsilon} \epsilon^{-1})} \sum_{k > n} kq^{-k}r^{-k} \ll 
q^{O(n^{1 - 2\epsilon} \epsilon^{-1})} nq^{-\frac{n}{2} + \epsilon n}.
\end{equation}
In view of  \cref{epsilonDefEq} the above is $\ll q^{-\frac{n}{2} + o(n)}$ as required.

Now we handle the case where $D = \lambda D_0^2$, with $\lambda \in \F_q$ and $D_0 \in \F_q[u]$. 
Since $F$ is irreducible by assumption, we get that $-\lambda \in \F_q^\times \setminus {\F_q^\times}^2$ and $D_0 \neq 0$ so
\begin{equation}
L (t, \chi_F) = \frac{1}{ 1 + q t } \prod_{ \pi \mid D}  (1- (-t )^{\deg (\pi) }).
\end{equation}
Therefore, for any $t \in \mathbb{C}$ with $|t| \leq q^{-\frac{1}{2}}$ we have
\begin{equation} 
|L (t, \chi_F)|^{-1} = 
(1 +qt) \prod_{ \pi \mid D}  |1- (-t )^{\deg (\pi)} |^{-1}   \ll \prod_{\pi \mid D} \frac{1}{1 - |\pi|^{-\frac{1}{2}}}
\ll |D|^{\epsilon}.
\end{equation}  

In order to obtain a pointwise bound for $G$ in this case,
we shall bound the Euler product in \eqref{G-L-equation}. 
Setting $r = |t|$ as in \cref{CauchyEq}, we have

\begin{equation*} 
\begin{split}
&\prod_{\pi} \left| 1 - \frac{ t^{2 \deg (\pi)} \chi_F( \pi) } { (1 - t^{\deg(\pi)}) (1- \chi_F( \pi)  t^{ \deg (\pi) } ) } \right| \leq \prod_{\pi}\left ( 1 + \frac{ r^{ 2 \deg (\pi) }  }{ (1- r^{\deg (\pi)} )^2} \right) = \\ 
&\prod_{\pi}\left (1+ \frac{ 2 r^{ 3\deg (\pi)} }{ 1 - r^{ \deg (\pi)}  } \right) \frac{ 1}{1 - r^{ 2\deg (\pi)} }= 
\frac{1}{1 - q r^2 } \prod_{\pi} \left(1+ \frac{ 2 r^{ 3\deg (\pi)} }{ 1 - r^{ \deg (\pi)}  } \right) \end{split}  \end{equation*} 
in which the final Euler product converges for $r < q^{-\frac{1}{3}}$ and is uniformly bounded for $r \leq q^{-1/2}$, 
so \cref{G-L-equation} defines a holomorphic function in this disc.

By \cref{CauchyEq} we have 
\begin{equation}
\frac{1}{ 1 - q r^{2}} = \frac{1}{ 1 - q^{ - 2\epsilon} } \ll \epsilon^{-1}
\end{equation}
so the error term is bounded by
\begin{equation}
\sum_{k > n } k q^{-k} \frac{ \epsilon^{-1}}{ r^{k+1} } = 
\epsilon^{-1}   \sum_{k>n} k q^{ (\epsilon -\frac{1}{2}) (k+1) -1 }  \ll \epsilon^{-1}  q^{ - \frac{n}{2} + \epsilon n} 
\end{equation} 
which is $\ll q^{ -\frac{n}{2} + o(n) }$ in view of  \cref{epsilonDefEq}. 
\end{proof}

\section{Quadratic forms}

We establish here analogs over $\F_q[u]$ of several facts mentioned in \cite{Hoo}.

\begin{prop} \label{UniqueMatrixProp}

For every (binary) quadratic form 
\begin{equation}
Q(X,Y) = aX^2 + bXY + cY^2
\end{equation}
over $\F_q[u]$ there exists a unique symmetric $2 \times 2$ matrix $K$ over $\F_q[u]$ with
\begin{equation} \label{MatrixFormCorr}
Q(X,Y) = (X,Y)K(X,Y)^T.
\end{equation}

\end{prop}

\begin{proof}

For the existence of $K$ as above, just note that
\begin{equation*} \label{CorrMat}
(X,Y)
\begin{pmatrix} 
a & \frac{b}{2} \\ \frac{b}{2} & c 
\end{pmatrix}
(X,Y)^T = \left( aX + \frac{bY}{2}, \frac{bX}{2} + cY \right) (X,Y)^T = Q(X,Y).
\end{equation*}

For uniqueness, let $K$ be a symmetric matrix satisfying \cref{MatrixFormCorr}. Then
\begin{equation}
K_{11} = (1,0)K(1,0)^T = Q(1,0), \ K_{22} = (0,1)K(0,1)^T = Q(0,1)
\end{equation}
and since $K$ is symmetric, we get from \cref{MatrixFormCorr} that
\begin{equation*}
\begin{split}
2K_{12} &= K_{12} + K_{21} = (1,1)K(1,1)^T - (1,0)K(1,0)^T - (0,1)K(0,1)^T \\
&= Q(1,1) - Q(1,0) - Q(0,1)
\end{split}
\end{equation*}
so $K$ is indeed uniquely determined by $Q$.
\end{proof}

\begin{defi} \label{CorrMatrixDef}

Keep the notation of the above proposition. 
We say that the symmetric $2 \times 2$ matrix $K$ is the corresponding matrix to the quadratic form $Q$,
and define the discriminant $D$ of $Q$ to be the determinant
\begin{equation} \label{DiscrimDefEq}
D =  \det(K) = ac - \frac{b^2}{4}. 
\end{equation} 
In case the polynomial 
\begin{equation} \label{TheQuadraticPolEq}
F(T) = T^2 + D \in \F_q[u][T]
\end{equation}
is reducible over $\F_q[u]$,
that is, negative $D$ is a square in $\F_q[u]$,
we say that $Q$ is degenerate, and otherwise we say that it is nondegenerate.

\end{defi}

\begin{remark} \label{NonZeroARem}

For a nondegenerate form $Q(X,Y) = aX^2 + bXY + cY^2$, the polynomial $a = Q(1,0)$ is nonzero.
For if $a = 0$ then
\begin{equation}
-D = \frac{b^2}{4} - ac = \left( \frac{b}{2} \right)^2 - 0 \cdot c = \left( \frac{b}{2} \right)^2
\end{equation}
contrary to our assumption that $Q$ is nondegenerate.

\end{remark}

\begin{defi} \label{QuadFormNot}


Let the group $\mathrm{SL}_2(\mathbb F_q[u])$ act from the right on row vectors in $\F_q[u]^2$ via the dual of the usual action by multiplication.
This means that for a matrix of polynomials
\begin{equation}
M \in \mathrm{SL}_2(\mathbb F_q[u]), \ M = 
\begin{pmatrix} 
M_{11} & M_{12} \\ M_{21} & M_{22} 
\end{pmatrix}, \
M_{11}M_{22} - M_{12}M_{21} = 1,
\end{equation}
and $(x,y) \in \F_q[u]^2$, the action is given by
\begin{equation} \label{ActionOnVectors}
\begin{split}
(x,y) \star M &= (x,y) M^{-T} = (x,y)
\begin{pmatrix} 
M_{22} & -M_{21} \\ -M_{12} & M_{11} 
\end{pmatrix} \\
&= (M_{22}x - M_{12}y, -M_{21}x + M_{11}y).
\end{split}
\end{equation}

\end{defi}

It is straightforward to check that the stabilizer of the vector $(1, 0)$ is
\begin{equation} \label{UnipotentEq}
\{M \in \mathrm{SL}_2(\mathbb F_q[u]) : (1,0)M^{-T} = (1,0) \} = \left\{
\begin{pmatrix} 
1 & g \\ 0 & 1 
\end{pmatrix}
: g \in \F_q[u] \right\}. 
\end{equation}

\begin{notation} \label{BarredNotation}

A vector $(x,y) \in \F_q[u]^2$ is called primitive if $\gcd(x,y) = 1$, or equivalently,
if the ideal of $\F_q[u]$ generated by $x$ and $y$ contains $1$.
For such a vector, we denote by $\overline{x} \in \F_q[u]$ the polynomial of least degree for which
\begin{equation}
\overline{x}x  = 1 \mod y,
\end{equation}
and let $y_x \in \F_q[u]$ be the polynomial of least degree satisfying
\begin{equation} \label{invxmody}
\overline{x}x -  y_xy = 1.
\end{equation}
Put
\begin{equation} \label{BarredMatrix}
M_{(x,y)} = \begin{pmatrix} 
x & y_x \\ y & \overline{x} 
\end{pmatrix} \in \mathrm{SL}_2(\mathbb F_q[u])
\end{equation}
and note that 
\begin{equation} \label{MovingToOneZeroEq}
(1,0) = (x,y)M_{(x,y)}^{-T}.
\end{equation}
In particular, the primitive vectors form an orbit under the action of $\mathrm{SL}_2(\mathbb F_q[u])$.

\end{notation}

\begin{defi} \label{ActionOnFormsDef}

The group $\mathrm{SL}_2(\mathbb F_q[u])$ also acts from the right on quadratic forms by
\begin{equation} \label{ActionOnForms}
Q(X,Y) \star M = Q((X,Y)M^T) = Q( M_{11}X +M_{12} Y, M_{21} X + M_{22} Y).
\end{equation}
We say that two quadratic forms are equivalent if they belong to the same orbit in this action.

\end{defi}

For instance, if
\begin{equation}
Q(X,Y) = aX^2 + bXY + cY^2, \quad M = 
\begin{pmatrix} 
1 & g \\ 0 & 1 
\end{pmatrix}
\in \mathrm{SL}_2(\mathbb F_q[u])
\end{equation}
then we have
\begin{equation} \label{UnipExpAction}
\begin{split}
Q((X,Y)M^T) &= Q(X + gY, Y) = a(X + gY)^2 + b(X + gY)Y + cY^2 \\
&= aX^2 + (b + 2ag)XY + (ag^2 + bg + c)Y^2.
\end{split}
\end{equation}


We show that equivalent quadratic forms have the same discriminant.

\begin{prop} \label{DiscPreserver}

Let $Q$ be a quadratic form,
and let $M \in\mathrm{SL}_2(\mathbb F_q[u])$.
Then the discriminant of the quadratic form
$
Q'(X,Y) = Q((X,Y) M^T)
$
equals the discriminant of $Q$.

\end{prop}

\begin{proof}

For the symmetric $2 \times 2$ matrix $K$ corresponding to $Q$ in \cref{CorrMatrixDef} we have
\begin{equation}
Q'(X,Y) = Q((X,Y)M^T) = (X,Y)M^T K M (X, Y)^T
\end{equation}
so $Q'$ and $M^T K M$ satisfy the assumptions of \cref{UniqueMatrixProp} as the latter matrix is symmetric. 
We conclude that $M^T K M$ is the matrix corresponding to $Q'$, so the discriminant of $Q'$ is
\begin{equation}
\det(M^T K M) = \det(M^T) \det(K) \det(M) = 1 \cdot \det(K) \cdot 1 = \det(K)
\end{equation}
which is the discriminant of $Q$.
\end{proof}

%
%
%
%
%
%
%
%
%
%
%
%
%
%
%

\begin{defi} \label{ActionsonRepsDef}

A representation of a polynomial $A \in \F_q[u]$ by a quadratic form is an ordered pair $(Q,(x,y)) $ where $(x,y) \in \F_q[u]^2$, 
$Q$ is a quadratic form, and $Q(x,y)=A$. 
The representation is said to be primitive if $(x,y)$ is primitive and $Q$ is nondegenerate.  
One checks that the actions defined in \cref{QuadFormNot} and \cref{ActionOnFormsDef} combine to a coordinatewise action of $\mathrm{SL}_2(\mathbb F_q[u])$ from the right on (primitive) representations of $A$.
We call two representations equivalent if they belong to the same orbit in this action.


\end{defi}

Next we show that the action of $\mathrm{SL}_2(\mathbb F_q[u])$ on primitive representations is free,
namely that the stabilizer of any primitive representation is trivial.

\begin{prop} \label{FreeActionProp}

Let $Q$ be a nondegenerate quadratic form over $\F_q[u]$, and let $v \in \F_q[u]^2$ be a primitive vector.
Then the only matrix $M \in \mathrm{SL}_2(\mathbb F_q[u])$ that satisfies
\begin{equation} \label{StabilizingRepAssum}
Q((X,Y)M^T) = Q(X,Y), \quad vM^{-T} = v,
\end{equation}
is the identity matrix.

\end{prop}

\begin{proof}

Write $Q(X,Y) = aX^2 + bXY + cY^2$, and suppose first that $v = (1,0)$.
In this case, we know that $(1,0)M^{-T} = (1,0)$ so by \cref{UnipotentEq}, there exists a polynomial $g \in \F_q[u]$ such that
\begin{equation}
M = 
\begin{pmatrix} 
1 & g \\ 0 & 1 
\end{pmatrix}.
\end{equation}
Using our assumption that $M$ stabilizes $Q$ and \cref{UnipExpAction} we get
\begin{equation}
\begin{split}
aX^2 + bXY + cY^2 &= Q(X,Y) = Q((X,Y)M^T) \\
&=  aX^2 + (b + 2ag)XY + (ag^2 + bg + c)Y^2.
\end{split}
\end{equation}

Since $Q$ is nondegenerate, we get from \cref{CorrMatrixDef} that $4ac - b^2 \neq 0$.
It follows that either $a \neq 0$ or $b \neq 0$ (or both).
In case $a \neq 0$, from comparing the coefficients of $XY$ above, we get that $2ag = 0$ so $g = 0$.
If $a = 0$ then $b \neq 0$, so from equating the coefficients of $Y^2$ above, we deduce that $bg = 0$ hence $g = 0$.
We have thus shown that $M$ is indeed the identity matrix in case $v = (1,0)$.

Assume now that $v$ is an arbitrary primitive vector.
In \cref{BarredMatrix} we have written a matrix $M_v \in \mathrm{SL}_2(\mathbb F_q[u])$ such that $ v = (1,0)M_v^{T}$.
We set
\begin{equation}
Q'(X,Y) = Q((X,Y) M_v^T),
\end{equation}
infer from \cref{DiscPreserver} that $Q'$ is a nondegenerate quadratic form,
and note that $Q(X,Y) = Q'((X,Y)M_v^{-T})$.
Now if $M \in \mathrm{SL}_2(\mathbb F_q[u])$ is a matrix satisfying \cref{StabilizingRepAssum} then 
\begin{equation*}
Q'((X,Y) M^T M_v^{-T}) = Q'((X,Y)M_v^{-T}), \quad (1,0)M_v^T M^{-T} = (1,0)M_v^T
\end{equation*}
so
\begin{equation*}
Q'((X,Y) (M_v^{-1} M M_v)^T ) = Q'((X,Y)), \quad (1,0) \left( M_v^{-1} M M_v \right)^{-T} = (1,0).
\end{equation*}
From the previous special case where the primitive vector was $(1,0)$ we conclude that $M_v^{-1}MM_v$ is the identity, so $M$ is the identity as well.
\end{proof}

\begin{defi} \label{OneZeroAssocSolDef}

If $Q$ is a nondegenerate quadratic form of discriminant $D$ as in \cref{CorrMatrixDef}, and $(Q, (1,0))$ represents a polynomial $A$, 
then we can write 
\begin{equation}
Q(X,Y) = AX^2 + bXY + cY^2
\end{equation}
so from the definition of the discriminant in \cref{DiscrimDefEq} we get
\begin{equation} \label{AssociatorEq}
\left( \frac{b}{2} \right)^2 + D = \frac{b^2}{4} + Ac - \frac{b^2}{4} = Ac \equiv 0 \mod A.
\end{equation}
We say that $f = \frac{b}{2}$ is the solution of the equation
\begin{equation} \label{QuadCongEq}
F(T) = T^2 + D \equiv 0 \mod A
\end{equation}
associated to the primitive representation $(Q, (1,0))$ of $A$. 
Note that $A \neq 0$ by \cref{NonZeroARem}.

\end{defi}

\begin{prop} \label{UnipotentEquivalence}

Let $(Q, (1,0))$ and $(Q', (1,0))$ be representations of a polynomial $A \in \F_q[u]$ by nondegenerate quadratic forms of discriminant $D$. Then the solutions to \cref{QuadCongEq} associated to the representations $(Q, (1,0))$ and $(Q', (1,0))$ coincide if and only if these representations are equivalent.

\end{prop}

\begin{proof}

Suppose first that the representations $(Q, (1,0))$ and $(Q', (1,0))$ are equivalent.
By the definitions in \cref{ActionOnVectors} and \cref{ActionOnForms},
equivalence means that there exists a matrix $M \in \mathrm{SL}_2(\mathbb F_q[u])$ such that
\begin{equation} \label{OneZeroFixedByM}
(1,0)M^{-T} = (1,0), \quad Q'(X,Y) = Q((X,Y)M^T).
\end{equation}
We have checked in \cref{UnipotentEq} that the first equality above implies 
\begin{equation}
M = 
\begin{pmatrix} 
1 & g \\ 0 & 1 
\end{pmatrix}
\end{equation}
for some $g \in \F_q[u]$.

Since $Q(1,0) = A$, we can write
\begin{equation}
Q(X,Y) = AX^2 + bXY + cY^2
\end{equation}
and get from the second equality in \cref{OneZeroFixedByM} and \cref{UnipExpAction} that
\begin{equation}
Q'(X,Y) = Q((X,Y)M^T) = AX^2 + (b + 2gA)XY + (Ag^2 + bg + c)Y^2
\end{equation}
so the coefficient of $XY$ in $Q'$ is
\begin{equation}
2gA + b \equiv b \mod A. 
\end{equation}
Hence the solution of \cref{QuadCongEq} associated to $(Q', (1,0))$ via \cref{AssociatorEq} is $f = \frac{b}{2}$,
which is also the solution associated to $(Q, (1,0))$.

Suppose now that the representations $(Q,(1,0))$ and $(Q', (1,0))$ of $A$ give rise to the same solution of \cref{QuadCongEq}.
We can therefore write
\begin{equation} \label{TwoFormsOneEq}
Q(X,Y) = AX^2 + bXY + cY^2, \quad Q'(X,Y) = AX^2 + b'XY + c'Y^2
\end{equation}
with $b'$ congruent to $b$ mod $A$.
Since $Q$ and $Q'$ are of discriminant $D$,
using \cref{CorrMatrixDef} we see that $4Ac = b^2 + 4D$ and that $4Ac' = b'^2 + 4D$.
By \cref{NonZeroARem}, $A$ is nonzero so we can rewrite our forms as
\begin{equation} \label{TwoFormsTakeTwo}
Q(X,Y) = AX^2 + bXY + \frac{b^2 + 4D}{4A}Y^2, \ Q'(X,Y) = AX^2 + b'XY + \frac{b'^2 + 4D}{4A}Y^2.
\end{equation}

In order to exhibit the equivalence of our representations, we take
\begin{equation}
M = 
\begin{pmatrix} 
1 & \frac{b' - b}{2A} \\ 0 & 1 
\end{pmatrix}
\in \mathrm{SL}_2(\mathbb F_q[u])
\end{equation}
and note that $(1,0) = (1,0)M^{-T}$.
Using \cref{TwoFormsTakeTwo} and \cref{UnipExpAction} we get
\begin{equation*}
\begin{split}
&Q((X,Y)M^T) = \\
&AX^2 + \left( b + 2A \cdot \frac{b'-b}{2A} \right)XY + \left(A \left(\frac{b'-b}{2A} \right)^2 + b \cdot \frac{b'-b}{2A} + \frac{b^2 + 4D}{4A} \right)Y^2 = \\
&AX^2 + b'XY + \left( \frac{b'^2-2bb' + b^2}{4A} +  \frac{2bb'-2b^2}{4A} + \frac{b^2 + 4D}{4A} \right)Y^2 = \\
&AX^2 + b'XY + \left( \frac{b'^2 + 4D}{4A} \right)Y^2 = Q'(X,Y)
\end{split}
\end{equation*}
so our representations are indeed equivalents as \cref{OneZeroFixedByM} holds.
\end{proof}

%

\begin{defi} \label{GenAssocDef}

To a primitive representation $(Q,(x,y))$ of a polynomial $A \in \F_q[u]$ by a quadratic form
\begin{equation}
Q(X,Y) = aX^2 + bXY + cY^2
\end{equation}
of discriminant $D$,
using \cref{BarredNotation} we associate the solution
\begin{equation} \label{GeneralAsscoiation}
f = a y_x x + \frac{b}{2}(\overline{x}x + y_x y) + c \overline{x}y
\end{equation}
of the congruence $T^2 \equiv -D \mod A$ from \cref{QuadCongEq}.
Note that this agrees with our previous definition $f = \frac{b}{2}$ for the case $(x,y) = (1,0)$.

Now we check that \cref{GeneralAsscoiation} is indeed a solution for \cref{QuadCongEq}.
We use \cref{BarredNotation} to associate with $(Q,(x,y))$ the quadratic form
\begin{equation}
\begin{split}
Q_{x,y}(X,Y) &= Q \left((X,Y) M^{T}_{(x,y)} \right) = Q(xX + y_xY, yX + \overline{x}Y) \\ 
&= a(xX + y_xY)^2 + b(xX + y_xY)(yX + \overline{x}Y) + c(yX + \overline{x}Y)^2
\end{split}
\end{equation}
and use \cref{MovingToOneZeroEq} to conclude that $(Q, (x,y))$ is equivalent to $(Q_{x,y}, (1,0))$.
One readily checks that dividing the coefficient of $XY$ above by $2$ gives the right hand side of \cref{GeneralAsscoiation},
so the latter is indeed a solution of \cref{QuadCongEq}. 
In other words, the solution of \cref{QuadCongEq} associated to $(Q,(x,y))$ is the solution of \cref{QuadCongEq} associated to the equivalent representation $(Q_{x,y}, (1,0))$ in \cref{OneZeroAssocSolDef}.
At last note that by \cref{NonZeroARem} we have
\begin{equation} \label{RepValueNonZeroEq}
A = Q(x,y) = Q_{x,y}(1,0) \neq 0.
\end{equation}

\end{defi}

\begin{cor} \label{GenAssocSolsEquivaCor}

Primitive representations $(Q,(x,y))$ and $(Q', (x',y'))$ of a polynomial $A$ by quadratic forms of discriminant $D$
are equivalent if and only if their associated solutions to the equation $T^2 + D \equiv 0 \mod A$ coincide. 

\end{cor}

\begin{proof}

The representations $(Q,(x,y))$ and $(Q', (x',y'))$ are equivalent if and only if $(Q_{x,y},(1,0))$ and $(Q'_{x',y'}, (1,0))$ are equivalent.
From \cref{UnipotentEquivalence} applied to $Q_{x,y}$ and $Q'_{x',y'}$,
we get that $(Q_{x,y},(1,0))$ and $(Q'_{x',y'}, (1,0))$ are equivalent if and only if they give rise to the same solution
for the congruence $T^2 + D \equiv 0 \mod A$.
Our corollary now follows because the solution associated to $(Q,(x,y))$ is the one associated to $(Q_{x,y},(1,0))$,
and the solution associated to $(Q',(x',y'))$ is also associated to $(Q'_{x',y'},(1,0))$.
\end{proof}

\begin{prop} \label{RepCorrToGivenSolProp}

Let $A \in \F_q[u]$ be a nonzero polynomial, and let $D \in \F_q[u]$ be a polynomial with $-D$ not a square.
Then every solution $f$ of
\begin{equation}
T^2 + D \equiv 0 \mod A
\end{equation}
arises from a primitive representation of $A$ by a quadratic form of discriminant $D$.

\end{prop}

\begin{proof}

Consider the quadratic form
\begin{equation}
Q(X,Y) = AX^2 + 2fXY + \frac{f^2 + D}{A}Y^2 \in \F_q[u][X,Y]
\end{equation}
that satisfies $Q(1,0) = A$. By \cref{CorrMatrixDef}, the discriminant of $Q$ is
\begin{equation}
A \cdot \frac{f^2 + D}{A} - \frac{(2f)^2}{4}   = f^2 + D - f^2 = D.
\end{equation}
By \cref{AssociatorEq}, $f = \frac{2f}{2}$ is associated to the primitive representation $(Q, (1,0))$ of $A$, as required.
\end{proof}

\begin{prop} \label{ExponentiationChangeProp}

As in \cref{GenAssocDef}, let $(Q, (x,y))$ be a primitive representation of a polynomial $A \in \F_q[u]$ by a quadratic form 
\begin{equation}
Q(X,Y) = aX^2 + bXY + cY^2
\end{equation}
of discriminant $D$, and let 
\begin{equation} \label{ExpsAssocSol}
f = a y_x x + \frac{b}{2}(\overline{x}x + y_x y) + c \overline{x}y
\end{equation}
be the associated solution to the equation $T^2 + D \equiv 0 \mod A$. 
Suppose that $y \neq 0$. Then
\begin{equation}
e \left( \frac{hf}{A} \right) = e \left( \frac{h\overline{x}}{y} \right)
\end{equation}
for any polynomial $h \in \F_q[u]$ satisfying
\begin{equation} \label{FirstDeghIneq}
\deg(h)  < \deg(A) - \deg(b) - 1
\end{equation}
and
\begin{equation} \label{SecondDeghIneq}
\deg(h) < \deg(A) + \deg(y) - \deg(a) - \deg(x) - 1.
\end{equation}

\end{prop}

\begin{proof}

We have an equality of rational functions
\begin{equation} \label{RatFuncEq}
\begin{split}
&\frac{c \overline{x} y^2 + ax(\overline{x} x -1) + \frac{b}{2}y(2\overline{x} x-1)}{yA} + \frac{ax+\frac{b}{2}y}{yA} = \\ 
&\frac{\overline{x}(ax^2 + bxy + cy^2)}{yA} = \frac{\overline{x}Q(x,y)}{yA} = \frac{\overline{x} A}{yA} = \frac{\overline{x}}{y}. 
\end{split}
\end{equation}
By \cref{invxmody}, we have $\overline{x} x - 1 = y_x y$ and thus $2\overline{x}x - 1 = \overline{x} x + y_x y$,
so plugging these in the first term of \cref{RatFuncEq} we get
\begin{equation} 
\frac{c \overline{x} y^2 + axy_x y + \frac{b}{2}y(\overline{x} x + y_x y)}{yA} = \frac{f}{A} 
\end{equation}
from the definition of the associated solution $f$ in \cref{ExpsAssocSol}.
We conclude from the above and \cref{RatFuncEq} that
\begin{equation}
\frac{f}{A} + \frac{ax + \frac{b}{2}y}{yA} = \frac{\overline{x}}{y}.
\end{equation}

Multiplying the above by a polynomial $h$ and exponentiating, we obtain
\begin{equation}
e \left( \frac{h\overline{x}}{y} \right) = e \left( \frac{hf}{A} \right) e \left( \frac{hax + \frac{b}{2}hy}{yA} \right).
\end{equation}
The second factor in the right hand side above equals $1$ as soon as 
\begin{equation}
\deg \left( hax + \frac{b}{2}hy \right) < \deg(yA) - 1.
\end{equation}
The latter holds in case the two inequalities
\begin{equation*}
\deg(h) < \deg(A) + \deg(y) - \deg(a) - \deg(x) - 1, \ \deg(h)  < \deg(A) - \deg(b) - 1
\end{equation*} 
are satisfied.
\end{proof}

\begin{defi}

Let $D \in \F_q[u]$ be a polynomial for which 
\begin{equation}
F(T) = T^2 + D \in \F_q(u)[T]
\end{equation}
is irreducible over $\F_q(u)$, that is negative $D$ is not a square of a polynomial.
We say that $D$ is indefinite if the infinite place of $\F_q(u)$ splits in the splitting field of $F$.
By \cite[Proposition 14.6]{Ros}, this is equivalent to the degree of $D$ being even and the leading coefficient of $-D$ being a square in $\F_q^\times$. 
Otherwise (if the infinite place of $\F_q(u)$ is ramified or inert in the splitting field of $F$) we say that $D$ is definite. 
A nondegenerate quadratic form is called definite (respectively, indefinite) if its discriminant is definite (respectively, indefinite).


\end{defi}

\subsection{Definite quadratic forms}

\begin{defi} \label{StandardDefiniteDef}

We say that a definite quadratic form 
\begin{equation}
Q(X,Y) = aX^2 + bXY + cY^2 \in \F_q[u][X,Y]
\end{equation}
is standard if $\deg c \geq \deg a > \deg b$. 

\end{defi}

\begin{remark} \label{LowerBoundDegCRem}

For the discriminant $D$ of $Q$ we have
\begin{equation}
\deg(a) + \deg(c) = \deg(D).
\end{equation}
Indeed, otherwise
$
\deg(ac) = \deg(a) + \deg(c) \neq \deg(D),
$
and thus 
\begin{equation}
\begin{split}
2\deg(b) &= \deg(b^2) = \deg(4ac - 4D) = \max\{\deg(ac), \deg(D)\} \\
&\geq \deg(ac) = \deg(a) + \deg(c) > \deg(b) + \deg(b) = 2\deg(b),
\end{split}
\end{equation} 
a contradiction.
We infer that
\begin{equation} \label{DegaDefiniteUpperBound}
\deg(a) = \frac{\deg(a) + \deg(a)}{2} \leq \frac{\deg(a) + \deg(c)}{2} = \frac{\deg(D)}{2}.
\end{equation}

\end{remark}

\begin{prop} \label{DefiniteDeg} 
For a standard definite quadratic form as above, and $x,y \in \F_q[u]$ we have
\begin{equation*}
\deg(Q(x,y)) = \max ( \deg (a) + 2 \deg (x), \deg (c) + 2 \deg (y)) > \deg(bxy).
\end{equation*}
\end{prop}

\begin{proof} 

We have 
\begin{equation}
\deg( ax^2 ) = \deg (a) + 2\deg (x), \quad \deg(cy^2) = \deg (c) + 2 \deg (y), 
\end{equation}
so since $\deg(c), \deg(a) > \deg(b)$ we get
\begin{equation} \label{Degbxy}
\begin{split} 
\deg ( bxy) &= \deg (b)+ \deg (x) + \deg (y) < \frac{\deg (a) + \deg (c)}{2} +\deg (x) + \deg (y) \\
&= \frac{ \deg (ax^2) + \deg( cy^2)}{2} \leq \max ( \deg (ax^2), \deg (cy^2)).
\end{split}
\end{equation}

Suppose toward a contradiction that the leading terms of $ax^2$ and $-cy^2$ are equal.
Then the leading terms of $a$ and $-c$ are equal up to multiplication by the square of a monomial, 
so the leading term of $-ac$ is a square. Since $\deg (ac) > \deg(b^2)$, we conclude that the leading term of the negated discriminant
\begin{equation}
-D = \frac{b^2}{4} - ac
\end{equation} 
of $Q$ is a square, which contradicts the assumption that $Q$ is definite.  

We infer that the leading terms of $ax^2 $ and $cy^2$ do not cancel each other out, 
so \cref{Degbxy} implies that $\deg (bxy) < \deg ( ax^2 + cy^2)$ and thus the desired statement. 
\end{proof}

\begin{defi} \label{ShortVecDef}

We say that a primitive vector $(x,y) \in \F_q[u]^2$ is a short vector of a definite quadratic form $Q$ if 
\begin{equation} \label{DefShortVecEq}
\deg (Q(x,y)) = \min \{ \deg (Q(v)) \mid v \in \F_q[u]^2, \ v \ \text{is primitive} \}.
\end{equation}
By \cref{RepValueNonZeroEq}, the degrees of polynomials primitively represented by $Q$ form a (nonempty) subset of the nonnegative integers.
Such a subset necessarily has a least element, so $Q$ admits short vectors.

\end{defi}

\begin{prop} \label{ShortVectorsStandardFormProp}

Let $Q(X,Y) = aX^2 + bXY + cY^2$ be a standard definite quadratic form over $\F_q[u]$. 
Then the short vectors of $Q$ are
\begin{equation}
\begin{cases}
\F_q^\times \times \{0\} & \deg(a) < \deg(c) \\
\F_q \times \F_q \setminus \{(0,0)\} & \deg(a) = \deg(c).
\end{cases}
\end{equation}


\end{prop}

\begin{proof} 

We first determine the minimum in the right hand side of \cref{DefShortVecEq}.
Invoke \cref{DefiniteDeg}, and note that the minimum value of 
\begin{equation}
\deg(Q(x,y)) = \max( \deg(a)+ 2\deg(x), \deg(c) + 2\deg(y))
\end{equation}
over all primitive vectors $(x,y) \in \F_q[u]^2$ is attained (at least) whenever $\deg(x)$ and $\deg(y)$ are as small as possible,
subject to $gcd(x,y) = 1$.
That is, the minimum occurs (at least) in case
\begin{equation}
\{\deg(x), \deg(y)\} = \{0, -\infty\},
\end{equation}
so this minimum is $\min (\deg(a),\deg(c))$ which is $\deg(a)$ since the form $Q$ is standard. 

To determine all short vectors, let $(x,y) \in \F_q[u]^2$ be a primitive vector. If $\deg(c) > \deg(a)$, then we have
\begin{equation} \label{MinimalityCondEq}
\max( \deg(a)+ 2\deg(x), \deg(c) + 2\deg(y)) = \deg(a)
\end{equation}
if and only if $y=0$ and $x$ is a nonzero constant polynomial.
Otherwise, since $Q$ is standard we have $\deg(c) = \deg(a)$ so \cref{MinimalityCondEq} is satisfied if and only if $x$ and $y$ are both constant polynomials (but not both zero).
\end{proof}

The following proposition shows that the action of $\mathrm{SL}_2(\mathbb F_q[u])$ on representations restricts to an action on representations by short vectors of definite quadratic forms.

\begin{prop} \label{ShortVecAction}

Let $v$ be a short vector of a definite quadratic form $Q$ over $\F_q[u]$,
and let $M \in \mathrm{SL}_2(\mathbb F_q[u])$.
Then the primitive vector $v' = vM^{-T}$ is a short vector of the definite quadratic form $Q'(X,Y) = Q((X,Y) M^T)$. 

\end{prop}

\begin{proof}

The form $Q'$ is definite since its discriminant is the discriminant of $Q$ by \cref{DiscPreserver}.
Since $Q(X,Y) = Q'((X,Y)M^{-T})$, and $v$ is a short vector of $Q$,
for any primitive vector $v_0 \in \F_q[u]^2$ we have
\begin{equation}
\begin{split}
\deg(Q'(v_0)) &= \deg(Q(v_0M^{T}))
\geq \deg(Q(v)) \\ &= \deg(Q'(vM^{-T})) = \deg(Q'(v'))
\end{split}
\end{equation}
so $v'$ is indeed a short vector of $Q'$.
\end{proof}

\begin{defi}

Let $Q$ be a definite quadratic form over $\F_q[u]$, and let $v$ be a short vector of $Q$.
A standardizing matrix of $Q$ at $v$ is a matrix $M \in \mathrm{SL}_2(\mathbb F_q[u])$ for which $vM^{-T} = (1,0)$ and
$
Q((X,Y) M^T)
$
is a standard definite quadratic form.

\end{defi}

\begin{prop}

There exists a unique standardizing matrix of $Q$ at $v$.
  
\end{prop}

\begin{proof}

We start by proving existence.
By \cref{ShortVecDef}, the vector $v$ is primitive, 
so as in \cref{MovingToOneZeroEq} there exists a matrix $M_{v} \in \mathrm{SL}_2(\mathbb F_q[u])$ with $vM_v^{-T} = (1,0)$.
By \cref{ShortVecAction}, the vector $(1,0)$ is then a short vector of the definite quadratic form 
\begin{equation}
Q'(X,Y) = Q((X,Y) M_v^T) = aX^2 + bXY + cY^2.
\end{equation}

By \cref{NonZeroARem}, $a$ is nonzero so division with remainder provides us with a polynomial $g \in \F_q[u]$ for which $\deg(b - ag) < \deg(a)$.
Setting
\begin{equation}
N = 
\begin{pmatrix}
1 & -\frac{g}{2} \\
0 & 1
\end{pmatrix}
\in \mathrm{SL}_2(\mathbb F_q[u])
\end{equation}
and applying \cref{UnipExpAction}, we get the quadratic form
\begin{equation}
S(X,Y) = Q'((X,Y) N^T) = aX^2 + (b - ag)XY + c_0Y^2
\end{equation}
for some $c_0 \in \F_q[u]$.

Applying \cref{ShortVecAction} again, we find that the vector $(1,0) = (1,0)N^{-T}$ is a short vector of the definite quadratic form $S$.
Therefore
\begin{equation}
\deg(c_0) = \deg(S(0,1)) \geq \deg(S(1,0)) = \deg(a)
\end{equation}
so $S$ is standard by \cref{StandardDefiniteDef}.
The existence part of our proposition then follows by taking $M = M_vN$ as
\begin{equation}
vM^{-T} = v(M_vN)^{-T} = vM_v^{-T}N^{-T} = (1,0)N^{-T} = (1,0)
\end{equation}
and
\begin{equation}
Q((X,Y)M^T) = Q((X,Y) N^T M_v^T) = Q'((X,Y)N^T) = S(X,Y).
\end{equation}

To demonstrate uniqueness, let $M_1, M_2 \in \mathrm{SL}_2(\mathbb F_q[u])$ be standardizing matrices of $Q$ at $v$.
Then we have $vM_1^{-T} = (1,0) = vM_2^{-T}$ so 
\begin{equation} \label{InvCompositionFixingOneZeroEq}
(1,0)(M_1^{-1} M_2)^{-T} = (1,0)
\end{equation}
and the quadratic forms 
\begin{equation} \label{TwoStandardFormEq}
S_1(X,Y) = Q((X,Y) M_1^T) = aX^2 + bXY + cY^2, \ S_2(X,Y) = Q((X,Y) M_2^T)
\end{equation}
are standard definite.
From \cref{InvCompositionFixingOneZeroEq} and \cref{UnipotentEq} we get that 
\begin{equation} \label{M1InverseM2Eq}
M_1^{-1}M_2 = 
\begin{pmatrix}
1 & g \\
0 & 1
\end{pmatrix}
\end{equation}
for some $g \in \F_q[u]$,
so from \cref{TwoStandardFormEq} and \cref{UnipExpAction} we get that
\begin{equation*}
\begin{split}
S_2(X,Y) &= Q((X,Y)M_2^T) = S_1((X,Y) M_2^T M_1^{-T}) 
= S_1((X,Y)  (M_1^{-1}M_2)^T  ) \\ &=
aX^2 + (b + 2ag)XY + (ag^2 + bg + c)Y^2.
\end{split}
\end{equation*}

Suppose toward a contradiction that $g \neq 0$.
As the form $S_1$ is standard definite we know that $\deg(a) > \deg(b)$,
so since $S_2$ is also standard definite, we get from the above that
\begin{equation}
\deg(a) > \deg(b + 2ag) = \deg(2ag) \geq \deg(a)
\end{equation}
which is an absurdity.
We conclude that $g = 0$ and thus $M_1^{-1}M_2$ is the identity in view of \cref{M1InverseM2Eq}, 
so $M_1 = M_2$ as required.
\end{proof}

\begin{defi}

Let $Q$ be a definite quadratic form, let $v$ be a short vector of $Q$, let $M$ be the standardizing matrix of $Q$ at $v$, and let $w \in \F_q[u]^2$ be a primitive vector. 
We say that the standard quadratic form $Q((X,Y) M^T)$ is the standardization of $Q$ at $v$,
and that the representation
\begin{equation}
(Q((X,Y)M^T), wM^{-T})
\end{equation}
is the standardization at $v$ of the representation $(Q,w)$.
A primitive representation $(S,w)$ by a definite quadratic form is called standard if $S$ is standard. 

\end{defi}

\begin{thm} \label{StndrdzThm}

Let $(Q,w)$ be a primitive representation of a polynomial $A$ by a quadratic form $Q$ of definite discriminant $D$.
Then the function that maps a short vector $v$ of $Q$ to the standardization of $(Q,w)$ at $v$ is a bijection between the set of short vectors of $Q$ and the set of those standard representations of $A$ that are equivalent to $(Q,w)$.

\end{thm}

\begin{proof}

To show that our function is injective, let $v_1,v_2$ be short vectors of $Q$, let $M_1,M_2 \in \mathrm{SL}_2(\mathbb F_q[u])$ be the standardizing matrices of $Q$ at $v_1$ and $v_2$ respectively, and suppose that the standardization of $(Q,w)$ at $v_1$ coincides with the standardization of $(Q,w)$ at $v_2$, namely
\begin{equation}
w M_1^{-T} = w M_2^{-T}, \quad Q((X,Y) M_1^T) = Q((X,Y) M_2^T).
\end{equation}
We can rewrite the above as
\begin{equation}
w  = w (M_2 M_1^{-1})^{-T}, \quad Q(X,Y) = Q((X,Y) (M_ 2 M_1^{-1} )^T)
\end{equation}
so from the freeness of the action of $\mathrm{SL}_2(\mathbb F_q[u])$ on primitive representations established in \cref{FreeActionProp},
we conclude that $M_2 M_1^{-1}$ is the identity matrix.
We therefore have $M_1 = M_2$ and thus
\begin{equation}
v_1 = (1,0)M_1^T = (1,0)M_2^T = v_2
\end{equation}
so injectivity is proven.

To obtain surjectivity, let $(S,z)$ be a standard representation of $A$ which is equivalent to $(Q,w)$.
We can therefore find a matrix $M \in \mathrm{SL}_2(\mathbb F_q[u])$ with
\begin{equation}
S((X,Y) M^{T}) = Q(X,Y), \quad z M^{-T} = w.
\end{equation}
The qudratic form $S$ is standard, so by \cref{ShortVectorsStandardFormProp}, the vector $(1,0)$ is a short vector of $S$.
We conclude from \cref{ShortVecAction} that $v = (1,0)M^{-T}$ is a short vector of $Q$,
hence $M^{-1}$ is the standardizing matrix of $Q$ at $v$,
and $(S,z)$ is the standardization of $(Q,w)$ at $v$.
\end{proof}

\begin{cor} \label{QuadraticCor}

Let $D \in \F_q[u]$ be definite.
Consider the function
\begin{equation*}
(aX^2 + bXY + cY^2, (x,y)) \mapsto (ax^2 + bxy + cy^2, a y_x x + \frac{b}{2}(\overline{x}x + y_x y) + c \overline{x}y)
\end{equation*}
which maps a standard representation $(S,w)$ by a quadratic form of discriminant $D$ to the represented polynomial $A = S(w)$ and the associated solution $f$ to the congruence $T^2 + D \equiv 0 \mod A$. 
Then the image of this function is
\begin{equation} \label{DoubleFuncImageEq}
\{(A, f) : A \in \F_q[u] \setminus \{0\}, \ f \in \F_q[u]/(A), \ f^2 + D \equiv 0 \mod A  \}.
\end{equation}
Moreover, the preimage of any $(f,A)$ as above is either a set of $q-1$ representations, all satisfying $\deg(a) < \deg(c)$,
or a set of $q^2 - 1$ representations, all satisfying $\deg(a) = \deg(c)$.

\end{cor}

\begin{proof}

The fact that the image of our function is contained in \cref{DoubleFuncImageEq} is immediate from \cref{GenAssocDef}.
Taking $(A,f)$ from the latter set, 
\cref{RepCorrToGivenSolProp} provides us with a primitive representation $(Q,w)$ of $A$ by a quadratic form of discriminant $D$ such that $f$ is the associated solution to this representation.
Standardizing $(Q,w)$ at a short vector of $Q$, 
we obtain a standard representation 
\begin{equation}
(S(X,Y) = aX^2 + bXY + cY^2, w')
\end{equation}
of $A$, which is equivalent to $(Q,w)$.
It follows from \cref{GenAssocSolsEquivaCor} that $f$ is also the solution associated to $(S,w')$, 
so our function maps $(S,w')$ to $(A,f)$,
hence its image is indeed given by \cref{DoubleFuncImageEq}.

By the other implication in \cref{GenAssocSolsEquivaCor}, the preimage of $(A,f)$ under our function consists of all those standard representations of $A$ that are equivalent to $(S,w')$.
These representation are in bijection with the short vectors of $S$ in view of \cref{StndrdzThm}.
Since $S$ is a standard definite quadratic form, 
it follows from \cref{ShortVectorsStandardFormProp} that this set has $q-1$ elements in case $\deg(a) < \deg(c)$ and $q^2-1$ elements in case $\deg(a) = \deg(c)$.

Now if 
\begin{equation}
(S'(X,Y) = a'X^2 + b'XY + c'Y^2, w'')
\end{equation}
is any other representation (of $A$) in the preimage of $(A,f)$, 
then it is equivalent to $(S,w')$. 
It follows from \cref{ShortVecAction} that the number of short vectors of $(S', w'')$ equals the number of short vectors of $(S,w')$,
so since $q -1 \neq q^2 - 1$, 
we conclude from \cref{ShortVectorsStandardFormProp} that $\deg(a') < \deg(c')$ in case $\deg(a) < \deg(c)$ and that $\deg(a') = \deg(c')$ if $\deg(a) = \deg(c)$.
\end{proof}


\begin{notation}

For a definite polynomial $D \in \F_q[u]$ we denote by 
\begin{equation*}
S(D) = \left\{(a,b,c) \in \F_q[u]^3 : \deg(c) \geq \deg(a) > \deg(b), \ ac - \frac{b^2}{4} = D \right\}
\end{equation*}
the set of all standard quadratic forms of discriminant $D$.

\end{notation}

\begin{cor} \label{NumberStandardDefiniteFormsBound}

For every $\epsilon > 0$ we have $|S(D)| \ll |D|^{\frac{1}{2} + \epsilon}$.

\end{cor}

\begin{proof}

For $i \in \{0,1\}$ set
\begin{equation}
S(D;i) = \{(a,b,c) \in S(D) : \deg(c) \equiv i \mod 2\}
\end{equation}
and note that $S(D) = S(D;0) \cup S(D;1)$ so it suffices to show that 
\begin{equation}
|S(D;i)| \ll |D|^{\frac{1}{2} + \epsilon}, \quad i \in \{0,1\}.
\end{equation}

Fix $i \in \{0,1\}$, let $(a,b,c) \in S(D;i)$, and let $n \geq \deg(D)$ be an integer with 
\begin{equation}
n \equiv i \ \mathrm{mod} \ 2.
\end{equation}
It follows from \cref{DefiniteDeg} that for corpime polynomials $x,y \in \F_q[u]$ with
\begin{equation}
\deg(x) < \frac{n-\deg(a)}{2}, \quad \deg(y) = \frac{n - \deg(c)}{2},
\end{equation} 
the polynomial
$
A = ax^2 + bxy + cy^2
$
has degree $n$.
From the count of pairs of coprime polynomials $(x,y)$ in \cite[Proof of Lemma 7.3]{ABR}, and \cref{LowerBoundDegCRem}, it follows that the number of standard representations of degree $n$ polynomials by quadratic forms of discriminant $D$ is $\gg$ 
\begin{equation}
|S(D;i)| q^{\frac{n - \deg(a)}{2} + \frac{n-\deg(c)}{2}} = |S(D;i)| q^{n - \frac{\deg(D)}{2}}.
\end{equation}

On the other hand, using \cref{QuadraticCongNot}, this number is $\ll$
\begin{equation}
\sum_{A \in \mathcal{M}_n} \rho(A;F)
\end{equation}
in view of \cref{QuadraticCor}.
By \cref{AverageNumberOfSolsQuadCongProp} the above is $\ll q^n |D|^{\epsilon}$ so 
\begin{equation*}
|S(D;i) | \ll q^{\frac{\deg(D)}{2}} |D|^{\epsilon} = |D|^{\frac{1}{2} + \epsilon} \end{equation*} 
as desired.
\end{proof}

\begin{remark} Keeping track of all the constants in the proofs of \cref{NumberStandardDefiniteFormsBound} and \cref{AverageNumberOfSolsQuadCongProp} would give a precise estimate for the number of standard quadratic forms, weighted by the inverse of their number of short vectors, in terms of a special value of the $L$-function. This would be an analogue of a classical proof of the class number formula, though we have here avoided the relationship between quadratic forms and ideal classes. \end{remark}

\subsection{Indefinite quadratic forms}


\begin{defi} \label{StandardIndefiniteDef}

We say that a primitive representation $(Q,(x,y))$ of a polynomial $A \in \F_q[u]$ by a quadratic form $Q(X,Y) = aX^2 + bXY + cY^2$ of indefinite discriminant $D$ is standard if there exists a nonnegative integer $s$ such that 
\begin{equation} \label{StandardabcIndefEq}
\deg(a) \leq \frac{\deg(D)}{2}-s, \quad \deg(b) \leq \frac{\deg(D)}{2}, \quad \deg(c) \leq \frac{\deg(D)}{2} + s
\end{equation}
and
\begin{equation} \label{StandardxyIndefEq}
\deg(x) \leq \frac{ \deg (A)}{2} - \frac{\deg (D)}{4} + \frac{s}{2}, \quad \deg (y) \leq \frac{\deg (A)}{2} - \frac{\deg (D)}{4} - \frac{s}{2}.
\end{equation}
Define the weight of a standard representation as above to be 
\begin{equation}
\omega(Q,(x,y)) = \omega_s =  
\begin{cases}
\frac{1}{q^3-q} &s=0 \\
\frac{1}{(q-1) q^{s+1}} &s>0.
\end{cases}
\end{equation}

%

\end{defi}

We show that the weight is well-defined.

\begin{prop} \label{ParityofsProp}

At least one of the inequalities in \cref{StandardabcIndefEq} is an equality, 
and (at least) one of the inequalities in \cref{StandardxyIndefEq} is an equality,
so the integer $s$ is uniquely determined by the standard representation $(Q,(x,y))$ of the polynomial $A$
and satisfies
\begin{equation} \label{aModTwoCongruenceForS}
s \equiv \frac{\deg(D)}{2} - \deg(A) \mod 2, \quad s \leq \frac{\deg(D)}{2}.
\end{equation}

\end{prop}

\begin{proof}

For the first claim note that if all inequalities in \cref{StandardabcIndefEq} were strict, we would have
\begin{equation*}
\deg(D) = \deg \left( ac - \frac{b^2}{4} \right) \leq \max\{\deg(a) + \deg(c), 2\deg(b)\} < \deg(D).
\end{equation*}
which is contradictory.

We turn to the second claim. 
Since $Q(x,y) = A$ we have
\begin{equation} \label{TripleMaxEq}
\begin{split}
\deg(A) &= \deg(Q(x,y)) = \deg(ax^2 + bxy + cy^2)\\
&\leq \max\{\deg(a) + 2\deg(x), \deg(b) + \deg(x) + \deg(y), \deg(c) + 2\deg(y)\}
\end{split}
\end{equation}
and we assume that the maximum is attained by the first element above.
Then from \cref{StandardabcIndefEq} and \cref{StandardxyIndefEq} we get
\begin{equation} \label{LestEq}
\deg(a) + 2\deg(x) \leq \frac{\deg(D)}{2} - s + \deg(A) - \frac{\deg(D)}{2} + s = \deg(A)
\end{equation}
so lest we arrive using \cref{TripleMaxEq} and the above at $\deg(A) < \deg(A)$, all our inequalities must be equalities.
In particular 
\begin{equation} \label{CaseOfEqualityForDegxEq}
\deg(x) = \frac{ \deg (A)}{2} - \frac{\deg (D)}{4} + \frac{s}{2}.
\end{equation}

A calculation similar to that in \cref{LestEq}, using \cref{StandardabcIndefEq} and \cref{StandardxyIndefEq}, shows that in case the maximum in \cref{TripleMaxEq} is attained by the second element, \cref{CaseOfEqualityForDegxEq} still holds.
In case the maximum is attained by the third element in \cref{TripleMaxEq}, 
we get that
\begin{equation}
\deg(y) = \frac{\deg (A)}{2} - \frac{\deg (D)}{4} - \frac{s}{2}.
\end{equation} 

In all three cases $s$ is uniquely determined by $(Q,(x,y))$ and
\begin{equation}
\pm s + \deg(A) - \frac{\deg(D)}{2} \in \{2 \deg(x), 2 \deg(y)\} \subseteq 2\mathbb{Z}
\end{equation}
so the congruence in \cref{aModTwoCongruenceForS} holds.
At last note that
\begin{equation}
0 \leq \deg(a) \leq \frac{\deg(D)}{2} - s
\end{equation}
in view of \cref{StandardabcIndefEq} and \cref{NonZeroARem}, so the inequality in \cref{aModTwoCongruenceForS} holds.
\end{proof}


\begin{defi} \label{PlaneValuationDef}

A valuation on $\mathbb F_q[u]^2$ is a function 
\begin{equation}
v \colon \mathbb F_q[u]^2 \to \mathbb Z \cup \{-\infty\}
\end{equation}
satisfying the following three conditions.
\begin{enumerate}

\item For every $a,x,y \in \F_q[u]$ we have
\begin{equation*}
v(ax,ay) = \deg (a) + v(x,y);
\end{equation*}

\item For all $x_1, x_2, y_1, y_2 \in \F_q[u]$ we have 
\begin{equation*}
v(x_1+x_2, y_1+y_2) \leq \max \{ v(x_1,y_1), v(x_2,y_2) \};
\end{equation*}

\item The values of $v$ on nonzero vectors are bounded below, i.e. 
\begin{equation*}
\inf_{(x,y) \neq (0,0)} v(x,y) >-\infty.
\end{equation*}

\end{enumerate}

\end{defi}
For $v$ a valuation, let  \begin{equation*} m_v = \inf_{(x,y) \neq (0,0)} v(x,y) .\end{equation*} Because $m_v$ is the infimum of a set of integers bounded below, $m_v$ is attained by some $x,y$.

\begin{notation}

For integers $\gamma$ and $\delta$ one readily checks that the function 
\begin{equation}
v_{\gamma, \delta}(x,y) = \max\{\gamma + \deg(x), \delta + \deg(y)\}
\end{equation}
is a valuation.
We say that a valuation $v$ is standard if there exist integers $\gamma \leq \delta$ such that $v = v_{\gamma, \delta}$.
In this case $m_v = v(1,0) = \gamma$.

\end{notation}

\begin{defi} \label{ActionOnValuationsDef}

We have an action of $\mathrm{SL}_2(\mathbb F_q[u])$ from the right on valuations by
\begin{equation}
(v \star M)(x,y) = v((x,y)M^T), \quad M \in \mathrm{SL}_2(\mathbb F_q[u]), \ x,y \in \F_q[u].
\end{equation}
Valuations in the same orbit will be called equivalent.

\end{defi}

We calculate the order of the stabilizer of a standard valuation.

\begin{prop} \label{FiniteStabValProp}

For a standard valuation $v_{\gamma, \delta}$ define the subgroup
\begin{equation}
G_{\gamma, \delta} = \{M \in \mathrm{SL}_2(\mathbb F_q[u]) : v_{\gamma, \delta} \star M = v_{\gamma, \delta}\}.
\end{equation}
Then
\begin{equation}
|G_{\gamma, \delta}| = 
\begin{cases}
q^3 - q &\gamma = \delta \\
(q-1)q^{\delta - \gamma + 1} &\gamma < \delta.
\end{cases}
\end{equation}

\end{prop}

\begin{proof}

For each  integer $s \geq 0$, we define a subgroup of  $\mathrm{SL}_2(\mathbb F_q[u])$ by
\begin{equation}
H_s = 
\begin{cases}
\mathrm{SL}_2(\mathbb F_q) &s=0 \\
\left\{
\begin{pmatrix}
\lambda &f \\
0 &\lambda^{-1}
\end{pmatrix}
: \lambda \in \F_q^\times, \ f \in \F_q[u], \ \deg(f) \leq s \right\} &s \geq 1
\end{cases}
\end{equation}
and claim that $G_{\gamma, \delta} = H_{\delta - \gamma}$. 
The asserted number of elements in $G_{\gamma, \delta}$ is immediate from this claim.


To prove one inclusion let $M \in G_{\gamma, \delta}$. We have
\begin{equation}
\begin{split}
\gamma = v_{\gamma,\delta}(1,0) &= v_{\gamma, \delta}((1,0)M^T) = v_{\gamma, \delta}(M_{11}, M_{21}) \\ 
&= \max\{\gamma + \deg(M_{11}), \delta + \deg(M_{21}) \}.
\end{split}
\end{equation}
We deduce that $\deg(M_{11})$ and $\deg(M_{21})$ are nonpositive, and in case $\gamma < \delta$ we can moreover say that $M_{21} = 0$. 
Similarly we have
\begin{equation}
\begin{split}
\delta = v_{\gamma,\delta}(0,1) &= v_{\gamma, \delta}((0,1) M^T) = v_{\gamma, \delta}(M_{12}, M_{22}) \\ 
&= \max\{\gamma + \deg(M_{12}), \delta + \deg(M_{22}) \}
\end{split}
\end{equation}
so $\deg(M_{22}) \leq 0$.
In case $\gamma = \delta$ we infer that $\deg(M_{12}) \leq 0$ as well,
while in case $\gamma < \delta$ we get that $\deg(M_{12}) \leq \delta - \gamma$.
Since our matrices have determinant $1$, this establishes the inclusion $G_{\gamma, \delta} \leq H_{\delta - \gamma}$ towards our claim.

For the other inclusion pick $M \in H_{\delta - \gamma}$.
In case $\gamma = \delta$ this is a matrix of polynomials of nonpositive degree, so we have
\begin{equation*}
\begin{split}
v_{\gamma, \delta}((x,y)M^T) &= v_{\gamma, \delta}(M_{11}x + M_{12}y, M_{21}x + M_{22}y) \\
&= \gamma + \max\{\deg(M_{11}x + M_{12}y), \deg(M_{21}x + M_{22}y) \} \\
&\leq \gamma + \max\{\deg{x},\deg{y}\} = v_{\gamma, \delta}(x,y).
\end{split}
\end{equation*}

In case $\gamma < \delta$ we have 
\begin{equation}
\deg(M_{11}) =0, \ \deg(M_{12}) \leq \delta - \gamma, \ \deg(M_{21}) = -\infty, \ \deg(M_{22}) = 0
\end{equation}
so in this case we have the similar inequality
\begin{equation*}
\begin{split}
v_{\gamma, \delta}((x,y)M^T) &= v_{\gamma, \delta}(M_{11}x + M_{12}y, M_{21}x + M_{22}y) \\
&= \max\{\gamma + \deg(M_{11}x + M_{12}y), \delta + \deg(M_{21}x + M_{22}y) \} \\
&\leq \max\{\gamma + \max\{\deg(M_{11}x), \deg(M_{12}y)\}, \delta + \deg(y) \} \\
&\leq \max\{\gamma + \deg{x}, \delta + \deg{y}\} = v_{\gamma, \delta}(x,y).
\end{split}
\end{equation*}

Since $M^{-1} \in H_{\delta - \gamma}$, we can plug $M^{-1}$ in place of $M$ and then plug $(x,y)M^{T}$ in place of $(x,y)$,
getting the inequality $v_{\gamma, \delta}(x,y) \leq v_{\gamma, \delta}((x,y)M^T)$. 
In conjunction with the above we have $v_{\gamma, \delta}((x,y)M^T) = v_{\gamma, \delta}(x,y)$ so $M$ is in $G_{\gamma, \delta}$ and thus $H_{\delta - \gamma} \leq G_{\gamma, \delta}$. This concludes the proof that $G_{\gamma,\delta} = H_{\delta - \gamma}$. 
\end{proof} 



\begin{prop} \label{StandardMembersProp}

Every valuation $v$ on $\F_q[u]^2$ is equivalent to a unique standard valuation.



\end{prop}

\begin{proof} 

Let $(x,y) \neq (0,0)$ be a vector attaining the minimal valuation, namely
\begin{equation}
v(x,y) = m_v.
\end{equation}
From \cref{PlaneValuationDef}(1) and the minimality of $(x,y)$ we get that
\begin{equation}
\begin{split}
v(x,y) &= v \left(\gcd(x,y)\frac{x}{\gcd(x,y)}, \gcd(x,y)\frac{y}{\gcd(x,y)} \right) \\
&= \deg(\gcd(x,y)) + v \left(\frac{x}{\gcd(x,y)},\frac{y}{\gcd(x,y)} \right) \\
&\geq \deg(\gcd(x,y)) + v(x,y) 
\end{split}
\end{equation}
so $\deg(\gcd(x,y)) \leq 0$, hence $x$ and $y$ are coprime.

We can therefore take $M_{(x,y)} \in \mathrm{SL}_2(\mathbb F_q[u]) $ to be the matrix from \cref{BarredMatrix} satisfying 
$
(1,0) = (x,y)M_{(x,y)}^{-T}.
$
For the valuation $v' = v \star M_{(x,y)}$ we then have 
\begin{equation*}
v'(z,w) = v((z,w)M_{(x,y)}^T) \geq m_v, \ (z,w) \neq (0,0), \quad v'(1,0) = v(x,y) = m_v 
\end{equation*}
so $v'(1,0) = m_{v'}$.

Let $(z,1) \in \F_q[u]^2$ be a vector with
\begin{equation}
v'(z,1) = \min\{v'(x,1) : x \in \F_q[u]\}.
\end{equation}
Let 
\begin{equation}
M = 
\begin{pmatrix}
1 & z \\
0 & 1
\end{pmatrix}
\in \mathrm{SL}_2(\mathbb F_q[u])
\end{equation}
and note that $(1,0)M^T = (1,0)$.
For the valuation $v'' = v' \star M$ we then have
\begin{equation*}
v''(t,w) = v'((t,w)M^T) \geq m_{v'}, \ (t,w) \neq (0,0), \quad v''(1,0) = v'(1,0) = m_{v'}
\end{equation*}
so $v''(1,0) = m_{v''}$.
Moreover
\begin{equation*}
\begin{split}
\min\{v''(x,1) : x \in \F_q[u]\} &= \min\{v'((x,1)M^T) : x \in \F_q[u]\} \\
&= \min\{v'(x+z, 1) : x \in \F_q[u]\} \\
&= \min\{v'(x, 1) : x \in \F_q[u]\} = v'(z,1) = v''(0,1).
\end{split}
\end{equation*}

Set $\gamma = v''(1,0) = m_{v''}$ and $\delta = v''(0,1)$ so that $\gamma \leq \delta$. We claim that
\begin{equation}
v'' = v_{\gamma, \delta}
\end{equation}
and thus $v$ is equivalent to the standard valuation $v''$. 
To prove the claim, note first that by \cref{PlaneValuationDef}(2) and \cref{PlaneValuationDef}(1) we have
\begin{equation}
\begin{split}
v''(x,y) &= v''(x(1,0) + y(0,1)) \leq \max\{v''(x(1,0)), v''(y(0,1))\} \\
&= \max\{\deg(x) + v''(1,0), \deg(y) + v''(0,1)\} = v_{\gamma, \delta}(x,y).
\end{split}
\end{equation}

Suppose toward a contradiction that there exists a vector $(x,y)$ with
\begin{equation}
v''(x,y) < v_{\gamma, \delta}(x,y)
\end{equation}
so in particular $(x,y) \neq (0,0)$.
If $\gamma + \deg (x) > \delta + \deg (y)$ then we have
\begin{equation*}
\begin{split}
\gamma + \deg (x) &= v''(1,0) + \deg(x) = v''(x(1,0))  = v''((x,y) + (0,-y))  \\
&\leq \max \{ v''(x,y), v''(-y(0,1)) \} = \max \{v''(x,y), v''(0,1) + \deg (y) \}  \\
&= \max\{v''(x,y), \delta + \deg(y)\} < \max\{v_{\gamma, \delta}(x,y), \gamma + \deg (x)\} = \gamma + \deg(x),
\end{split}
\end{equation*}
a contradiction. 
Similarly, if $\delta + \deg y > \gamma + \deg x$, we have
\begin{equation*}
\begin{split}
\delta + \deg (y) &= v''(y(0,1)) \leq \max \{ v''(x,y), v''(-x(1,0)) \} \\
&= \max \{ v''(x,y), \gamma + \deg (x)\} < \delta + \deg (y),
\end{split}
\end{equation*}
a contradiction. 

Finally, if $\gamma + \deg (x) = \delta + \deg (y)$, so in particular 
\begin{equation}
\deg x  = \delta - \gamma + \deg(y) \geq \deg (y) \geq 0, 
\end{equation}
we can use division with remainder in $\F_q[u]$ to write 
\begin{equation}
x = wy + r, \quad \deg (r) < \deg (y) \leq \deg (x), \ \deg (w) = \deg (x) - \deg (y).
\end{equation} 
Since $v''(0,1) = \min \{v''(z,1) : z \in \F_q[u]\}$ we have
\begin{equation*}
\begin{split}
\delta + \deg (y) &= v''(0,1) + \deg(y) \leq v''(w,1) + \deg (y) = v''(wy,y) = v''(x-r,y) \\
&\leq \max\{ v''(x,y), v''(-r(1,0)) \} = \max \{ v''(x,y) , v'''(1,0) + \deg (r) \} \\
&= \max \{ v''(x,y) , \gamma + \deg (r) \} < \max \{ v_{\gamma, \delta}(x,y) , \gamma + \deg (x) \} = \delta + \deg(y),
\end{split}
\end{equation*} 
the final contradiction.

We have seen that $v$ is equivalent to the standard valuation $v_{\gamma, \delta}$.
To prove uniqueness, assume that $v$ is also equivalent to some standard valuation $v_{\gamma', \delta'}$.
We conclude that $v_{\gamma, \delta}$ is equivalent to $v_{\gamma', \delta'}$ so
\begin{equation}
\gamma = v_{\gamma, \delta}(1,0) = m_{v_{\gamma, \delta}} = m_{v_{\gamma', \delta'}} = v_{\gamma', \delta'}(1,0) = \gamma'.
\end{equation}
Since the valuations $v_{\gamma, \delta}$ and $v_{\gamma', \delta'}$ belong to the same orbit under the action of $\mathrm{SL}_2(\mathbb F_q[u])$,
their stabilizers are conjugate subgroups of $\mathrm{SL}_2(\mathbb F_q[u])$, so they have the same cardinality.
We conclude from \cref{FiniteStabValProp} that
\begin{equation}
\begin{cases}
q^3 - q &\gamma = \delta \\
(q-1)q^{\delta - \gamma + 1} &\gamma < \delta
\end{cases} = 
\begin{cases}
q^3 - q &\gamma = \delta' \\
(q-1)q^{\delta' - \gamma + 1} &\gamma < \delta'.
\end{cases}
\end{equation}
As $q^3 - q$ is not equal to $q-1$ times a power of $q$, we infer that $\delta' = \delta$ and thus 
$v_{\gamma, \delta} = v_{\gamma', \delta'}$ as required for uniqueness.
\end{proof}


\begin{prop} \label{IndefFormFactorizationProp}

For every indefinite quadratic form $Q(X,Y)$ over $\F_q[u]$ there exist linear forms $L_1(X,Y)$ and $L_2(X,Y)$ over $\F_q((u^{-1}))$ such that
\begin{equation}
Q(X,Y) = L_1(X,Y)L_2(X,Y).
\end{equation} 
This factorization is unique up to scaling $L_1$ by an element of $\F_q((u^{-1}))^\times$ and $L_2$ by its inverse, 
and up to changing the order of the factors.

\end{prop}

\begin{proof} 

The discriminant $D = ac - b^2/4$ of $Q(X,Y) = aX^2 + bXY + cY^2$ is indefinite,
namely the infinite place of $\F_q(u)$ splits in the splitting field of $F(T) = T^2 + D$ over $\F_q(u)$,
or equivalently $-D$ is a square in the completion $\F_q((u^{-1}))$ of $\F_q(u)$ at infinity.
This means that there exists a unique (unordered) pair of scalars $\lambda_1, \lambda_2 \in \F_q((u^{-1}))$ such that
\begin{equation*}
\begin{split}
Q(X,Y) &= Y^2 \left( a\left(\frac{X}{Y}\right)^2 + b\frac{X}{Y} + c \right) = 
Y^2 a \left( \frac{X}{Y} - \lambda_1 \right) \left( \frac{X}{Y} - \lambda_2 \right) \\
&= a(X - \lambda_1Y)(X - \lambda_2Y).
\end{split}
\end{equation*}

\end{proof}

\begin{notation} 

Using the notation of \cref{IndefFormFactorizationProp}, for a primitive representation $(Q,(x,y))$ we define a function on $\F_q[u]^2$ by
\begin{equation*}
v_{ (x,y)}^Q (z,w) = \max \{ \deg (L_1(z,w)) - \deg (L_1(x,y)), \deg (L_2(z,w)) - \deg (L_2(x,y)) \}
\end{equation*}
where the degree of a nonzero element of $\mathbb F_q((u^{-1}))$ is the degree of its highest-order term in $u$. 
It follows from the uniqueness part of \cref{IndefFormFactorizationProp} that the function $v_{(x,y)}^Q$ is well-defined.

\end{notation} 

\begin{lem} \label{valuation-quadratic-form}

The function $v_{(x,y)}^Q(z,w)$ is a valuation, and it satisfies
\begin{equation} \label{valuation-quadratic-form-lower-bound}
v_{(x,y)}^Q(z,w) \geq \frac{\deg(Q(z,w)) - \deg(Q(x,y))}{2}.
\end{equation}

\end{lem}

\begin{proof} 

To check \cref{PlaneValuationDef}(1), we just need to note that for every polynomial $a \in \F_q[u]$ we have
\begin{equation*}
\deg(L_i(az,aw)) = \deg(a) + \deg(L_i(z,w)), \quad i \in \{1,2\}.
\end{equation*}
For \cref{PlaneValuationDef}(2), one has to observe that for $i \in {1,2}$ we have
\begin{equation*}
\deg(L_i(z + r, w + s)) = \deg(L_i(z,w) + L_i(r,s)) \leq \max\{ \deg(L_i(z,w)), \deg(L_i(r,s)) \}.
\end{equation*}

To verify \cref{PlaneValuationDef}(3) note that twice the value of the function equals
\begin{equation*}  
\begin{split} 
&2 \max \{ \deg (L_1(z,w)) - \deg (L_1(x,y)), \deg (L_2(z,w)) - \deg (L_2(x,y)) \} \geq \\
&\deg (L_1(z,w)) - \deg (L_1(x,y)) + \deg (L_2(z,w)) - \deg (L_2(x,y)) = \\
&\deg(L_1(z,w)L_2(z,w)) - \deg(L_1(x,y)L_2(x,y))= \deg (Q(z,w)) - \deg (Q(x,y)) 
\end{split}
\end{equation*}
and $Q(z,w) \neq 0$ for $(z,w) \neq (0,0)$ by the homogeneity of $Q$ and \cref{RepValueNonZeroEq}, 
so the above is at least $-\deg(Q(x,y))$.
This concludes the verification of \cref{valuation-quadratic-form-lower-bound} and of all the conditions a valuation must satisfy. 
\end{proof}

Associating a valuation to a representation is an $\mathrm{SL}_2(\F_q[u])$-equivariant operation, as we shall now see.

\begin{prop} \label{FormValuationEquivarianceProp}

Let $(Q,(x,y))$ be a representation by an indefinite quadratic form, and let $M \in \mathrm{SL}_2(\F_q[u])$ be a matrix.
Then
\begin{equation}
v_{(x,y) \star M}^{Q \star M} = v_{(x,y)}^Q \star M.
\end{equation}

\end{prop}

\begin{proof}

Since $Q(X,Y) = L_1(X,Y)L_2(X,Y)$, we get from \cref{ActionOnFormsDef} that
\begin{equation}
(Q \star M)(X,Y) = Q( (X,Y)M^T) = L_1((X,Y)M^T)  L_2((X,Y)M^T)
\end{equation}
and from \cref{QuadFormNot} that $(x,y) \star M = (x,y)M^{-T}$.
Therefore
\begin{equation*}
\begin{split}
v_{ (x,y) \star M}^{Q \star M} (z,w) &= \max_{i \in \{1,2\}} \{ \deg (L_i((z,w)M^T)) - \deg (L_i((x,y)M^{-T}M^T)) \} \\
&= \max_{i \in \{1,2\}} \{ \deg (L_i((z,w)M^T)) - \deg (L_i(x,y)) \} = (v_{(x,y)}^Q \star M) (z,w)
\end{split}
\end{equation*}
for every vector $(z,w) \in \F_q[u]^2$, in view of \cref{ActionOnValuationsDef}.
\end{proof}

\begin{remark} 

We can think of the set of valuations on $\mathbb F_q[u]^2$ as an analogue of the upper half-plane, 
on which $\mathrm{SL}_2(\mathbb F_q[u])$ acts, 
and our set of (standard) representatives of each $\mathrm{SL}_2(\mathbb F_q[u])$-orbit as an analog of the usual fundamental domain for the action of $\mathrm{SL}_2(\mathbb{Z})$ on the upper half-plane. 
To each indefinite quadratic form $Q$ one associates a geodesic in the upper half-plane, 
which for us consists of the valuations $v_{(x,y)}^Q$ for the various vectors $(x,y)$. 
We will show that the standard representations correspond to points on this geodesic that lie in the fundamental domain.

\end{remark}

\begin{lem} \label{StandardValuationImpliesStandardRepresentation}

Let $(Q, (x,y))$ be a primitive representation of $A \in \F_q[u]$ by an indefinite quadratic form.
If the associated valuation is standard, namely
\begin{equation} 
v_{(x,y)}^Q = v_{\gamma, \delta}, \quad \gamma \leq \delta,
\end{equation}
then $(Q,(x,y))$ is a standard representation of weight
\begin{equation}
\omega(Q,(x,y)) = \omega_{\delta - \gamma} = 
\begin{cases}
\frac{1}{q^3-q} &\gamma = \delta \\
\frac{1}{(q-1) q^{\delta - \gamma +1}} &\gamma<\delta.
\end{cases}
\end{equation}

\end{lem}

\begin{proof} It follows from \cref{valuation-quadratic-form} and our assumptions that 
\begin{equation} \label{zwQuadraticFormInequality}
\begin{split} 
&\deg(a z^2 + bzw + c w^2) =  \deg (Q(z,w)) \leq 2 v_{(x,y)}^Q(z,w) + \deg (Q(x,y)) = \\
&2v_{\gamma, \delta}(z,w) + \deg(A) = \max \{ 2\gamma + 2\deg(z), 2\delta + 2\deg (w) \}  + \deg (A).
\end{split}
\end{equation}
Taking $w=0, \ z=1$ above, we see that 
\begin{equation} \label{QFormaboundEq}
\deg (a) \leq 2\gamma + \deg(A),
\end{equation}
taking $z=0, \ w=1$, we see that 
\begin{equation} \label{QFormcboundEq}
\deg (c) \leq 2 \delta + \deg (A),
\end{equation}
taking $z = u^{\delta - \gamma}, \ w=1$ in \cref{zwQuadraticFormInequality}, and using \cref{QFormaboundEq}, \cref{QFormcboundEq} we see that
\begin{equation*}
\begin{split}
&\deg(b) + \delta - \gamma = \deg(bu^{\delta - \gamma}) = \deg(Q(u^{\delta - \gamma},1) - au^{2\delta - 2\gamma} - c) \leq\\
&\max\{\deg(Q(u^{\delta - \gamma},1)), \deg(au^{2\delta - 2\gamma}), \deg(c)\} \leq 2\delta + \deg(A)
\end{split}
\end{equation*}
so
\begin{equation} \label{QFormbboundEq}
\deg (b) \leq \gamma + \delta + \deg (A).
\end{equation}

From \cref{RepValueNonZeroEq} we get that $Q(x,y) \neq 0$ so for the linear forms $L_1$ and $L_2$ from \cref{IndefFormFactorizationProp} we have that
$\deg(L_1(x,y))$ and $\deg(L_2(x,y))$ are finite, hence
\begin{equation} 
0 = v_{(x,y)}^Q (x,y) = v_{\gamma, \delta}(x,y) = \max \{ \gamma + \deg (x), \delta + \deg (y)\}.
\end{equation}
Therefore
\begin{equation} \label{QFormxyboundEq}
\deg (x) \leq - \gamma, \quad \deg (y) \leq - \delta.
\end{equation}

Set $s = \delta - \gamma$. Our proposition reduces to showing that the discriminant $D = ac - b^2/4$ of $Q$ satisfies
\begin{equation} \label{discriminant-valuation-identity} 
\deg (D) = 2\gamma + 2\delta + 2 \deg (A).
\end{equation}
Indeed it follows from $s=\delta-\gamma$ and \cref{discriminant-valuation-identity} that
\begin{equation} \label{TheGammaDeltaEq}
\gamma = \frac{\deg(D)}{4} - \frac{\deg(A)}{2} -\frac{s}{2}, \quad \delta = \frac{\deg(D)}{4} - \frac{\deg(A)}{2} + \frac{s}{2}
\end{equation}
so \cref{StandardabcIndefEq} and \cref{StandardxyIndefEq} follow from \cref{QFormaboundEq}, \cref{QFormbboundEq}, \cref{QFormcboundEq}, and \cref{QFormxyboundEq}.

To check \cref{discriminant-valuation-identity}, write 
\begin{equation*}
L_1 (X,Y) =\alpha_1 X + \beta_1 Y, \ L_2(X,Y) = \alpha_2 X + \beta_2 Y, \ \alpha_1, \beta_1, \alpha_2, \beta_2 \in \F_q((u^{-1}))
\end{equation*} 
so that 
\[Q(X,Y) = (\alpha_1 X + \beta_1 Y) (\alpha_2 X + \beta_2 Y)= \alpha_1 \alpha_2 X^2 + (\alpha_1 \beta_2 + \beta_1 \alpha_2) XY + \beta_1 \beta_2 Y^2 \] 
and thus
\begin{equation} \label{NegativeDinTermsOfLinearForms}
-D= \frac{ (\alpha_1 \beta_2 + \beta_1 \alpha_2)^2 }{4} - \alpha_1 \alpha_2 \beta_1 \beta_2 = \frac{  (\alpha_1 \beta_2 - \beta_1 \alpha_2)^2 }{4}.
\end{equation}
Therefore, it suffices to show that
\begin{equation} \label{determinant-valuation-identity} 
\deg( \alpha_1 \beta_2 - \beta_1 \alpha_2) = \gamma + \delta + \deg (A).
\end{equation}

For $i \in \{1,2\}$ we have
\begin{equation}
\begin{split}
\gamma = v_{\gamma, \delta}(1,0) = v_{(x,y)}^Q(1,0) &\geq \deg(L_i(1,0)) - \deg(L_i(x,y)) \\ &= \deg(\alpha_i) - \deg(L_i(x,y))
\end{split}
\end{equation}
and similarly 
\begin{equation}
\delta = v_{\gamma, \delta}(0,1) = v_{(x,y)}^Q(0,1) \geq \deg(\beta_i) - \deg(L_i(x,y))
\end{equation}
so
\begin{equation}
\deg(\alpha_i) \leq \gamma + \deg(L_i(x,y)), \quad \deg(\beta_i) \leq \delta + \deg(L_i(x,y)).
\end{equation}
We conclude that
\begin{equation*}
\begin{split}
&\deg(\alpha_1\beta_2 - \beta_1\alpha_2) \leq \max\{\deg(\alpha_1) + \deg(\beta_2), \deg(\beta_1) + \deg(\alpha_2)\} \leq \\
&\gamma + \delta + \deg(L_1(x,y) L_2(x,y)) = \gamma + \delta + \deg(Q(x,y)) = \gamma + \delta + \deg(A)
\end{split}
\end{equation*}
so we have established one inequality towards \cref{determinant-valuation-identity}.

Assume for contradiction that the inequality above is strict. 
Denoting by $\alpha_i^{(0)}$ the coefficient of $\alpha_i$ in degree $\gamma + \deg(L_{i}(x,y))$ for $i \in \{1,2\}$,
and by $\beta_i^{(0)}$ the coefficient of $\beta_i$ in degree $\delta + \deg(L_{i}(x,y))$,
we can interpret our assumption for contradiction as
\begin{equation}
\det
\begin{pmatrix}
\alpha_1^{(0)} &\alpha_2^{(0)} \\
\beta_1^{(0)} &\beta_2^{(0)}
\end{pmatrix}
= \alpha_1^{(0)}\beta_2^{(0)} - \beta_1^{(0)}\alpha_2^{(0)} = 0.
\end{equation}

Let $(r,t) \in \F_q^2$ be a nonzero vector in the kernel of the matrix above. 
By examining the coefficients in degree $\delta + \deg(L_i(x,y))$ we see that
\begin{equation*}
\deg (L_i(ru^{\delta - \gamma},t)) =\deg ( r  u^{\delta - \gamma}\alpha_i + t \beta_i ) <  \delta + \deg (L_i(x,y)), \quad i \in \{1,2\}.
\end{equation*}
Since (at least) one of the scalars $r,t$ is nonzero, we get that 
\begin{equation*}
\begin{split}
\delta &= \max\{\deg(r) + \delta, \deg(t) + \delta \} = v_{\gamma, \delta}(ru^{\delta - \gamma},t) \\
&= v_{(x,y)}^Q (ru^{\delta - \gamma},t) = \max_{i \in \{1,2\}} \{\deg (L_i(ru^{\delta - \gamma},t)) - \deg (L_i(x,y))\} < \delta
\end{split}
\end{equation*}
which is a contradiction.
This verifies \cref{determinant-valuation-identity}, completing the proof. 
\end{proof}

\begin{lem} \label{StandardVRepresentationImpliesStandardValuation}

For every standard representation $(Q,(x,y))$ of a polynomial $A$ by an indefinite quadratic form, there exist integers $\gamma \leq \delta$ such that
\begin{equation} 
v_{(x,y)}^Q = v_{\gamma, \delta}.
\end{equation}

\end{lem}

\begin{proof} 

Following \cref{StandardIndefiniteDef}, and \cref{TheGammaDeltaEq} we define
\begin{equation} 
\gamma = \frac{\deg(D)}{4} - \frac{\deg(A)}{2} -\frac{s}{2}, \quad \delta = \frac{\deg(D)}{4} - \frac{\deg(A)}{2} + \frac{s}{2}
\end{equation}
and note that $\gamma, \delta$ are indeed integers by \cref{ParityofsProp}.
Our assumption that $(Q,(x,y))$ is standard then gives
\begin{equation} \label{xyInequalitiesTwoEqualities}
\deg(x) \leq -\gamma, \ \deg (y) \leq -\delta, \ \delta - \gamma = s, \ \gamma + \delta =\frac{\deg (D)}{2} - \deg (A).
\end{equation}

By \cref{IndefFormFactorizationProp}, for $i \in \{1,2\}$ there exist linear forms
\begin{equation}
L_i (X,Y) =\alpha_i X + \beta_i Y, \quad \alpha_i, \beta_i \in \F_q((u^{-1}))
\end{equation} 
such that
\begin{equation*}
Q(X,Y) = L_1(X,Y)L_2(X,Y) = \alpha_1 \alpha_2 X^2 + (\alpha_1 \beta_2 + \beta_1 \alpha_2) XY + \beta_1 \beta_2 Y^2
\end{equation*}
and 
\begin{equation} \label{SquareExpressionForNegativeD}
-D= \frac{  (\alpha_1 \beta_2 - \beta_1 \alpha_2)^2 }{4}
\end{equation}
as in \cref{NegativeDinTermsOfLinearForms}.

It follows from our expression for $Q$, \cref{StandardabcIndefEq}, and \cref{SquareExpressionForNegativeD} that
\begin{equation}
\deg (\alpha_1 \beta_2 + \beta_1 \alpha_2) , \deg (\alpha_1 \beta_2 -\beta_1 \alpha_2) \leq \frac{\deg (D)}{2} 
\end{equation}
and therefore that
\begin{equation}
\deg (\alpha_1 \beta_2), \deg (\beta_1 \alpha_2) \leq \frac{\deg (D)}{2}.
\end{equation}
We further infer from our expression for $Q$, \cref{StandardabcIndefEq}, and \cref{xyInequalitiesTwoEqualities} that
\begin{equation}
\deg (\alpha_1\alpha_2) \leq \frac{\deg (D)}{2} - \delta + \gamma, \quad 
\deg (\beta_1 \beta_2 ) \leq \frac{\deg (D)}{2} + \delta - \gamma.
\end{equation}

By \cref{xyInequalitiesTwoEqualities}, for $i \in \{1,2\}$ we have
\begin{equation}
\begin{split}
\deg (L_i(x,y) ) &\leq \max \{ \deg (\alpha_i) + \deg (x), \deg (\beta_i) + \deg (y) \} \\
&\leq \max \{ \deg (\alpha_i) - \gamma , \deg (\beta_i) - \delta \}
\end{split}
\end{equation} 
which either gives a lower bound on the degree of $\alpha_i$ or a lower bound on the degree of $\beta_i$ (or both). 
Combined with the upper bounds on $\deg(\alpha_i \alpha_{3-i})$ and $\deg (\beta_{i} \alpha_{3-i})$ in the first case, or $\deg (\alpha_{i} \beta_{3-i}) $ and $\deg(\beta_i \beta_{3-i}) $ in the second case, 
we obtain using \cref{xyInequalitiesTwoEqualities} that
\begin{equation*}
\begin{split}
\deg(\alpha_i) &\leq 
\begin{cases}
\frac{\deg (D)}{2} - \delta + \gamma - \deg(\alpha_{3-i}) \\
\frac{\deg(D)}{2} - \deg(\beta_{3-i})
\end{cases}
\leq \frac{\deg (D)}{2} - \delta - \deg(L_{3-i}(x,y))
\\
&= \deg (A) + \gamma - \deg(L_{3-i}(x,y)) = \deg(L_i(x,y)) + \gamma
\end{split}
\end{equation*}
and
\begin{equation*}
\begin{split}
\deg(\beta_i) &\leq 
\begin{cases}
\frac{\deg (D)}{2} - \deg(\alpha_{3-i}) \\
\frac{\deg(D)}{2} + \delta - \gamma - \deg(\beta_{3-i})
\end{cases}
\leq \frac{\deg (D)}{2} - \gamma - \deg(L_{3-i}(x,y))
\\
&= \deg (A) + \delta - \deg(L_{3-i}(x,y)) = \deg(L_i(x,y)) + \delta.
\end{split}
\end{equation*}

The bounds on $\deg(\alpha_i), \deg(\beta_i)$ imply that for $(z,w) \in \F_q[u]^2$ we have
\begin{equation*}
\begin{split}
v_{(x,y)}^Q(z,w) &= \max_{i \in \{1,2\}} \{ \deg (\alpha_i z + \beta_i w) - \deg (L_i(x,y)) \} \\
&\leq \max_{i \in \{1,2\}} \{ \max\{\deg(\alpha_i) + \deg(z), \deg(\beta_i) + \deg(w)\} - \deg (L_i(x,y)) \} \\
&\leq \max \{ \gamma + \deg (z), \delta + \deg (w)\} = v_{\gamma, \delta}(z,w).
\end{split}
\end{equation*}
We must prove that this inequality is in fact an equality.

Assume toward a contradiction that for some (nonzero) vector $(z,w)$ the inequality above is strict,
namely
\begin{equation} \label{SmallerDegreeAssumToContrEq}
\max_{i \in \{1,2\}} \{ \deg (\alpha_i z + \beta_i w) - \deg (L_i(x,y)) \} < v_{\gamma, \delta}(z,w).
\end{equation}
For $i \in \{1,2\}$ let $\alpha_i^{(0)}$ be the coefficient of $\alpha_i$ in degree $\deg(L_i(x,y)) + \gamma$,
and let $\beta_i^{(0)}$ be the coefficient of $\beta_i$ in degree $\deg(L_i(x,y)) + \delta$.
Let moreover $z^{(0)}$ be the coefficient of $z$ in degree $v_{\gamma, \delta}(z,w) - \gamma$, 
and let $w^{(0)}$ be the coefficient of $w$ in degree $v_{\gamma, \delta}(z,w) - \delta$.
Note that all the coefficients in degrees higher than these are necessarily zero, 
and that $(z^{(0)}, w^{(0)}) \neq (0,0)$ because $(z,w) \neq (0,0)$.

In the notation above, \cref{SmallerDegreeAssumToContrEq} translates to
\begin{equation}
\begin{pmatrix}
z^{(0)} &w^{(0)}
\end{pmatrix}
\begin{pmatrix}
\alpha_1^{(0)} &\alpha_2^{(0)} \\
\beta_1^{(0)} &\beta_2^{(0)}
\end{pmatrix}
=
\begin{pmatrix}
0 &0
\end{pmatrix}
\end{equation}
so the determinant of the matrix above vanishes, that is
\begin{equation}
\alpha_1^{(0)} \beta_2^{(0)} - \alpha_2^{(0)} \beta_1^{(0)} = 0.
\end{equation}
We conclude, using \cref{xyInequalitiesTwoEqualities}, that 
\begin{equation*}
\begin{split}
\deg(\alpha_1\beta_2 - \alpha_2\beta_1) &< \deg(L_1(x,y)) + \deg(L_2(x,y)) + \gamma + \delta = \deg(A) + \gamma + \delta \\
&= \frac{\deg(D)}{2}
\end{split}
\end{equation*}
which contradicts \cref{SquareExpressionForNegativeD}.
\end{proof}

\begin{cor} \label{IndefQuadraticCor}

Let $D \in \F_q[u]$ be indefinite.
Consider the function
\begin{equation*}
(aX^2 + bXY + cY^2, (x,y)) \mapsto (ax^2 + bxy + cy^2, a y_x x + \frac{b}{2}(\overline{x}x + y_x y) + c \overline{x}y)
\end{equation*}
which maps a standard representation $(Q,v)$ by a quadratic form of discriminant $D$ to the represented polynomial $A = Q(v)$ and the associated solution $f$ to the congruence $T^2 + D \equiv 0 \mod A$. 
Then the image of this function is
\begin{equation} \label{IndefDoubleFuncImageEq}
\{(A, f) : A \in \F_q[u] \setminus \{0\}, \ f \in \F_q[u]/(A), \ f^2 + D \equiv 0 \mod A  \}.
\end{equation}
Moreover the number of elements in the preimage of any $(A,f)$ from \cref{IndefDoubleFuncImageEq} equals the inverse of the weight of any representation in this preimage.

\end{cor}

\begin{proof}

It is immediate from \cref{GenAssocDef} that the image of our function is contained in \cref{IndefDoubleFuncImageEq}. 
For $(A,f)$ in this set, 
\cref{RepCorrToGivenSolProp} gives us a primitive representation $(Q,(x,y))$ of $A$ by a quadratic form of discriminant $D$ such that the associated solution is $f$.
By \cref{StandardMembersProp} there exist integers $\gamma \leq \delta$ and a matrix $M \in \mathrm{SL}_2(\F_q[u])$ such that
\begin{equation}
v_{(x,y)}^Q \star M = v_{\gamma, \delta}.
\end{equation}
\cref{FormValuationEquivarianceProp} then implies that
\begin{equation} \label{MstandardizedValuationEq}
v_{(x,y) \star M}^{Q \star M} = v_{\gamma, \delta}.
\end{equation}
\cref{StandardValuationImpliesStandardRepresentation} tells us that the representation $(Q', (x',y')) = (Q \star M, (x,y) \star M)$ is a standard representation of $A$ with weight
\begin{equation} \label{InverseWeightEq}
\omega(Q', (x',y')) = 
\begin{cases}
\frac{1}{q^3-q} &\gamma = \delta \\
\frac{1}{(q-1) q^{\delta - \gamma +1}} &\gamma<\delta.
\end{cases}
\end{equation}
It follows from \cref{GenAssocSolsEquivaCor} that our function maps $(Q', (x',y'))$ to $(A,f)$ so its image is indeed given by \cref{IndefDoubleFuncImageEq}.

By \cref{GenAssocSolsEquivaCor}, 
the preimage of $(A,f)$ under our function consists of the standard representations of $A$ that are equivalent to $(Q', (x',y'))$.
These are parametrized by matrices $N \in \mathrm{SL}_2(\F_q[u])$ for which the representation
\begin{equation}
(Q' \star N, (x',y') \star N)
\end{equation}
is standard.
By \cref{StandardValuationImpliesStandardRepresentation} and \cref{StandardVRepresentationImpliesStandardValuation} this is equivalent to the valuation
\begin{equation}
v_{(x',y') \star N}^{Q' \star N}
\end{equation}
being standard.
In view of \cref{FormValuationEquivarianceProp} and \cref{MstandardizedValuationEq}, 
we are looking for the set of all $N \in \mathrm{SL}_2(\F_q[u])$ for which the valuation
\begin{equation}
v_{(x',y')}^{Q'} \star N = v_{\gamma, \delta} \star N
\end{equation}
is standard.

Using the uniqueness part of \cref{StandardMembersProp}, we see that the valuation above is standard if and only if
\begin{equation}
v_{\gamma, \delta} \star N = v_{\gamma, \delta}.
\end{equation}
We conclude from \cref{FiniteStabValProp} that the number of elements in the preimage of $(A,f)$ is 
\begin{equation}
\#\{N \in \mathrm{SL}_2(\mathbb F_q[u]) : v_{\gamma, \delta} \star N = v_{\gamma, \delta}\} = 
\begin{cases}
q^3 - q &\gamma = \delta \\
(q-1)q^{\delta - \gamma + 1} &\gamma < \delta.
\end{cases}
\end{equation}
At last, note that the above is the inverse of the weight of the representation $(Q',(x',y'))$ given in \cref{InverseWeightEq}.
\end{proof}

\begin{notation} \label{LowerNotsNotation}

Following \cref{StandardIndefiniteDef}, for an indefinite polynomial $D$ in $\F_q[u]$ we set
\begin{equation*}
\begin{split}
S(D) = \Bigg\{(s,a,b,c) \in \mathbb{Z} \times \F_q[u]^3 : &\deg(a) \leq \frac{\deg(D)}{2} - s, \ \deg(b) \leq \frac{\deg(D)}{2} \\
&\deg(c) \leq \frac{\deg(D)}{2} + s, \ ac - \frac{b^2}{4} = D, \ s \geq 0 \Bigg\}.
\end{split}
\end{equation*}
For $(s,a,b,c) \in S(D)$ let $a_{(0)}, b_{(0)}, c_{(0)}$ be the coefficients of $a,b,c$ in degrees
\begin{equation}
\frac{\deg(D)}{2} - s, \quad \frac{\deg(D)}{2}, \quad \frac{\deg(D)}{2} + s.
\end{equation} 
For a standard representation $ax^2 + bxy + cy^2$ of a degree $n$ polynomial, 
we denote by $x_{(0)}, y_{(0)}$ the coefficients of $x,y$ in degrees 
\begin{equation}
\frac{n}{2} - \frac{\deg(D)}{4} + \frac{s}{2}, \quad \frac{n}{2} - \frac{\deg(D)}{4} - \frac{s}{2}.
\end{equation}

\end{notation}

\begin{cor} \label{NumberStandardIndefiniteFormsBound}

For every $\epsilon > 0$ we have 
\begin{equation}
\sum_{(s,a,b,c) \in S(D)} \omega_s \ll |D|^{\frac{1}{2} + \epsilon}.
\end{equation}

\end{cor}

\begin{proof}

For $i \in \mathbb{Z}/2\mathbb{Z}$ set $S(D;i) = \{(s,a,b,c) \in S(D) : s \equiv i \ \mathrm{mod} \ 2\}$, 
and note that it is enough to obtain the bound
\begin{equation}
\sum_{(s,a,b,c) \in S(D;i)} \omega_s \ll |D|^{\frac{1}{2} + \epsilon}.
\end{equation}

Fix $i \in \mathbb{Z}/2\mathbb{Z}$, and let $(s,a,b,c) \in S(D;i)$.
Arguing as in the proof of the first claim in \cref{ParityofsProp}, we find that (at least) one of the coefficients $a_{(0)}, b_{(0)}, c_{(0)}$ is nonzero.
Therefore choosing $x_{0}, y_{0} \in \F_q$ via
\begin{equation}
\left(x_{0},y_{0}\right) = 
\begin{cases}
(1,0) &a_{(0)} \neq 0 \\
(0,1) &a_{(0)} = 0, \ c_{(0)} \neq 0 \\
(1,1) &a_{(0)} = 0, \ c_{(0)} = 0, \ b_{(0)} \neq 0
\end{cases}
\end{equation}
we see that
\begin{equation}
a_{(0)}x_{0}^2 + b_{(0)}x_{0}y_{0} + c_{(0)}y_{0}^2 \neq 0.
\end{equation}

Let $n \geq \frac{\deg(D)}{2} + s$ be an integer with
\begin{equation}
n \equiv \frac{\deg(D)}{2} + i \mod 2,
\end{equation} 
and let $x,y \in \F_q[u]$ be coprime polynomials with  
\begin{equation*}
\deg(x) \leq \frac{n}{2} - \frac{\deg(D)}{4} + \frac{s}{2}, 
\ \deg(y) \leq \frac{n}{2} - \frac{\deg(D)}{4} - \frac{s}{2}, \
x_{(0)} = x_0, \ y_{(0)} = y_0.
\end{equation*} 
It follows from our choices that the polynomial
$
A = ax^2 + bxy + cy^2
$
has degree $n$, so this representation is standard.
From the count of coprime pairs $(x,y)$ in \cite[Proof of Lemma 7.3]{ABR}, it follows that the weighted number of standard representations of degree $n$ polynomials by forms of discriminant $D$ is $\gg$ 
\begin{equation*}
\sum_{(s,a,b,c) \in S(D;i)} \omega_{s} q^{\frac{n}{2} - \frac{\deg(D)}{4} + \frac{s}{2}} q^{\frac{n}{2} - \frac{\deg(D)}{4} - \frac{s}{2}} =
q^{n - \frac{\deg(D)}{2}} \sum_{(s,a,b,c) \in S(D;i)} \omega_{s}.
\end{equation*}

On the other hand, using \cref{QuadraticCongNot}, the number of such representations is $\ll$
\begin{equation}
\sum_{A \in \mathcal{M}_n} \rho(A;F)
\end{equation}
in view of \cref{QuadraticCor}.
By \cref{AverageNumberOfSolsQuadCongProp} the above is $\ll q^n |D|^{\epsilon}$ so 
\begin{equation*}
\sum_{(s,a,b,c) \in S(D;i)} \omega_{s} \ll q^{\frac{\deg(D)}{2}} |D|^{\epsilon} = |D|^{\frac{1}{2} + \epsilon} 
\end{equation*} 
as desired.
\end{proof}

\section{Primes in quadratic sequences}

We state and prove a uniform version of \cref{MainRes}.

\begin{thm} \label{UniformBHstatement}

Fix $0 < \delta \leq 1$. 
Let $p$ be an odd prime, let 
\begin{equation} \label{uniform-q-assumption}  
q > \max \left\{ \left( 16 p e \delta^{-1}  \right)^3 ,  \left (  96 e p^2 \delta^{-1}   \right )^2   \right\}
\end{equation} 
be a power of $p$, and let  
\begin{equation} \label{gammaPowerSaving}
\gamma = 1 - \frac{\delta}{12p-1}\min \left\{ 2 - 6  \log_q \left( 16 p e \delta^{-1}  \right)   , \frac{ 1}{2p} - \frac{\log_q  \left (  96 e p^2 \delta^{-1}   \right )}{p} \right\}. 
\end{equation}

Let $d$ be a nonnegative integer and $D$ a polynomial in $\mathbb F_q[u]$ with 
\begin{equation}
\deg(D) \leq 2d(1-\delta).
\end{equation}
Let $F(T)=T^2+D$, and assume that $F$ is irreducible. Then
\begin{equation} \label{UniformBHsavings}
\sum_{f \in \mathcal{M}_d} \Lambda(f^2+D) = \mathfrak{S}_q(F )q^d + O(q^{\gamma d}), \quad d \to \infty,
\end{equation}
with the implied constant depending only on $\delta$ and $q$.
\end{thm}

Note that $\gamma<1$ in view of \cref{uniform-q-assumption}, so \cref{UniformBHsavings} always gives a power saving. 
If $q >(96 ep^2 )^2 =  \max \{ (16pe)^3, (96 e p^2)^2 \} $ we can choose $\delta<1$ satisfying \cref{uniform-q-assumption},
and in this way obtain a power saving bound for $d$ sufficiently large depending on $\deg (D)$.  
Specifically, we obtain a power savings of  $\frac{ 1}{2p (12p-1) } \left ( 1- 2 \log_q \left( 96 e p^2 \right) \right) $ as this term always dominates in \cref{gammaPowerSaving}.

\begin{proof}

The identity $\Lambda = -1 * (\mu \cdot \deg)$ expressing von Mangoldt in terms of M\"{o}bius gives
\begin{equation} \label{ConvId}
\Lambda(F(f)) =  
- \sum_{k=1}^{2d} k \mathop{\sum_{A \in \mathcal{M}_k} \sum_{B \in \mathcal{M}_{2d-k}}}_{AB = F(f)} \mu(A)
\end{equation}
for any polynomial $f \in \mathcal{M}_d$. 
Summing \cref{ConvId} over all degree $d$ monic polynomials $f \in \F_q[u]$ we get
\begin{equation} \label{InitialBatemanHornSum}
\sum_{f \in \mathcal{M}_d} \Lambda(F(f)) = 
- \sum_{k=1}^{2d} k \sum_{f \in \mathcal{M}_d} 
\mathop{\sum_{A \in \mathcal{M}_k} \sum_{B \in \mathcal{M}_{2d-k}}}_{AB = F(f)} \mu(A).
\end{equation}

Fix $\epsilon = \epsilon(q) \in (0,\delta/4)$. The contribution of the range $(1 + \epsilon)d \leq k \leq 2d$ is 
\begin{equation} \label{QuadrapleMobSumEq}
- \sum_{(1+ \epsilon)d \leq k \leq 2d} k
\sum_{B \in \mathcal{M}_{2d-k}} \sum_{\substack{g \in \F_q[u] \\ \deg(g) < 2d-k \\ F(g) \equiv 0 \ \mathrm{mod} \ B}} \
\sum_{\substack{f \in \mathcal{M}_d \\ f \equiv g \ \mathrm{mod} \ B}}
\mu \left(\frac{f^2 + D}{B} \right).
\end{equation}
Writing $f = g + CB$ we get
\begin{equation*}
- \sum_{(1+ \epsilon)d \leq k \leq 2d} k
\sum_{B \in \mathcal{M}_{2d-k}} \sum_{\substack{g \in \F_q[u] \\ \deg(g) < 2d-k \\ F(g) \equiv 0 \ \mathrm{mod} \ B}} \  
\sum_{\substack{C \in \mathcal{M}_{k-d}}}
\mu \left( BC^2 + 2gC + \frac{g^2 + D}{B} \right).
\end{equation*}
We note that the quadratic polynomials 
\begin{equation}
G(T) = BT^2 + 2gT + \frac{g^2 + D}{B} \in \F_q[u][T]
\end{equation}
in the M\"{o}bius above are separable in the variable $T$ for every $B,g$.
Indeed the discriminant of $G$ in $T$ is
\begin{equation}
(2g)^2 - 4B\frac{g^2+D}{B} = 4g^2 - 4(g^2 + D) = -4D \neq 0.
\end{equation}

We can therefore apply \cref{MobiusVStraceThm} with
\begin{equation*}
c_1 = \max\{2d-k, \deg(D) - (2 d - k ) \}, \ c_2=0, \ g=1, \ \mathcal{I} = \mathcal{M}_{k-d}, \ \beta_1 = (1 + 2\gamma_1)^2
\end{equation*}
obtaining
\begin{equation*}
\begin{split} 
\sum_{\substack{C \in \mathcal{M}_{k-d}}}
\mu \left( BC^2 + 2gC + \frac{g^2 + D}{B} \right) &\ll  q^{ (k-d) (1-\alpha_1) }  \beta_1^{ 2  \max \{2d-k, \deg (D) - (2 d - k ) \} }  \\
&\leq q^{ (k-d) (1-\alpha_1) }\left(   \beta_1^{ 2 (2d-k)}+ \beta_1^{ 2 \deg (D) -2 (2d-k)} \right)  \\ &= q^{ (k-d) (1-\alpha_1) } \left(  (1+ 2\gamma_1)^ { 4 (2d-k)}  + (1+2\gamma_1)^{ 4 \deg (D) - 4 (2d-k)} \right) 
\end{split} 
\end{equation*}
for any $\alpha_1$ and $0 < \gamma_1 \leq 1$ satisfying \eqref{alpha-gamma-assumption}, namely
\begin{equation} \label{alpha-gamma-assumptionONE}
0 < \alpha_1 < \frac{1}{2p} + \frac{\log_q \gamma_1}{p} - 2\log_q(1 + 2\gamma_1).
\end{equation}

By \cref{AverageNumberOfSolsQuadCongProp}, the contribution of each $(1 + \epsilon)d \leq k \leq 2d$ to \cref{QuadrapleMobSumEq} is then 
\begin{equation*} 
\begin{split}
&\ll  q^{ (k-d) (1-\alpha_1) }  \left(  (1+ 2\gamma_1)^ { 4 (2d-k)}  + (1+2\gamma_1)^{ 4 \deg (D) - 4 (2d-k)} \right)  \sum_{B \in \mathcal{M}_{2d-k}} \sum_{\substack{g \in \F_q[u] \\ \deg(g) < 2d-k \\ F(g) \equiv 0 \ \mathrm{mod} \ B}}  1
\\ &  \ll q^{ (k-d) (1-\alpha_1) }  \left(  (1+ 2\gamma_1)^ { 4 (2d-k)}  + (1+2\gamma_1)^{ 4 \deg (D) - 4 (2d-k)} \right) \cdot|D|^{o(1)}  q^{ 2d-k}  \\
& =|D|^{o(1)} q^d  q^{ - \alpha_1 (k-d) } (1+ 2\gamma_1)^{ 4 (2d-k)} +|D|^{o(1)} q^d  q^{ - \alpha_1 (k-d) } (1+ 2\gamma_1)^{ 4 \deg (D) - 4 (2d-k)}. 
\end{split} 
\end{equation*}
The first term is exponentially decreasing as a function of $k$. Hence for $k \geq (1+\epsilon)d$ it is $\ll$
\begin{equation}\label{alpha-1-error-equation}
|D|^{o(1)}  q^{d} q^{ - \alpha_1 \epsilon d} (1+ 2 \gamma_1)^ { 4 (1-\epsilon)d}. 
\end{equation} 

Therefore, in order to obtain power savings, we need 
\begin{equation}\label{alpha-1-savings-equation} 
q^{ - \alpha_1 \epsilon } (1 + 2\gamma_1)^{ 4 (1-\epsilon)} < 1. 
\end{equation} 
If we assume for the moment that \cref{alpha-1-savings-equation} holds, 
using our assumptions that $ \epsilon \leq \frac{\delta}{4} \leq  \frac{1}{2} $, we get 
\begin{equation} 
q^{ - \alpha_1  } (1 + 2\gamma_1)^{ 4} < 1 
\end{equation}
so the second term $|D|^{o(1)} q^d  q^{ - \alpha_1 (k-d) } (1+ 2\gamma_1)^{ 4 \deg (D) - 4 (2d-k)}$
is exponentially decreasing as a function of $k$, and thus is $\ll$
\begin{equation}
|D|^{o(1)}  q^{d} q^{ - \alpha_1 \epsilon d} (1+ 2 \gamma_1)^ { 4 \deg (D) - 4  (1-\epsilon)d}
\end{equation} 
which is also bounded by \eqref{alpha-1-error-equation} since 
\begin{equation}
\deg (D) \leq 2 (1-\delta) d \leq 2 (1-\epsilon ) d.
\end{equation} 

Consequently, the contribution of the range $(1+\epsilon)d \leq k \leq 2d$ to \cref{InitialBatemanHornSum} is $\ll$
\begin{equation}\label{alpha-1-nice-equation} 
d^2 |D|^{o(1)}  q^{d} q^{ - \alpha_1 \epsilon d} (1+ 2 \gamma_1)^ { 4 (1-\epsilon)d} 
\end{equation} 
as long as we have \eqref{alpha-gamma-assumptionONE} and  \eqref{alpha-1-savings-equation}.  
We now specialize 
\begin{equation*} 
\epsilon = \frac{\delta}{12p-1}, \gamma_1 =\frac{ \epsilon}{ 2 (4 - 2 \epsilon) p }  =  \frac{\delta}{ 2 p (48 p-4 -2\delta) } , \alpha_1 = \frac{1}{2p} + \frac{\log_q \gamma_1}{p} - \frac{ 4 \gamma_1}{\log q} 
\end{equation*}
which satisfies the second inequality in \eqref{alpha-gamma-assumptionONE} because  
\begin{equation} 
\alpha_1 =  \frac{1}{2p} + \frac{\log_q \gamma_1}{p} - \frac{ 4 \gamma_1}{\log q} <  \frac{1}{2p} + \frac{\log_q \gamma_1}{p} -  2\log_q (1 +2 \gamma_1 ) 
\end{equation} 
and satisfies \eqref{alpha-1-savings-equation}, and thus the first inequality in \cref{alpha-gamma-assumptionONE}, because 
\begin{equation} \label{finding-alpha-1-savings} 
\begin{split}
q^{ - \alpha_1 \epsilon } (1 + 2\gamma_1)^{ 4 (1-\epsilon)} &<  q^{ - \alpha_1 \epsilon} e ^{ 8 \gamma_1 (1-\epsilon)} = 
q^{ - \frac{\epsilon}{2p} } \gamma_1 ^{ -\frac{\epsilon}{p}}  e^{ 4 \epsilon  \gamma_1  + 8 (1-\epsilon) \gamma_1} \\
&= q^{ - \frac{\epsilon}{2p} } \gamma_1 ^{ -\frac{\epsilon}{p}}  e^{ \frac{\epsilon}{p} } = \left( \frac{q^{\frac{1}{2}} \delta  }{ 2p e (48p-4-2\delta)} \right)^{-\frac{\epsilon}{p}}  
\end{split}
\end{equation}
 which is $<1$ since $q> (96 e p^2   \delta^{-1}) ^2 >  (2 pe (48p-4-2\delta) \delta^{-1} )^2$ by \cref{uniform-q-assumption}.
 
Applying \eqref{finding-alpha-1-savings}, and using our assumption $|D| \leq q^{2 d (1-\delta)}$,  
which guarantees that $d^2 |D|^{o(1)}$ is bounded by any exponential in $d$, 
we conclude that the total contribution of the range $(1+\epsilon)d \leq k \leq 2d$ to \cref{InitialBatemanHornSum} is $\ll$
\begin{equation} 
q^d  \left( \frac{ 2p e (48p-4-2\delta)}{q^{\frac{1}{2} } \delta  } \right)^{d \frac{\epsilon}{p} }  = \left( q^d \right)^{ 1 - \frac{ \delta} { 2 p (12p-1) } ( 1 - 2 \log_q (  \frac{ 2 p e (48p -4 -2\delta)}{\delta})) }.
\end{equation}  
This is bounded by $q^{ \gamma d}$, for our choice of $\gamma$ in \cref{gammaPowerSaving}.

The contribution of the range $k < (1 + \epsilon)d$ to \cref{InitialBatemanHornSum} is 
\begin{equation}
-\sum_{1 \leq k < (1 + \epsilon)d} k\sum_{\substack{A \in \mathcal{M}_k}} \mu(A) \rho_d(A; F)
\end{equation}
so by \cref{RhoFormula} from $k \leq d$ we get 
\begin{equation} \label{FirstMainTermEq}
-\sum_{k=1}^d kq^{d-k}\sum_{\substack{A \in \mathcal{M}_k}} \mu(A) \rho(A;F).
\end{equation}
By \cref{QuadCongCor}, from $d < k < (1 + \epsilon)d$ we have
\begin{equation} \label{TwoTermsMiddleRangeEq}
\begin{split}
&\sum_{d < k < (1 + \epsilon)d} -kq^{d-k} \sum_{A \in \mathcal{M}_k} \mu(A) \rho(A;F) + \\
&\sum_{d < k < (1 + \epsilon)d} -kq^{d-k} \sum_{\substack{h \in \F_q[u] \setminus \{0\} \\ \deg(h) < k-d}} e \left( \frac{-hu^d}{A} \right) 
\sum_{A \in \mathcal{M}_k} \mu(A) 
\sum_{\substack{f \in \F_q[u]/(A) \\ F(f) \equiv 0 \ \mathrm{mod} \ A }} e \left( \frac{hf}{A} \right). 
\end{split}
\end{equation}

Uniting the first term in \cref{TwoTermsMiddleRangeEq} with \cref{FirstMainTermEq}, and applying \cref{SingularSeriesProp}, we get 
\begin{equation} \label{TotalMainTermEq}
-\sum_{1 \leq k < (1 + \epsilon)d} kq^{d-k}\sum_{\substack{A \in \mathcal{M}_k}} \mu(A) \rho(A; F) = \mathfrak{S}_q(F)q^d + 
o \left( q^{\frac{d}{2}} \right)
\end{equation}
which gives us our main term and an admissible error term.
The second term in \cref{TwoTermsMiddleRangeEq} is $\ll$
\begin{equation} \label{SumBeforeQuadraticTransform}
d^2 \sup_{d < k < (1 + \epsilon)d} \ \sup_{\substack{h \in \F_q[u]}} \left|
\sum_{A \in \mathcal{M}_k} \mu(A) 
\sum_{\substack{f \in \F_q[u]/(A) \\ F(f) \equiv 0 \ \mathrm{mod} \ A}} e \left( \frac{hf}{A} \right) \right|.
\end{equation}

By \cref{QuadraticCor}, and \cref{ExponentiationChangeProp}, the sum in absolute value above, in the definite case, equals
\begin{equation} \label{DefiniteQuadraticTransform}
\sum_{\substack{a,b,c \in \F_q[u] \\ \deg (c) \geq \deg (a) > \deg (b) \\ 4ac - b^2 = 4D}}
\frac{1 + q \cdot \mathbf{1}_{\deg(c) > \deg(a)} }{q^2-1} 
\sum_{\substack{x,y \in \F_q[u] \\ \gcd(x,y) =1, \ y \neq 0 \\ ax^2 + bxy + cy^2 \in \mathcal{M}_k}} 
\mu(ax^2 + bxy + cy^2) e \left( \frac{h \overline{x}}{y} \right) 
\end{equation}
where we have excluded $y=0$ because then we have a factor of $\mu(ax^2)$ which is zero.
Indeed if it were nonzero, then $x$ would be a nonzero constant,
so from \cref{DegaDefiniteUpperBound} and our initial assumption on $\deg_u(F)$ we would get
\begin{equation} \label{yiszerocontreq}
\begin{split}
k &= \deg(A) = \deg(ax^2 + bxy + cy^2) = \deg(ax^2) \\
&= \deg(a) \leq \frac{\deg(D)}{2} = \frac{\deg_u(F)}{2} \leq \frac{2d(1 - \delta)}{2} < d
\end{split}
\end{equation}
which is impossible because we are in the range $d < k < (1 + \epsilon)d$.


Let us now check that the assumptions of \cref{ExponentiationChangeProp} are indeed met here,
namely that \cref{FirstDeghIneq} and \cref{SecondDeghIneq} hold.
Using the second line in \cref{yiszerocontreq}, and the fact that $\deg(b) < \deg(a)$ we get
\begin{equation}
\deg(h) < k - d = \deg(A) - d \leq \deg(A) - \deg(a) \leq \deg(A) - \deg(b) - 1
\end{equation}
so \cref{FirstDeghIneq} is verified.
By \cref{DefiniteDeg} and \cref{yiszerocontreq} we have
\begin{equation*}
\begin{split}
&\deg(A) + \deg(y) - \deg(a) - \deg(x) = \\ 
&\max\{\deg(a) + 2\deg(x), \deg(c) + 2\deg(y)\} + \deg(y) - \deg(a) - \deg(x) \geq \\ 
&\max\{\deg(x) + \deg(y), \deg(c) - d(1 - \delta) - \deg(x) \}.
\end{split}
\end{equation*}

If toward a contradiction \cref{SecondDeghIneq} fails, then the above is at most $\deg(h) + 1$ which is bounded by $k-d$.
The latter does not exceed $\epsilon d$, so
\begin{equation}
\deg(x),\deg(y) \leq \epsilon d, \ \deg(c) \leq d(1 - \delta) + \deg(x) + \epsilon d
\end{equation}
and thus
\begin{equation*}
\begin{split}
k &= \deg(A) = \deg(ax^2 + bxy + cy^2) \\ &= \max\{\deg(a) + 2\deg(x), \deg(c) + 2\deg(y)\}
\\ &\leq  \max\{d(1 - \delta) + 2\epsilon d, d(1 - \delta) + 4\epsilon d \} = d(1 - \delta + 4\epsilon) \leq k (1 - \delta + 4\epsilon)
\end{split}
\end{equation*}
a contradiction since $\epsilon < \delta/4$.
Our invocation of \cref{ExponentiationChangeProp} is thus justified.


We then apply the triangle inequality to the sum over $a,b,c,y$ in \cref{DefiniteQuadraticTransform}, to get $\ll$
\begin{equation} \label{abcxyDefiniteSumEq}
\sum_{\substack{a,b,c \in \F_q[u] \\ \deg (c) \geq \deg (a) > \deg (b) \\ 4ac - b^2 = 4D}} 
\sum_{\substack{y \in \F_q[u] \\ y \neq 0 \\ \deg(y) \leq \frac{k - \deg(c)}{2}}}
\left| \sum_{\substack{ x \in S_{a,b,c,y} \\ \gcd(x,y) = 1}} \mu(ax^2 + bxy + cy^2) e \left( \frac{h \overline{x}}{y} \right)  \right|
\end{equation}
where
\begin{equation}
S_{a,b,c,y} = \{ x \in \F_q[u] : ax^2 + bxy + cy^2 \in \mathcal{M}_k \}.
\end{equation}

We claim that for $a,b,c,y$ as above, the set $S_{a,b,c,y}$ is a disjoint union of at most two intervals in $\F_q[u]$, the degree of which is at most $(k-\deg(a))/2$.
To show this, recall from \cref{DefiniteDeg} that since $D = ac - b^2/4$ is definite, we have
\begin{equation}
\begin{split}
k &=  \deg(ax^2 + bxy + cy^2) \\
&= \max \{ \deg (a) + 2 \deg (x), \deg (c) + 2 \deg (y)\} > \deg(bxy).
\end{split}
\end{equation}

In case $\deg(c) + 2\deg(y) < k$, the leading coefficient $1$ of the monic polynomial $ax^2 + bxy + cy^2$ is the leading coefficient of $ax^2$, that is the leading coefficient of $a$ times the square of the leading coefficient of $x$.
Hence, if the leading coefficient of $a$ is not a square in $\F_q^\times$, the set $S_{a,b,c,y}$ is empty.
Otherwise, if the leading coefficient of $a$ is $\lambda^2$, for some $\lambda \in \F_q^\times$, then
\begin{equation}
S_{a,b,c,y} = \lambda^{-1} \cdot \mathcal{M}_{\frac{k - \deg(a)}{2}} \cup (-\lambda^{-1}) \cdot \mathcal M_{\frac{k - \deg(a)}{2} } .
\end{equation}

Suppose now $\deg(c) + 2\deg(y) = k$. 
If $cy^2$ is monic, we have
\begin{equation}
S_{a,b,c,y} = \left\{x \in \F_q[u] : \deg(x) < \frac{k-\deg(a)}{2} \right\}.
\end{equation}
If $cy^2$ is not monic, then the set $S_{a,b,c,y}$ is empty in case $k \not\equiv \deg(a) \ \mathrm{mod} \ 2$,
while in case $k \equiv \deg(a) \ \mathrm{mod} \ 2$, denoting by $f_0$ the leading coefficient of a polynomial $f \in \F_q[u]$, we get
\begin{equation}
S_{a,b,c,y} = \left\{x \in \F_q[u] : \deg(x) = \frac{k - \deg(a)}{2}, \ x_0^2 = \frac{1 -  c_0 y_0^2}{a_0} \right\}
\end{equation}
which is a disjoint union of two (possibly empty) intervals corresponding to polynomials with leading coefficient equal to one of the square roots of $(1-c_0y_0^2)/a_0$ in $\F_q^\times$. 
This concludes the verification of our claim in all cases. 

We use \cref{NumberStandardDefiniteFormsBound} to bound the number of triples $(a,b,c)$ in the outer sum of \cref{abcxyDefiniteSumEq}, and recall from \cref{LowerBoundDegCRem} that $\deg(c) \leq \deg(D)$, so it suffices to control
\begin{equation} \label{CoreSumEq}
|D|^{\frac{1}{2} + o(1)}
\max_{\substack{a,b,c \in \F_q[u] \\ 4ac - b^2 = 4D \\ \deg(b) < \deg(a) \leq \deg(c) \leq \deg(D)}}
\sum_{\substack{y \in \F_q[u] \\ y \neq 0 \\  \deg(y) \leq \frac{k - \deg(c)}{2} }}
\left| \sum_{\substack{ x \in \mathcal{I}_{a,b,c,y} \\ \gcd(x,y) = 1}} \mu(ax^2 + bxy + cy^2) e \left( \frac{h \overline{x}}{y} \right)  \right|
\end{equation}
where $\mathcal{I}_{a,b,c,y}$ is an interval in $\F_q[u]$ with
\begin{equation}
  \deg(\mathcal{I}_{a,b,c,y})  \leq \frac{k - \deg(a) }{2}
\end{equation}
for all $a,b,c,y$.

Fixing $a,b,c$, we define the polynomial
\begin{equation}
F_y(T) = aT^2 + byT + cy^2 \in \F_q[u][T]
\end{equation}
for any $y \in \F_q[u] \setminus \{0\}$, and note that its discriminant in the variable $T$ is
\begin{equation}
(by)^2 - 4acy^2 = y^2(b^2 - 4ac) = (2y)^2 \cdot (-D) \neq 0
\end{equation}
so $F_y$ is a separable polynomial in $T$. 

Setting 
\begin{equation} 
n = \frac{k - \deg(c)}{2} , \hspace{10pt}  c_1 = k ,  \hspace{10pt}  c_2 = - \frac{ k- \deg(a)}{2} , \hspace{10pt}  c_3 =\frac{k - \deg(a) }{2} 
\end{equation}  we see that the coefficient of $T^i$ in $F_y(T)$ has degree at most $c_1 + i c_2$ for $i \in \{0,1,2\}$.  
Therefore, \cref{MobiusBeatsKloostermanCor} allows us to bound \cref{CoreSumEq} by 
\begin{equation*} 
|D|^{\frac{1}{2} + o(1)}
\max_{\substack{a,b,c \in \F_q[u] \\ 4ac - b^2 = 4D \\ \deg(b) < \deg(a) \leq \deg(c) \leq \deg(D)}}
 q^{\frac{k - \deg(c)}{2} }  q^{\frac{k - \deg(a) }{2}   (1 - \alpha_2 ) }  \beta_2^{ 2k - 3  \frac{ k- \deg(a)}{2} }     (1+3\gamma_2)^{ \frac{k - \deg(c)}{2}   } \end{equation*}
where $\alpha_2$ and $0 < \gamma_2 \leq 1$ satisfy \cref{weird-alpha-gamma-assumption}, namely
\begin{equation} \label{weird-alpha-gamma-assumptionTWO}
0 < \alpha_2 < \min \left\{\frac{1}{2} - 10\log_q(1 + 2\gamma_2) + \log_q(1 + 3\gamma_2), \frac{1}{2p} +  \frac{\log_q \gamma_2}{p} - 2\log_q(1 + 2\gamma_2) \right\}
\end{equation}
and $ \beta_2 = (1+2\gamma_2)^2$.


We can separate the terms involving $k$ from those involving $a,b,c$, rewriting things as
\begin{equation*} 
\begin{split}
  &  |D|^{\frac{1}{2} + o(1)}
\max_{\substack{a,b,c \in \F_q[u] \\ 4ac - b^2 = 4D \\ \deg(b) < \deg(a) \leq \deg(c) \leq \deg(D)}} q^{ - \frac{ \deg (D)}{2} + \frac{ \deg (a)}{2} \alpha_2 }   \beta_2^{ 3 \frac{ \deg(a)}{2} }  (1 + 3 \gamma_2)^{ - \frac{ \deg (c)}{2 } }  
   q^{ k (1 - \frac{\alpha_2}{2} ) }  \beta_2^{ \frac{k }{2} }  
   (1+3\gamma_2)^{ \frac{k }{2} }. 
\end{split} 
\end{equation*}
Observe that the terms depending on $\deg (a)$ are increasing, 
and those depending on $\deg (c)$ are decreasing. 
We may therefore replace $\deg (a)$ by its upper bound $\frac{\deg (D)}{2}$ and $\deg (c)$ by its lower bound $\frac{\deg (D)}{2} $, obtaining
\begin{equation} \label{FiveLineEq}
\begin{split} 
&\ll q^{ (\frac{1}{2} + o(1)) \deg (D) } q ^{ - \frac{ \deg (D)}{2} +  \frac{ \deg (D)}{4} \alpha_2 }   \beta_2^{  \frac{ 3 \deg(D)}{4} }  (1 + 3 \gamma_2)^{ - \frac{ \deg (D)}{4} }     q^{ k (1 - \frac{\alpha_2}{2} ) }  \beta_2^{ \frac{k }{2} }  
   (1+3\gamma_2)^{ \frac{k }{2} } \\
   &=   \left( q^{    \frac{\alpha_2}{4} + o(1)}  \beta_2^{ \frac{3}{4}}  (1+3 \gamma_2)^{- \frac{1}{4}} \right)^{ \deg (D)}  \left( q^{  1 - \frac{\alpha_2}{2} } \beta_2^{ \frac{1}{2} } (1+3\gamma_2)^{ \frac{1}{2}} \right)^{ k }  \\
   &= \left( q^{   \frac{\alpha_2}{4} + o(1)}  (1+ 2 \gamma_2) ^{ \frac{3}{2}}  (1+3 \gamma_2)^{- \frac{1}{4}} \right)^{ \deg (D)}  \left( q^{  1 - \frac{\alpha_2}{2} } (1+ 2 \gamma_2) (1+3\gamma_2)^{ \frac{1}{2}} \right)^{ k }\\
     &\leq  \left( q^{  \frac{\alpha_2}{4} + o(1)}  (1+ 2 \gamma_2) ^{ \frac{3}{2}}  (1+3 \gamma_2)^{- \frac{1}{4}} \right)^{ 2 d (1-\delta) }  \left( q^{  1 - \frac{\alpha_2}{2} } (1+ 2 \gamma_2) (1+3\gamma_2)^{ \frac{1}{2}} \right)^{ d (1+\epsilon)  }\\
& =\left( q^{ 1+   \epsilon - \frac{ ( \delta+ \epsilon)  \alpha_2}{2} + o(1)   }  (1 + 2 \gamma_2)^{ (4 + \epsilon- 3 \delta) }  (1+3 \gamma_2)^{ \frac{ \delta+ \epsilon}{2} } \right)^d.
\end{split}  
\end{equation} 

We now specialize to 
\begin{equation*} 
\epsilon = \frac{ \delta}{ 12p-1}, \gamma_2 = \frac{ 3  \delta} { 4 (12p-1) -  (15p-4) \delta } , \alpha_2 = \frac{1}{2p} + \frac{ \log_q \gamma_2  }{p} -  \frac{ 4 \gamma_2}{\log q}    
\end{equation*} 
which satisfies the third inequality in \cref{weird-alpha-gamma-assumptionTWO} because
\begin{equation*}  
\alpha_2 =   \frac{1}{2p} +  \frac{ \log_q  \gamma_2}{p}- \frac{ 4\gamma_2}{\log q} 
 < \frac{1}{2p} +  \frac{ \log_q  \gamma_2}{p}- 2 \log_q (1 + 2 \gamma_2).
\end{equation*} 

To check that the second inequality in \cref{weird-alpha-gamma-assumptionTWO} holds, we first note that
\begin{equation}
q>  3.57\ldots =  e^{ \frac{84}{66 } }  \geq e ^{ \frac{84}{ 33 (p-1) }} = e^{ \frac{42}{ 4 (12p-1) -(15p-4)  }  \frac{2p} {p-1} }.
\end{equation}
As a result, since $0 < \gamma_2 \leq 1$ we have
\begin{equation*}
\gamma_2^{\frac{1}{p} } \frac{(1+2 \gamma_2)^8}{ (1+3 \gamma_2) } \leq  (1+2 \gamma_2)^7 \leq e^{ 14 \gamma_2} = e^{ \frac{ 42  \delta} { 4 (12p-1) - (15p-4) \delta } } \leq e^{ \frac{42}{ 4 (12p-1) -(15p-4)  }} <   q^{ \frac{p-1}{2p} }. 
\end{equation*}
Taking logarithms to base $q$ gives
\begin{equation}
\frac{ \log_q  \gamma_2}{p}+ 8 \log_q (1 + 2 \gamma_2)  -  \log_q (1+ 3 \gamma_2 )    \leq  \frac{p-1}{2p}
\end{equation}
or equivalently
\begin{equation*} 
\begin{split}
& \frac{1}{2p} +  \frac{ \log_q  \gamma_2}{p}- 2 \log_q (1 + 2 \gamma_2) \leq  \frac{1}{2} - 10\log_q ( 1+ 2 \gamma_2)  + \log_q (1+ 3 \gamma_2 )
\end{split} 
\end{equation*} 
which implies that the second inequality in \cref{weird-alpha-gamma-assumptionTWO} holds as the third does.

Multiplying \cref{FiveLineEq} by the factor $d^2$ from \cref{SumBeforeQuadraticTransform}, 
we absorb it, together with $q^{o(d)}$, into an exponential savings in $d$, so we get
\begin{equation*} 
\begin{split} d^2  &\left( q^{ 1+   \frac{\delta}{12p-1}  - \frac{ 6p \delta   \alpha_2}{12p-1} + o(1)  }  (1 + 2 \gamma_2)^{ (4 - \frac{ 36p-4}{ 12p-1}  \delta) }  (1+3 \gamma_2)^{ \frac{6p \delta }{12p-1}} \right)^d\\
\ll &\left( q^{ 1+   \frac{\delta}{12p-1}  - \frac{ 6p \delta   \alpha_2}{12p-1}   }   e ^{ 2\gamma_2  \left( 4 - \frac{ 36p-4}{ 12p-1}  \delta \right) + 3\gamma_2 \frac{6p \delta }{12p-1}  }  \right)^d\\ 
  = &\left( q^{ 1+   \frac{\delta}{12p-1}  - \frac{ 3 \delta }{12p-1}    }  \gamma_2^{ - \frac{6 \delta}{12p-1}}  e ^{  \frac{ 24 p \gamma_2 \delta}{12p-1} + 2\gamma_2  \left( 4 - \frac{ 36p-4}{ 12p-1}  \delta \right) + 3\gamma_2 \frac{6p \delta }{12p-1}  }  \right)^d\\ 
= &\left( q^{ 1 - \frac{2 \delta}{12p-1} } \gamma_2^{ - \frac{6\delta}{12p-1}}  e^{ \gamma_2   \frac{  8 (12p-1) - (30p-  8) \delta   }{ 12p-1} } \right)^d  \\
= &\left( q^{ 1 - \frac{2 \delta}{12p-1} } \gamma_2^{ - \frac{6\delta}{12p-1}} e^{ \frac{6 \delta}{12p-1}}  \right)^d \\
= &\left(q^d \right)^{ 1 - \frac{2 \delta }{12p-1} \left( 1-    3 \log_q ( e/\gamma_2 ) \right) } \\
= &\left(q^d \right)^{ 1 - \frac{2 \delta }{12p-1} \left( 1-    3 \log_q ( e \frac{   4 (12p-1)-  (15p-4)  \delta }{ 3  \delta} ) \right) }  .
\end{split}   
\end{equation*}  
In particular, by our definition of $\gamma$ in \cref{gammaPowerSaving}, this is bounded by $q^{ \gamma d}$.
This also verifies the first inequality in \cref{weird-alpha-gamma-assumptionTWO}.

In case $D$ is indefinite, we get from \cref{LowerNotsNotation}, \cref{IndefQuadraticCor} and \cref{ExponentiationChangeProp} that the sum in absolute value in \cref{SumBeforeQuadraticTransform} equals
\begin{equation}
\sum_{(s,a,b,c) \in S_D} \omega_s
\sum_{\substack{x,y \in \F_q[u] \\ \gcd(x,y)=1, y \neq 0 \\ \deg(x) \leq \frac{k}{2} - \frac{\deg(D)}{4} + \frac{s}{2} \\ \deg(y) \leq \frac{k}{2} - \frac{\deg(D)}{4} - \frac{s}{2} \\ ax^2 + bxy + cy^2 \in \mathcal{M}_k}}
\mu(ax^2 + bxy + cy^2) e \left( \frac{h\overline{x}}{y} \right).
\end{equation}
The condition $y \neq 0$ is justified here in the same way as in the definite case, only that here we need to refer to \cref{StandardabcIndefEq} instead of \cref{DegaDefiniteUpperBound}.

We check that the assumptions of \cref{ExponentiationChangeProp} are satisfied in this case.
From our initial assumption on $\deg_u(F)$ we get that
\begin{equation*}
\deg(h) \leq k-d-1 = \deg(A) - d-1 < \deg(A) - \frac{\deg(D)}{2} -1 \leq \deg(A) - \deg(b) - 1
\end{equation*}
so \cref{FirstDeghIneq} is satisfied.
Moreover we have
\begin{equation*}
\begin{split}
&\deg(h) + \deg(a) + \deg(x) - \deg(A) < k-d + \frac{\deg(D)}{2} - s + \frac{k}{2} - \frac{\deg(D)}{4} + \frac{s}{2} - k = \\
&\frac{\deg(D)}{4} - \frac{s}{2} - d + \frac{k}{2} \leq \frac{d(1 - \delta)}{2} - d + \frac{(1 + \epsilon)d}{2} = 
\frac{d(\epsilon - \delta)}{2} \leq -1 \leq \deg(y) - 1
\end{split}
\end{equation*}
so \cref{SecondDeghIneq} is satisfied as well.

Arguing as in the definite case, we arrive at
\begin{equation}
\sum_{(s,a,b,c) \in S_D} \omega_s
\sum_{\substack{y \in \F_q[u] \setminus \{0\} \\ \deg(y) \leq \frac{k}{2} - \frac{\deg(D)}{4} - \frac{s}{2}}}
\left| \sum_{\substack{x \in S_{s,a,b,c,y} \\ \gcd(x,y)=1}}
\mu(ax^2 + bxy + cy^2) e \left( \frac{h\overline{x}}{y} \right) \right|
\end{equation}
where
\begin{equation*}
S_{s,a,b,c,y} = \{x \in \F_q[u] : \deg(x) \leq \frac{k}{2} - \frac{\deg(D)}{4} + \frac{s}{2}, \ ax^2 + bxy + cy^2 \in \mathcal{M}_k \}.
\end{equation*}
We can rewrite the set above as
\begin{equation*}
S_{s,a,b,c,y} = \{x \in \F_q[u] : \deg(x) \leq \frac{k}{2} - \frac{\deg(D)}{4} + \frac{s}{2}, \
a_{(0)}x_{(0)}^2 + b_{(0)}x_{(0)}y_{(0)} + c_{(0)}y_{(0)}^2 = 1\}
\end{equation*}
where (for instance) $x_{(0)}, y_{(0)}$ are the coefficients of $x,y$ in degrees
\begin{equation}
\frac{k}{2} - \frac{\deg(D)}{4} + \frac{s}{2}, \quad \frac{k}{2} - \frac{\deg(D)}{4} - \frac{s}{2}
\end{equation}
as in \cref{LowerNotsNotation}. 
Therefore the set $S_{a,b,c,y}$ is a disjoint union of at most two intervals in $\F_q[u]$, 
corresponding to the solutions of the (possibly degenerate) quadratic equation in $x_{(0)}$.

As in the definite case it is thus enough to control
\begin{equation} \label{IndefiniteAbsoluteValuedSumEq}
\sum_{(s,a,b,c) \in S_D} \omega_s
\sum_{\substack{y \in \F_q[u] \setminus \{0\} \\ \deg(y) \leq \frac{k}{2} - \frac{\deg(D)}{4} - \frac{s}{2}}}
\left| \sum_{\substack{x \in \mathcal{I}_{s,a,b,c,y} \\ \gcd(x,y) = 1}}
\mu(F_y(x)) e \left( \frac{h\overline{x}}{y} \right) \right|
\end{equation}
where $F_y(T)$ is the separable polynomial  $aT^2 + byT + cy^2$, 
and $\mathcal{I}_{s,a,b,c,y}$ is an interval in $\F_q[u]$ with
\begin{equation}
\deg(\mathcal{I}_{s,a,b,c,y}) \leq \frac{k}{2} - \frac{\deg(D)}{4} + \frac{s}{2}.
\end{equation}

Applying \cref{MobiusBeatsKloostermanCor} with
\begin{equation*}
n = \frac{k}{2} - \frac{\deg(D)}{4} - \frac{s}{2}, \quad c_1 =k  , \quad c_2 = \frac{ \deg (D)}{4} - \frac{s}{2}  - \frac{k}{2} , \quad c_3 = \frac{k}{2} - \frac{\deg(D)}{4} + \frac{s}{2}
\end{equation*}
we get a bound of 
\begin{equation*}
q^{  \left( \frac{k}{2} - \frac{\deg(D)}{4} - \frac{s}{2}\right) +  \left( \frac{k}{2} - \frac{\deg(D)}{4} + \frac{s}{2} \right) (1- \alpha_2 ) } \beta_2^{ 2 k+3 \left(  \frac{ \deg (D)}{4} - \frac{s}{2}  - \frac{k}{2}\right) }  (\beta_2^{ 0} +1 )  (1+3\gamma_2)^{  \frac{k}{2} - \frac{\deg(D)}{4} - \frac{s}{2} }
\end{equation*}
for the sum over $y$ in \cref{IndefiniteAbsoluteValuedSumEq}. 
The above can be rewritten as
\begin{equation*} 
\begin{split}   
2 \left( q^{ \frac{ 2- \alpha_2 }{2}}  \beta_2^{ \frac{1}{2}}  (1+ 3 \gamma_2)^{\frac{1}{2} }  \right)^k  \left( q^{ - \frac{ 2 - \alpha_2}{4}  }     \beta_2^{ \frac{3}{4}  }  (1+ 3 \gamma_2) ^{ - \frac{1}{4} } \right)^{\deg (D)}  \left( q^{ - \frac{\alpha_2}{2} } \beta_2^{ - \frac{3}{2} } (1+ 3\gamma_2)^{ - \frac{1}{2} } \right)^s \\ \ll     \left( q^{ \frac{ 2- \alpha_2 }{2}}  \beta_2^{ \frac{1}{2}}  (1+ 3 \gamma_2)^{\frac{1}{2} }  \right)^k  \left( q^{ - \frac{ 2 - \alpha_2}{4} }      \beta_2^{ \frac{3}{4}  }  (1+ 3 \gamma_2) ^{ - \frac{1}{4} } \right)^{\deg (D)}
\end{split} 
\end{equation*}
since $s\geq 0$ and the term being raised to the power $s$ is a product of factors that are individually at most $1$, hence is bounded by $1$.  

Summing over $S_D$ we get from \cref{NumberStandardIndefiniteFormsBound} that
\begin{equation*}   
\begin{split} 
&\left( q^{ \frac{ 2- \alpha_2 }{2}}  \beta_2^{ \frac{1}{2}}  (1+ 3 \gamma_2)^{\frac{1}{2} }  \right)^k  \left( q^{ - \frac{ 2 - \alpha_2}{4}   }   \beta_2^{ \frac{3}{4}  }  (1+ 3 \gamma_2) ^{ - \frac{1}{4}}  \right)^{\deg (D)}  \sum_{(s,a,b,c) \in S_D} \omega_s = \\
&\left( q^{ 1 - \frac{\alpha_2 }{2}}  \beta_2^{ \frac{1}{2}}  (1+ 3 \gamma_2)^{\frac{1}{2} }  \right)^k  \left( q^{ \frac{  \alpha_2}{4} +  o(1)  }  \beta_2^{ \frac{3}{4}  }  (1+ 3 \gamma_2) ^{ - \frac{1}{4} } \right)^{\deg (D)}.
\end{split} 
\end{equation*}
This is identical to the bound obtained in the definite case, 
more specifically on the second line of \cref{FiveLineEq}.
We may thus give the same argument (choosing the same $\epsilon, \alpha_2, \gamma_2$), and again obtain a bound which is $\ll q^{ \gamma d}$.
\end{proof}

\begin{remark} 

The optimal value of $\epsilon$ depends on $q,p,\delta$. 
As $\delta$ becomes smaller, the contribution of the range $d < k < (1+\epsilon)d$ becomes more difficult to bound, forcing us to lower $\epsilon$. As $q$ grows, this contribution becomes easier to bound (even compared to the contribution from $k \geq (1+\epsilon ) d$), allowing us to raise $\epsilon$. 

There is likely no closed-form formula for the exact optimal value of $\epsilon$, and if there was it would make our formulas distressingly complicated, so we have chosen to approximate. 
Specifically, we have chosen $\epsilon$ to roughly optimize the range of $q,\delta$ in which we have some savings, rather than to optimize the amount of savings when $q$ is large and $\delta \sim 1$. (This would require a much larger value of $\epsilon$, close to $\frac{1}{4p+1}$, obtaining power savings tending to $\frac{1}{ 8p^2 +4p}$ as $q\to \infty$ and $\delta \to 1$ with $p$ fixed. )

The specific nature of our choice of $\epsilon$ is that it makes the first lower bound in \eqref{uniform-q-assumption} proportional to $\delta^{-3}$. We have chosen $\epsilon$ this way because making that lower bound proportional to $\delta^{-2}$ is impossible, 
requiring $\epsilon=0$. We could choose an intermediate growth rate (the optimum should be roughly $\delta^{-2} \log (\delta^{-1})$), but this would again give a messier formula, for a mild gain. 
\end{remark}

\section{Trace functions vs Primes}

\begin{lem} \label{MobiusProgressionFormula}

For a prime $\pi \in \F_q[u]$ and an integer $k \geq \deg(\pi)$ we have
\begin{equation} \label{TheThreeMobCasesEq}
\sum_{\substack{A \in \mathcal{M}_k \\ \pi \mid A}} \mu(A) =
\begin{cases}
-1 &k \equiv 0 \mod \deg(\pi) \\
q &k \equiv 1 \mod \deg(\pi) \\
0 &\text{otherwise}.
\end{cases}
\end{equation}

\end{lem}

\begin{proof}

We rewrite our sum as
\begin{equation}
\begin{split}
\sum_{B \in \mathcal{M}_{k- \deg(\pi)}} \mu(B\pi) &=
\sum_{\substack{B \in \mathcal{M}_{k- \deg(\pi)} \\ \pi \nmid B}} \mu(B\pi) = 
\mu(\pi) \sum_{\substack{B \in \mathcal{M}_{k- \deg(\pi)} \\ \pi \nmid B}} \mu(B) \\
&= \sum_{\substack{B \in \mathcal{M}_{k - \deg(\pi)} \\ \pi \mid B}} \mu(B) - \sum_{C \in \mathcal{M}_{k - \deg(\pi)}} \mu(C)
\end{split}
\end{equation}
and induct on $k$.
For the base case $k < 2\deg(\pi)$ the last sum over $B$ above is empty, 
so we are only left with minus the sum over $C$ which equals $-1$ in case $k = \deg(\pi)$,
equals $q$ in case $k = \deg(\pi) + 1$, and otherwise vanishes by \cite[Exercise 2.12]{Ros}.
This matches the right hand side of \cref{TheThreeMobCasesEq}, so the base case is established.
If $k \geq 2\deg(\pi)$ then the sum over $C$ vanishes, and the lemma follows from the induction hypothesis.
\end{proof}

We shall now deduce \cref{ShiftedCharacterVsPrimesCor} from \cref{MobiusProgressionFormula}, \cref{intro-trace-interval-bound}, and \cref{LinearMobiusVStraceFunctionThm}.

\begin{proof}

The identity $\Lambda = (\mu \cdot \deg) * (-1)$ gives
\begin{equation} \label{ConvolutedMangoldtCharacterEq}
\sum_{f \in \mathcal{M}_n} \chi(f+h) \Lambda(f) =
-\sum_{k=1}^n k \sum_{A \in \mathcal{M}_k} \mu(A) \sum_{B \in \mathcal{M}_{n-k}} \chi(AB + h).
\end{equation} 

For any $k \leq \zeta n$ and any $A \in \mathcal{M}_k$ that is not divisible by $\pi$, the contribution to \cref{ConvolutedMangoldtCharacterEq} is $\ll$
\begin{equation}
n \sum_{\substack{C \in \F_q[u] \\ \deg(C) < n-k}} \chi(AC + AT^{n-k} + h)
\end{equation}
where $C = B - T^{n-k}$.
Since $\pi \nmid A$, we are in the situation of \cref{DirichletTraceFuncEx},
so we can invoke \cref{intro-trace-interval-bound} and get that the above is $\ll nq^{\frac{n-k}{2}}|\pi|^{\log_q(3)}$.
The contribution from all such $k$ and $A$ is thus $\ll$ 
\begin{equation} \label{LowKrangeBoundEq}
\max_{k \leq \zeta n} n^2 |\mathcal{M}_k| q^{\frac{n-k}{2}} |\pi|^{\log_q(3)} \leq n^2 q^{\frac{n(1 + \zeta)}{2}} |\pi|^{\log_q(3)} \ll q^{\frac{n(1 + \zeta + 2\epsilon)}{2}} |\pi|^{\log_q(3)}
\end{equation}
for any $\epsilon > 0$.

The contribution to \cref{ConvolutedMangoldtCharacterEq} of all $\deg(\pi) \leq k \leq \zeta n$ and all $A \in \mathcal{M}_k$ that are divisible by $\pi$ is $\ll$
\begin{equation*}
\begin{split}
\max_{k \geq \deg(\pi)} n^2 \left| \sum_{\substack{A \in \mathcal{M}_k \\ \pi \mid A}} \mu(A) \sum_{B \in \mathcal{M}_{n-k}} \chi(h)\right| &=
\max_{k \geq \deg(\pi)} n^2 q^{n-k} |\chi(h)| \left| \sum_{\substack{A \in \mathcal{M}_k \\ \pi \mid A}} \mu(A) \right| \\
&\ll \max_{k \geq \deg(\pi)} n^2q^{n-k} \ll n^2q^{n-\deg(\pi)}
\end{split}
\end{equation*} 
in view of \cref{MobiusProgressionFormula}.

The contribution of all $k \geq \zeta n$ to \cref{ConvolutedMangoldtCharacterEq} is $\ll$
\begin{equation*}
\max_{k \geq \zeta n} n^2 \left| \sum_{B \in \mathcal{M}_{n-k}} \sum_{A \in \mathcal{M}_k} \mu(A)  \chi(AB + h) \right| \ll
\max_{\substack{k \geq \zeta n \\ B \in \mathcal{M}_{n-k}}} n^2q^{n-k} \left| \sum_{A \in \mathcal{M}_k} \mu(A)  \chi(AB + h) \right|
\end{equation*} 
and by \cref{LinearMobiusVStraceFunctionThm} this is $\ll$
\begin{equation*}
\max_{k \geq \zeta n} n^2q^{n-k}|\mathcal{M}_k|^{1 - \frac{1}{2p} + \frac{\log_q(2ep)}{p}}|\pi|^{\log_q(3)} \ll
\max_{k \geq \zeta n} q^{n - \frac{k}{2p} + \frac{k\log_q(2ep)}{p}+ \epsilon n}|\pi|^{\log_q(3)}
\end{equation*}
for any $\epsilon > 0$. 
Since $q > 4e^2p^2$ by assumption, the above is largest once $k$ is as small as possible, so we put $k = \zeta n$ and get
\begin{equation} \label{HighKrangeBoundEq}
q^{n(1 - \frac{\zeta}{2p} + \frac{\zeta \log_q(2ep)}{p} + \epsilon)}|\pi|^{\log_q(3)}.
\end{equation}

One readily checks that our choice of $\zeta$ in \cref{TheZetaConstantEq} is such that the bounds in \cref{LowKrangeBoundEq} and \cref{HighKrangeBoundEq} coincide, giving the final bound
\begin{equation}
q^{\frac{n(1 + \zeta)}{2} + \epsilon n} |\pi|^{\log_q(3)} + q^{n(1 + \epsilon) - \deg(\pi)}.
\end{equation}

\end{proof}

\begin{prop} \label{CompleteExponentialInfinitameSum}

Let $\pi \in \F_q[u]$ be a prime, and let $t \colon \F_q[u]/(\pi) \to \mathbb{C}$ be an infinitame trace function arising from a sheaf $\mathcal F$ whose geometric monodromy representation does not admit the trivial representation $\overline{\mathbb{Q}_\ell}$ as a quotient.
For an integer $n \geq \deg(\pi)$ and a polynomial $h \in \F_q[u]$ we then have 
\begin{equation}
\left| \sum_{\substack{f \in \F_q[u] \\ \deg(f) < n}} t(f) e\left( \frac{hf}{\pi}\right) \right| \leq c(t)q^{n}|\pi|^{-\frac{1}{2}}.
\end{equation}

\end{prop}

\begin{proof}

In every residue class mod $\pi$ there are $q^{n - \deg(\pi)}$ polynomials of degree less than $n$, so
\begin{equation} \label{EquidistributionInResidueClassesEq}
\left| \sum_{\substack{f \in \F_q[u] \\ \deg(f) < n}} t(f) e\left( \frac{hf}{\pi}\right) \right| =
\frac{q^n}{|\pi|}\left| \sum_{f \in \F_q[u]/(\pi)} t(f) e\left( \frac{hf}{\pi}\right) \right|.
\end{equation}
Setting $\mathcal{F}_{h} = \mathcal{F} \otimes \mathcal{L}_\psi(hx)$, and using the Grothendieck-Lefschetz trace formula, we get
\begin{equation} \label{TriangleGrothendieckLefschetzEq}
\left| \sum_{f \in \F_q[u]/(\pi)} t(f) e \left( \frac{hf}{\pi} \right) \right| = 
\left| \sum_{y \in \mathbb{A}^1(\F_q[u]/(\pi))} t_{\mathcal{F}_h}(y) \right| \leq
\sum_{i=0}^2 \left| \text{tr}(\mathrm{Frob}_{|\pi|} , H^i_c(\mathbb{A}^1_{\overline{\F_q[u]/(\pi)}}, \overline{\mathcal{F}_h}) ) \right|
\end{equation}
where $\overline{\mathcal{F}_h}$ is the base change of $\mathcal{F}_h$ to the algebraic closure of $\F_q[u]/(\pi)$.

For $i = 0$ there is no cohomology by the fact that $\mathcal{F}$ has no finitely supported sections, \cref{AS-sheaf-properties}(5), and \cref{general-tensor-product}(2).
For $i = 2$ the cohomology equals the geometric monodromy coinvariants of $\mathcal{F}_h$. 
These vanish for $h=0$ in view of our assumption that the geometric monodromy representation of $\mathcal{F} = \mathcal{F}_0$ does not admit trivial quotients, and also vanish for $h \neq 0$ because $\mathcal{F}$ is infinitame hence its geometric monodromy representation does not have Aritn-Schreier quotients.

Consequently, using \cref{EP2} and \cref{AffineEulerChar} we get
\[\dim H^1_c(\mathbb{A}^1_{\overline{\F_q[u]/(\pi)}}, \overline{\mathcal{F}_h}) = 
-\chi(\mathbb{A}^1_{\overline{\F_q[u]/(\pi)}}, \overline{\mathcal{F}_h}) = 
\swan_{\infty}(\mathcal{F}_h) - \rank(\mathcal{F}_h) + \sum_{x \in | \mathbb{A}^1 |}\cond_x(\mathcal{F}_h)  .\]
From \cref{general-tensor-product}(5) and \cref{conductor-little-lemmas}(5) we get that the above equals 
\begin{equation} \label{FourTermsCounductorEq}
\swan_\infty(\mathcal{F}_h) - \rank(\mathcal F) +
c_F(\mathcal{F}_h)  - \swan'_\infty(\mathcal{F}_h).
\end{equation}

In case $h = 0$ the above reduces to $c(t) - r(t)$ becuase $\mathcal{F}$ is infinitame.
In case $h \neq 0$ we still have $c_F(\mathcal{F}_h) = c_F(\mathcal{F})$ in view of \cref{ComplexesInvarsDef},
and since $\mathcal{F}$ is infinitame, the local monodromy at $\infty$ of $\mathcal{F}_h$ is a direct sum of $\rank(\mathcal F)$ copies of the local monodromy of $\mathcal{L}_{\psi}(hx)$, so \cref{FourTermsCounductorEq} equals
\begin{equation*}
c_F(\mathcal{F}) + \rank(\mathcal{F})\mathrm{slope}_\infty(\mathcal{L}_\psi(hx)) - \rank(\mathcal{F})\max\{\mathrm{slope}_\infty(\mathcal{L}_\psi(hx))-1,0\} - \rank(\mathcal F)
\end{equation*}
where the slopes are taken with respect to the representation of the inertia group $I_\infty$ on the generic fiber.
Since $\mathrm{slope}_\infty(\mathcal{L}_\psi(hx)) = 1$ by \cref{AS-sheaf-properties}(3), the above equals $c(t)$.

Since $\mathcal{F}_h$ is mixed of nonpositive weights by \cref{general-tensor-product}(4), 
each eigenvalue of $\mathrm{Frob}_{|\pi|}$ acting on $ H^1_c(\mathbb{A}^1_{\overline{\F_q[u]/(\pi)}}, \overline{\mathcal{F}_h})$ is of absolute value at most $|\pi|^{\frac{1}{2}}$ by Deligne's bound,
so \cref{TriangleGrothendieckLefschetzEq} is bounded by
\begin{equation}
\dim H^1_c(\mathbb{A}^1_{\overline{\F_q[u]/(\pi)}}, \overline{\mathcal{F}_h}) |\pi|^{\frac{1}{2}} \leq c(t) |\pi|^{\frac{1}{2}}.
\end{equation}
It follows from \cref{EquidistributionInResidueClassesEq} that our original sum is bounded by
\begin{equation}
q^n|\pi|^{-1}c(t)|\pi|^{\frac{1}{2}} = c(t)q^n|\pi|^{-\frac{1}{2}}
\end{equation}
as required.
\end{proof}

\begin{cor} \label{PolyaVinogradovLem}

With assumptions as above, for an integer $0 \leq d < \deg(\pi)$ we have 
\begin{equation}
\left| \sum_{\substack{f \in \F_q[u] \\ \deg(f) < d}} t(f) \right| \leq c(t)|\pi|^{\frac{1}{2}}.
\end{equation}

\end{cor}

\begin{proof}

By \cref{IndicatorAdditiveCharsProp} and \cref{CompleteExponentialInfinitameSum} we have
\begin{equation*}
\begin{split}
\left| \sum_{\substack{f \in \F_q[u] \\ \deg(f) < d}} t(f) \right| &= 
\frac{q^d}{|\pi|} \left| \sum_{\substack{h \in \F_q[u] \\ \deg(h) <\deg(\pi)-d }} \ \sum_{\substack{f \in \F_q[u] \\ \deg(f) < \deg(\pi)}} t(f) e \left( \frac{hf}{\pi} \right) \right| \\
&\leq \sup_{\substack{h \in \F_q[u] \\ \deg(h) <\deg(\pi)-d }} \left|  \sum_{\substack{f \in \F_q[u] \\ \deg(f) < \deg(\pi)}} t(f) e \left( \frac{hf}{\pi} \right) \right| \leq c(t)q^{\deg(\pi)}|\pi|^{-\frac{1}{2}}
= c(t)|\pi|^{\frac{1}{2}}.
\end{split}
\end{equation*}

\end{proof}

We shall now deduce \cref{InfinitameVSprimesCor}.

\begin{proof}

The identity $\Lambda = (\mu \cdot \deg) * (-1)$ gives
\begin{equation} \label{ConvolutedMangoldtTraceEq}
\sum_{f \in \mathcal{M}_n} t(f) \Lambda(f) =
-\sum_{k=1}^n k \sum_{A \in \mathcal{M}_k} \mu(A) \sum_{B \in \mathcal{M}_{n-k}} t(AB).
\end{equation} 

For any $k \leq \zeta n$ and any $A \in \mathcal{M}_k$ that is not divisible by $\pi$, the contribution to \cref{ConvolutedMangoldtTraceEq} is $\ll$
\begin{equation}
n \sum_{\substack{C \in \F_q[u] \\ \deg(C) < n-k}} t(AC + AT^{n-k})
\end{equation}
where $C = B - T^{n-k}$.
Since $\pi \nmid A$, the function $C \mapsto t(AC + AT^{n-k})$ satisfies the hypothesis of \cref{CompleteExponentialInfinitameSum} and \cref{PolyaVinogradovLem},
so the above is bounded by
\begin{equation}
n c(t) |\pi|^{\frac{1}{2}} \left(1 + \frac{q^{n-k}}{|\pi|} \right).
\end{equation}
The contribution from all such $k$ and $A$ is thus $\ll$ 
\begin{equation} \label{LowKrangeBoundTraceEq}
\max_{k \leq \zeta n} n^2 c(t)   |\pi|^{\frac{1}{2}} |\mathcal{M}_k| \left(1 + \frac{q^{n-k}}{|\pi|} \right) \ll
c(t)   |\pi|^{\frac{1}{2}} \left(q^{n(\zeta + \epsilon)} + \frac{q^{n(1 + \epsilon)}}{|\pi|} \right) 
\end{equation}
for any $\epsilon > 0$.

The contribution to \cref{ConvolutedMangoldtTraceEq} of all $\deg(\pi) \leq k \leq \zeta n$ and all $A \in \mathcal{M}_k$ that are divisible by $\pi$ is $\ll$
\begin{equation*}
\begin{split}
\max_{k \geq \deg(\pi)} n^2 \left| \sum_{\substack{A \in \mathcal{M}_k \\ \pi \mid A}} \mu(A) \sum_{B \in \mathcal{M}_{n-k}} t(0)\right| &=
\max_{k \geq \deg(\pi)} n^2 q^{n-k} |t(0)| \left| \sum_{\substack{A \in \mathcal{M}_k \\ \pi \mid A}} \mu(A) \right| \\
&\ll \max_{k \geq \deg(\pi)} n^2q^{n-k}r(t) \ll n^2q^{n-\deg(\pi)}r(t)
\end{split}
\end{equation*} 
in view of \cref{MobiusProgressionFormula}.

The contribution of all $k \geq \zeta n$ to \cref{ConvolutedMangoldtTraceEq} is $\ll$
\begin{equation*}
\max_{k \geq \zeta n} n^2 \left| \sum_{B \in \mathcal{M}_{n-k}} \sum_{A \in \mathcal{M}_k} \mu(A)  t(AB) \right| \ll
\max_{\substack{k \geq \zeta n \\ B \in \mathcal{M}_{n-k}}} n^2q^{n-k} \left| \sum_{A \in \mathcal{M}_k} \mu(A) t(AB) \right|
\end{equation*} 
and by \cref{LinearMobiusVStraceFunctionThm} this is $\ll$
\begin{equation}
\begin{split}
&n^2 \max_{k \geq \zeta n} q^{n-k} |\mathcal{M}_k|^{1 - \frac{1}{2p} + \frac{\log_q(2ep)}{p}}|\pi|^{\log_q \left( r(t) \left( 1 + \frac{1}{2p} \right) + \frac{c(t)}{2p} \right)} \ll \\
&\max_{k \geq \zeta n} q^{n - \frac{k}{2p} + \frac{k\log_q(2ep)}{p}+ \epsilon n}|\pi|^{\log_q \left( r(t) \left( 1 + \frac{1}{2p} \right) + \frac{c(t)}{2p} \right)}
\end{split}
\end{equation}
for any $\epsilon > 0$. 
The above is largest once $k = \zeta n$ so we have the bound
\begin{equation} \label{HighKrangeBoundTraceEq}
q^{n - \frac{\zeta n}{2p} + \frac{\zeta n \log_q(2ep)}{p}+ \epsilon n}|\pi|^{\log_q \left( r(t) \left( 1 + \frac{1}{2p} \right) + \frac{c(t)}{2p} \right)}
\end{equation}

It follows from the choice of $\zeta$ in \cref{TheZetaConstantTraceEq}, and our assumption on $n$ that
\begin{equation}
\begin{split}
c(t)|\pi|^{\frac{1}{2}}q^{n(\zeta + \epsilon)} &= c(t) q^{n(\zeta + \epsilon) + \frac{1}{2}\deg(\pi)} \leq 
c(t) q^{n(\frac{1}{1 + 2\delta} + \zeta + \epsilon)} \\
&\ll q^{n(1 - \frac{\zeta}{2p} + \frac{\zeta  \log_q(2ep)}{p}+ \epsilon )}|\pi|^{\log_q \left( r(t) \left( 1 + \frac{1}{2p} \right) + \frac{c(t)}{2p} \right)}
\end{split}
\end{equation}
so the bound in \cref{HighKrangeBoundTraceEq} dominates the first summand in \cref{LowKrangeBoundTraceEq}, 
hence we can use
\begin{equation}
q^{n(1 - \frac{\zeta}{2p} + \frac{\zeta  \log_q(2ep)}{p}+ \epsilon )}|\pi|^{\log_q \left( r(t) \left( 1 + \frac{1}{2p} \right) + \frac{c(t)}{2p} \right)} + (c(t) + r(t))\frac{q^{n(1 + \epsilon)}}{|\pi|^{\frac{1}{2}}}
\end{equation}
as a final bound.
\end{proof}

\section{Acknowledgments} 

We are grateful to de Jong for a very helpful discussion of the proof of \cref{ccg-like}.

\end{document}